\documentclass[reqno,10pt,a4paper]{amsart}

\usepackage[a4paper,left=24mm,right=24mm,top=27mm,bottom=30mm]{geometry}
\usepackage{amsmath,amssymb,mathrsfs}
\usepackage{graphicx}
\usepackage{enumerate,color,xcolor}
\usepackage{microtype}
\usepackage{showkeys}

\usepackage[normalem]{ulem}
\usepackage[colorlinks,
            linkcolor=red,
            anchorcolor=blue,
            citecolor=green
            ]{hyperref}
\newtheorem{theorem}{Theorem}[section]
\newtheorem{proposition}[theorem]{Proposition}
\newtheorem{lemma}[theorem]{Lemma}
\newtheorem{corollary}[theorem]{Corollary}
\newtheorem{definition}[theorem]{Definition}
\newtheorem{remark}[theorem]{Remark}

\numberwithin{equation}{section}
\allowdisplaybreaks[2]

\newcommand{\beq}{\begin{equation}}
\newcommand{\eeq}{\end{equation}}
\newcommand{\ben}{\begin{eqnarray}}
\newcommand{\een}{\end{eqnarray}}
\newcommand{\beno}{\begin{eqnarray*}}
\newcommand{\eeno}{\end{eqnarray*}}
\newcommand{\normm}[1]{{ \vert\kern-0.25ex \vert\kern-0.25ex \vert #1
		\vert\kern-0.25ex \vert\kern-0.25ex \vert}}
\let\f=\frac

\newcommand{\be}{\begin{equation} \label}
	\newcommand{\ee}{\end{equation}}
\newcommand{\bea}{\begin{eqnarray}\label}
	\newcommand{\eea}{\end{eqnarray}}
\newcommand{\bas}{\begin{eqnarray*}}
	\newcommand{\eas}{\end{eqnarray*}}
\newcommand{\bit}{\begin{itemize}}
	\newcommand{\eit}{\end{itemize}}

\newcommand{\N}{{\mathbb N}}

\newcommand{\R}{{\mathbb R}}

\newcommand{\pa}{\partial}
\newcommand{\eps}{\varepsilon}

\newcommand{\ba}{\begin{aligned}}
\newcommand{\ea}{\end{aligned}}
 \def\na{\nabla}

\let\pa=\partial

\let\lam=\lambda

\let\f=\frac

\let\D=\Delta
\let\Lam=\Lambda
\let\Om=\Omega

\def\pa{\partial}

\def\virgp{\raise 2pt\hbox{,}}
\def\cdotpv{\raise 2pt\hbox{;}}

\def\C{\mathop{\mathbb C\kern 0pt}\nolimits}
\def\DD{\mathop{\mathbb D\kern 0pt}\nolimits}
\def\EE{\mathop{{\mathbb E \kern 0pt}}\nolimits}
\def\K{\mathop{\mathbb K\kern 0pt}\nolimits}
\def\N{\mathop{\mathbb N\kern 0pt}\nolimits}
\def\Q{\mathop{\mathbb Q\kern 0pt}\nolimits}
\def\R{\mathop{\mathbb R\kern 0pt}\nolimits}
\def\SS{\mathop{\mathbb S\kern 0pt}\nolimits}

\def\<{\langle}
\def\>{\rangle}

\def\ga{\gamma}
\def\al{\alpha}
\def\be{\beta}
\def\de{\delta}

\def\gs{\gtrsim}
\def\ls{\lesssim}
\def\S{\mathbb{S}}

\begin{document}

\title[Boltzmann equation with radial anharmonic potentials]{Equilibria and linear stability for the Boltzmann equation with radial anharmonic confining potentials}

\author{Ling-Bing He}
\address[Ling-Bing He]{Department of Mathematical Sciences, Tsinghua University, Beijing 100084, P.R. China}
\email{hlb@tsinghua.edu.cn}
\author{Jie Ji}
\address[J. Ji]{School of Mathematics, Nanjing University  of Aeronautics and Astronautics\\
	Nanjing 211106,  P.R. China.}
\email{jij\_24@nuaa.edu.cn}
\author{Wu-Wei Li}
\address[Wu-Wei Li]{Department of Mathematical Sciences, Tsinghua University, Beijing 100084, P.R. China}
\email{ww-li20@mails.tsinghua.edu.cn}

\begin{abstract}
We study the Boltzmann equation in the whole space under the radial anharmonic confining potentials $\Phi(x)=|x|^p/p$ for $p>2$ and $\Phi(x)=\langle x\rangle^p/p$ for $1<p<2$. We first classify all positive finite-mass-and-energy entropy-invariant, equivalently zero-entropy-production, solutions. For $p>2$, the nonlinear equilibrium manifold is parametrized by mass, temperature, and the three components of angular momentum; for $1<p<2$, integrability excludes rotating equilibria and only mass and energy remain as equilibrium parameters. We then identify the five-dimensional stationary space of the equation linearized about an arbitrary equilibrium and construct an explicit projection determined by the conserved moments. After normalization, the collision term takes the form $C_Me^{-\widetilde\Phi(x)}\mathsf L$, so its microscopic coercivity degenerates at spatial infinity. A far-field weight-transfer estimate compensates for this degeneracy. After subtracting the stationary projection, the corresponding semigroup solution converges algebraically in exponentially weighted $L^2$ spaces. The rate is governed by the growth exponent $p$ and the gap between the two weights, up to an arbitrarily small loss. For $1<p<2$, the mismatch between the two-dimensional nonlinear equilibrium manifold and the five-dimensional linear stationary space yields a conditional obstruction to nonlinear asymptotic attraction for perturbations carrying nonzero angular momentum.
\end{abstract}

\keywords{Boltzmann equation, confining potential, equilibrium, angular momentum, hypocoercivity, algebraic decay}

\subjclass[2020]{35Q20, 35B35, 35B40, 82C40}

\hypersetup{
  pdftitle={Equilibria and linear stability for the Boltzmann equation with radial anharmonic confining potentials},
  pdfauthor={Ling-Bing He, Jie Ji, and Wu-Wei Li}
}

\maketitle
\setcounter{tocdepth}{1}
\tableofcontents
%%%%%%%%%%%%%%%%%%%%%%%%%%%%%%%%%%%%%%%%%%%%%%%%%%
	
\section{Introduction}
External confinement changes both the invariant states and the relaxation mechanism of an inhomogeneous kinetic equation. For the Boltzmann equation in the whole space, the isotropic harmonic potential is exceptional: its Hamiltonian flow supports nontrivial time-periodic Maxwellian modes in addition to the static Maxwell--Boltzmann state. This phenomenon goes back to Boltzmann \cite{Boltzmann_1876}; a modern classification of the corresponding special macroscopic modes is given in \cite{carrapatoso_special_2024}, and a nonlinear stability theory for the Landau equation with harmonic confinement was developed in \cite{cao_landau_2024}.

The purpose of this paper is to understand what remains of this structure when the quadratic confinement is replaced by a radial anharmonic potential. We consider
\[
\Phi(x)=\frac1p|x|^p,\qquad p>2,
\qquad\text{and}\qquad
\Phi(x)=\frac1p\langle x\rangle^p,\qquad 1<p<2.
\]
The two regimes have sharply different equilibrium geometries. Superquadratic confinement is strong enough to make rotating Maxwellians integrable, whereas subquadratic confinement excludes every nonzero rotation. At the linear level, however, radial symmetry still produces three stationary angular-momentum modes in both cases. A distinction from the uniformly coercive linear models studied in \cite{carrapatoso_special_2024} is that linearization about the full equilibrium introduces a spatial multiplier comparable to $e^{-\Phi(x)}$ in front of the collision operator. The microscopic coercivity therefore degenerates at infinity, and exponential hypocoercive relaxation is replaced by a quantitative algebraic decay mechanism.

\subsection{The Boltzmann equation}
	Consider the Boltzmann equation 
		\begin{equation}\label{Boltzmann}
		\partial_tF + v \cdot \nabla_xF - \nabla_x\Phi \cdot \nabla_vF = Q(F, F).
	\end{equation}
	Here $F(t,x,v)\ge0$ is the particle distribution at time $t\ge0$, position $x\in\mathbb R^3$, and velocity $v\in\mathbb R^3$. In \eqref{Boltzmann}, $-\nabla_x\Phi$ is the acceleration generated by the external confining potential $\Phi=\Phi(x)$.
	The Boltzmann collision operator $Q$ acts only on the variable $v$, and is defined as
	\begin{equation*}
		Q(F, G) = \int_{\mathbb{R}^3 \times \mathbb{S}^2}B(v - v_*, \sigma)[F'_*G' - F_*G]d\sigma dv_*,
	\end{equation*}
	where we use the short hand
	\begin{equation*}
		G = G(t, x, v), \:\:\: G' = G(t, x, v'), \:\:\: F_* = F(t, x, v_*), \:\:\: F'_* = F(t, x, v'_*),
	\end{equation*}
	and $v', v'_*$ are defined in terms of $ v, v_*,\sigma$ by
	\begin{equation}\label{v'v'*}
		v' = \frac{v + v_*}{2} + \frac{|v - v_*|}{2}\sigma, \:\:\: v'_* = \frac{v + v_*}{2} - \frac{|v - v_*|}{2}\sigma,\quad\sigma\in\S^2.
	\end{equation}
	This follows the conservation law of momentum and energy of the collision:
	\begin{equation}\label{conserv}
		v' + v'_* = v + v_*, \:\:\: |v'|^2 + |v'_*|^2 = |v|^2 + |v_*|^2.
	\end{equation}
	We impose the following assumptions on the Boltzmann collision kernel $B=B(v-v_*,\sigma)$:
	\begin{itemize}
\item[$\mathbf{(A1)}$] The Boltzmann kernel $B$ has the product form
		\begin{equation*}
		B(v-v_*,\sigma)=|v-v_*|^\gamma b(\cos\theta),
		\qquad \gamma\in(-3,1],
		\qquad \cos\theta=\frac{v-v_*}{|v-v_*|}\cdot\sigma.
		\end{equation*}
	\item[$\mathbf{(A2)}$] The nonnegative angular kernel satisfies either Grad's cutoff condition
	\[
	\int_{\mathbb S^2}b(\widehat z\cdot\sigma)\,d\sigma<\infty,
	\qquad \widehat z\in\mathbb S^2,
	\]
	or the non-cutoff condition
	\[
	c_b\theta^{-2-2s}\le b(\cos\theta)\le C_b\theta^{-2-2s},
	\qquad 0<\theta\le\frac{\pi}{2},
	\]
	for some $s\in(0,1)$ and constants $0<c_b\le C_b<\infty$.
		\item[$\mathbf{(A3)}$] Without loss of generality, we assume that $b(\cos\theta)$ is supported in $0\le\theta\le\pi/2$, equivalently $(v-v_*)\cdot\sigma\ge0$. Otherwise, $b$ is replaced by its symmetrized form
	\begin{equation*}
		\overline b(\widehat z\cdot\sigma)
		:=\big[b(\widehat z\cdot\sigma)+b(-\widehat z\cdot\sigma)\big]
		\mathbf 1_{\{\widehat z\cdot\sigma\ge0\}},
		\qquad \widehat z=\frac{v-v_*}{|v-v_*|},
	\end{equation*}
	where $\mathbf{1}_A$ is the characteristic function of the set $A$.
	\end{itemize}
\begin{remark}
For long-range interactions generated by an inverse power law between particles, let $q>2$ denote the exponent of the intermolecular potential. Then $\gamma=(q-5)/(q-1)$ and $s=1/(q-1)$, so that $\gamma+4s=1$. We use $q$ here to avoid confusion with the exponent $p$ of the external confining potential.
The Grad's cutoff assumption refers to the hypothesis that the interaction between particles is short-ranged. This assumption simplifies the mathematical analysis by avoiding the complications associated with long-range interactions and singularities in the collision operator.
\end{remark}
%%%%%%%%%%%%%%%%%%%%%%%%%%%%%%%%%%%%%%%%%%%%%%%%%%

\subsection{Entropy-invariant solutions}
The equilibrium classification is naturally tied to Boltzmann's \(H\)-theorem. If the entropy $\mathscr{H}(F)$ and the entropy dissipation $\mathscr{D}(F)$ are defined by
	\begin{equation}\label{entropy}
		\mathscr{H}(F)(t):=\int_{\R^6}F\log Fdvdx,\quad
		\mathscr{D}(F)(t):=-\int_{\R^6}Q(F,F)\log F dvdx,
	\end{equation} then the $H$-theorem states that
	\begin{equation}\label{H-theorem}
		\f d{dt}\mathscr{H}(F)(t)+\mathscr{D}(F)(t)=0.
	\end{equation}
Whenever the entropy identity is justified, a constant-entropy trajectory has zero entropy production. For a positive solution, $\mathscr D(F)\equiv0$ is equivalent to $Q(F,F)\equiv0$, and the equation then reduces to the Hamiltonian transport equation
	\ben\label{transportequation}
	\partial_tF + v \cdot \nabla_xF - \nabla_x\Phi \cdot \nabla_vF = 0.
	\een
This motivates the following terminology.
\begin{definition}\label{Defoentropy-invariant}
	A sufficiently regular positive function $M=M(t,x,v)>0$ is called an \emph{entropy-invariant solution}, or equivalently a \emph{zero-entropy-production solution}, of \eqref{Boltzmann} if
\ben\label{integcondi}\int_{\R^6}(\Phi(x)+ |v|^2 + 1)M(t, x, v)dxdv < +\infty,\een
and, for every $t\ge0$,
	\begin{equation}\label{equationforeisol}
		\partial_tM + v \cdot \nabla_xM - \na_x\Phi \cdot \nabla_vM = 0 =Q(M, M).
	\end{equation}
\end{definition}
\begin{remark}
If $\mathscr H(M)$ is finite and the $H$-theorem is valid for $M$, Definition \ref{Defoentropy-invariant} is equivalent to constancy of the entropy. We use the definition above because the classification argument only requires the two identities in \eqref{equationforeisol}. The integrability condition \eqref{integcondi} guarantees finite mass and energy. Angular momentum is an additional conserved moment whenever $\int_{\R^6}\|x\wedge v\|_2M\,dx\,dv<\infty$; it is not part of the definition.
\end{remark}

It follows that if \( Q(M, M) = 0 \), then \( M \) must be a local Maxwellian, that is,
\[
M = A \exp\left\{ -\frac{\beta}{2} |v - u|^2 \right\},
\]
where \( A \), \( \beta \) and the average velocity \( u \) are functions of \( x \) and \( t \). Since the local Maxwellian \( M \) also satisfies the first equation in \eqref{equationforeisol}, it can be rigorously shown for a sufficiently general external potential \( \Phi(x) \) that \( M \) should be the Maxwell-Boltzmann distribution, i.e.,
\ben\label{Maxwell-Boltzmann distrubution} M=A_0\exp\bigg\{-\beta\bigg(\f{|v|^2}2+\Phi(x)\bigg)\bigg\}, \een
where $A_0$ and $\beta$ are constants independent of $x$ and $t$, determined by the mass and total energy. For the isotropic harmonic potential, additional time-periodic Maxwellian solutions exist; see \cite{Boltzmann_1876,carrapatoso_special_2024,cao_landau_2024}. This raises the corresponding classification problem for radial nonquadratic potentials.

\subsection{Conservation laws and the conserved-moment map}
Throughout this paper, we focus on the following external potentials:
\begin{equation}\label{potential in this work}
	\:\:\: \Phi(x) = \frac{1}{p}|x|^p, \:\:\: \text{with} ~\: p \in (2, \infty);\quad\Phi(x) = \frac{1}{p}\<x\>^{p}, \:\:\: \text{with}~ \: p \in (1, 2) ,
\end{equation}
where $\<\cdot\>:=(1+|\cdot|^2)^{1/2}$. For $1<p<2$, the regularization $\langle x\rangle^p$ avoids the singularity of the second derivatives of $|x|^p$ at the origin. Both potentials are radial, and hence the Hamiltonian transport preserves angular momentum in addition to mass and total energy.
For these potentials, the following conservation laws hold:
	\begin{lemma}\label{CLaw}
		Let $F = F(t, x, v)$ be a sufficiently regular global solution to \eqref{Boltzmann} with finite mass and energy. Then, for all $t>0$,
		\begin{equation*}
			\frac{d}{dt}\int_{\R^6}(1, \Phi(x) + \frac{1}{2}|v|^2)F(t, x, v)dxdv = 0.
		\end{equation*}
		If in addition $\int_{\R^6}\|x\wedge v\|_2F(t,x,v)\,dx\,dv<\infty$, then
		\begin{equation*}
			\frac{d}{dt}\int_{\R^6}(x\wedge v)F(t,x,v)\,dx\,dv=0,
		\end{equation*}
		where $a \wedge b = ab^\tau - ba^\tau=(a_ib_j-a_jb_i)_{3\times 3}$ for $a,b\in\R^3$.

%		\noindent$\bullet$ H-Theorem:
%		\begin{equation*}
%			\frac{d}{dt}\mathscr{H}[F] + \mathscr{D}[F] = 0,
%		\end{equation*}
%		where
%		\begin{equation*}
%			\mathscr{H}[F] = \int F\ln Fdxdv; \:\:\: \mathscr{D}[F] = \int Q(F, F)\ln Fdxdv \ge 0.
%		\end{equation*}
	\end{lemma}
	The identities follow by testing the equation against the collision invariants $1$ and $|v|^2/2$, and against the transport invariants $\Phi(x)+|v|^2/2$ and $x\wedge v$. The last invariant is a consequence of rotational symmetry: $x\wedge\nabla\Phi(x)=0$ for a radial potential. This motivates the following conserved-moment map.
\begin{definition}Let $F=F(t,x,v)$ be a global solution to \eqref{Boltzmann} with initial data $F_0=F_0(x,v)$ and finite angular momentum. $\Psi$ is called the conserved-moment map if
\begin{equation}\label{ConservedMap}
	\Psi(F(t)) := (\Psi_1(F)(t),\Psi_2(F)(t),\Psi_3(F)(t)):=\int_{\R^6}(1, \Phi(x) + \frac{1}{2}|v|^2, x \wedge v)F(t, x, v)dxdv.
\end{equation}
The conservation laws give $\Psi(F(t))=\Psi(F_0)$ for all $t>0$; we therefore write simply $\Psi(F)$ or $\Psi(F_0)$ when no time dependence needs to be emphasized.
\end{definition}
\begin{remark}
The components $\Psi_1(F)$, $\Psi_2(F)$, and $\Psi_3(F)$ represent mass, total energy, and angular momentum, respectively. The three independent entries of the skew-symmetric matrix $x\wedge v$ are equivalent, up to the standard identification, to the components of the vector $x\times v$.
\end{remark}

\subsection{Related works}

The stability and convergence rate to the equilibrium for \eqref{Boltzmann} is an important and old problem. Given the vast literature, we will briefly review previous works closely related to our research.

$\bullet$ In the absence of an external potential, prior research has explored diverse methods to study global dynamics. Within the framework of perturbation theory, we refer readers to \cite{alexandre_global_2011, guo_landau_2002, guo_boltzmann_2004, gressman_global_2011, liu_energy_2004,duan_global_2021} and the references therein, which investigate micro-macro decompositions inspired by Grad's 13 moments method (see \cite{grad_boltzmanns_1965}). Additionally, \cite{desvillettes_trend_2005} introduces the entropy-production method, which involves analyzing an appropriate set of ordinary differential systems. Meanwhile, \cite{gualdani_factorization_2017, villani_hypocoercivity_2009} examines the hypocoercivity method in exponentially weighted spaces and discusses the theory of space enlargement in polynomially weighted spaces (see also related developments in \cite{cao_propagation_2024-1, carrapatoso_landau_2017}). When the spatial variable \( x \) is in \(\mathbb{R}^3\), we refer readers to \cite{bardos_global_2016, chaturvedi_stability_2021, luk_stability_2019} and the references therein for stability results, which are enabled by the dispersion effect.

%Without an external potential, previous studies have focused on various approaches to the global dynamics. In the context of perturbation theory, we direct readers to \cite{alexandre_global_2011,guo_landau_2002,guo_boltzmann_2004,gressman_global_2011,liu_energy_2004} and references therein, which delve into the micro-macro decomposition inspired by Grad’s 13 moments method (see \cite{grad_boltzmanns_1965}). Additionally, \cite{desvillettes_trend_2005} presents the entropy-production method, analyzing a suitable set of ordinary differential systems, while \cite{gualdani_factorization_2017,villani_hypocoercivity_2009} explores the hypocoercivity method in exponentially weighted spaces and discusses the theory of space enlargement in polynomially weighted spaces (see also developments in \cite{cao_propagation_2024-1,carrapatoso_landau_2017}). When the spatial variable $x$ is in $\R^3$, we refer readers to \cite{bardos_global_2016, chaturvedi_stability_2021, luk_stability_2019}and references therein for the stability  results thanks to the dispersion effect.

$\bullet$ In the presence of an external potential, significant attention has been devoted to the linearized version of \eqref{Boltzmann}. We define the Maxwell-Boltzmann distribution and the Maxwell distribution as follows:
\begin{equation}\label{MBd}
	\mathcal{M} := (2\pi)^{-\frac{3}{2}} e^{-(\Phi(x) + \frac{1}{2}|v|^2)}, \quad \mu := (2\pi)^{-\frac{3}{2}} e^{-\frac{1}{2}|v|^2}.
\end{equation}
For the linear stability analysis of \(\mathcal{M}\), we refer to \cite{bosi_bgk_2009} for insights into the BGK model, and to \cite{tabata_decay_1993, tabata_decay_1994, desvillettes_trend_2001, dolbeault_hypocoercivity_2015, helffer_hypoelliptic_2005, herau_isotropic_2004, carrapatoso_special_2024} and the references therein for advancements in the Fokker-Planck and linear Boltzmann models. It is important to note that these linear stability studies are conducted within the following frameworks:
\begin{equation*}
	\begin{aligned}
			\pa_t F+v\cdot\nabla_xF-x\cdot\nabla_vF=\mathcal{M}^{-\f12}(Q(\mathcal{M},\mathcal{M}^{\f12}F)+Q(\mathcal{M}^{\f12}F,\mathcal{M}))	\end{aligned}
\end{equation*}
or 
\begin{equation}\label{Linearinkineticeqs}
	\begin{aligned}
		\pa_t F+v\cdot\nabla_xF-x\cdot\nabla_vF= Q(\mu,F)+Q(F,\mu).
	\end{aligned}
\end{equation}
$\bullet$ The work most directly related to our linear analysis is Carrapatoso et al.\ \cite{carrapatoso_special_2024}. In an abstract class of linear kinetic equations with local conservation laws, they classify all special macroscopic modes and prove quantitative exponential convergence toward those modes. Their argument provides both a micro--macro route and a commutator route to hypocoercivity. We use the same classification perspective and part of the micro--macro architecture, in particular the separation of Poincar\'e--Korn-coercive fields from a finite-dimensional family of modes. These structural ingredients are part of the framework developed in \cite{carrapatoso_weighted_2022,carrapatoso_special_2024} and are not claimed here as new.

There are, however, two precise differences. First, after a general equilibrium is normalized, the linearized equation in this paper contains
\begin{equation}\label{IntroDegenerateCollision}
 C_Me^{-\widetilde\Phi(x)}\mathsf L
\end{equation}
instead of a collision operator whose microscopic coercivity is uniform in $x$. The prefactor in \eqref{IntroDegenerateCollision} tends to zero at spatial infinity. Appendix B of \cite{carrapatoso_special_2024} allows a weaker velocity-space coercivity estimate, but the collision operator there still has no confining prefactor that vanishes with $x$; it therefore does not directly cover \eqref{IntroDegenerateCollision}. The new analytical input of the present paper is the far-field weight-transfer mechanism that compensates for this spatial loss of dissipation and leads to algebraic, rather than exponential, relaxation.

Second, the nonlinear equilibrium statement must be distinguished from the linear classification of special modes. Appendix C.5 of \cite{carrapatoso_special_2024} already observes that suitable special macroscopic modes can be exponentiated to produce solutions of the nonlinear Boltzmann equation. Our classification result is complementary: it proves that, for the two anharmonic potentials considered here, every positive finite-mass-and-energy zero-entropy-production solution is of this form, determines exactly which formal rotations are integrable, and parametrizes the admissible equilibria by their conserved moments. More explicitly, Proposition \ref{propcha} gives
\[
\log M\in\operatorname{span}\left\{1,\Phi(x)+\frac12|v|^2,x\wedge v\right\},
\]
whereas the corresponding linear stationary modes satisfy
\begin{equation}\label{coef1}
\frac{f}{\mathcal M}\in\operatorname{span}\left\{1,\Phi(x)+\frac12|v|^2,x\wedge v\right\}.
\end{equation}
Thus the finite-dimensional formulas are deliberately consistent with \cite{carrapatoso_special_2024}; the completeness, the subquadratic/superquadratic integrability dichotomy, the moment parametrization, and the spatially degenerate decay theory are the points specific to the present work.

The harmonic-potential Landau equation was studied nonlinearly in \cite{cao_landau_2024}, including the classification and stability of time-periodic Maxwell--Boltzmann states. That work also identifies the high-order energy estimates required by a nonlinear theory. In the anharmonic setting considered here, uniform-in-time propagation of the corresponding high-order Sobolev norms is already open for the linearized flow. We therefore restrict the present stability result to the linear equation.

\medskip

\subsection{Main difficulties and contributions}
It is useful to separate the framework adapted from earlier hypocoercivity arguments from the inputs specific to the present equation. The macro--micro decomposition, the use of weighted Poincar\'e and Poincar\'e--Korn inequalities, and the isolation of finitely many macroscopic modes follow the general strategy of \cite{carrapatoso_weighted_2022,carrapatoso_special_2024}. In this paper these tools must be reorganized after normalization around an arbitrary, possibly rotating, equilibrium and combined with a new estimate for a collision strength that vanishes at spatial infinity.

The principal contributions are the following.

\underline{(1). Complete equilibrium classification and moment parametrization.}
We classify all positive finite-mass-and-energy entropy-invariant solutions, not merely a prescribed family of Maxwellian ansatzes. For $p>2$, rotating equilibria are integrable and form a five-parameter family; for $1<p<2$, every nonzero rotation destroys spatial integrability and the equilibrium family is two-dimensional. We then prove the corresponding parametrization by mass, energy, and, when admissible, angular momentum.

\underline{(2). Stationary projection around an arbitrary equilibrium.}
The full linearized generator has five stationary directions in both regimes. The Gram matrix gives an explicit projection onto these directions for arbitrary initial moments. Its role is finite-dimensional bookkeeping rather than a new hypocoercive principle. The substantive step is to normalize a general rotating equilibrium by an invertible change of variables while tracking the equation, all five conserved moments, and the exponential norms.

\underline{(3). Spatially degenerate coercivity and algebraic relaxation.}
For the static Maxwell--Boltzmann state, the normalized equation has the schematic form
\ben\label{pertubofcauchylandau}
\partial_tf+v\cdot\nabla_xf-\nabla_x\Phi\cdot\nabla_vf
=e^{-\Phi(x)}\mathsf Lf,
\qquad
\mathsf Lf=Q(\mu,f)+Q(f,\mu).
\een
Thus the standard velocity-space coercivity of $\mathsf L$ is multiplied by a coefficient tending to zero as $|x|\to\infty$. At the macroscopic level, thirteen local coefficients must be controlled using only five global conservation laws. We adapt the weighted Poincar\'e--Korn closure to the transformed moment system, use evolution equations to control the remaining finite-dimensional coefficients, and remove the final five directions by the conserved moments. The far-field weight-transfer inequalities of Section 4 then exchange part of the velocity weight for confinement weight along the Hamiltonian flow. Interpolation between two exponential weights yields the algebraic rate in Theorem \ref{linearstabilityp<2}, up to an arbitrarily small loss. No sharpness or matching lower bound is claimed.

\underline{(4). A nonlinear geometric consequence in the subquadratic regime.}
For $1<p<2$, the two-dimensional nonlinear equilibrium manifold does not contain the three angular-momentum directions that remain stationary for the linearized equation. We formulate this mismatch as a quantitative failure of asymptotic attraction for any global nonlinear trajectory carrying nonzero angular momentum in a topology that controls that moment. This is an obstruction theorem, not a claim of nonlinear Lyapunov instability or a construction of global nonlinear solutions.

\subsection{Notation and functional setting}
 $(i)$  We use the notations $a\ls b(a\gs b)$ and $a\ls_c b(a\gs_c b)$ to indicate that there is a constant $C$ which is uniform or depends on parameter $c$ and may be different on different lines, such that $a\leq Cb(a\geq Cb)$. We use the notation $a\sim b$($a\sim_c b$) whenever $a\ls b$ and $b\ls a$($a\ls_c b$ and $a\gs_cb$).
 
 $(ii)$ We denote $C_{a_1,a_2,\cdots,a_n}$(or $C(a_1,a_2,\cdots,a_n)$) by a constant depending on parameters $a_1,a_2,\cdots,a_n$. Moreover,  parameter $\varepsilon$ is used  to represent different positive numbers much less than 1 and determined in different cases.
 
 $(iii).$ The Japanese bracket is defined as $\<x\>:=(1+|x|^2)^{1/2}$. $\R^+$ is used to denote the set $\{x\ge0|x\in\R\}$ and we also specify the set $\R^+_{>0}:=\{x>0|x\in\R\}$.
 $\mathbf{1}_\Om$ is the characteristic function of the set $\Om$. We use $(\cdot,\cdot)_{\R^3},(\cdot,\cdot)_{L^2_v}$, and $(\cdot,\cdot)_{L^2_{x, v}}$ to denote the inner product in $\R^3$, $L^2_{\R^3_v}$, and $L^2_{\R^3_v\times\R^3_x}$, respectively. Occasionally, we employ $(\cdot,\cdot)$ to represent the inner product briefly without ambiguity.
 
 $(v).$ %$\mathbb{M}_3(\R)$ represents the set of $3\times3$ matrices with all components in $\R$. 
 We use $\mathbb{I}$ to represent the unit matrix or identity operator. $\mathbb{I}_{3\times 3}$ is used to emphasize that it is a $3\times3$ unit matrix.  If $A=(a_{ij})_{m\times n}$, then $\|A\|_2^2:=\sum_{i=1}^m\sum_{j=1}^n |a_{ij}|^2$ and $|A|:=\sqrt{A^\tau A}$. Here $A^\tau$ denotes the transpose of $A$.  If $T,S$ are two operators, then  the commutator $[T,S]$ is defined by $[T,S]=TS-ST$.

Other notations will be provided as required throughout the paper.

%\subsubsection{Function spaces} 

\subsection{Main results}
\subsubsection{Equilibria and conserved-moment parametrization}
Our first result gives the complete classification of positive entropy-invariant solutions.
	\begin{theorem}\label{EISform}
		Let $M=M(t,x,v)>0$ be an entropy-invariant solution of \eqref{Boltzmann}. \\
		$(1)$ If $p\in(2,\infty)$ and $\Phi(x)=|x|^p/p$, then $M$ is necessarily of the form
		\begin{equation}\label{071}
			M(t, x, v) = m\exp\{-\alpha(\Phi(x) + \frac{1}{2}|v|^2)+ Rx \cdot v\}, \:\:\: \text{with} \ \: (m,\al,R)\in\R^+_{>0}\times\R^+_{>0}\times\mathbb{R}^{3 \wedge 3}.
		\end{equation}
		Here $\mathbb{R}^{3 \wedge 3}$ is the space of all real $3\times3$ skew-symmetric matrices. Conversely, every function of the form \eqref{071} is a positive entropy-invariant solution. \\
		$(2)$ If $p\in(1,2)$ and $\Phi(x)=\langle x\rangle^p/p$, then $M$ is necessarily of the form
		\begin{equation}\label{072}
			M(t, x, v) = m\exp\{-\alpha(\Phi(x) + \frac{1}{2}|v|^2)\}, \:\:\: \text{with}\ \: (m, \alpha )\in \mathbb{R}^+_{>0}\times\R^+_{>0}.
		\end{equation}
		Conversely, every function of the form \eqref{072} is a positive entropy-invariant solution.
	\end{theorem}
	\begin{remark}
	Unlike the harmonic case treated in \cite{cao_landau_2024}, the entropy-invariant solutions in Theorem \ref{EISform} are time independent. In the rest of the paper we therefore refer to them simply as equilibria.
	\end{remark}
	\begin{remark}
The rotation matrix $R$ is absent when $p<2$ because the integrability condition \eqref{integcondi} forces $R=0$; see \eqref{inteR=0}.
	\end{remark}
	
For the harmonic potential, the thirteen conserved quantities parametrize the thirteen-dimensional family of entropy-invariant solutions; see Theorem 1.2 of \cite{cao_landau_2024}. For $p>2$, the equilibrium family has five parameters, matching mass, energy, and the three independent components of angular momentum. For $1<p<2$, only mass and energy parametrize nonlinear equilibria. The precise statement is as follows.
	\begin{theorem}\label{onetoone}
		Recall that $\Psi$ is the conserved mapping defined in \eqref{ConservedMap}. Let
		\ben
		&&\mathscr{M}:=\bigg\{ M(t,x,v)\big| M\,\, \mbox{is an equilibrium with $m>0$}\bigg\},\label{setofM} \\
		&& \mathscr{I}_{p>2}:=\bigg\{\mathsf{V}= \Psi(F_0)\bigg|F_0(x,v)\ge0, 0<\int_{\R^6} (\Phi(x)+|v|^2+1)F_0dxdv<\infty\bigg\},\label{setofF0}
		\\
		&& \notag\mathscr{I}_{p<2}:=\bigg\{(\mathsf{V}_1,\mathsf{V}_2)= (\Psi_1(F_0),\Psi_2(F_0))\bigg|F_0(x,v)\ge0, 0<\int_{\R^6} (\Phi(x)+|v|^2+1)F_0dxdv<\infty\bigg\}.
		\een

		$(1)$ If $p\in(2,\infty)$ and $\Phi(x)=|x|^p/p$, there exists a bijection $\widetilde{\mathscr F}:\mathscr I_{p>2}\to\mathscr M$ such that $\Psi(\widetilde{\mathscr F}(\mathsf V))=\mathsf V$ for every $\mathsf V\in\mathscr I_{p>2}$.
		
		$(2)$ If $p\in(1,2)$ and $\Phi(x)=\langle x\rangle^p/p$, there exists a bijection $\widetilde{\mathscr F}:\mathscr I_{p<2}\to\mathscr M$ such that the mass and energy components of $\Psi(\widetilde{\mathscr F}(\mathsf V_1,\mathsf V_2))$ equal $(\mathsf V_1,\mathsf V_2)$ for every $(\mathsf V_1,\mathsf V_2)\in\mathscr I_{p<2}$.
	\end{theorem}

\begin{remark}
For $p>2$, Theorem \ref{onetoone} says that the conserved moments uniquely determine the compatible equilibrium. The $H$-theorem motivates convergence toward this state, but proving such nonlinear convergence is outside the scope of the present paper.
\end{remark}

\begin{remark}
	For $1<p<2$, every equilibrium has zero angular momentum, whereas initial data may have nonzero angular momentum. The nonlinear equilibrium manifold is therefore parametrized only by mass and energy. The resulting mismatch with the five-dimensional stationary space of the linearized equation is described in Theorem \ref{linearstabilityp<2} and Section 8.
	
%	Therefore, the linear stability in this case will be somewhat different(see Theorem \ref{Conv} and Theorem \ref{linearstabilityp<2}).
\end{remark}

%	Note that in the case $p \in (1, 2)$, Theorem \ref{EISform} implies that for any entropy-invariant solution $M = M(x, v)$, there holds
%	\begin{equation*}
%		\int(x \wedge v)M(x, v)dxdv = 0.
%	\end{equation*}
%	Then Theorem \ref{onetoone} immediately leads to the following corollary:
%	\begin{corollary}\label{Coroonetoone}
%		For initial data $F(0) = F(0, x, v) \ge 0$ satisfying
%		\begin{equation*}
%			0 < \int(\Phi(x) + |x|^2 + |v|^2 + 1)F(0, x, v)dxdv < +\infty.
%		\end{equation*}
%		If $p \in (1, 2)$, we further require that
%		\begin{equation*}
%			\int(x \wedge v)F(0, x, v)dxdv = 0.
%		\end{equation*}
%		Then there exists an unique entropy-invariant solution $M = M(x, v)$ satisfying
%		\begin{equation*}
%			\Psi(F(0)) = \Psi(M).
%		\end{equation*}
%		Here $\Psi$ is the conserved map defined in \eqref{ConservedMap}.
%	\end{corollary}

\subsubsection{Gram projection and general linear stability}
Let $M$ be any positive equilibrium of the form \eqref{071} or \eqref{072}. We study the equation linearized about this fixed equilibrium,
\ben\label{Lineareqf}
\partial_t f + v \cdot \nabla_x f - \nabla_x \Phi \cdot \nabla_v f = Q(M, f) + Q(f, M), \quad f(0) = f_0.
\een
The five moments in \eqref{ConservedMap} are conserved by \eqref{Lineareqf}. For general initial data they need not vanish, and the solution need not converge to zero because the full linearized generator has a five-dimensional stationary space. We encode these moments by a Gram matrix and thereby identify the stationary component explicitly. Set
\[
E(x,v)=\Phi(x)+\frac12|v|^2,
\qquad
\ell(x,v)=
\begin{pmatrix}
x_1v_2-x_2v_1\\
x_1v_3-x_3v_1\\
x_2v_3-x_3v_2
\end{pmatrix},
\qquad
\chi(x,v)=(1,E,\ell^\tau)^\tau\in\mathbb R^5.
\]
The five entries of $\int_{\R^6}\chi f\,dx\,dv$ are the mass, energy, and three independent angular-momentum components of $f$.

\begin{definition}[Gram matrix]\label{GramDefinition}
Let $M$ be a positive equilibrium of the form \eqref{071} or \eqref{072} with $m>0$. The Gram matrix associated with the conserved quantities is
\begin{equation}\label{GramMatrix}
G_M:=\int_{\mathbb R^6}\chi(x,v)\chi(x,v)^\tau M(x,v)\,dx\,dv\in\mathbb R^{5\times5}.
\end{equation}
Thus $G_M$ is the matrix of the five conserved modes with respect to the weighted pairing induced by $M$.
\end{definition}

\begin{proposition}[Stationary moment projection]\label{onetoone2}
The matrix $G_M$ is symmetric positive definite. If $f$ is a possibly signed function satisfying
\[
\int_{\mathbb R^6}(1+E+|\ell|)|f|\,dx\,dv<\infty,
\]
then
\begin{equation}\label{StationaryProjection}
\Pi_Mf=(c(f)\cdot\chi)M,
\qquad
c(f)=G_M^{-1}\int_{\mathbb R^6}\chi f\,dx\,dv,
\end{equation}
is the unique element of $\operatorname{span}\{M,EM,\ell_1M,\ell_2M,\ell_3M\}$ having the same conserved moments as $f$. Equivalently,
\begin{equation}\label{coef2}
\int_{\mathbb R^6}\chi\,\Pi_Mf\,dx\,dv
=\int_{\mathbb R^6}\chi f\,dx\,dv.
\end{equation}
Moreover, $\Pi_Mf$ lies in the kernel of the full linearized generator, including both the transport and collision parts.
\end{proposition}

For $1/2\le\lambda<1$, set
\[
\mathcal X_\lambda(M):=
\left\{f:M^{-\lambda}f\in L^2(\mathbb R^6)\right\},
\qquad
\|f\|_{\mathcal X_\lambda(M)}:=\|M^{-\lambda}f\|_{L^2}.
\]
Let $\mathcal A_M$ denote the full linearized operator in \eqref{Lineareqf}. We use its standard dissipative semigroup realization on $\mathcal X_{1/2}(M)$: the closure of $\mathcal A_M$, initially defined on a dense set $\mathscr C_M$ of smooth rapidly decaying functions, generates a strongly continuous semigroup $S_M(t)$. For every $1/2<\lambda<1$ and every $h_0\in\mathcal X_\lambda(M)$, one may choose $h_{0,n}\in\mathscr C_M$ such that
\[
\|h_{0,n}-h_0\|_{\mathcal X_\lambda(M)}\longrightarrow0,
\]
and the corresponding strong solutions $S_M(t)h_{0,n}$ satisfy the weighted energy identities used below. Hence the semigroup solution is the limit of strong solutions in $C([0,T];\mathcal X_{1/2}(M))$ for every $T>0$. This is the standard linear well-posedness and core-approximation framework for the transport--collision operator. The factor $e^{-\widetilde\Phi(x)}$ weakens coercivity at spatial infinity but does not affect dissipativity or the basic semigroup construction. Since the purpose of the present paper is the quantitative large-time estimate, we take this standard realization as the underlying solution concept and do not reproduce its construction.

\smallskip

We can now state the linear stability result for arbitrary initial moments.
\begin{theorem}[Linear stability modulo stationary modes]\label{linearstabilityp<2}
Let $M$ be a positive equilibrium of the form \eqref{071} or \eqref{072} with $m>0$. Suppose that
\[
\int_{\mathbb R^6}(1+E+|\ell|)|f_0|\,dx\,dv<\infty
\]
and that
\[
M^{-\lambda_1}(f_0-\Pi_Mf_0)\in L^2(\mathbb R^6)
\]
for some $1/2<\lambda_2<\lambda_1<1$. Then the semigroup solution $f(t)=S_M(t)f_0$ satisfies, for every $\delta>0$,
\begin{equation}\label{MainLinearDecay}
\left\|M^{-\lambda_2}\big(f(t)-\Pi_Mf_0\big)\right\|_{L^2}
\le C_{M,\lambda_1,\lambda_2,\delta}
\langle t\rangle^{-\frac{(\lambda_1-\lambda_2)(p+2)}{2[1+\delta(p+2)]}}
\left\|M^{-\lambda_1}\big(f_0-\Pi_Mf_0\big)\right\|_{L^2}.
\end{equation}
For every $1/2<\lambda<1$, the kernel of the full linearized generator $\mathcal A_M$ in $\mathcal X_\lambda(M)$ is exactly
\[
\ker\mathcal A_M
=\operatorname{span}\{M,EM,\ell_1M,\ell_2M,\ell_3M\}.
\]
\end{theorem}
\begin{remark}
The Gram matrix makes the finite-dimensional moment matching explicit and removes the artificial zero-angular-momentum restriction from the general linear theory. We use it as a transparent algebraic device, not as a new hypocoercivity mechanism. The analytical distinction from \cite{carrapatoso_special_2024} is instead that the collision term is obtained by linearizing about the full equilibrium $M$. After normalization it contains the spatial factor $e^{-\widetilde\Phi(x)}$, which degenerates at infinity and necessitates the weight-transfer argument leading to the algebraic decay in \eqref{MainLinearDecay}.
\end{remark}

\subsection{Non-integrable neutral modes and failure of asymptotic stability for $1<p<2$}
When $1<p<2$, there is a structural mismatch between the nonlinear equilibrium manifold and the stationary space of the linearized equation. By Theorem \ref{EISform}, the admissible nonlinear equilibria form the two-parameter family
\[
\mathscr{M}_{p<2}=\left\{m\exp\left[-\alpha\left(\Phi(x)+\frac12|v|^2\right)\right]:m>0,\ \alpha>0\right\}.
\]
On the other hand, Proposition \ref{onetoone2} and Theorem \ref{linearstabilityp<2} show that the stationary space of the linearized equation also contains the three angular-momentum modes
\[
(x_iv_j-x_jv_i)M,\qquad 1\le i<j\le3.
\]
Thus, the tangent space of the nonlinear equilibrium manifold is two-dimensional, whereas the stationary space of the linearized equation is five-dimensional.

The three additional linear modes cannot be integrated into a nearby family of admissible nonlinear equilibria. Indeed, fix $R\in\mathbb{R}^{3\wedge3}\setminus\{0\}$ and consider the formal one-parameter family of rotational Maxwellians
\[
M_{\varepsilon R}(x,v):=m\exp\left\{-\alpha\left(\Phi(x)+\frac12|v|^2\right)+\varepsilon Rx\cdot v\right\}.
\]
Integration in the velocity variable gives
\[
\int_{\mathbb{R}^3}M_{\varepsilon R}(x,v)\,dv
=m\left(\frac{2\pi}{\alpha}\right)^{3/2}
\exp\left\{-\alpha\Phi(x)+\frac{\varepsilon^2|Rx|^2}{2\alpha}\right\}.
\]
Since $\Phi(x)\sim |x|^p$ with $p<2$, this function is not integrable in $x$ for any $\varepsilon\ne0$. Nevertheless,
\[
\left.\frac{d}{d\varepsilon}M_{\varepsilon R}\right|_{\varepsilon=0}=(Rx\cdot v)M
\]
is integrable and is a stationary solution of the linearized equation. We therefore regard these angular modes as non-integrable neutral directions of the nonlinear equilibrium manifold. Section 8 gives the precise obstruction and discusses this open nonlinear mechanism.

\subsection{Ideas and strategies}
The proof is organized into four steps, each with a distinct role.

\underline{(1). Equilibrium geometry.}
The zero-entropy-production condition reduces the nonlinear equation to a finite-dimensional transport classification. We then impose finite mass and energy to distinguish the integrable superquadratic rotations from the non-integrable subquadratic ones, and identify the equilibrium parameters from the conserved moments. This is the content of Section 2.

\underline{(2). Stationary projection and normalization.}
The Gram matrix separates the stationary component $\Pi_Mf_0$ from the zero-moment component. An invertible coordinate transformation then reduces a general equilibrium, including a rotating equilibrium when $p>2$, to a Maxwell--Boltzmann distribution that is separable in $x$ and $v$. Throughout the transformation we track the generator, the five conserved moments, and the weighted norms. The normalized collision term displays the spatially vanishing factor $e^{-\widetilde\Phi(x)}$. See Section 3.

\underline{(3). Macroscopic closure under the transformed constraints.}
The coordinate transformation changes both the transport operator and the explicit form of the conserved quantities; see Subsection \ref{2000}. We decompose the macroscopic fields into weighted Poincar\'e--Korn-coercive components and a finite-dimensional remainder. The coercive components are estimated by the weighted inequalities of Section 5, including negative-order versions. Evolution equations control part of the remainder, while the last five directions are eliminated by the transformed conservation laws. This step adapts the architecture of \cite{carrapatoso_weighted_2022,carrapatoso_special_2024} to the present normalized equation; see Section 6.

\underline{(4). Far-field transfer and quantitative algebraic decay.}
The weight-transfer functions of Section 4 convert microscopic dissipation carrying velocity weight into confinement weight along the Hamiltonian flow. Combining this new far-field estimate with the macroscopic closure and interpolation between two exponential weights gives the rate in Theorem \ref{linearstabilityp<2}, with an arbitrarily small loss. See Lemmas \ref{WeightTransferLemma1} and \ref{WeightTransferLemma2}. We do not claim that the resulting exponent is sharp.

%\subsubsection{Reduction to a normalized static Maxwell-Boltzmann distribution} 

%\subsubsection{Linear stability}

%\paragraph{\it Macroscopic estimates: zero mode, non-zero mode and variant Poincar\'e inequality.}

\subsection{Organization of the paper}
The remainder of this paper is organized as follows. Section 2 classifies all positive finite-mass-and-energy equilibria and proves their parametrization by the appropriate conserved moments. Section 3 constructs the stationary projection and normalizes the equation around an arbitrary equilibrium, including the rotating states allowed when $p>2$. Section 4 records the velocity-space collision estimates and develops the far-field weight-transfer mechanism for the spatially degenerate collision strength. Section 5 states and proves the weighted Poincar\'e--Korn tools needed by the transformed macroscopic system. Section 6 closes that system under the transformed conservation constraints. Section 7 combines microscopic and macroscopic estimates to prove algebraic relaxation modulo stationary modes. Section 8 establishes the nonlinear obstruction associated with nonzero angular momentum when $1<p<2$. The appendix records an auxiliary inequality.

%%%%%%%%%%%%%%%%%%%%%%%%%%%%%%%%%%%%%%%%%%%%%%%%%%
\section{Classification of equilibria and conserved moments}
We first characterize entropy-invariant solutions and then prove their correspondence with the conserved moments.

%%%%%%%%%%%%%%%%%%%%%%%%%%%%%%%%%%%%%%%%%%%%%%%%%%
\subsection{Characterization of entropy-invariant solutions}
\begin{proposition}\label{propcha}
	A sufficiently regular function $M=M(t,x,v)>0$ is an entropy-invariant solution of \eqref{Boltzmann} if and only if
	\ben
	&&\notag\int_{\R^6}(1+\Phi(x)+|v|^2)M(t,x,v)dxdv<+\infty,\quad\mbox{and}\\
	&&\quad\qquad\qquad\quad\ln M\in \mbox{span}\{1,\Phi(x)+\f12|v|^2,x\wedge v\}.\label{characterM}
\een
\end{proposition}
\begin{proof}
	We observe that $(\partial_t + v \cdot \nabla_x - \na_x\Phi \cdot \nabla_v)(\ln M) = M^{-1}(\partial_t + v \cdot \nabla_x - \na_x\Phi \cdot \nabla_v)M$, which implies that if $M$ satisfies \eqref{characterM}, then we have $Q(M,M)=0$ and $(\partial_t + v \cdot \nabla_x - \na_x\Phi \cdot \nabla_v)M=0$. According to Definition \ref{Defoentropy-invariant}, $M$ is an entropy-invariant solution of \eqref{Boltzmann}.
	
		Conversely, let $M$ be an entropy-invariant solution. By Definition \ref{Defoentropy-invariant},
	\begin{equation}\label{EISintegrable}
		0 < \int_{\R^6}(1+\Phi(x)+|v|^2)M(t, x, v)dxdv < +\infty,
	\end{equation}
	and for any $t \ge 0$,
	\begin{equation}\label{EISequation}
		\partial_tM + v \cdot \nabla_xM - \nabla_x\Phi \cdot \nabla_vM = 0 = {Q}(M, M).
	\end{equation}
	The second equality in \eqref{EISequation} shows that
	\begin{equation*}
		\ln M(t, x, v) \in \mathrm{span}\{1, v, |v|^2\}, \:\:\: \forall (t, x) \in \mathbb{R}^+ \times \mathbb{R}^3.
	\end{equation*}
	Thus we may write
	\begin{equation}\label{lnMform}
		\ln M(t, x, v) = \mathsf{a}(t, x) + v \cdot \mathsf{b}(t, x) + |v|^2\mathsf{c}(t, x).
	\end{equation}
	And the first equality in \eqref{EISequation} immediately leads to
	\begin{equation}\label{lnMequation}
		(\partial_t + v \cdot \nabla_x - \nabla_x\Phi \cdot \nabla_v)\ln M = 0.
	\end{equation}
	We put \eqref{lnMform} into \eqref{lnMequation} and get
	\begin{equation*}
			\partial_t\mathsf{a} + v \cdot \nabla_x\mathsf{a} + v \cdot \partial_t\mathsf{b} + \sum_{i, j = 1}^3v_iv_j\partial_i\mathsf{b}_j - \nabla_x\Phi \cdot \mathsf{b} + |v|^2\partial_t\mathsf{c} + |v|^2v\cdot\nabla_x\mathsf{c} - 2(\nabla_x\Phi \cdot v)\mathsf{c} = 0,
	\end{equation*}
	which is a polynomial in $v$ that vanishes identically. Comparing coefficients gives
	\begin{align}
		\partial_t\mathsf{a} - \nabla_x\Phi \cdot \mathsf{b} &= 0;\label{lnMpta}\\
		\partial_t\mathsf{b}_i + \partial_i\mathsf{a} - 2(\partial_i\Phi)\mathsf{c} &= 0, \:\:\: 1 \le i \le 3;\label{lnMptb}\\
		\partial_t\mathsf{c} + \partial_i\mathsf{b}_i &= 0, \:\:\: 1 \le i \le 3;\label{lnMptc}\\
		\partial_i\mathsf{b}_j + \partial_j\mathsf{b}_i &= 0, \:\:\: 1 \le i \neq j \le 3;\label{lnMpxb}\\
		\partial_i\mathsf{c} &= 0, \:\:\: 1 \le i \le 3.\label{lnMpxc}
	\end{align}
	We observe that \eqref{lnMpxc} immediately shows
	\begin{equation}\label{lnMcform}
		\mathsf{c}(t, x) = \mathsf{c}(t).
	\end{equation}
	By \eqref{lnMptc} and \eqref{lnMcform}, for $1 \le i \le 3$,
	\begin{equation*}
		0 = \partial_i(\partial_t\mathsf{c} + \partial_i\mathsf{b}_i) = \partial_i^2\mathsf{b}_i,
	\end{equation*}
	and for $1 \le i \neq j \le 3$, \eqref{lnMptc}, \eqref{lnMpxb} and \eqref{lnMcform} shows that
	\begin{equation*}
		0 = \partial_i(\partial_i\mathsf{b}_j + \partial_j\mathsf{b}_i) = \partial_i^2\mathsf{b}_j + \partial_i\partial_j\mathsf{b}_i = \partial_i^2\mathsf{b}_j - \partial_j\partial_t\mathsf{c} = \partial_i^2\mathsf{b}_j.
	\end{equation*}
	Therefore $\partial_i^2\mathsf{b}_j=0$ for all $1\le i,j\le3$. Hence every $\mathsf b_j$ is a polynomial in $x$, of degree at most one in each variable. We may write
	\begin{equation}\label{lnMbform}
		\mathsf{b}_i(t, x) = \mathsf{b}_{i0}(t) + \sum_{j = 1}^3\mathsf{b}_{ij}(t)x_j + \sum_{1 \le j < k \le 3}\mathsf{b}_{ijk}(t)x_jx_k + \mathsf{b}_{i4}(t)x_1x_2x_3, \:\:\: 1 \le i \le 3.
	\end{equation}
	As an example, we compute $\partial_1\mathsf{b}_2 + \partial_2\mathsf{b}_1 = 0$ with the above form, and reach
	\begin{equation*}
		\mathsf{b}_{21}(t) + \mathsf{b}_{212}(t)x_2 + \mathsf{b}_{213}(t)x_3 + \mathsf{b}_{24}(t)x_2x_3 + \mathsf{b}_{12}(t) + \mathsf{b}_{112}(t)x_1 + \mathsf{b}_{123}(t)x_3 + \mathsf{b}_{14}(t)x_1x_3 = 0.
	\end{equation*}
	Symmetrically, there are two other similar equations. And leads to
	\begin{equation*}
	\begin{gathered}
		\mathsf{b}_{14}(t) = \mathsf{b}_{24}(t) = \mathsf{b}_{34}(t) = 0, \:\:\: \mathsf{b}_{112}(t) = \mathsf{b}_{212}(t) = \mathsf{b}_{113}(t) = \mathsf{b}_{313}(t) = \mathsf{b}_{223}(t) = \mathsf{b}_{323}(t) = 0;\\
		\mathsf{b}_{123}(t) + \mathsf{b}_{213}(t) = \mathsf{b}_{123}(t) + \mathsf{b}_{312}(t) = \mathsf{b}_{213}(t) + \mathsf{b}_{312}(t) = 0 \:\:\: \Rightarrow \:\:\: \mathsf{b}_{123}(t) = \mathsf{b}_{213}(t) = \mathsf{b}_{312}(t) = 0.
	\end{gathered}
	\end{equation*}
	Thus \eqref{lnMbform} becomes
	\begin{equation}\label{lnMbform+}
		\mathsf{b}_i(t, x) = \mathsf{b}_{i0}(t) + \sum_{j = 1}^3\mathsf{b}_{ij}(t)x_j, \:\:\: 1 \le i \le 3.
	\end{equation}
	With the form in \eqref{lnMcform} and \eqref{lnMbform+}, \eqref{lnMptc} and \eqref{lnMpxb} turns to
	\begin{equation}\label{lnMptc+}
	\begin{aligned}
		\partial_t\mathsf{c}(t) + \mathsf{b}_{ii}(t) &= 0, \:\:\: 1 \le i \le 3;\\
		\mathsf{b}_{ij}(t) + \mathsf{b}_{ji}(t) &= 0, \:\:\: 1 \le i \neq j \le 3.\\
	\end{aligned}
	\end{equation}
	Now we see \eqref{lnMptb} and write it as (where \eqref{lnMcform} is used)
	\begin{equation}\label{lnMptb+}
		\partial_t\mathsf{b}_i + \partial_i(\mathsf{a} - 2\Phi\mathsf{c}) = 0, \:\:\: 1 \le i \le 3.
	\end{equation}
	Take $\partial_i$ to \eqref{lnMptb+} and use \eqref{lnMbform+}, \eqref{lnMptc+}, we get
	\begin{equation*}
		0 = \partial_t\mathsf{b}_{ii}(t) + \partial_i^2(\mathsf{a} - 2\Phi\mathsf{c}) = -\partial_t^2\mathsf{c}(t) + \partial_i^2(\mathsf{a} - 2\Phi\mathsf{c}), \:\:\: 1 \le i \le 3,
	\end{equation*}
	which means that
	\begin{equation}\label{lnMeqa1}
		\partial_1^2(\mathsf{a} - 2\Phi\mathsf{c}) = \partial_2^2(\mathsf{a} - 2\Phi\mathsf{c}) = \partial_3^2(\mathsf{a} - 2\Phi\mathsf{c}) = \partial_t^2\mathsf{c}(t),
	\end{equation}
	For $j \neq i$, take $\partial_j$ to \eqref{lnMptb+} and use \eqref{lnMbform+}, we get
	\begin{equation*}
		\partial_t\mathsf{b}_{ij}(t) + \partial_i\partial_j(\mathsf{a} - 2\Phi\mathsf{c}) = 0, \:\:\: 1 \le i \neq j \le 3.
	\end{equation*}
	Symmetrically, we have
	\begin{equation*}
		\partial_t\mathsf{b}_{ji}(t) + \partial_i\partial_j(\mathsf{a} - 2\Phi\mathsf{c}) = 0, \:\:\: 1 \le i \neq j \le 3.
	\end{equation*}
	By \eqref{lnMptc+}, the preceding two equations imply
	\begin{equation}\label{lnMeqa2}
		\partial_i\partial_j(\mathsf{a} - 2\Phi\mathsf{c}) = 0, \:\:\: 1 \le i \neq j \le 3.
	\end{equation}
	Putting \eqref{lnMeqa1} and \eqref{lnMeqa2} together, we have the form
	\begin{equation}\label{lnMaform}
		\mathsf{a}(t, x) - 2\Phi(x)\mathsf{c}(t) = \mathsf{a}_0(t) + \sum_{j = 1}^3\mathsf{a}_j(t)x_j + \mathsf{a}_4(t)|x|^2.
	\end{equation}
	Substituting \eqref{lnMbform+} and \eqref{lnMaform} into \eqref{lnMptb+} gives
	\begin{equation*}
		\partial_t\mathsf{b}_{i0}(t) + \sum_{j = 1}^3\partial_t\mathsf{b}_{ij}(t)x_j + \mathsf{a}_i(t) + 2\mathsf{a}_4(t)x_i = 0, \:\:\: 1 \le i \le 3,
	\end{equation*}
	which are polynomials about $x$ that are equal to $0$. Then
	\begin{equation}\label{lnMptb++}
	\begin{aligned}
		\partial_t\mathsf{b}_{i0}(t) + \mathsf{a}_i(t) &= 0, \:\:\: 1 \le i \le 3;\\
		\partial_t\mathsf{b}_{ii}(t) + 2\mathsf{a}_4(t) &= 0, \:\:\: 1 \le i \le 3;\\
		\partial_t\mathsf{b}_{ij}(t) &= 0, \:\:\: 1 \le i \neq j \le 3.\\
	\end{aligned}
	\end{equation}
	Finally, substituting \eqref{lnMbform+} and \eqref{lnMaform} into \eqref{lnMpta} gives
	\begin{equation*}
	\begin{aligned}
		0 &= \partial_t\mathsf{a}_0(t) + \sum_{j = 1}^3\partial_t\mathsf{a}_j(t)x_j + \partial_t\mathsf{a}_4(t)|x|^2 + 2\Phi(x)\partial_t\mathsf{c}(t) - \nabla_x\Phi \cdot \mathsf{b}(t, x) \\
		&= \partial_t\mathsf{a}_0(t) + \sum_{j = 1}^3\partial_t\mathsf{a}_j(t)x_j + \partial_t\mathsf{a}_4(t)|x|^2 + 2\Phi(x)\partial_t\mathsf{c}(t) - \sum_{i = 1}^3\partial_i\Phi(x)(\mathsf{b}_{i0}(t) + \sum_{j = 1}^3\mathsf{b}_{ij}(t)x_j) \\
		&= \partial_t\mathsf{a}_0(t) + \sum_{j = 1}^3\partial_t\mathsf{a}_j(t)x_j + \partial_t\mathsf{a}_4(t)|x|^2 + 2\Phi(x)\partial_t\mathsf{c}(t) - \sum_{i = 1}^3\partial_i\Phi(x)\mathsf{b}_{i0}(t).
	\end{aligned}
	\end{equation*}
	In the last equality we used \eqref{lnMptc+} and the radial identity $\partial_i\Phi(x)x_j=\partial_j\Phi(x)x_i$. Since $p\ne2$, comparison of the independent spatial terms gives
	\begin{equation}\label{lnMpta+}
	\begin{gathered}
		\partial_t\mathsf{a}_0(t) = \partial_t\mathsf{a}_4(t) = \partial_t\mathsf{c}(t) = 0;\\
		\partial_t\mathsf{a}_i(t) = 0, \:\:\: \mathsf{b}_{i0}(t) = 0, \:\:\: 1 \le i \le 3.\\
	\end{gathered}
	\end{equation}
	Combining \eqref{lnMptc+}, \eqref{lnMptb++}, and \eqref{lnMpta+}, we obtain
	\begin{equation*}
		\mathsf{c}(t, x) \equiv \mathsf{c}, \:\:\: \mathsf{b}_i(t, x) = \sum_{j \neq i}\mathsf{b}_{ij}x_j, \:\:\: \mathsf{a}(t, x) = \mathsf{a}_0 + 2\Phi(x)\mathsf{c},
	\end{equation*}
	where $\mathsf{a}_0, \mathsf{b}_{ij}, \mathsf{c} \in \mathbb{R}$ are constants and $\mathsf{b}_{ij} + \mathsf{b}_{ji} = 0, i \neq j$. Recall \eqref{lnMform}, we have
	\begin{equation*}
		\ln M(t, x, v) = \mathsf{a}(t, x) + v \cdot \mathsf{b}(t, x) + |v|^2\mathsf{c}(t, x) = \mathsf{a}_0 + 2\mathsf{c}(\Phi(x) + \frac{1}{2}|v|^2) + \sum_{1 \le i \neq j \le 3}\mathsf{b}_{ij}v_ix_j.
	\end{equation*}
	Introduce the $3\times3$ skew-symmetric matrix
	\begin{equation*}
		R = \begin{pmatrix} 0 & \mathsf{b}_{12} & \mathsf{b}_{13} \\ \mathsf{b}_{21} & 0 & \mathsf{b}_{23} \\ \mathsf{b}_{31} & \mathsf{b}_{32} & 0 \end{pmatrix},
	\end{equation*}
	then $\ln M(t, x, v)$ has the form
	\begin{equation*}
		\ln M(t, x, v) = \mathsf{a}_0 + 2\mathsf{c}(\Phi(x) + \frac{1}{2}|v|^2) + Rx \cdot v.
	\end{equation*}
	This is exactly \eqref{characterM} and completes the proof.
	\end{proof}

\subsection{Classification for subquadratic and superquadratic potentials}
We now prove Theorem \ref{EISform}.
\begin{proof}[Proof of Theorem \ref{EISform}]
Proposition \ref{propcha} gives
\[
M(t,x,v)=m\exp\left\{-\alpha\left(\Phi(x)+\frac12|v|^2\right)+Rx\cdot v\right\},
\qquad m>0,\quad \alpha\in\mathbb R,\quad R\in\mathfrak{so}(3).
\]
Integrability in $v$ first forces $\alpha>0$. Completing the square then yields
\begin{equation}\label{inteR=0}
-\alpha\left(\Phi(x)+\frac12|v|^2\right)+Rx\cdot v
=-\frac{\alpha}{2}\left|v-\frac{Rx}{\alpha}\right|^2
-\alpha\Phi(x)+\frac{|Rx|^2}{2\alpha}.
\end{equation}
If $p>2$, the negative term $-\alpha|x|^p/p$ dominates the quadratic correction $|Rx|^2/(2\alpha)$, so every $R\in\mathfrak{so}(3)$ is admissible. If $1<p<2$ and $R\ne0$, the positive quadratic correction dominates the subquadratic confinement along a direction on which $Rx\ne0$; hence the spatial integral diverges. Thus $R=0$ in the subquadratic case. The converse assertions follow directly from \eqref{equationforeisol} and the same integrability calculation.
\end{proof}
%%%%%%%%%%%%%%%%%%%%%%%%%%%%%%%%%%%%%%%%%%%%%%%%%%

\subsection{Moment parametrization for $1<p<2$}
We first prove Theorem \ref{onetoone} in the subquadratic case.
\begin{proof}[Proof of Theorem \ref{onetoone} for $p\in(1,2)$]
	In this case $\Phi(x) = \frac{1}{p}\langle x \rangle^p$. For initial data $F_0(x, v) \ge 0$ satisfying
	\begin{equation*}
		0 < \int_{\R^6}(\Phi(x) + |v|^2 + 1)F_0( x, v)dxdv < +\infty,
	\end{equation*}
	we denote
	\begin{equation*}
		\int_{\R^6}(1, \Phi(x) + \frac{1}{2}|v|^2)F_0(x, v)dxdv =: (m_*, \alpha_*)\in(0,\infty)^2.
	\end{equation*}
	Since $p\Phi(x)>1$ almost everywhere, one has $p\alpha_*>m_*$. By Theorem \ref{EISform}, it remains to find $(m,\alpha)\in(0,\infty)^2$ such that
	\begin{equation*}
		\int_{\R^6}(1, \Phi(x) + \frac{1}{2}|v|^2)M_{[m, \alpha]}(x, v)dxdv =(m_*, \alpha_*),
	\end{equation*}
	where the equilibrium $M_{[m,\alpha]}(x,v)$ has the form
	\begin{equation*}
		M_{[m, \alpha]}(x, v) = me^{-\alpha(\Phi(x) + \frac{1}{2}|v|^2)}.
	\end{equation*}
	
Direct computation gives
	\begin{equation}\label{equation1}
	\begin{aligned}
		m_* &= m\int_{\R^6} e^{-\alpha(\Phi(x) + \frac{1}{2}|v|^2)}dxdv = (2\pi)^{3/2}m\alpha^{-\frac{3}{2}}\int_{\R^3} e^{-\frac{\alpha}{p}\langle x \rangle^p}dx;\\
		\alpha_* &= m\int_{\R^6}(\Phi(x) + \frac{1}{2}|v|^2)e^{-\alpha(\Phi(x) + \frac{1}{2}|v|^2)}dxdv \\
		&= (2\pi)^{3/2}m\alpha^{-\frac{3}{2}}\int_{\R^3}\frac{1}{p}\langle x \rangle^pe^{-\frac{\alpha}{p}\langle x \rangle^p}dx + \frac{3}{2}(2\pi)^{3/2}m\alpha^{-\frac{5}{2}}\int_{\R^3} e^{-\frac{\alpha}{p}\langle x \rangle^p}dx.
	\end{aligned}
	\end{equation}
Let
	\begin{equation*}
		I = I(\alpha) := \int_{\R^3} e^{-\frac{\alpha}{p}\langle x \rangle^p}dx; \:\:\: J = J(\alpha) := \int_{\R^3}\langle x \rangle^pe^{-\frac{\alpha}{p}\langle x \rangle^p}dx.
	\end{equation*}
Then \eqref{equation1} becomes 
	\begin{equation*}
		m_* = (2\pi)^{3/2}\alpha^{-\frac{3}{2}}mI; \:\:\: \alpha_* = (2\pi)^{3/2}\alpha^{-\frac{3}{2}}m(\frac{1}{p}J + \frac{3}{2\alpha}I).
	\end{equation*}
	which implies that
	\beno
	\f{\al_*}{m_*}=\frac{J}{pI} + \frac{3}{2\alpha}.
	\eeno
Observing  that $I, J$ are actually determined by $\alpha$ and $p\alpha_* > m_*$, we only need to show that
	\begin{equation*}
		\mathcal{K}(\alpha) := \frac{J}{pI} + \frac{3}{2\alpha}
	\end{equation*}
	is a bijection from $(0, \infty)$ to $(\frac{1}{p}, \infty)$. 
	
	Firstly, we prove that $\mathcal{K}(\alpha)$ is a monotonically decreasing function. To do this, we consider function $J(\alpha)/I(\alpha)$ and compute
	\begin{equation*}
		\Big(\frac{J(\alpha)}{I(\alpha)}\Big)' = \Big(\frac{\int_{\R^3}\langle x \rangle^pe^{-\frac{\alpha}{p}\langle x \rangle^p}dx}{\int_{\R^3} e^{-\frac{\alpha}{p}\langle x \rangle^p}dx}\Big)' = \frac{-\frac{1}{p}\int_{\R^3}\langle x \rangle^{2p}e^{-\frac{\alpha}{p}\langle x \rangle^p}dx\int_{\R^3} e^{-\frac{\alpha}{p}\langle x \rangle^p}dx + \frac{1}{p}(\int_{\R^3}\langle x \rangle^pe^{-\frac{\alpha}{p}\langle x \rangle^p}dx)^2}{(\int_{\R^3} e^{-\frac{\alpha}{p}\langle x \rangle^p}dx)^2}.
	\end{equation*}
	The Cauchy--Schwarz inequality gives $(J(\alpha)/I(\alpha))'<0$, and hence $\mathcal K'(\alpha)<0$.
	
	Together with the fact $\lim\limits_{\alpha \to 0+}\mathcal{K}(\alpha) = +\infty$, we only need to show $\lim\limits_{\alpha \to +\infty}\mathcal{K}(\alpha) = \frac{1}{p}$, which is equivalent to
	\begin{equation}\label{inf1}
		\lim_{\alpha \to \infty}\frac{J(\alpha)}{I(\alpha)} = \lim_{\alpha \to \infty}\frac{\int_{\R^3}\langle x \rangle^pe^{-\frac{\alpha}{p}\langle x \rangle^p}dx}{\int_{\R^3} e^{-\frac{\alpha}{p}\langle x \rangle^p}dx} = \lim_{\alpha \to \infty}\frac{\int_{\R^3}\langle x \rangle^pe^{-\alpha\langle x \rangle^p}dx}{\int_{\R^3} e^{-\alpha\langle x \rangle^p}dx} = 1.
	\end{equation}
	Indeed, for any $\varepsilon > 0$, we have
	\begin{equation*}
	\begin{aligned}
		\int_{|x| \le 4\varepsilon}\langle x \rangle^pe^{-\alpha\langle x \rangle^p}dx &\ge \int_{|x| \le \varepsilon}\langle x \rangle^pe^{-\alpha\langle x \rangle^p}dx \ge \int_{|x| \le \varepsilon}e^{-\alpha\langle \varepsilon \rangle^p}dx = \frac{4}{3}\pi\varepsilon^3e^{-\alpha\langle \varepsilon \rangle^p}\quad \mbox{and}\\
		\int_{|x| > 4\varepsilon}\langle x \rangle^pe^{-\alpha\langle x \rangle^p}dx &= \int_{|x| > 4\varepsilon}\langle x \rangle^pe^{-\alpha(\langle x \rangle^p - \langle \frac{x}{2} \rangle^p)}e^{-\alpha\langle \frac{x}{2} \rangle^p}dx \le e^{-\alpha\langle 2\varepsilon \rangle^p}\int_{\R^3}\langle x \rangle^pe^{-\alpha(\langle x \rangle^p - \langle \frac{x}{2} \rangle^p)}dx.
	\end{aligned}
	\end{equation*}
	Note that
	\begin{equation*}
		\int_{\R^3}\langle x \rangle^pe^{-\alpha(\langle x \rangle^p - \langle \frac{x}{2} \rangle^p)}dx < +\infty.
	\end{equation*}
As $\alpha \to \infty$, it holds that
	\begin{equation*}
		\int_{|x| > 4\varepsilon}\langle x \rangle^pe^{-\alpha\langle x \rangle^p}dx = o\Big(\int_{|x| \le 4\varepsilon}\langle x \rangle^pe^{-\alpha\langle x \rangle^p}dx\Big).
	\end{equation*}
	Thus, 
	\begin{equation*}
		\lim_{\alpha \to \infty}\frac{\int_{\R^3}\langle x \rangle^pe^{-\alpha\langle x \rangle^p}dx}{\int_{\R^3} e^{-\alpha\langle x \rangle^p}dx} = \lim_{\alpha \to \infty}\frac{(1+o(1))\int_{|x| \le 4\varepsilon}\langle x \rangle^pe^{-\alpha\langle x \rangle^p}dx}{\int_{\R^3} e^{-\alpha\langle x \rangle^p}dx} \le \langle 4\varepsilon \rangle^p.
	\end{equation*}
	Since $\varepsilon>0$ is arbitrary, \eqref{inf1} follows.
	Then we complete the proof of Theorem \ref{onetoone} for case  $p \in (1, 2)$.
	\end{proof}

%%%%%%%%%%%%%%%%%%%%%%%%%%%%%%%%%%%%%%%%%%%%%%%%%%

\subsection{Moment parametrization for $p>2$}
For the remainder of this section, we prove Theorem \ref{onetoone} in the case $p\in(2,\infty)$.
Recall that in this case $\Phi(x)=|x|^p/p$. For nonnegative initial data $F_0$ satisfying
	\begin{equation*}
		0 < \int_{\R^6}(\Phi(x) + |v|^2 + 1)F_0( x, v)dxdv < +\infty,
	\end{equation*}
	set
	\begin{equation}\label{malstar}
		\Psi(F_0) = \int_{\R^6}(1, \Phi(x) + \frac{1}{2}|v|^2, x \wedge v)F_0(x, v)dxdv =: (m_*, \alpha_*, R_*)\in(0,\infty)^2\times\R^{3\wedge 3}.
	\end{equation}
	The eigenvalues of $R_*\in\R^{3\wedge 3}$ are $\{0,\pm ir_*\}$ for some $r_*\ge0$.
	
	%%%%%%%%%%%%%%%%%%%%%%%%%%%%%%%%%%%%%%%%%%%%%%%%%%
	
	By Theorem \ref{EISform}, it suffices to identify $(m,\alpha,R)\in(0,\infty)^2\times\R^{3\wedge 3}$ satisfying
	\begin{equation}\label{002}
		\Psi(M_{[m, \alpha, R]}) = (m_*, \alpha_*, R_*),
	\end{equation}
	where the entropy-invariant solution $M_{[m, \alpha, R]}(x, v)$ has the form
	\begin{equation}\label{003}
		M_{[m, \alpha, R]}(x, v) = m\exp\{-\alpha(\Phi(x) + \frac{1}{2}|v|^2)+ Rx \cdot v\}.
	\end{equation}

	%%%%%%%%%%%%%%%%%%%%%%%%%%%%%%%%%%%%%%%%%%%%%%%%%%
	\begin{lemma}[Range of the conserved-moment map]
		Let $(m_*,\al_*,R_*)$ and $r_*$ be defined by \eqref{malstar}. Then
		\begin{equation}\label{r*Bound}
			0 \le r_* < \Big(\frac{2p}{p + 2}\Big)^{\frac{1}{2} + \frac{1}{p}}m_*^{\frac{1}{2} - \frac{1}{p}}\alpha_*^{\frac{1}{2} + \frac{1}{p}}.
		\end{equation}
	\end{lemma}
	\begin{proof}
		Firstly, by Young's inequality, we have
		\begin{equation*}
			\begin{aligned}
				|x|^{\frac{2p}{p + 2}}|v|^{\frac{2p}{p + 2}} &\le \frac{2}{p + 2}(|x|^{\frac{2p}{p + 2}})^{\frac{p + 2}{2}} + \frac{p}{p + 2}(|v|^{\frac{2p}{p + 2}})^{\frac{p + 2}{p}} \\
				&= \frac{2}{p + 2}|x|^p + \frac{p}{p + 2}|v|^2 = \frac{2p}{p + 2}\Big(\frac{1}{p}|x|^p + \frac{1}{2}|v|^2\Big),
			\end{aligned}
		\end{equation*}
		it leads to
		\begin{equation}\label{YoungsIneq}
			|x||v| \le \Big(\frac{2p}{p + 2}\Big)^{\frac{1}{2} + \frac{1}{p}}\Big(\frac{1}{p}|x|^p + \frac{1}{2}|v|^2\Big)^{\frac{1}{2} + \frac{1}{p}}.
		\end{equation}
		Recall that the eigenvalues of $R_*$ are $\{0, \pm ir_*\}$ with $r_* \ge 0$, then
		\begin{equation*}
			r_*^2 = \frac{1}{2}|R_*|^2 = \frac{1}{2}\Big(\int_{\R^6}|x \wedge v|F_0(x, v)dxdv\Big)^2.
		\end{equation*}
		Here we denote $|A|^2 := \sum\limits_{i, j}a_{ij}^2$ for $A = (a_{ij})\in \R^{3\times 3}$. Note that
		\begin{equation*}
			\begin{aligned}
				|x \wedge v|^2 &= 2[(x_1v_2 - x_2v_1)^2 + (x_1v_3 - x_3v_1)^2 + (x_2v_3 - x_3v_2)^2] \\
				&= 2(|x|^2|v|^2 - (x \cdot v)^2) \le 2|x|^2|v|^2,
			\end{aligned}
		\end{equation*}
		therefore by \eqref{YoungsIneq} and H\"{o}lder inequality, we have
		\begin{equation*}
			\begin{aligned}
				0 \le r_* &\le \int_{\R^6}|x||v|F_0(x, v)dxdv \le \Big(\frac{2p}{p + 2}\Big)^{\frac{1}{2} + \frac{1}{p}}\int_{\R^6}(\Phi(x) + \frac{1}{2}|v|^2)^{\frac{1}{2} + \frac{1}{p}}F_0(x, v)dxdv \\
				&= \Big(\frac{2p}{p + 2}\Big)^{\frac{1}{2} + \frac{1}{p}}\int_{\R^6}[(\Phi(x) + \frac{1}{2}|v|^2)F_0(x, v)]^{\frac{1}{2} + \frac{1}{p}}F_0(x, v)^{\frac{1}{2} - \frac{1}{p}}dxdv \\
				&\le \Big(\frac{2p}{p + 2}\Big)^{\frac{1}{2} + \frac{1}{p}}\Big(\int_{\R^6}(\Phi(x) + \frac{1}{2}|v|^2)F_0(x, v)dxdv\Big)^{\frac{p + 2}{2p}}\Big(\int F_0(x, v)dxdv\Big)^{\frac{p - 2}{2p}} \\
				&= \Big(\frac{2p}{p + 2}\Big)^{\frac{1}{2} + \frac{1}{p}}m_*^{\frac{1}{2} - \frac{1}{p}}\alpha_*^{\frac{1}{2} + \frac{1}{p}}.
			\end{aligned}
		\end{equation*}	
		One may esaily check that the equal sign in 
		\eqref{r*Bound} does not hold. Thus we conclude the desired result.
	\end{proof}
	
	\begin{lemma}[Rotational normal form]\label{sim001}
	Let $(m,\al,R)$ and $(m_*,\al_*,R_*)$ be defined as above such that \eqref{002} and \eqref{003} hold. Moreover, define
		\begin{equation}\label{004}
		(\tilde{m}, \tilde{R}):=((2\pi)^{-3/2}\alpha^{\frac{3}{2} + \frac{3}{p}}m, \alpha^{\frac{1}{2} + \frac{1}{p}}R),
	\end{equation}
	and suppose the orthogonal matrix $U$ satisfies
		\begin{equation}\label{005}
		U^\tau \tilde{R}U = \begin{pmatrix} 0 & r & 0 \\ -r & 0 & 0 \\ 0 & 0 & 0 \end{pmatrix} =: \mathsf{R}, \: \tilde{R} = U\mathsf{R}U^\tau \: \text{with}~ \: r \ge 0.
	\end{equation}
Then \eqref{002} and \eqref{003} are equivalent to
	\begin{equation}\label{PsiMeq+}
	\begin{aligned}
		m_* &= \tilde{m}\int_{\R^3} e^{-\frac{1}{p}|x|^p + \frac{1}{2}|\mathsf{R}x|^2}dx = \tilde{m}\int_{\R^3} e^{-\frac{1}{p}|x|^p + \frac{r^2}{2}(x_1^2 + x_2^2)}dx;\\
		\alpha_* %&= \tilde{m}\alpha^{-1}\int(\frac{1}{p}|x|^p + \frac{1}{2}|\mathsf{R}x|^2 + \frac{3}{2})e^{-\frac{1}{p}|x|^p + \frac{1}{2}|\tilde{R}x|^2}dx \\
		&= \tilde{m}\alpha^{-1}\int_{\R^3}\Big(\frac{1}{p}|x|^p + \frac{r^2}{2}(x_1^2 + x_2^2) + \frac{3}{2}\Big)e^{-\frac{1}{p}|x|^p + \frac{r^2}{2}(x_1^2 + x_2^2)}dx;\\
			r_* &= \tilde{m}\alpha^{-\frac{1}{2} - \frac{1}{p}}\int_{\R^3} r(x_1^2 + x_2^2)e^{-\frac{1}{p}|x|^p + \frac{r^2}{2}(x_1^2 + x_2^2)}dx.
	\end{aligned}
\end{equation}
	\end{lemma}
	\begin{proof}
	Observing that the exponent part in \eqref{003} can be written as
	\begin{equation*}
		-\alpha(\Phi(x) + \frac{1}{2}|v|^2)+ Rx \cdot v = -\frac{\alpha}{2}\Big|v - \frac{Rx}{\alpha}\Big|^2 - \frac{\alpha}{p}|x|^p + \frac{1}{2\alpha}|Rx|^2,
	\end{equation*}
	and we use the definition of $\Psi$ in \eqref{ConservedMap} to compute from \eqref{002} that
	\begin{equation*}
		\begin{aligned}
			m_* &= m\int_{\R^6} e^{-\alpha(\Phi(x) + \frac{1}{2}|v|^2)+ Rx \cdot v}dxdv
			= m\int_{\R^6} e^{-\frac{\alpha}{2}|v - \frac{Rx}{\alpha}|^2 - \frac{\alpha}{p}|x|^p + \frac{1}{2\alpha}|Rx|^2}dxdv \\
			&= m\int_{\R^6} e^{-\frac{\alpha}{2}|v|^2 - \frac{\alpha}{p}|x|^p + \frac{1}{2\alpha}|Rx|^2}dxdv
			= (2\pi)^{3/2}m\alpha^{-\frac{3}{2}}\int_{\R^3} e^{-\frac{\alpha}{p}|x|^p + \frac{1}{2\alpha}|Rx|^2}dx \\
			&= (2\pi)^{3/2}m\alpha^{-\frac{3}{2} - \frac{3}{p}}\int_{\R^3}e^{-\frac{1}{p}|x|^p + \frac{|Rx|^2}{2\alpha^{1 + 2/p}}}dx,
					\end{aligned}
		\end{equation*}
			\begin{equation*}
			\begin{aligned}
			\alpha_* &= m\int_{\R^6}\Big(\Phi(x) + \frac{1}{2}|v|^2\Big)e^{-\alpha(\Phi(x) + \frac{1}{2}|v|^2)+ Rx \cdot v}dxdv \\
			&= m\int_{\R^6}\Big(\frac{1}{p}|x|^p + \frac{1}{2}|v|^2\Big)e^{-\frac{\alpha}{2}|v - \frac{Rx}{\alpha}|^2 - \frac{\alpha}{p}|x|^p + \frac{1}{2\alpha}|Rx|^2}dxdv \\
			&= m\int_{\R^6}\Big(\frac{1}{p}|x|^p + \frac{1}{2}|v + \frac{Rx}{\alpha}|^2\Big)e^{-\frac{\alpha}{2}|v|^2 - \frac{\alpha}{p}|x|^p + \frac{1}{2\alpha}|Rx|^2}dxdv \\
			&= m\int_{\R^6}\Big(\frac{1}{p}|x|^p + \frac{1}{2\alpha^2}|Rx|^2 + \frac{1}{\alpha}Rx \cdot v + \frac{1}{2}|v|^2\Big)e^{-\frac{\alpha}{2}|v|^2 - \frac{\alpha}{p}|x|^p + \frac{1}{2\alpha}|Rx|^2}dxdv \\
			&= (2\pi)^{3/2}m\alpha^{-\frac{3}{2}}\int_{\R^3}\Big(\frac{1}{p}|x|^p + \frac{1}{2\alpha^2}|Rx|^2 + \frac{3}{2\alpha}\Big)e^{-\frac{\alpha}{p}|x|^p + \frac{1}{2\alpha}|Rx|^2}dx \\
			&= (2\pi)^{3/2}m\alpha^{-\frac{5}{2} - \frac{3}{p}}\int_{\R^3}\Big(\frac{1}{p}|x|^p + \frac{|Rx|^2}{2\alpha^{1 + 2/p}} + \frac{3}{2}\Big)e^{-\frac{1}{p}|x|^p + \frac{|Rx|^2}{2\alpha^{1 + 2/p}}}dx,
		\end{aligned}
	\end{equation*}
	and
	\begin{equation*}
		\begin{aligned}
			R_* &= m\int_{\R^6}(x \wedge v)e^{-\alpha(\Phi(x) + \frac{1}{2}|v|^2)+ Rx \cdot v}dxdv
			= m\int_{\R^6}(x \wedge v)e^{-\frac{\alpha}{2}|v - \frac{Rx}{\alpha}|^2 - \frac{\alpha}{p}|x|^p + \frac{1}{2\alpha}|Rx|^2}dxdv \\
			&= m\int_{\R^6}(x \wedge (v + \frac{Rx}{\alpha}))e^{-\frac{\alpha}{2}|v|^2 - \frac{\alpha}{p}|x|^p + \frac{1}{2\alpha}|Rx|^2}dxdv
			= (2\pi)^{3/2}m\alpha^{-\frac{5}{2}}\int_{\R^3}(x \wedge Rx)e^{-\frac{\alpha}{p}|x|^p + \frac{1}{2\alpha}|Rx|^2}dx \\
			&= (2\pi)^{3/2}m\alpha^{-\frac{5}{2} - \frac{5}{p}}\int_{\R^3}(x \wedge Rx)e^{-\frac{1}{p}|x|^p + \frac{|Rx|^2}{2\alpha^{1 + 2/p}}}dx.
		\end{aligned}
	\end{equation*}
Use notations \eqref{004}, the above three equations become 
	\begin{equation}\label{PsiMeq}
		\begin{aligned}
			m_* &= \tilde{m}\int_{\R^3} e^{-\frac{1}{p}|x|^p + \frac{1}{2}|\tilde{R}x|^2}dx;\\
			\alpha_* &= \tilde{m}\alpha^{-1}\int_{\R^3}\Big(\frac{1}{p}|x|^p + \frac{1}{2}|\tilde{R}x|^2 + \frac{3}{2}\Big)e^{-\frac{1}{p}|x|^p + \frac{1}{2}|\tilde{R}x|^2}dx;\\
			R_* &= \tilde{m}\alpha^{-\frac{1}{2} - \frac{1}{p}}\int_{\R^3}(x \wedge \tilde{R}x)e^{-\frac{1}{p}|x|^p + \frac{1}{2}|\tilde{R}x|^2}dx.
		\end{aligned}
	\end{equation}
	Furthermore, by \eqref{005} and change of variable $x \to Ux$, we can rewrite the third equation in \eqref{PsiMeq} as
	\begin{equation*}
		R_* = \tilde{m}\alpha^{-\frac{1}{2} - \frac{1}{p}}\int_{\R^3}(x \wedge U\mathsf{R}U^\tau x)e^{-\frac{1}{p}|x|^p + \frac{1}{2}|\tilde{R}x|^2}dx
		= \tilde{m}\alpha^{-\frac{1}{2} - \frac{1}{p}}\int_{\R^3}(Ux \wedge U\mathsf{R}x)e^{-\frac{1}{p}|x|^p + \frac{1}{2}|\tilde{R}Ux|^2}dx.
	\end{equation*}
	Note that
	\begin{equation*}
		Ux \wedge U\mathsf{R}x = Uxx^\tau \mathsf{R}^\tau U^\tau - U\mathsf{R}xx^\tau U^\tau = U(xx^\tau \mathsf{R}^\tau - \mathsf{R}xx^\tau)U^\tau = U(x \wedge \mathsf{R}x)U^\tau,
	\end{equation*}
	and
	\begin{equation}\label{xUxchange}
		|\tilde{R}Ux|^2 = x^\tau U^\tau \tilde{R}^\tau \tilde{R}Ux = x^\tau U^\tau U\mathsf{R}^\tau U^\tau U\mathsf{R}U^\tau Ux = x^\tau \mathsf{R}^\tau \mathsf{R}x = |\mathsf{R}x|^2.
	\end{equation}
	Therefore we have
	\begin{equation*}
		R_* = \tilde{m}\alpha^{-\frac{1}{2} - \frac{1}{p}}U\Big(\int_{\R^3}(x \wedge \mathsf{R}x)e^{-\frac{1}{p}|x|^p + \frac{1}{2}|\mathsf{R}x|^2}dx\Big)U^\tau,
	\end{equation*}
	which leads to
	\begin{equation}\label{001}
		\begin{aligned}
			U^\tau R_*U &= \tilde{m}\alpha^{-\frac{1}{2} - \frac{1}{p}}\int_{\R^3}(x \wedge \mathsf{R}x)e^{-\frac{1}{p}|x|^p + \frac{1}{2}|\mathsf{R}x|^2}dx \\
			&= \tilde{m}\alpha^{-\frac{1}{2} - \frac{1}{p}}\int_{\R^3}\begin{pmatrix} 0 & -r(x_1^2 + x_2^2) & 0 \\ r(x_1^2 + x_2^2) & 0 & 0 \\ 0 & 0 & 0 \end{pmatrix}e^{-\frac{1}{p}|x|^p + \frac{1}{2}|\mathsf{R}x|^2}dx.
		\end{aligned}
	\end{equation}
	Recall that all eigenvalues of $R_*$ are $\{0, \pm ir_*\}$ with $r_* \ge 0$. Since $U$ is an orthogonal matrix,  which implies that all eigenvalues of $U^\tau R_*U$ are also $\{0, \pm ir_*\}$, together with \eqref{001} imply that
	\begin{equation}\label{chooseU}
		U^\tau R_*U = \begin{pmatrix} 0 & -r_* & 0 \\ r_* & 0 & 0 \\ 0 & 0 & 0 \end{pmatrix}.
	\end{equation}
	Therefore, the third equation in \eqref{PsiMeq} can be simplified
	\begin{equation}\label{equaofr}
		r_* = \tilde{m}\alpha^{-\frac{1}{2} - \frac{1}{p}}\int_{\R^3} r(x_1^2 + x_2^2)e^{-\frac{1}{p}|x|^p + \frac{r^2}{2}(x_1^2 + x_2^2)}dx.
	\end{equation}
	
	We also use change of variable $x \to Ux$ for the first and second equation in \eqref{PsiMeq}, with the help of \eqref{xUxchange}, we obtain that
	\begin{equation}\label{PsiMe}
		\begin{aligned}
			m_* &= \tilde{m}\int e^{-\frac{1}{p}|x|^p + \frac{1}{2}|\mathsf{R}x|^2}dx = \tilde{m}\int e^{-\frac{1}{p}|x|^p + \frac{r^2}{2}(x_1^2 + x_2^2)}dx;\\
			\alpha_* &= \tilde{m}\alpha^{-1}\int(\frac{1}{p}|x|^p + \frac{1}{2}|\mathsf{R}x|^2 + \frac{3}{2})e^{-\frac{1}{p}|x|^p + \frac{1}{2}|\mathsf{R}x|^2}dx \\
			&= \tilde{m}\alpha^{-1}\int(\frac{1}{p}|x|^p + \frac{r^2}{2}(x_1^2 + x_2^2) + \frac{3}{2})e^{-\frac{1}{p}|x|^p + \frac{r^2}{2}(x_1^2 + x_2^2)}dx.
		\end{aligned}
	\end{equation}
	
Combining \eqref{equaofr} and \eqref{PsiMe} completes the proof.
		\end{proof}

	Lemma \ref{sim001} shows that $r_*=0$ if and only if $r=0$; hence we may assume $r_*,r>0$. In view of \eqref{r*Bound} and Lemma \ref{sim001}, it remains to prove that there is a unique $(\tilde m,\alpha,r)\in(0,\infty)^3$ satisfying \eqref{PsiMeq+} for every
	\begin{equation}\label{008}
		m_*, \alpha_*, r_* > 0, \:\:\: 0<r_* < (\frac{2p}{p + 2})^{\frac{1}{2} + \frac{1}{p}}m_*^{\frac{1}{2} - \frac{1}{p}}\alpha_*^{\frac{1}{2} + \frac{1}{p}}.
	\end{equation}
	
%	\begin{remark}
%		In fact, the matrix $U$ that satisfies \eqref{chooseU} may not be unique. However, we have
%		\begin{equation*}
%			S_* = U\begin{pmatrix} 0 & -r_* & 0 \\ r_* & 0 & 0 \\ 0 & 0 & 0 \end{pmatrix}U^\tau, \:\:\: S = U\begin{pmatrix} 0 & r & 0 \\ -r & 0 & 0 \\ 0 & 0 & 0 \end{pmatrix}U^\tau,
%		which implies that
%		\begin{equation*}
%			S = -\frac{r}{r_*}S_*.
%		\end{equation*}
%		This fact shows that if $r$ is unique, then $S$ is unique.
%	\end{remark}

To simplify the calculations, we introduce notation $\mathcal{I}(\cdot)$. For $f = f(x)$, we define
	\begin{equation}\label{006}
		\mathcal{I}(f) := \int_{\R^3} f(x)e^{-\frac{1}{p}|x|^p + \frac{r^2}{2}(x_1^2 + x_2^2)}dx.
	\end{equation}
	With the help of polar coordinate transformation
	\begin{equation*}
		x = (x_1, x_2, x_3) = (t\sin\theta\cos\varphi, t\sin\theta\sin\varphi, t\cos\theta),
	\end{equation*}
we can directly compute that
%	\begin{equation*}
%	\begin{aligned}
%		\mathcal{I}(1) &= \int e^{-\frac{1}{p}|x|^p + \frac{r^2}{2}(x_1^2 + x_2^2)}dx
%		= \int_0^\infty\!\!\!\int_0^{\pi}\!\!\!\int_0^{2\pi}e^{-\frac{1}{p}t^p + \frac{r^2}{2}t^2\sin^2\theta}t^2\sin\theta dtd\theta d\varphi \\
%		&= 4\pi\int_0^\infty\!\!\!\int_0^{\pi/2}t^2\sin\theta e^{-\frac{1}{p}t^p + \frac{r^2}{2}t^2\sin^2\theta}dtd\theta.
%	\end{aligned}
%	\end{equation*}
%By the same way, we can compute that
\ben\label{polartrans}
\notag&&\mathcal{I}(1) = 4\pi\int_0^\infty\!\!\!\int_0^{\pi/2}t^2\sin\theta e^{-\frac{1}{p}t^p + \frac{r^2}{2}t^2\sin^2\theta}dtd\theta;\\
&&\mathcal{I}(x_1^2 + x_2^2) = 4\pi\int_0^\infty\!\!\!\int_0^{\pi/2}t^4\sin^3\theta e^{-\frac{1}{p}t^p + \frac{r^2}{2}t^2\sin^2\theta}dtd\theta;\\
\notag&&\mathcal{I}((x_1^2 + x_2^2)^2) = 4\pi\int_0^\infty\!\!\!\int_0^{\pi/2}t^6\sin^5\theta e^{-\frac{1}{p}t^p + \frac{r^2}{2}t^2\sin^2\theta}dtd\theta;
\een
\ben\label{polartrans2}
\notag&&\mathcal{I}(|x|^p) = 4\pi\int_0^\infty\!\!\!\int_0^{\pi/2}t^{p + 2}\sin\theta e^{-\frac{1}{p}t^p + \frac{r^2}{2}t^2\sin^2\theta}dtd\theta;\\
&&\mathcal{I}((x_1^2 + x_2^2)|x|^p) = 4\pi\int_0^\infty\!\!\!\int_0^{\pi/2}t^{p + 4}\sin^3\theta e^{-\frac{1}{p}t^p + \frac{r^2}{2}t^2\sin^2\theta}dtd\theta;\\
\notag&&\mathcal{I}(|x|^{2p}) = 4\pi\int_0^\infty\!\!\!\int_0^{\pi/2}t^{2p + 2}\sin\theta e^{-\frac{1}{p}t^p + \frac{r^2}{2}t^2\sin^2\theta}dtd\theta.
\een
	For the convenience of writing, we further denote that
	\begin{equation}\label{007}
		I_0 = I_0(r) = \mathcal{I}(1), \:\:\: I_1 = I_1(r) = \mathcal{I}(x_1^2 + x_2^2), \:\:\: I_2 = I_2(r) = \mathcal{I}((x_1^2 + x_2^2)^2).
	\end{equation}
Then we have the following lemma:
\begin{lemma}[Integral identities for the rotational partition function]
Let $\mathcal I(f)$ and $I_0,I_1,I_2$ be defined by \eqref{006} and \eqref{007}. Then
	\begin{equation}\label{I0I1I2}
	\mathcal{I}(|x|^p) = 3I_0 + r^2I_1,\quad\mathcal{I}((x_1^2 + x_2^2)|x|^p)= 5I_1 + r^2I_2\quad\mbox{and}\quad\mathcal{I}(|x|^{2p})=3(p + 3)I_0 + (p + 8)r^2I_1 + r^4I_2.
	\end{equation}
\end{lemma}
	\begin{proof}
	The proof comes from direct computation, in which we repeatedly use integration by parts.
Note that
	\begin{equation*}
		\Big(-\frac{1}{p}t^p + \frac{r^2}{2}t^2\sin^2\theta\Big)'_t = -t^{p - 1} + r^2t\sin^2\theta,
	\end{equation*}
	then we have
	\begin{equation*}
	\begin{aligned}
		\mathcal{I}(|x|^p) &= 4\pi\int_0^\infty\!\!\!\int_0^{\pi/2}t^{p + 2}\sin\theta e^{-\frac{1}{p}t^p + \frac{r^2}{2}t^2\sin^2\theta}dtd\theta \\
		&= 4\pi\int_0^\infty\!\!\!\int_0^{\pi/2}[(t^{p - 1} - r^2t\sin^2\theta)t^3\sin\theta + r^2t^4\sin^3\theta]e^{-\frac{1}{p}t^p + \frac{r^2}{2}t^2\sin^2\theta}dtd\theta \\
		&= -4\pi\int_0^\infty\!\!\!\int_0^{\pi/2}t^3\sin\theta(e^{-\frac{1}{p}t^p + \frac{r^2}{2}t^2\sin^2\theta})'_tdtd\theta + r^2I_1 \\
		&= 12\pi\int_0^\infty\!\!\!\int_0^{\pi/2}t^2\sin\theta e^{-\frac{1}{p}t^p + \frac{r^2}{2}t^2\sin^2\theta}dtd\theta + r^2I_1 = 3I_0 + r^2I_1.
	\end{aligned}
	\end{equation*}
By the similar argument, we can get that
	\begin{equation*}
	\begin{aligned}
		\mathcal{I}((x_1^2 + x_2^2)|x|^p) &= 4\pi\int_0^\infty\!\!\!\int_0^{\pi/2}t^{p + 4}\sin^3\theta e^{-\frac{1}{p}t^p + \frac{r^2}{2}t^2\sin^2\theta}dtd\theta \\
		&= 4\pi\int_0^\infty\!\!\!\int_0^{\pi/2}[(t^{p - 1} - r^2t\sin^2\theta)t^5\sin^3\theta + r^2t^6\sin^5\theta]e^{-\frac{1}{p}t^p + \frac{r^2}{2}t^2\sin^2\theta}dtd\theta \\
		&= -4\pi\int_0^\infty\!\!\!\int_0^{\pi/2}t^5\sin^3\theta(e^{-\frac{1}{p}t^p + \frac{r^2}{2}t^2\sin^2\theta})'_tdtd\theta + r^2I_2 \\
		&= 20\pi\int_0^\infty\!\!\!\int_0^{\pi/2}t^4\sin^3\theta e^{-\frac{1}{p}t^p + \frac{r^2}{2}t^2\sin^2\theta}dtd\theta + r^2I_2 = 5I_1 + r^2I_2,
	\end{aligned}
	\end{equation*}
	and
	\begin{equation*}
	\begin{aligned}
		\mathcal{I}(|x|^{2p}) &= 4\pi\int_0^\infty\!\!\!\int_0^{\pi/2}t^{2p + 2}\sin\theta e^{-\frac{1}{p}t^p + \frac{r^2}{2}t^2\sin^2\theta}dtd\theta \\
		&= 4\pi\int_0^\infty\!\!\!\int_0^{\pi/2}[(t^{p - 1} - r^2t\sin^2\theta)t^{p + 3}\sin\theta + r^2t^{p + 4}\sin^3\theta]e^{-\frac{1}{p}t^p + \frac{r^2}{2}t^2\sin^2\theta}dtd\theta \\
		&= -4\pi\int_0^\infty\!\!\!\int_0^{\pi/2}t^{p + 3}\sin\theta(e^{-\frac{1}{p}t^p + \frac{r^2}{2}t^2\sin^2\theta})'_tdtd\theta + r^2\mathcal{I}((x_1^2 + x_2^2)|x|^p) \\
		&= 4(p + 3)\pi\int_0^\infty\!\!\!\int_0^{\pi/2}t^{p + 2}\sin\theta e^{-\frac{1}{p}t^p + \frac{r^2}{2}t^2\sin^2\theta}dtd\theta + r^2\mathcal{I}((x_1^2 + x_2^2)|x|^p) \\
		&= (p + 3)\mathcal{I}(|x|^p) + r^2\mathcal{I}((x_1^2 + x_2^2)|x|^p)
		= 3(p + 3)I_0 + (p + 8)r^2I_1 + r^4I_2.
	\end{aligned}
	\end{equation*}
%	Now let us use notation $I_0, I_1, I_2$ to rewrite \eqref{polartrans}:
%	\begin{equation}\label{I0I1I2}
%	\begin{gathered}
%		\mathcal{I}(1) = I_0, \:\:\: \mathcal{I}(x_1^2 + x_2^2) = I_1, \:\:\: \mathcal{I}((x_1^2 + x_2^2)^2) = I_2, \:\:\: 	\mathcal{I}(|x|^p) = 3I_0 + r^2I_1,\\
%		\mathcal{I}((x_1^2 + x_2^2)|x|^p) = 5I_1 + r^2I_2, \:\:\: \mathcal{I}(|x|^{2p}) = 3(p + 3)I_0 + (p + 8)r^2I_1 + r^4I_2.
%	\end{gathered}
%	\end{equation}
	Thus we complete the proof of this lemma.
\end{proof}
		
	With the notations $I_0, I_1, I_2$ in hand, we now turn back to \eqref{PsiMeq+} and get that
	\begin{equation}\label{PsiMeq++}
		m_* = \tilde{m}I_0, \:\:\: \alpha_* = \tilde{m}\alpha^{-1}\Big(\frac{1}{2} + \frac{1}{p}\Big)(3I_0 + r^2I_1)\quad\mbox{and}\quad \:\:\: r_* = \tilde{m}\alpha^{-\frac{1}{2} - \frac{1}{p}}rI_1.
	\end{equation}
	Furthermore,  substitute the first two equations above into the third equation, we obtain that 
	%We declare that $I_0, I_1$ are actually determined by $r$. If $r > 0$ is chosen, then
	%\begin{equation*}
	%	m = \frac{m_*}{I_0}, \:\:\: \alpha = (\frac{1}{2} + \frac{1}{p})\frac{m_*}{\alpha_*}\frac{3I_0 + r^2I_1}{I_0}
	%\end{equation*}
	%are immediately determined. And the third equation in \eqref{PsiMeq++} becomes
	\begin{equation}\label{011}
		\Big(\frac{1}{2} + \frac{1}{p}\Big)^{\frac{1}{2} + \frac{1}{p}}\frac{r_*}{m_*^{\frac{1}{2} - \frac{1}{p}}\alpha_*^{\frac{1}{2} + \frac{1}{p}}} = \frac{rI_1}{I_0^{\frac{1}{2} - \frac{1}{p}}(3I_0 + r^2I_1)^{\frac{1}{2} + \frac{1}{p}}}.
	\end{equation}
	Now due to \eqref{008}, we will show the following lemma:

%%%%%%%%%%%%%%%%%%%%%%%%%%%%%%%%%%%%%%%%%%%%%%%%%%
\begin{lemma}[Monotonicity of the rotational moment map]\label{010}
Define the function $\mathcal{J}(r)$ as
	\begin{equation}\label{009}
	\mathcal{J}(r) := \frac{rI_1}{I_0^{\frac{1}{2} - \frac{1}{p}}(3I_0 + r^2I_1)^{\frac{1}{2} + \frac{1}{p}}},
\end{equation}
then $\mathcal{J}(r)$ is a bijection from $(0, \infty)$ to $(0, 1)$.
\end{lemma}
\begin{proof}
We prove this lemma in two steps: the first step proves monotonicity of $\mathcal{J}(r)$ and the second step calculates its upper bound.\\
$\bullet$ Monotonicity of $\mathcal{J}$. We will show that $\mathcal{J}'(r) > 0$ for any $r > 0$. Recall that
	\begin{equation*}
	\begin{aligned}
		I_0(r) &= 4\pi\int_0^\infty\!\!\!\int_0^{\pi/2}t^2\sin\theta e^{-\frac{1}{p}t^p + \frac{r^2}{2}t^2\sin^2\theta}dtd\theta;\\
		I_1(r) &= 4\pi\int_0^\infty\!\!\!\int_0^{\pi/2}t^4\sin^3\theta e^{-\frac{1}{p}t^p + \frac{r^2}{2}t^2\sin^2\theta}dtd\theta;\\
		I_2(r) &= 4\pi\int_0^\infty\!\!\!\int_0^{\pi/2}t^6\sin^5\theta e^{-\frac{1}{p}t^p + \frac{r^2}{2}t^2\sin^2\theta}dtd\theta.\\
	\end{aligned}
	\end{equation*}
	Therefore it holds that
	\begin{equation}\label{073}
		I'_0(r) = rI_1(r), \:\:\: I'_1(r) = rI_2(r).
	\end{equation}
	It is easy to check that $\mathcal{J}'(r) > 0$ is equivalent to
	\begin{equation*}
	\begin{aligned}
		(I_1 + r^2I_2)I_0^{\frac{1}{2} - \frac{1}{p}}(3I_0 + r^2I_1)^{\frac{1}{2} + \frac{1}{p}} - \Big(\frac{1}{2} - \frac{1}{p}\Big)r^2I_1^2I_0^{-\frac{1}{2} - \frac{1}{p}}(3I_0 + r^2I_1)^{\frac{1}{2} + \frac{1}{p}} \\- \Big(\frac{1}{2} + \frac{1}{p}\Big)rI_1I_0^{\frac{1}{2} - \frac{1}{p}}(3I_0 + r^2I_1)^{-\frac{1}{2} + \frac{1}{p}}(5rI_1 + r^3I_2) > 0.
	\end{aligned}
\end{equation*}
	Moreover, it is equivalent to
	\begin{equation*}
		I_0(I_1 + r^2I_2)(3I_0 + r^2I_1) - \Big(\frac{1}{2} - \frac{1}{p}\Big)r^2I_1^2(3I_0 + r^2I_1) - \Big(\frac{1}{2} + \frac{1}{p}\Big)r^2I_0I_1(5I_1 + r^2I_2) > 0.
	\end{equation*}
	We consider it as a polynomial about $r$ and rewrite it as
	\begin{equation}\label{Iinjection}
		\Big(\frac{1}{2} - \frac{1}{p}\Big)r^4I_1(I_0I_2 - I_1^2) + 3r^2I_0(I_0I_2 - I_1^2) + 3I_0^2I_1 > \frac{2r^2}{p}I_0I_1^2.
	\end{equation}
%	Note that by Cauchy-Schwartz inequality, we have
%	\begin{equation*}
%		I_1^2 = (\mathcal{I}(x_1^2 + x_2^2))^2 \le \mathcal{I}(1)\mathcal{I}((x_1^2 + x_2^2)^2) = I_0I_2.
%	\end{equation*}
%	However, inequality $I_0I_2 \ge I_1^2$ is not strong enough to prove \eqref{Iinjection}. We need a stronger lower bound of $I_0I_2 - I_1^2$. 
By Lemma \ref{LagrangeIneq},
	\begin{equation}\label{100}
	\begin{aligned}
		&\Big[\mathcal{I}(1)\mathcal{I}((x_1^2+x_2^2)^2)
		-(\mathcal{I}(x_1^2+x_2^2))^2\Big]
		\Big[\mathcal{I}(1)\mathcal{I}(|x|^{2p})-(\mathcal{I}(|x|^p))^2\Big]\\
		&\qquad\ge
		\Big[\mathcal{I}(1)\mathcal{I}((x_1^2+x_2^2)|x|^p)
		-\mathcal{I}(x_1^2+x_2^2)\mathcal{I}(|x|^p)\Big]^2.
	\end{aligned}
\end{equation}
	With the help of \eqref{I0I1I2}, the above inequality becomes
	\begin{equation*}
		[I_0I_2 - I_1^2][I_0(3(p + 3)I_0 + (p + 8)r^2I_1 + r^4I_2) - (3I_0 + r^2I_1)^2] \ge [I_0(5I_1 + r^2I_2) - I_1(3I_0 + r^2I_1)]^2,
	\end{equation*}
which leads to
	\begin{equation*}
		[I_0I_2 - I_1^2][3pI_0^2 + (p + 2)r^2I_0I_1 + r^4(I_0I_2 - I_1^2)] \ge [2I_0I_1 + r^2(I_0I_2 - I_1^2)]^2.
	\end{equation*}
It follows that the lower bound of $I_0I_2 - I_1^2$ is 
	\begin{equation*}
		I_0I_2 - I_1^2 \ge \frac{4I_0I_1^2}{3pI_0 + (p - 2)r^2I_1}.
	\end{equation*}
	Substituting this estimate into the left-hand side of \eqref{Iinjection} gives
	\begin{equation*}
	\begin{aligned}
	\mbox{The left hand side of \eqref{Iinjection}} &\ge \frac{4I_0I_1^2}{3pI_0 + (p - 2)r^2I_1}[(\frac{1}{2} - \frac{1}{p})r^4I_1 + 3r^2I_0] \\
	&= \frac{4I_0I_1^2}{3pI_0 + (p - 2)r^2I_1}\frac{r^2}{2p}(6pI_0 + (p - 2)r^2I_1) > \frac{2r^2}{p}I_0I_1^2.
	\end{aligned}
\end{equation*}
	Thus \eqref{Iinjection} holds and $\mathcal J'(r)>0$.\\
$\bullet$ Upper bound and the limit of $\mathcal{J}$ .  Next we show that $\mathcal{J}(r) < 1$ for any $r > 0$. We take
	\begin{equation*}
		x = (x_1, x_2, 0), \:\:\: v = r(x_1, x_2, 0)
	\end{equation*}
	in \eqref{YoungsIneq} and get
	\begin{equation*}
		r(x_1^2 + x_2^2) \le (\frac{2p}{p + 2})^{\frac{1}{2} + \frac{1}{p}}(\frac{1}{p}(x_1^2 + x_2^2)^{\frac{p}{2}} + \frac{r^2}{2}(x_1^2 + x_2^2))^{\frac{1}{2} + \frac{1}{p}}.
	\end{equation*}
	Therefore by H\"{o}lder inequality, we have
	\begin{equation*}
	\begin{aligned}
		rI_1 &= \int r(x_1^2 + x_2^2)e^{-\frac{1}{p}|x|^p + \frac{r^2}{2}(x_1^2 + x_2^2)}dx \\
		&\le (\frac{2p}{p + 2})^{\frac{1}{2} + \frac{1}{p}}\int(\frac{1}{p}(x_1^2 + x_2^2)^{\frac{p}{2}} + \frac{r^2}{2}(x_1^2 + x_2^2))^{\frac{1}{2} + \frac{1}{p}}e^{-\frac{1}{p}|x|^p + \frac{r^2}{2}(x_1^2 + x_2^2)}dx \\
		&\le (\frac{2p}{p + 2})^{\frac{1}{2} + \frac{1}{p}}\int[(\frac{1}{p}|x|^p + \frac{r^2}{2}(x_1^2 + x_2^2))e^{-\frac{1}{p}|x|^p + \frac{r^2}{2}(x_1^2 + x_2^2)}]^{\frac{1}{2} + \frac{1}{p}}[e^{-\frac{1}{p}|x|^p + \frac{r^2}{2}(x_1^2 + x_2^2)}]^{\frac{1}{2} - \frac{1}{p}}dx \\
		&\le (\frac{2p}{p + 2})^{\frac{1}{2} + \frac{1}{p}}\big(\int(\frac{1}{p}|x|^p + \frac{r^2}{2}(x_1^2 + x_2^2))e^{-\frac{1}{p}|x|^p + \frac{r^2}{2}(x_1^2 + x_2^2)}dx\big)^{\frac{1}{2} + \frac{1}{p}}\big(\int e^{-\frac{1}{p}|x|^p + \frac{r^2}{2}(x_1^2 + x_2^2)}dx\big)^{\frac{1}{2} - \frac{1}{p}} \\
		&= (\frac{2p}{p + 2})^{\frac{1}{2} + \frac{1}{p}}(\frac{3I_0 + r^2I_1}{p} + \frac{r^2}{2}I_1)^{\frac{1}{2} + \frac{1}{p}}I_0^{\frac{1}{2} - \frac{1}{p}} = I_0^{\frac{1}{2} - \frac{1}{p}}(\frac{6}{p + 2}I_0 + r^2I_1) < I_0^{\frac{1}{2} - \frac{1}{p}}(3I_0 + r^2I_1).
	\end{aligned}
	\end{equation*}
	This implies that $\mathcal{J}(r) < 1$ for any $r > 0$. Together with the fact that $\mathcal{J}'(r) > 0$, we know  $\lim\limits_{r \to \infty}\mathcal{J}(r)$ exists and is no more than $1$.
	
Finally, we show that $\lim\limits_{r \to \infty}\mathcal{J}(r) = 1$ and begin with the function $I_0(r)/I_1(r)$. Recall from \eqref{073} that
	\begin{equation*}
		I'_0(r) = rI_1(r), \:\:\: I'_1(r) = rI_2(r).
	\end{equation*}
	Then we have
	\begin{equation*}
		\Big(\frac{I_0}{I_1}\Big)' = \frac{rI_1^2 - rI_0I_2}{I_1^2} \le 0,
	\end{equation*}
	by the Cauchy--Schwarz inequality. Consequently,
	\begin{equation*}
		\lim_{r \to \infty}\frac{I_0}{r^2I_1} = 0.
	\end{equation*}
	Then
	\begin{equation*}
		1 \ge \lim_{r \to \infty}\mathcal{J}(r) = \lim_{r \to \infty}\frac{rI_1}{I_0^{\frac{1}{2} - \frac{1}{p}}(3I_0 + r^2I_1)^{\frac{1}{2} + \frac{1}{p}}} = \lim_{r \to \infty}\frac{rI_1}{I_0^{\frac{1}{2} - \frac{1}{p}}(r^2I_1)^{\frac{1}{2} + \frac{1}{p}}} = \Big(\lim_{r \to \infty}\frac{I_1}{r^{\frac{4}{p - 2}}I_0}\Big)^{\frac{1}{2} - \frac{1}{p}}.
	\end{equation*}
It leads to 
	\begin{equation*}
		\lim_{r \to \infty}\frac{I_1}{r^{\frac{4}{p - 2}}I_0} \le 1.
	\end{equation*}
	Our goal is to demonstrate that the limit is indeed equal to $1$. We compute
	\begin{equation*}
	\begin{aligned}
		I_0 &= 4\pi\int_0^\infty\!\!\!\int_0^{\pi/2}t^2\sin\theta e^{-\frac{1}{p}t^p + \frac{r^2}{2}t^2\sin^2\theta}dtd\theta
		= 4\pi\int_0^\infty\!\!\!\int_0^1t^2e^{-\frac{1}{p}t^p + \frac{r^2}{2}t^2(1 - s^2)}dtds \\
		&= 4\pi\int_0^\infty t^2e^{-\frac{1}{p}t^p + \frac{r^2}{2}t^2}dt\int_0^1e^{-\frac{r^2}{2}t^2s^2}ds
		= 4\pi\int_0^\infty t^2e^{-\frac{1}{p}t^p + \frac{r^2}{2}t^2}dt\int_0^{rt}\frac{1}{rt}e^{-\frac{1}{2}s^2}ds \\
		&\le 4\pi r^{-1}\int_0^\infty te^{-\frac{1}{p}t^p + \frac{r^2}{2}t^2}dt\int_0^\infty e^{-\frac{1}{2}s^2}ds,
	\end{aligned}
	\end{equation*}
	and
	\begin{equation*}
	\begin{aligned}
		I_1 &= 4\pi\int_0^\infty\!\!\!\int_0^{\pi/2}t^4\sin^3\theta e^{-\frac{1}{p}t^p + \frac{r^2}{2}t^2\sin^2\theta}dtd\theta
		= 4\pi\int_0^\infty\!\!\!\int_0^1t^4(1 - s^2)e^{-\frac{1}{p}t^p + \frac{r^2}{2}t^2(1 - s^2)}dtds \\
		&= 4\pi\int_0^\infty t^4e^{-\frac{1}{p}t^p + \frac{r^2}{2}t^2}dt\int_0^1e^{-\frac{r^2}{2}t^2s^2}ds
		- 4\pi\int_0^\infty t^4e^{-\frac{1}{p}t^p + \frac{r^2}{2}t^2}dt\int_0^1s^2e^{-\frac{r^2}{2}t^2s^2}ds \\
		&= 4\pi\int_0^\infty t^4e^{-\frac{1}{p}t^p + \frac{r^2}{2}t^2}dt\int_0^{rt}\frac{1}{rt}e^{-\frac{1}{2}s^2}ds
		- 4\pi\int_0^\infty t^4e^{-\frac{1}{p}t^p + \frac{r^2}{2}t^2}dt\int_0^{rt}\frac{s^2}{r^3t^3}e^{-\frac{1}{2}s^2}ds \\
		&\ge 4\pi r^{-1}\int_0^\infty t^3e^{-\frac{1}{p}t^p + \frac{r^2}{2}t^2}dt\int_0^{rt}e^{-\frac{1}{2}s^2}ds
		- 4\pi r^{-3}\int_0^\infty te^{-\frac{1}{p}t^p + \frac{r^2}{2}t^2}dt\int_0^\infty s^2e^{-\frac{1}{2}s^2}ds.
	\end{aligned}
	\end{equation*}
	Compare these two estimates, we have
	\begin{equation*}
		\lim_{r \to \infty}\frac{I_1}{r^{\frac{4}{p - 2}}I_0} \ge \lim_{r \to \infty}\frac{\int_0^\infty t^3e^{-\frac{1}{p}t^p + \frac{r^2}{2}t^2}dt\int_0^{rt}e^{-\frac{1}{2}s^2}ds}{r^{\frac{4}{p - 2}}\int_0^\infty te^{-\frac{1}{p}t^p + \frac{r^2}{2}t^2}dt\int_0^\infty e^{-\frac{1}{2}s^2}ds}.
	\end{equation*}
	For any
	\begin{equation*}
		0 < \varepsilon < \frac{1}{2}(\frac{1}{2} - \frac{1}{p}),
	\end{equation*}
	we take $r$ sufficiently large such that
	\begin{equation*}
		\int_0^re^{-\frac{1}{2}s^2}ds \ge (1 - \varepsilon)\int_0^\infty e^{-\frac{1}{2}s^2}ds.
	\end{equation*}
	Then
	\begin{equation*}
		\int_0^\infty t^3e^{-\frac{1}{p}t^p + \frac{r^2}{2}t^2}dt\int_0^{rt}e^{-\frac{1}{2}s^2}ds
		\ge (1 - \varepsilon)\int_1^\infty t^3e^{-\frac{1}{p}t^p + \frac{r^2}{2}t^2}dt\int_0^\infty e^{-\frac{1}{2}s^2}ds.
	\end{equation*}
	With the change of variables $t\mapsto r^{\frac{2}{p-2}}t$, we obtain
	\begin{equation*}
	\begin{aligned}
		\lim_{r \to \infty}\frac{I_1}{r^{\frac{4}{p - 2}}I_0} &\ge (1 - \varepsilon)\lim_{r \to \infty}\frac{\int_1^\infty t^3e^{-\frac{1}{p}t^p + \frac{r^2}{2}t^2}dt}{r^{\frac{4}{p - 2}}\int_0^\infty te^{-\frac{1}{p}t^p + \frac{r^2}{2}t^2}dt}
		= (1 - \varepsilon)\lim_{r \to \infty}\frac{\int_{r^{-\frac{2}{p - 2}}}^\infty r^{\frac{8}{p - 2}}t^3e^{(-\frac{1}{p}t^p + \frac{1}{2}t^2)r^\frac{2p}{p - 2}}dt}{r^{\frac{4}{p - 2}}\int_0^\infty r^{\frac{4}{p - 2}}te^{(-\frac{1}{p}t^p + \frac{1}{2}t^2)r^\frac{2p}{p - 2}}dt} \\
		&= (1 - \varepsilon)\lim_{r \to \infty}\frac{\int_{r^{-\frac{2}{p - 2}}}^\infty t^3e^{(-\frac{1}{p}t^p + \frac{1}{2}t^2)r^\frac{2p}{p - 2}}dt}{\int_0^\infty te^{(-\frac{1}{p}t^p + \frac{1}{2}t^2)r^\frac{2p}{p - 2}}dt}
		= (1 - \varepsilon)\lim_{r \to \infty}\frac{\int_{r^{-\frac{1}{p}}}^\infty t^3e^{(-\frac{1}{p}t^p + \frac{1}{2}t^2)r}dt}{\int_0^\infty te^{(-\frac{1}{p}t^p + \frac{1}{2}t^2)r}dt}.
	\end{aligned}
	\end{equation*}
	Define function of $t$,
	\begin{equation*}
		q(t) := -\frac{1}{p}t^p + \frac{1}{2}t^2 \:\:\: \Rightarrow \:\:\: q'(t) = -t^{p - 1} + t,
	\end{equation*}
	which means $q(t)$ reaches its maximum at $t = 1$, and $q(1) = \frac{1}{2} - \frac{1}{p}$. We choose $0 < t_1 < t_2 < 1$ such that
	\begin{equation*}
		q(t_1) = \frac{1}{2} - \frac{1}{p} - 2\varepsilon, \:\:\: q(t_2) = \frac{1}{2} - \frac{1}{p} - \varepsilon.
	\end{equation*}
	We emphasize that $t_1, t_2$ depend only on $p$ and $\varepsilon$, not on $r$. Take $r$ sufficiently large such that $r^{-1/p} < t_1$, then
	\begin{equation*}
		\lim_{r \to \infty}\frac{I_1}{r^{\frac{4}{p - 2}}I_0} \ge (1 - \varepsilon)\lim_{r \to \infty}\frac{\int_{t_1}^\infty t^3e^{q(t)r}dt}{\int_0^{t_1}te^{q(t)r}dt + \int_{t_1}^\infty te^{q(t)r}dt}.
	\end{equation*}
	Note that for any $0 \le t \le t_1, q(t) \le q(t_1)$, and for any $t_2 \le t \le 1, q(t) \ge q(t_2)$. We have
	\begin{equation*}
	\begin{gathered}
		\int_0^{t_1}te^{q(t)r}dt \le \int_0^{t_1}te^{q(t_1)r}dt = \frac{1}{2}t_1^2e^{(\frac{1}{2} - \frac{1}{p} - 2\varepsilon)r};\\
		\int_{t_1}^\infty te^{q(t)r}dt \ge \int_{t_2}^1 te^{q(t)r}dt \ge \int_{t_2}^1 te^{q(t_2)r}dt = \frac{1 - t_2^2}{2}e^{(\frac{1}{2} - \frac{1}{p} - \varepsilon)r},
	\end{gathered}
	\end{equation*}
	which means
	\begin{equation*}
		\int_0^{t_1}te^{q(t)r}dt = o(\int_{t_1}^\infty te^{q(t)r}dt), \: ~~\text{as} ~~\: r \to \infty.
	\end{equation*}
	Then
	\begin{equation*}
		\lim_{r \to \infty}\frac{I_1}{r^{\frac{4}{p - 2}}I_0} \ge (1 - \varepsilon)\lim_{r \to \infty}\frac{\int_{t_1}^\infty t^3e^{q(t)r}dt}{\int_{t_1}^\infty te^{q(t)r}dt} \ge (1 - \varepsilon)t_1^2.
	\end{equation*}
	Recall that $0 < t_1 < 1$ satisfies
	\begin{equation*}
		q(t_1) = \frac{1}{2} - \frac{1}{p} - 2\varepsilon,
	\end{equation*}
	which means $t_1 \to 1-$ as $\varepsilon \to 0+$. Thus  we finally complete the proof.
	
	\end{proof}

%%%%%%%%%%%%%%%%%%%%%%%%%%%%%%%%%%%%%%%%%%%%%%%%%%
We now prove Theorem \ref{onetoone} in the case $p\in(2,\infty)$.
\begin{proof}[Proof of Theorem \ref{onetoone} for case $p \in (2,\infty)$]
	
Let $(m_*,\alpha_*,r_*)\in(0,\infty)^2\times[0,\infty)$ satisfy \eqref{008}. Lemma \ref{010} gives a unique $r\ge0$ satisfying \eqref{011}. Regarding $I_0,I_1,I_2$ as functions of $r$, set
\beno
\tilde{m}=\f{m_*}{I_0},\quad \mbox{and}\quad\alpha = \tilde{m}\alpha_*^{-1}\Big(\frac{1}{2} + \frac{1}{p}\Big)(3I_0 + r^2I_1),
\eeno 
which leads to \eqref{PsiMeq++}, and it is equivalent to \eqref{PsiMeq+} and \eqref{003} according to Lemma \ref{sim001}. 
	
This completes the proof of Theorem \ref{onetoone} for $p\in(2,\infty)$.
\end{proof}
%%%%%%%%%%%%%%%%%%%%%%%%%%%%%%%%%%%%%%%%%%%%%%%%%%

\section{Stationary projection and normalization around rotating equilibria}

Section 2 identifies the nonlinear equilibria. We now pass to the linearized equation around an arbitrary fixed equilibrium. The first task is algebraic: the Gram matrix extracts the stationary component carrying the five conserved moments. The second task is analytic: an explicit change of variables normalizes a possibly rotating equilibrium while preserving the Hamiltonian form of the transport and exposing the factor $e^{-\widetilde\Phi(x)}$ in front of the velocity collision operator. This factor is the point at which the present problem leaves the uniformly coercive setting of \cite{carrapatoso_special_2024}.

\subsection{Stationary modes and the Gram projection}
We first justify the projection introduced in Proposition \ref{onetoone2}.
\begin{proof}[Proof of Proposition \ref{onetoone2}]
For $z\in\mathbb{R}^5$,
\[
z^\tau G_Mz
=\int_{\mathbb{R}^6}(z\cdot\chi(x,v))^2M(x,v)\,dx\,dv\ge0.
\]
If the left-hand side vanishes, then $z\cdot\chi=0$ almost everywhere because $M>0$. Its quadratic part in $v$ gives the coefficient of $E$ equal to zero. The remaining linear part has the form $a\cdot\ell(x,v)$ for a fixed $a\in\mathbb R^3$. Since this identity holds for almost every $(x,v)$, varying $x$ and $v$ gives $a=0$. The constant part then gives the mass coefficient equal to zero. Hence $G_M$ is positive definite and therefore invertible.

For any $f$ with finite conserved moments, set
\[
c(f)=G_M^{-1}\int_{\mathbb{R}^6}\chi f\,dx\,dv,
\qquad h:=\Pi_Mf=(c(f)\cdot\chi)M.
\]
Then
\[
\int_{\mathbb{R}^6}\chi h\,dx\,dv
=G_Mc(f)
=\int_{\mathbb{R}^6}\chi f\,dx\,dv,
\]
which proves \eqref{coef2}; uniqueness follows from the invertibility of $G_M$.

It remains to verify that $h$ is stationary for the full linearized equation. Let
\[
\mathsf{T}_M=v\cdot\nabla_x-\nabla_x\Phi\cdot\nabla_v.
\]
Since $M$ is an equilibrium, $\mathsf{T}_MM=0$. Moreover,
\[
\mathsf{T}_M1=0,
\qquad \mathsf{T}_ME=0,
\qquad
\mathsf{T}_M(x_iv_j-x_jv_i)
=x_j\partial_i\Phi-x_i\partial_j\Phi=0,
\]
because $\Phi$ is radial. Thus $\mathsf{T}_Mh=0$. For each fixed $x$, the quotient $h/M$ belongs to $\operatorname{span}\{1,v_1,v_2,v_3,|v|^2\}$, the space of collision invariants. The standard null-space characterization of the linearized collision operator therefore yields
\[
Q(M,h)+Q(h,M)=0.
\]
Consequently, $h=\Pi_Mf$ belongs to the kernel of the full linearized generator.
\end{proof}

\subsection{Reduction to the zero-moment problem}
Let $M$ be a fixed positive equilibrium. For arbitrary initial data $f_0$, Proposition \ref{onetoone2} gives the decomposition
\[
f_0=\Pi_Mf_0+(I-\Pi_M)f_0.
\]
The first term is stationary, while the second has zero mass, energy, and angular momentum. Therefore it suffices to establish decay for the solution issued from $(I-\Pi_M)f_0$. From this point through Section 7, $f$ denotes this zero-moment component. By Theorem \ref{EISform}, the reference equilibrium has the form
	\begin{equation*}
	M(x, v) = m\exp\{-\alpha(\Phi(x) + \frac{1}{2}|v|^2)+ Rx \cdot v\}, \qquad m,\alpha>0,\quad R\in\R^{3\wedge 3}.
\end{equation*}
In the case $p \in (1, 2)$, one has $R=0$. The zero-moment component satisfies
	\begin{equation}\label{Lineareqf+}
		\partial_tf + v \cdot \nabla_xf - \nabla_x\Phi \cdot \nabla_vf = {Q}(M, f) + {Q}(f, M).
	\end{equation}
The five moments are conserved by this equation and vanish for all time:
	\begin{equation}\label{012}
		\int_{\R^6}(1, \Phi(x) + \frac{1}{2}|v|^2, x \wedge v)f(t, x, v)dxdv = 0.
	\end{equation}
We now transform $M$ into a normalized, separable equilibrium and track the corresponding changes in \eqref{Lineareqf+}, \eqref{012}, and the weighted norms.

\subsection{Collision invariance under elementary transformations}
We first record how translations, scalings, and rotations act on the collision operator ${Q}$.
	\begin{lemma}\label{TransQ}
		For smooth functions $g = g(v)$ and  $h = h(v)$.

		\indent$\bullet$ Let $u \in \mathbb{R}^3$ and $g_1(v) = g(v + u), h_1(v) = h(v + u)$, then
		\begin{equation*}
			{Q}(g, h)(v + u) = {Q}(g_1, h_1)(v).
		\end{equation*}

		\indent$\bullet$ Let $r>0$ and $g_2(v) = g(rv), h_2(v) = h(rv)$, then
		\begin{equation*}
			{Q}(g, h)(rv) = r^{3 + \gamma}{Q}(g_2, h_2)(v).
		\end{equation*}

		\indent$\bullet$ Let U be a $3 \times 3$ orthogonal matrix and $g_3(v) = g(Uv), h_3(v) = h(Uv)$, then
		\begin{equation*}
			{Q}(g, h)(Uv) = {Q}(g_3, h_3)(v).
		\end{equation*}
	\end{lemma}
	\begin{proof}
		We prove it step by step.
		
		$\bullet$ For $g_1(v) = g(v + u), h_1(v) = h(v + u)$, by definition of the collision operator ${Q}$, we have
		\begin{multline*}
			{Q}(g, h)(v + u) = \int_{\R^3\times\S^2}\Big[g\Big(\frac{v + u + v_*}{2} - \frac{|v + u - v_*|}{2}\sigma\Big)h\Big(\frac{v + u + v_*}{2} - \frac{|v + u - v_*|}{2}\sigma\Big) \\- g(v_*)h(v + u)]|v + u - v_*|^\gamma b\Big(\frac{v + u - v_*}{|v + u - v_*|} \cdot \sigma\Big)d\sigma dv_*
		\end{multline*}
		Use change of variable $v_* \to v_* + u$ and by \eqref{v'v'*}, we have 
		\begin{equation*}
		\begin{aligned}
			{Q}(g, h)(v + u) &= \int_{\R^3\times\S^2}[g(v'_* + u)h(v' + u) - g(v_* + u)h(v + u)]|v - v_*|^\gamma b\Big(\frac{v - v_*}{|v - v_*|} \cdot \sigma\Big)d\sigma dv_* \\
			&= \int_{\R^3\times\S^2}[g_1(v'_*)h_1(v') - g_1(v_*)h_1(v)]|v - v_*|^\gamma b\Big(\frac{v - v_*}{|v - v_*|} \cdot \sigma\Big)d\sigma dv_* = {Q}(g_1, h_1)(v).
		\end{aligned}
		\end{equation*}

		$\bullet$ For $g_2(v) = g(rv), h_2(v) = h(rv)$, also by definition of ${Q}$, we have
		\begin{multline*}
			{Q}(g, h)(rv) = \int_{\R^3\times\S^2}[g\Big(\frac{rv + v_*}{2} - \frac{|rv - v_*|}{2}\sigma\Big)h\Big(\frac{rv + v_*}{2} - \frac{|rv - v_*|}{2}\sigma\Big) \\- g(v_*)h(rv)]|rv - v_*|^\gamma b\Big(\frac{rv - v_*}{|rv - v_*|} \cdot \sigma\Big)d\sigma dv_*
		\end{multline*}
		Use change of variable $v_* \to rv_*$ and \eqref{v'v'*}, we have 
		\begin{equation*}
		\begin{aligned}
			{Q}(g, h)(rv) &= \int_{\R^3\times\S^2}[g(rv'_*)h(rv') - g(rv_*)h(rv)]r^\gamma|v - v_*|^\gamma b\Big(\frac{v - v_*}{|v - v_*|} \cdot \sigma\Big)r^3d\sigma dv_* \\
			&= r^{3 + \gamma}\int_{\R^3\times\S^2}[g_2(v'_*)h_2(v') - g_2(v_*)h_2(v)]|v - v_*|^\gamma b\Big(\frac{v - v_*}{|v - v_*|} \cdot \sigma\Big)d\sigma dv_* = r^{3 + \gamma}{Q}(g_2, h_2)(v).
		\end{aligned}
		\end{equation*}

		$\bullet$ For $g_3(v) = g(Uv), h_3(v) = h(Uv)$, by definition of ${Q}$ we have
		\begin{multline*}
			{Q}(g, h)(Uv) = \int_{\R^3\times\S^2}\Big[g\Big(\frac{Uv + v_*}{2} - \frac{|Uv - v_*|}{2}\sigma\Big)h\Big(\frac{Uv + v_*}{2} - \frac{|Uv - v_*|}{2}\sigma\Big) \\- g(v_*)h(Uv)]|Uv - v_*|^\gamma b\Big(\frac{Uv - v_*}{|Uv - v_*|} \cdot \sigma\Big)d\sigma dv_*
		\end{multline*}
		Use change of variable $(\sigma, v_*) \to (U\sigma, Uv_*)$ and \eqref{v'v'*}, we have 
		\begin{equation*}
		\begin{aligned}
			{Q}(g, h)(Uv) &= \int_{\R^3\times\S^2}[g(Uv'_*)h(Uv') - g(Uv_*)h(Uv)]|v - v_*|^\gamma b\Big(\frac{v - v_*}{|v - v_*|} \cdot \sigma\Big)d\sigma dv_* \\
			&= \int_{\R^3\times\S^2}[g_3(v'_*)h_3(v') - g_3(v_*)h_3(v)]|v - v_*|^\gamma b\Big(\frac{v - v_*}{|v - v_*|} \cdot \sigma\Big)d\sigma dv_* = {Q}(g_3, h_3)(v).
		\end{aligned}
		\end{equation*}
		Then we complete the proof of this lemma.
	\end{proof}

%%%%%%%%%%%%%%%%%%%%%%%%%%%%%%%%%%%%%%%%%%%%%%%%%%
We first reduce \eqref{Lineareqf+} to a normalized form in the subquadratic case.
	\subsection{Reduction for $p\in(1,2)$}
	In this case $\Phi(x) = \frac{1}{p}\langle x \rangle^p$ and
	\begin{equation*}
		M(x, v) = m\exp\{-\alpha(\Phi(x) + \frac{1}{2}|v|^2)\},\qquad m,\alpha>0.
	\end{equation*}
	We first set
	\begin{equation}\label{f11}
			M_1(x, v) = M(x, \alpha^{-\frac{1}{2}}v) = me^{-\alpha\Phi(x) - \frac{1}{2}|v|^2}\quad\mbox{and}\quad		f_1(t, x, v) = f(t, x, \alpha^{-\frac{1}{2}}v).
	\end{equation}
	By derivative rules, we have
	\begin{equation*}
		\nabla_vf_1(t, x, v) = \alpha^{-\frac{1}{2}}\nabla_vf(t, x, \alpha^{-\frac{1}{2}}v) \:\:\: \Rightarrow \:\:\: \nabla_vf(t, x, \alpha^{-\frac{1}{2}}v) = \alpha^{\frac{1}{2}}\nabla_vf_1(t, x, v).
	\end{equation*}
Replacing $(t,x,v)$ by $(t,x,\alpha^{-1/2}v)$ in \eqref{Lineareqf+} and using Lemma \ref{TransQ}, we obtain
	\begin{equation*}
		\partial_tf_1 + \alpha^{-\frac{1}{2}}(v \cdot \nabla_xf_1 - \alpha\nabla_x\Phi \cdot \nabla_vf_1) = \alpha^{-\frac{3 + \gamma}{2}}({Q}(M_1, f_1) + {Q}(f_1, M_1)).
	\end{equation*}
	Next we set
	\begin{equation}\label{f222}
		M_2(x,v)=M_1(x,v),\quad f_2(t, x, v) = f_1(\alpha^{\frac{1}{2}}t, x, v),
	\end{equation}
	and derive the equation for $f_2$ is
	\begin{equation}\label{Lineareq<2}
		\partial_tf_2 + v \cdot \nabla_xf_2 - \alpha\nabla_x\Phi \cdot \nabla_vf_2 = \alpha^{-1 - \frac{\gamma}{2}}({Q}(M_2, f_2) + {Q}(f_2, M_2)).
	\end{equation}
	Recall from \eqref{012} that our conservation condition for $f$ is 
	\begin{equation*}
		\int_{\R^6}(1, \Phi(x) + \frac{1}{2}|v|^2, x \wedge v)f(t, x, v)dxdv = 0, \:\:\: \forall t \ge 0,
	\end{equation*}
	and $f_1(t, x, v) = f(t, x, \alpha^{-\frac{1}{2}}v)$. By using change of variables, we have
	\begin{equation*}
	\begin{aligned}
		0 &= \int_{\R^6} f(t, x, v)dxdv = \alpha^{-\frac{3}{2}}\int_{\R^6} f_1(t, x, v)dxdv;\\
		0 &= \int_{\R^6}(\Phi(x) + \frac{1}{2}|v|^2)f(t, x, v)dxdv = \alpha^{-\frac{5}{2}}\int_{\R^6}(\alpha\Phi(x) + \frac{1}{2}|v|^2)f_1(t, x, v)dxdv;\\
		0 &= \int_{\R^6}(x \wedge v)f(t, x, v)dxdv = \alpha^{-2}\int_{\R^6}(x \wedge v)f_1(t, x, v)dxdv.
	\end{aligned}
	\end{equation*}
	In general, we have
	\begin{equation*}
		\int_{\R^6}(1, \alpha\Phi(x) + \frac{1}{2}|v|^2, x \wedge v)f_1(t, x, v)dxdv = 0, \:\:\: \forall t \ge 0.
	\end{equation*}
	Furthermore, we get from \eqref{f222} that
	\begin{equation}\label{f2}
		\int_{\R^6}(1, \alpha\Phi(x) + \frac{1}{2}|v|^2, x \wedge v)f_2(t, x, v)dxdv = 0, \:\:\: \forall t \ge 0.
	\end{equation}
	In \eqref{Lineareq<2}, we denote
	\begin{equation}\label{tildephi<2}
		\widetilde{\Phi} = \widetilde{\Phi}(x) = \alpha\Phi(x) = \frac{\alpha}{p}\langle x \rangle^p, \:\:\: \mu = \mu(v) = (2\pi)^{-3/2}e^{-\frac{1}{2}|v|^2}.
	\end{equation}
	Then $M_2(x, v)$ is in the form of variable separation:
	\begin{equation*}
		M_2(x, v) = (2\pi)^{3/2}me^{-\widetilde{\Phi}(x)}\mu(v).
	\end{equation*}
	We redefine \( f_2 \) and \(M_2\) as \( g \) and $\mathcal{M}$, then from \eqref{Lineareq<2}, the linear equation for $g$ reads like
	\begin{equation}\label{Lineareqg<2}
		\partial_tg + v \cdot \nabla_xg - \nabla_x\widetilde{\Phi} \cdot \nabla_vg = \al^{-1-\f \ga 2}(Q(\mathcal{M},g)+Q(g,\mathcal{M}))=(2\pi)^{3/2}m \al^{-1-\f \gamma 2}e^{-\widetilde{\Phi}}\mathsf{L}g,
	\end{equation}
	where the linear operator
	\begin{equation}\label{sfL}
		\mathsf{L} g := {Q}(\mu, g) + {Q}(g, \mu).
	\end{equation}
	%Then the linear equation for $g$ reads like
	%\begin{equation}
	%	\partial_tg + v \cdot \nabla_xg - \nabla_x\widetilde{\Phi} \cdot \nabla_vg = C_Me^{-\widetilde{\Phi}} L g.
	%\end{equation}
	And from \eqref{f2}, conservation conditions becomes
	\begin{equation}\label{Conserveg<2}
		\int_{\R^6}(1, \widetilde{\Phi}(x) + \frac{1}{2}|v|^2, x \wedge v)g(t, x, v)dxdv = 0, \:\:\: \forall t \ge 0.
	\end{equation}
	Thus, we conclude from \eqref{f11} and \eqref{f222} that if we let $\mathcal{M}(x,v)=M(x,\al^{-\f12}v)$ and \( g(t, x, v) = f(\alpha^{\frac{1}{2}} t, x, \alpha^{-\frac{1}{2}} v) \), then \( g \) satisfies linear equation \eqref{Lineareqg<2} and the conservation laws \eqref{Conserveg<2} with $\widetilde{\Phi}$ defined in \eqref{tildephi<2}.

%	To prove Theorem \ref{Conv}, we find that for any $\lambda \in (\frac{1}{2}, 1)$, there holds
%	\begin{equation*}
%		\|M^{-\lambda}f(t)\|_{L_{x, v}^2} \sim \|M_1^{-\lambda}f_1(t)\|_{L_{x, v}^2} \sim \|e^{\lambda(\widetilde{\Phi}(x) + \frac{1}{2}|v|^2)}g(\alpha^{-\frac{1}{2}}t)\|_{L_{x, v}^2}.
%	\end{equation*}

%%%%%%%%%%%%%%%%%%%%%%%%%%%%%%%%%%%%%%%%%%%%%%%%%%

	\subsection{Reduction for $p\in(2,\infty)$}
	In this case $\Phi(x) = \frac{1}{p}|x|^p$ and
	\begin{equation*}
	M(x, v) = m\exp\{-\alpha(\Phi(x) + \frac{1}{2}|v|^2)+ Rx \cdot v\},\qquad m,\alpha>0,\quad R\in\R^{3\wedge 3}.
	\end{equation*}
We will complete the reduction in several steps.
%%%%%%%%%%%%%%%%%%%%%%%%%%%%%%%%%%%%%%%%%%%%%%%%%%

	\bigskip\noindent$\bullet$ \textit{Translation.}
	Note that
	\begin{equation*}
		-\alpha(\Phi(x) + \frac{1}{2}|v|^2)+ Rx \cdot v = -\frac{\alpha}{2}\Big|v - \frac{Rx}{\alpha}\Big|^2 - \frac{\alpha}{p}|x|^p + \frac{1}{2\alpha}|Rx|^2.
	\end{equation*}
	We set
	\begin{equation}\label{m1}
		M_1(x, v) = M(x, v + \frac{Rx}{\alpha}) = me^{-\frac{\alpha}{2}|v|^2 - \frac{\alpha}{p}|x|^p + \frac{1}{2\alpha}|Rx|^2},
	\end{equation}
	and
	\begin{equation}\label{f1}
		f_1(t, x, v) = f(t, x, v + \frac{Rx}{\alpha}).
	\end{equation}
	By derivative rules we have
	\begin{equation*}
	\begin{gathered}
		\partial_tf_1(t, x, v) = \partial_tf(t, x, v + \frac{Rx}{\alpha}); \:\:\: \nabla_vf_1(t, x, v) = \nabla_vf(t, x, v + \frac{Rx}{\alpha});\\
		\nabla_xf_1(t, x, v) = \nabla_xf(t, x, v + \frac{Rx}{\alpha}) - \frac{R\nabla_v}{\alpha}f(t, x, v + \frac{Rx}{\alpha}),
	\end{gathered}
	\end{equation*}
	which leads to
	\begin{equation*}
	\begin{gathered}
		\partial_tf(t, x, v + \frac{Rx}{\alpha}) = \partial_tf_1(t, x, v); \:\:\: \nabla_vf(t, x, v + \frac{Rx}{\alpha}) = \nabla_vf_1(t, x, v);\\
		\nabla_xf(t, x, v + \frac{Rx}{\alpha}) = \nabla_xf_1(t, x, v) + \frac{R\nabla_v}{\alpha}f_1(t, x, v).
	\end{gathered}
	\end{equation*}
	Replacing $(t,x,v)$ by $(t,x,v+Rx/\alpha)$ in \eqref{Lineareqf+} and applying Lemma \ref{TransQ}, we obtain
	\begin{equation*}
		\partial_tf_1 + (v + \frac{Rx}{\alpha}) \cdot (\nabla_xf_1 + \frac{R\nabla_v}{\alpha}f_1) - |x|^{p - 2}x \cdot \nabla_vf_1 = {Q}(f_1, M_1) + {Q}(M_1, f_1),
	\end{equation*}
	and further derive the equation for $f_1$:
	\begin{multline}\label{Lineareqf1}
		\partial_tf_1 + v \cdot \nabla_xf_1 + \alpha^{-1}Rx \cdot \nabla_xf_1 - \alpha^{-1}Rv \cdot \nabla_vf_1 - (|x|^{p - 2}x + \alpha^{-2}R^2x) \cdot \nabla_vf_1 \\= {Q}(M_1, f_1) + {Q}(f_1, M_1).
	\end{multline}
	Here we have used the facts
	\begin{equation*}
		v \cdot R\nabla_v = -Rv \cdot \nabla_v, \:\:\: Rx \cdot R\nabla_v = -R^2x \cdot \nabla_v
	\end{equation*}
	since $R$ is skew-symmetric.

%%%%%%%%%%%%%%%%%%%%%%%%%%%%%%%%%%%%%%%%%%%%%%%%%%

	\bigskip\noindent$\bullet$ \textit{Scaling transformation}.
	Recall from \eqref{m1} that
	\begin{equation*}
		M_1(x, v) = me^{-\frac{\alpha}{2}|v|^2 - \frac{\alpha}{p}|x|^p + \frac{1}{2\alpha}|Rx|^2},
	\end{equation*}
	we set
	\begin{equation}\label{m22}
		M_2(x, v) = M_1(\alpha^{-\frac{1}{p}}x, \alpha^{-\frac{1}{2}}v) = me^{-\frac{1}{2}|v|^2 - \frac{1}{p}|x|^p + \frac{|Rx|^2}{2\alpha^{1 + 2/p}}},
	\end{equation}
	and
	\begin{equation}\label{f22}
		f_2(t, x, v) = f_1(t, \alpha^{-\frac{1}{p}}x, \alpha^{-\frac{1}{2}}v).
	\end{equation}
	By derivative rules we have
	\begin{equation*}
		\nabla_xf_2(t, x, v) = \alpha^{-\frac{1}{p}}\nabla_xf_1(t, \alpha^{-\frac{1}{p}}x, \alpha^{-\frac{1}{2}}v); \:\:\: \nabla_vf_2(t, x, v) = \alpha^{-\frac{1}{2}}\nabla_vf_1(t, \alpha^{-\frac{1}{p}}x, \alpha^{-\frac{1}{2}}v),
	\end{equation*}
	which leads to
	\begin{equation*}
		\nabla_xf_1(t, \alpha^{-\frac{1}{p}}x, \alpha^{-\frac{1}{2}}v) = \alpha^{\frac{1}{p}}\nabla_xf_2(t, x, v); \:\:\: \nabla_vf_1(t, \alpha^{-\frac{1}{p}}x, \alpha^{-\frac{1}{2}}v) = \alpha^{\frac{1}{2}}\nabla_vf_2(t, x, v).
	\end{equation*}
	Replacing $(t,x,v)$ by $(t,\alpha^{-1/p}x,\alpha^{-1/2}v)$ in \eqref{Lineareqf1} and applying Lemma \ref{TransQ}, we obtain
	\begin{multline*}
		\partial_tf_2 + \alpha^{-\frac{1}{2} + \frac{1}{p}}v \cdot \nabla_xf_2 + \alpha^{-1}Rx \cdot \nabla_xf_2 - \alpha^{-1}Rv \cdot \nabla_vf_2 - (\alpha^{-\frac{1}{2} + \frac{1}{p}}|x|^{p - 2}x + \alpha^{-\frac{3}{2} - \frac{1}{p}}R^2x) \cdot \nabla_vf_2 \\= \alpha^{-\frac{3 + \gamma}{2}}({Q}(M_2, f_2) + {Q}(f_2, M_2)).
	\end{multline*}
	Now we set
	\begin{equation}\label{betas1}
		\beta = \alpha^{-\frac{1}{2} + \frac{1}{p}}, \:\:\: R_1 = \alpha^{-\frac{1}{2} - \frac{1}{p}}R,
	\end{equation}
	then the above equation for $f_2$ becomes
	\begin{multline}\label{Lineareqf2}
		\partial_tf_2 + \beta(v \cdot \nabla_xf_2 + R_1x \cdot \nabla_xf_2 - R_1v \cdot \nabla_vf_2 - (|x|^{p - 2}x + R_1^2x) \cdot \nabla_vf_2) \\= \alpha^{-\frac{3 + \gamma}{2}}({Q}(M_2, f_2) + {Q}(f_2, M_2)).
	\end{multline}

%%%%%%%%%%%%%%%%%%%%%%%%%%%%%%%%%%%%%%%%%%%%%%%%%%

	\bigskip\noindent$\bullet$ \textit{Rotation transformation}.
	Recall from \eqref{m22} that
	\begin{equation*}
		M_2(x, v) = me^{-\frac{1}{2}|v|^2 - \frac{1}{p}|x|^p + \frac{|Rx|^2}{2\alpha^{1 + 2/p}}} = me^{-\frac{1}{2}|v|^2 - \frac{1}{p}|x|^p + \frac{1}{2}|R_1x|^2}.
	\end{equation*}
	There exists an orthogonal matrix $U$ such that
	\begin{equation}\label{defiU}
		U^\tau R_1U = \begin{pmatrix} 0 & r & 0 \\ -r & 0 & 0 \\ 0 & 0 & 0 \end{pmatrix} := \mathsf{R}, \: R_1 = U\mathsf{R}U^\tau \: \text{with} \: r \ge 0.
	\end{equation}
	Note that
	\begin{equation*}
		|R_1Ux|^2 = x^\tau U^\tau R_1^\tau R_1Ux = x^\tau \mathsf{R}^\tau \mathsf{R}x = |\mathsf{R}x|^2.
	\end{equation*}
	we set
	\begin{equation}\label{m33}
		M_3(x, v) = M_2(Ux, Uv) = me^{-\frac{1}{2}|v|^2 - \frac{1}{p}|x|^p + \frac{1}{2}|\mathsf{R}x|^2},
	\end{equation}
	and
	\begin{equation}\label{f33}
		f_3(t, x, v) = f_2(t, Ux, Uv).
	\end{equation}
	By derivative rules we have
	\begin{equation*}
		\nabla_xf_3(t, x, v) = U^\tau\nabla_xf_2(t, Ux, Uv); \:\:\: \nabla_vf_3(t, x, v) = U^\tau\nabla_vf_2(t, Ux, Uv);
	\end{equation*}
	and leads to
	\begin{equation*}
		\nabla_xf_2(t, Ux, Uv) = U\nabla_xf_3(t, x, v); \:\:\: \nabla_vf_2(t, Ux, Uv) = U\nabla_vf_3(t, x, v).
	\end{equation*}
	Replacing $(t,x,v)$ by $(t,Ux,Uv)$ in \eqref{Lineareqf2} and applying Lemma \ref{TransQ}, we obtain
	\begin{multline*}
		\partial_tf_3 + \beta(Uv \cdot U\nabla_xf_3 + R_1Ux \cdot U\nabla_xf_3 - R_1Uv \cdot U\nabla_vf_3 - (|x|^{p - 2}Ux + R_1^2Ux) \cdot U\nabla_vf_3) \\= \alpha^{-\frac{3 + \gamma}{2}}({Q}(M_3, f_3) + {Q}(f_3, M_3)).
	\end{multline*}
	Note the fact that
	\begin{equation*}
		Uv \cdot U\nabla_x = v^\tau U^\tau U\nabla_x = v^\tau\nabla_x = v \cdot \nabla_x,\quad Ux \cdot U\nabla_v = x \cdot \nabla_v
	\end{equation*}
	and
	\begin{equation*}
		\begin{aligned}
		&R_1Ux \cdot U\nabla_x = x^\tau U^\tau R_1^\tau U\nabla_x = x^\tau \mathsf{R}^\tau\nabla_x = \mathsf{R}x \cdot \nabla_x,\\
		&R_1Uv \cdot U\nabla_v = \mathsf{R}v \cdot \nabla_v, R_1^2Ux \cdot U\nabla_v = \mathsf{R}^2x \cdot \nabla_v.
		\end{aligned}
	\end{equation*}
Now we derive the equation for $f_3$:
	\begin{multline}\label{Lineareqf3}
		\partial_tf_3 + \beta(v \cdot \nabla_xf_3 + \mathsf{R}x \cdot \nabla_xf_3 - \mathsf{R}v \cdot \nabla_vf_3 - (|x|^{p - 2}x + \mathsf{R}^2x) \cdot \nabla_vf_3) \\= \alpha^{-\frac{3 + \gamma}{2}}({Q}(M_3, f_3) + {Q}(f_3, M_3)).
	\end{multline}

%%%%%%%%%%%%%%%%%%%%%%%%%%%%%%%%%%%%%%%%%%%%%%%%%%

	\bigskip\noindent$\bullet$ \textit{Normalization}.
	Recall from \eqref{betas1} that $\beta = \alpha^{-\frac{1}{2} + \frac{1}{p}}$, we set
	\begin{equation}\label{f44}
		M_4(x,v)=M_3(x,v),\quad f_4(t, x, v) = f_3(\beta^{-1}t, x, v) = f_3(\alpha^{\frac{1}{2} - \frac{1}{p}}t, x, v),
	\end{equation}
	and derive the equation for $f_4$:
	\begin{multline}\label{Lineareqf4}
		\partial_tf_4 + v \cdot \nabla_xf_4 + \mathsf{R}x \cdot \nabla_xf_4 - \mathsf{R}v \cdot \nabla_vf_4 - (|x|^{p - 2}x + \mathsf{R}^2x) \cdot \nabla_vf_4 \\= \alpha^{-1 - \frac{1}{p} - \frac{\gamma}{2}}({Q}(M_4, f_4) + {Q}(f_4, M_4)).
	\end{multline}

%%%%%%%%%%%%%%%%%%%%%%%%%%%%%%%%%%%%%%%%%%%%%%%%%%
We also track the transformed conserved quantities. Recall from \eqref{012} that the zero-moment condition for $f$ is
	\begin{equation*}
		\int_{\R^6}(1, \Phi(x) + \frac{1}{2}|v|^2, x \wedge v)f(t, x, v)dxdv = 0, \:\:\: \forall t \ge 0,
	\end{equation*}
	and relations \eqref{f1}\eqref{f22}\eqref{f33} and \eqref{f44} between $f, f_1, f_2$ and $ f_3$
	\begin{equation}\label{081}
		\begin{aligned}
		&f_1(t, x, v) = f(t, x, v + \frac{Rx}{\alpha}), \:\:\: f_2(t, x, v) = f_1(t, \alpha^{-\frac{1}{p}}x, \alpha^{-\frac{1}{2}}v), \\
		&\:\:\: f_3(t, x, v) = f_2(t, Ux, Uv),\:\:\: f_4(t, x, v) = f_3(\alpha^{\frac{1}{2} - \frac{1}{p}}t, x, v).
		\end{aligned}
	\end{equation}
	By using change of variables three times, we have
		\begin{equation*}
		\begin{aligned}
			0 &= \int_{\R^6} f(t, x, v)dxdv
			 = \int_{\R^6} f_1(t, x, v)dxdv \\
			&= \alpha^{-\frac{3}{2} - \frac{3}{p}}\int_{\R^6} f_2(t, x, v)dxdv
			 = \alpha^{-\frac{3}{2} - \frac{3}{p}}\int_{\R^6} f_3(t, x, v)dxdv;\\
		0 &= \int_{\R^6}(\Phi(x) + \frac{1}{2}|v|^2)f(t, x, v)dxdv
		= \int_{\R^6}(\frac{1}{p}|x|^p + \frac{1}{2}|v + \frac{Rx}{\alpha}|^2)f_1(t, x, v)dxdv \\
		&= \alpha^{-\frac{5}{2} - \frac{3}{p}}
		\int_{\R^6}(\frac{1}{p}|x|^p + \frac{1}{2}|v + R_1x|^2)f_2(t, x, v)dxdv\\
		&= \alpha^{-\frac{5}{2} - \frac{3}{p}}
		\int_{\R^6}(\frac{1}{p}|x|^p + \frac{1}{2}|Uv + R_1Ux|^2)f_3(t, x, v)dxdv.
	\end{aligned}
	\end{equation*}
	One may check that
	\begin{equation*}
		|Uv + R_1Ux|^2 = |Uv|^2 + 2Uv \cdot R_1Ux + |R_1Ux|^2 = |v|^2 + 2v \cdot \mathsf{R}x + |\mathsf{R}x|^2 = |v + \mathsf{R}x|^2.
	\end{equation*}
	Thus we have
	\begin{equation*}
		0 = \int_{\R^6}(\frac{1}{p}|x|^p + \frac{1}{2}|v + \mathsf{R}x|^2)f_3(t, x, v)dxdv.
	\end{equation*}
	Finally,
	\begin{multline*}
		0 = \int_{\R^6}(x \wedge v)f(t, x, v)dxdv = \int_{\R^6}(x \wedge (v + \frac{Rx}{\alpha}))f_1(t, x, v)dxdv \\
		= \alpha^{-2 - \frac{4}{p}}\int_{\R^6}(x \wedge (v + R_1x))f_2(t, x, v)dxdv
		= \alpha^{-2 - \frac{4}{p}}\int_{\R^6}(Ux \wedge (Uv + R_1Ux))f_3(t, x, v)dxdv,
	\end{multline*}
	and we compute
	\begin{equation*}
	\begin{gathered}
		Ux \wedge Uv = Uxv^\tau U^\tau - Uvx^\tau U^\tau = U(xv^\tau - vx^\tau)U^\tau = U(x \wedge v)U^\tau;\\
		Ux \wedge R_1Ux = Uxx^\tau U^\tau R_1^\tau - R_1Uxx^\tau U^\tau = Uxx^\tau U^\tau U\mathsf{R}^\tau U^\tau - U\mathsf{R}U^\tau Uxx^\tau U^\tau \\= Uxx^\tau \mathsf{R}^\tau U^\tau - U\mathsf{R}xx^\tau U^\tau = U(xx^\tau \mathsf{R}^\tau - \mathsf{R}xx^\tau)U^\tau = U(x \wedge \mathsf{R}x)U^\tau.
	\end{gathered}
	\end{equation*}
	And the third condition becomes
	\begin{equation*}
		0 = \int_{\R^6}(x \wedge (v + \mathsf{R}x))f_3(t, x, v)dxdv.
	\end{equation*}
	In general, we have
	\begin{equation*}
		\int_{\R^6}(1, \frac{1}{p}|x|^p + \frac{1}{2}|v + \mathsf{R}x|^2, x \wedge (v + \mathsf{R}x))f_3(t, x, v)dxdv = 0, \:\:\: \forall t \ge 0.
	\end{equation*}
Since $f_4(t, x, v) = f_3(\alpha^{\frac{1}{2} - \frac{1}{p}}t, x, v)$, at last we get
	\begin{equation}\label{conf4}
		\int_{\R^6}(1, \frac{1}{p}|x|^p + \frac{1}{2}|v + \mathsf{R}x|^2, x \wedge (v + \mathsf{R}x))f_4(t, x, v)dxdv = 0, \:\:\: \forall t \ge 0.
	\end{equation}

%%%%%%%%%%%%%%%%%%%%%%%%%%%%%%%%%%%%%%%%%%%%%%%%%%

	\bigskip\noindent$\bullet$ \textit{Summary of the reduction for the case $p \in (2, \infty)$}.
	We denote
	\begin{equation}\label{tildephi>2}
	\begin{gathered}
		\widetilde{\Phi} = \widetilde{\Phi}(x) = \frac{1}{p}|x|^p - \frac{1}{2}|\mathsf{R}x|^2 = \frac{1}{p}|x|^p - \frac{r^2}{2}(x_1^2 + x_2^2);\\
		\mu = \mu(v) = (2\pi)^{-3/2}e^{-\frac{1}{2}|v|^2}.
	\end{gathered}
	\end{equation}
	Then from \eqref{m33} and \eqref{f44}, we immediately have
	\begin{equation*}
		M_4(x, v) = (2\pi)^{3/2}me^{-\widetilde{\Phi}(x)}\mu(v),
	\end{equation*}
	we emphasize that $M_4(x, v)$ is in the form of variable separation. Besides, note that $\nabla_x\widetilde{\Phi} = |x|^{p - 2}x + \mathsf{R}^2x$, we redefine \( f_4 \) and $M_4$ as \( g \) and $\mathcal{M}$
	in \eqref{Lineareqf4}, then the linear equation for $g$ reads like
	\begin{equation}\label{Lineareqg>2}
		\begin{aligned}
		\partial_tg + v \cdot \nabla_xg + \mathsf{R}x \cdot \nabla_xg - \mathsf{R}v \cdot \nabla_vg - \nabla_x\widetilde{\Phi} \cdot \nabla_vg &= \alpha^{-1 - \frac{1}{p} - \frac{\gamma}{2}}(Q(\mathcal{M},g)+Q(g,\mathcal{M})) \\
		&=(2\pi)^{3/2}m\alpha^{-1 - \frac{1}{p} - \frac{\gamma}{2}}e^{-\widetilde{\Phi}}\mathsf{L}g.
		\end{aligned}
	\end{equation}
	%And denote linear operator
%	\begin{equation*}
	%	\mathsf{L} g := {Q}(\mu, g) + {Q}(g, \mu).
%	\end{equation*}
%	\begin{equation}
	%	\partial_tg + v \cdot \nabla_xg + Rx \cdot \nabla_xg - Rv \cdot \nabla_vg - \nabla_x\widetilde{\Phi} \cdot \nabla_vg = C_Me^{-\widetilde{\Phi}}\mathsf{L}g,
%	\end{equation}
And we write conservation conditions from \eqref{conf4} as
	\begin{equation}\label{Conserveg>2}
		\int_{\R^6}(1, \widetilde{\Phi}(x) + |	\mathsf{R}x|^2 + 	\mathsf{R}x \cdot v + \frac{1}{2}|v|^2, x \wedge (v + \mathsf{R}x))g(t, x, v)dxdv = 0, \:\:\: \forall t \ge 0.
	\end{equation}
Thus, we conclude from \eqref{081} that if we let $\mathcal{M}(x,v)=M(\al^{-\f1p}Ux,\al^{-\f12}Uv+\al^{-\f1p-1}RUx)$ and \( g(t, x, v) = f(\al^{\f12-\f1p}t,\al^{-\f1p}Ux,\al^{-\f12}Uv+\al^{-\f1p-1}RUx)\) with $U$ defined in \eqref{defiU}, then \( g \) satisfies linear equation \eqref{Lineareqg>2} and the conservation laws \eqref{Conserveg>2} with $\widetilde{\Phi}$ defined in \eqref{tildephi>2}.

%	To prove theorem \ref{Conv}, we find that for any $\lambda \in (\frac{1}{2}, 1)$, there holds
%	\begin{equation*}
%		\|M^{-\lambda}f(t)\|_{L_{x, v}^2} \sim \|M_1^{-\lambda}f_1(t)\|_{L_{x, v}^2} \sim \|M_2^{-\lambda}f_2(t)\|_{L_{x, v}^2} \sim \|M_3^{-\lambda}f_3(t)\|_{L_{x, v}^2} \sim \|e^{\lambda(\widetilde{\Phi}(x) + \frac{1}{2}|v|^2)}g(\alpha^{-\frac{1}{2} + \frac{1}{p}}t)\|_{L_{x, v}^2}.
%	\end{equation*}

%%%%%%%%%%%%%%%%%%%%%%%%%%%%%%%%%%%%%%%%%%%%%%%%%%

	\subsection{Summary of the reduction}\label{2000}
	Combining \eqref{Lineareqg<2} and \eqref{Lineareqg>2}, define
	\begin{equation}\label{gform}
		\begin{aligned}
			&g(t, x, v) = f(\alpha^{\frac{1}{2}} t, x, \alpha^{-\frac{1}{2}} v),~~p\in(1,2);\\
			&g(t, x, v) = f(\al^{\f12-\f1p}t,\al^{-\f1p}Ux,\al^{-\f12}Uv+\al^{-\f1p-1}RUx),~~p\in(2,\infty).
		\end{aligned}
	\end{equation}
Then $g$ satisfies
	\begin{equation}\label{Lineareqg}
		\partial_tg + v \cdot \nabla_xg + \mathsf{R}x \cdot \nabla_xg - \mathsf{R}v \cdot \nabla_vg - \nabla_x\widetilde{\Phi} \cdot \nabla_vg = C_Me^{-\widetilde{\Phi}}\mathsf{L}g,
	\end{equation}
	where $\mathsf{L}g=Q(\mu,g)+Q(g,\mu)$. The transformed potential $\widetilde{\Phi}$ is defined by
		\begin{equation}\label{013}
			\widetilde{\Phi}(x) = \frac{\alpha}{p}\langle x \rangle^p, \quad p \in (1, 2); \qquad \widetilde{\Phi}(x) = \frac{1}{p}|x|^p - \frac{1}{2}|\mathsf{R}x|^2, \quad p \in (2, \infty).
	\end{equation}
	with
			\begin{equation}\label{Rform}
			\mathsf{R} = r\begin{pmatrix} 0 & 1 & 0 \\ -1 & 0 & 0 \\ 0 & 0 & 0 \end{pmatrix},~~r=0~~\mbox{ for the case }p \in (1, 2).
		\end{equation}
The constant $C_M$ is
\begin{equation}\label{083}
C_M=(2\pi)^{3/2}m \al^{-1-\f \gamma 2},~~p\in(1,2);\quad C_M=(2\pi)^{3/2}m\alpha^{-1 - \frac{1}{p} - \frac{\gamma}{2}},~~p\in(2,\infty).
\end{equation} 
	%Where $C_M=$ is a constant determined by entropy-invariant solution $M(x, v)$, and
%Function $\widetilde{\Phi}(x)$ is defined differently in different cases: 
	The transformed conservation conditions follow from \eqref{Conserveg<2} and \eqref{Conserveg>2}. In both cases,
	\begin{equation}\label{Conserveg}
		\int_{\R^6}(1, \widetilde{\Phi}(x) + |	\mathsf{R}x|^2 + 	\mathsf{R}x \cdot v + \frac{1}{2}|v|^2, x \wedge (v + 	\mathsf{R}x))g(t, x, v)dxdv = 0, \:\:\: \forall t \ge 0.
	\end{equation}
	
\bigskip The same invertible coordinate transformation is applied to $M$, $\mathcal M$, $f$, and $g$; its Jacobian and the equivalence constants of the weighted norms depend only on the equilibrium parameters. Consequently, the original zero-moment problem is equivalent to the normalized problem stated in Theorem \ref{082} and proved in Section 7.

%%%%%%%%%%%%%%%%%%%%%%%%%%%%%%%%%%%%%%%%%%%%%%%%%%

\section{Spatially degenerate microscopic coercivity and far-field weight transfer}
The velocity-space coercivity and upper bounds for the linearized collision operator are standard inputs. The new difficulty is that the normalized equation multiplies this dissipation by $e^{-\widetilde\Phi(x)}$, which becomes arbitrarily small in the far field. We first record the collision estimates in the precise weighted form used below and then prove the weight-transfer inequalities that compensate for this spatial degeneracy.

\subsection{The linearized collision operator}
Recall that the linear Boltzmann operator is defined in \eqref{sfL}. We use the standard coercivity and upper-bound estimates for cutoff and non-cutoff kernels. For a smooth function $g=g(x,v)$, define its macroscopic part by
\begin{equation}\label{Defig0}
	\begin{aligned}
		g_0 &:= \Big(\int_{\R^3} gdv\Big)\mu + \sum_{i = 1}^3\Big(\int_{\R^3} v_igdv\Big)v_i\mu + \Big(\int_{\R^3}\frac{|v|^2 - 3}{6}gdv\Big)(|v|^2 - 3)\mu \\
		&:= a(t, x)\mu + v \cdot b(t, x)\mu + (|v|^2 - 3)c(t, x)\mu.
	\end{aligned}
\end{equation}
The remainder $g_1:=g-g_0$ is the microscopic part. Throughout this section, $\normm{\cdot}$ denotes a collision-dissipation norm. In the cutoff case it is equivalent to $\|\cdot\|_{L^2_{\gamma/2}}$. In the non-cutoff case we use the standard anisotropic norm, which is equivalent, up to the usual choice of polynomial weights, to
\[
\|h\|_{L^2_{\gamma/2+s}}
+\|h\|_{H^s_{\gamma/2}}
+\|(-\Delta_{\mathbb S^2})^{s/2}h\|_{L^2_{\gamma/2}}.
\]
Only the coercivity and upper-bound properties stated below are used in the proof. The following estimates are standard for the collision kernels in $\mathbf{(A1)}$--$\mathbf{(A3)}$; see \cite{alexandre_global_2011,gressman_global_2011,duan_global_2021,MR2013332,MR2209761,MR2366140}.

\begin{lemma}\label{UpperBoundL1}
	For every smooth function $g=g(v)$, there exists a constant $c>0$ such that
	\begin{equation*}
		(\mathsf{L}g, e^{\frac{1}{2}|v|^2}g)_{L_v^2} \le -c\normm{e^{\frac{1}{4}|v|^2}g_1}^2.
	\end{equation*}
\end{lemma}

Additionally, we have the following upper bound estimates:
\begin{lemma}\label{UpperBoundL2}
	For smooth functions $g=g(v)$ and $f=f(v)$ and for $\lambda\in(1/2,1)$, there exist constants $c=c(\lambda)>0$ and $C=C(\lambda)>0$ such that
	\begin{equation*}
		(\mathsf{L}g, e^{\lambda|v|^2}g)_{L_v^2} \le -c\normm{e^{\frac{\lambda}{2}|v|^2}g}^2 + C\|g\|_{L_v^2}^2,\quad |(\mathsf{L}g, e^{\lambda|v|^2}f)_{L_v^2}|\le C\normm{e^{\frac{\lambda}{2}|v|^2}g}\normm{e^{\frac{\lambda}{2}|v|^2} f}.
	\end{equation*}
\end{lemma}

\begin{lemma}\label{UpperBoundL0}
	For every smooth function $g=g(v)$ and every $\delta,l>0$, there exists $C(\delta,l)>0$ such that
	\begin{equation*}
		|(\mathsf{L}g, \<v\>^l)_{L_v^2}| \le C(\delta, l)\|e^{\delta|v|^2}g\|_{L_v^2}.
	\end{equation*}
\end{lemma}
\begin{proof}
We provide only a concise sketch of the proof for the non‑cutoff case. By definition, we have  
\[
|(\mathsf{L}g, \langle v\rangle^l)_{L_v^2}| \le |(Q(\mu,g),\langle v\rangle^l)_{L^2_v}| + |(Q(g,\mu),\langle v\rangle^l)_{L^2_v}|.
\]
Applying the phase and frequency decompositions from \cite{he_sharp_2018}, we may increase the weights at the first and second positions so as to reduce the weight at the third position, and assign all $2s$‑order derivatives to the third position. Consequently, for the first and second terms we obtain  
\[
\begin{aligned}
|(Q(\mu,g),\langle v\rangle^l)_{L^2_v}| &\lesssim C_l \,\|\mu^{1/2}\|_{L^2} \, \|\langle v\rangle^{l+10} g\|_{L^2} \, \|\langle v\rangle^{-10}\|_{H^{2s}_{\gamma+2s}} \\
&\lesssim C(\delta,l) \, \| e^{\delta |v|^2} g\|_{L^2}, \\[4pt]
|(Q(g,\mu),\langle v\rangle^l)_{L^2_v}| &\lesssim C_{\delta,l} \, \|e^{\f\delta 2|v|^2} g\| \, \|\mu^{1/2}\|_{L^2} \, \|\langle v\rangle^{-10}\|_{H^{2s}_{\gamma+2s}} \\
&\lesssim C(\delta,l) \, \| e^{\delta |v|^2} g\|_{L^2},
\end{aligned}
\]  
which yields the desired estimate.
\end{proof}

%%%%%%%%%%%%%%%%%%%%%%%%%%%%%%%%%%%%%%%%%%%%%%%%%%
\medskip
\subsection{Transfer of velocity weights to confinement weights}
We first introduce the notation
\begin{equation*}
	\mathsf{T} := v \cdot \nabla_x + \mathsf{R}x \cdot \nabla_x - \mathsf{R}v \cdot \nabla_v - \nabla_x\widetilde{\Phi} \cdot \nabla_v.
\end{equation*}
The following weight-transfer estimates are the main result of this subsection.
\begin{lemma}[The case $p \in (1, 2)$]\label{WeightTransferLemma1}
	For every $\lambda\in(1/2,1)$ and $0<\delta<1/3-1/(p+2)$, there exist constants $N\ge1$ and $C_1>0$ such that
	\begin{equation}\label{WeightTransferEq1}
		N\langle v \rangle^{-6}e^{(2\lambda - 1)\widetilde{\Phi} + \lambda|v|^2} - \mathsf{T}(x \cdot v)e^{2(\lambda - \frac{1}{p + 2} - \delta)(\widetilde{\Phi} + \frac{1}{2}|v|^2)} \ge C_1e^{2(\lambda - \frac{1}{p + 2} - \delta)(\widetilde{\Phi} + \frac{1}{2}|v|^2)}.
	\end{equation}
\end{lemma}
\begin{proof}
	For the case $p \in (1, 2)$, recall from \eqref{013} that $\widetilde{\Phi}(x) = \frac{\alpha}{p}\langle x \rangle^p$ and  $\mathsf{R} = 0$. Thus
	\begin{equation*}
		-\mathsf{T}(x \cdot v) = \alpha|x|^2\langle x \rangle^{p - 2} - |v|^2.
	\end{equation*}
	For every $0<\varepsilon<1$, there exists $M_\varepsilon>0$ such that
	\begin{equation*}
		\alpha|x|^2\langle x \rangle^{p - 2} - |v|^2 \ge (1 - \varepsilon)\alpha\langle x \rangle^p - |v|^2, \quad\forall |x| \ge M_\varepsilon.
	\end{equation*}
	Indeed, the inequality is equivalent to $\varepsilon|x|^2\ge1-\varepsilon$, so one may take $M_\varepsilon=\varepsilon^{-1/2}$. The difference between the two exponents is
	\begin{equation*}
		(2\lambda - 1)\widetilde{\Phi} + \lambda|v|^2 - 2(\lambda - \frac{1}{p + 2} - \delta)(\widetilde{\Phi} + \frac{1}{2}|v|^2) = (\frac{1}{p + 2} + \frac{\delta}{2})|v|^2 - (1 - \frac{2}{p + 2} - \delta)\widetilde{\Phi} + \delta(\widetilde{\Phi} + \frac{1}{2}|v|^2).
	\end{equation*}
	
	$\bullet$ Case 1: $|x| \ge M_\varepsilon$ and $(\frac{1}{p + 2} + \frac{\delta}{2})|v|^2 \ge (1 - \frac{2}{p + 2} - \delta)\widetilde{\Phi}$. We have 
	\begin{equation*}
		\begin{aligned}
		\mbox{The left hand side of \eqref{WeightTransferEq1}}	&\ge (N\langle v \rangle^{-6}e^{\delta(\widetilde{\Phi} + \frac{1}{2}|v|^2)} + (1 - \varepsilon)\alpha\langle x \rangle^p - |v|^2)e^{2(\lambda - \frac{1}{p + 2} - \delta)(\widetilde{\Phi} + \frac{1}{2}|v|^2)} \\
			&\ge \langle v \rangle^{-6}e^{\delta(\widetilde{\Phi} + \frac{1}{2}|v|^2)}e^{2(\lambda - \frac{1}{p + 2} - \delta)(\widetilde{\Phi} + \frac{1}{2}|v|^2)} \ge C_1e^{2(\lambda - \frac{1}{p + 2} - \delta)(\widetilde{\Phi} + \frac{1}{2}|v|^2)}
		\end{aligned}
	\end{equation*}
	for sufficiently large $N$ and some constant $C_1>0$.
	
	$\bullet$ Case 2: $|x| \ge M_\varepsilon$ and $(\frac{1}{p + 2} + \frac{\delta}{2})|v|^2 < (1 - \frac{2}{p + 2} - \delta)\widetilde{\Phi}$. We have
	\begin{equation*}
		0 \le |v|^2 < \frac{2p - 2\delta(p + 2)}{2 + \delta(p + 2)}\widetilde{\Phi}.
	\end{equation*}
	Then
	\begin{equation*}
		(1 - \varepsilon)\alpha\langle x \rangle^p - |v|^2 > p(1 - \varepsilon)\widetilde{\Phi} - \frac{2p - 2\delta(p + 2)}{2 + \delta(p + 2)}\widetilde{\Phi}.
	\end{equation*}
	Observing that
	\begin{equation*}
		p(1 - \varepsilon)(2 + \delta(p + 2)) > 2p - 2\delta(p + 2)
	\end{equation*}
	 is equivalent to
	\begin{equation*}
		\varepsilon < \frac{\delta(p^2 + 4p + 4)}{2p + \delta(p^2 + 2p)}.
	\end{equation*}
	Thus we can choose such $\varepsilon$ and then there exists a constant $C_1>0$ such that
	\begin{equation*}
		(1 - \varepsilon)\alpha\langle x \rangle^p - |v|^2 > C_1\langle x \rangle^p \ge C_1.
	\end{equation*}
	It yields that
		\begin{equation*}
		\begin{aligned}
			\mbox{The left hand side of \eqref{WeightTransferEq1}}	&\ge  - \mathsf{T}(x \cdot v)e^{2(\lambda - \frac{1}{p + 2} - \delta)(\widetilde{\Phi} + \frac{1}{2}|v|^2)} \\
			&\ge C_1e^{2(\lambda - \frac{1}{p + 2} - \delta)(\widetilde{\Phi} + \frac{1}{2}|v|^2)}.
		\end{aligned}
	\end{equation*}

	$\bullet$ Case 3: $|x| < M_\varepsilon$. Note that $\varepsilon$ is chosen, and related to $p$ and $\delta$. Then $M_\varepsilon = \varepsilon^{-1/2}$ is a constant depending on $p$ and $\delta$. And $|x| < M_\varepsilon$ shows that $\widetilde{\Phi}$ is actually bounded, then there exists some constant $c$ such that
	\begin{equation*}
		(\alpha|x|^2\langle x \rangle^{p - 2} - |v|^2)e^{2(\lambda - \frac{1}{p + 2} - \delta)(\widetilde{\Phi} + \frac{1}{2}|v|^2)} \ge - c\langle v \rangle^2e^{(\lambda - \frac{1}{p + 2} - \delta)|v|^2}.
	\end{equation*}
	Thus
	\begin{equation*}
		\mbox{The left-hand side of \eqref{WeightTransferEq1}}\ge N\langle v \rangle^{-6}e^{\lambda|v|^2} - c\langle v \rangle^2e^{(\lambda - \frac{1}{p + 2} - \delta)|v|^2} \ge C_1e^{2(\lambda - \frac{1}{p + 2} - \delta)(\widetilde{\Phi} + \frac{1}{2}|v|^2)}
	\end{equation*}
	for sufficiently large $N$ and some constant $C_1>0$.
	
	Combining the three cases completes the proof.
	
\end{proof}

\begin{lemma}[The case $p \in (2, \infty)$]\label{WeightTransferLemma2}
	For every $\lambda\in(1/2,1)$ and $0<\delta<1/3-1/(p+2)$, there exist constants $N\ge1$ and $C_1,C_2>0$ such that
	\begin{multline}\label{WeightTransferEq2}
		N\langle v \rangle^{-6}e^{(2\lambda - 1)\widetilde{\Phi} + \lambda|v|^2} - \mathsf{T}(x \cdot v)e^{2(\lambda - \frac{1}{p + 2} - \delta)(\widetilde{\Phi} + \frac{1}{2}|v|^2)} \\\ge C_1(|x|^p + |v|^2)e^{2(\lambda - \frac{1}{p + 2} - \delta)(\widetilde{\Phi} + \frac{1}{2}|v|^2)} \ge C_1e^{2(\lambda - \frac{1}{p + 2} - \delta)(\widetilde{\Phi} + \frac{1}{2}|v|^2)} - C_2.
	\end{multline}
\end{lemma}
\begin{proof}
	In the case $p \in (2, \infty)$,  recall from \eqref{013} that $\widetilde{\Phi}(x) = \frac{1}{p}|x|^p - \frac{1}{2}|\mathsf{R}x|^2$. Thus
	\begin{equation*}
		-\mathsf{T}(x \cdot v) = |x|^p - |v|^2 - 2\mathsf{R}x \cdot v - |\mathsf{R}x|^2.
	\end{equation*}
	For every $0<\varepsilon<1$, there exists $M_\varepsilon>0$ such that
	\begin{equation*}
		|x|^p - |v|^2 - 2\mathsf{R}x \cdot v - |\mathsf{R}x|^2 \ge (1 - \varepsilon)|x|^p - (1 + \varepsilon)|v|^2, \forall |x| \ge M_\varepsilon.
	\end{equation*}
	It is equivalent to $\varepsilon(|x|^p + |v|^2) - 2\mathsf{R}x \cdot v - |\mathsf{R}x|^2 \ge 0$. Note that
	\begin{equation*}
		\varepsilon(|x|^p + |v|^2) - 2\mathsf{R}x \cdot v - |\mathsf{R}x|^2 \ge \varepsilon(|x|^p + |v|^2) - \varepsilon|v|^2 - \frac{r^2}{\varepsilon}|x|^2 - r^2|x|^2 = \varepsilon|x|^2(|x|^{p - 2} - \frac{r^2}{\varepsilon^2} - \frac{r^2}{\varepsilon})
	\end{equation*}
	Thus we only need to choose $M_\varepsilon=(4r^2/\varepsilon^2)^{\f1{p-2}}$. We also noticed the difference between exponent part is
	\begin{equation*}
		(2\lambda - 1)\widetilde{\Phi} + \lambda|v|^2 - 2(\lambda - \frac{1}{p + 2} - \delta)(\widetilde{\Phi} + \frac{1}{2}|v|^2) = (\frac{1}{p + 2} + \frac{\delta}{2})|v|^2 - (1 - \frac{2}{p + 2} - \delta)\widetilde{\Phi} + \delta(\widetilde{\Phi} + \frac{1}{2}|v|^2).
	\end{equation*}
	
	$\bullet$ Case 1:  $|x| \ge M_\varepsilon$ and $(\frac{1}{p + 2} + \frac{\delta}{2})|v|^2 \ge (1 - \frac{2}{p + 2} - \delta)\widetilde{\Phi}$. We have that
	\begin{multline*}
		\mbox{The left hand side of \eqref{WeightTransferEq2}}	\ge (N(|v|^2 + 1)^{-3}e^{\delta(\widetilde{\Phi} + \frac{1}{2}|v|^2)} + (1 - \varepsilon)|x|^p - (1 + \varepsilon)|v|^2)e^{2(\lambda - \frac{1}{p + 2} - \delta)(\widetilde{\Phi} + \frac{1}{2}|v|^2)} \\
			\ge (|v|^2 + 1)^{-3}e^{\delta(\widetilde{\Phi} + \frac{1}{2}|v|^2)}e^{2(\lambda - \frac{1}{p + 2} - \delta)(\widetilde{\Phi} + \frac{1}{2}|v|^2)} \ge C_1(|x|^p + |v|^2)e^{2(\lambda - \frac{1}{p + 2} - \delta)(\widetilde{\Phi} + \frac{1}{2}|v|^2)}
	\end{multline*}
	for sufficiently large $N$ and some constant $C_1>0$.
	
	$\bullet$ Case 2: $|x| \ge M_\varepsilon$ and $(\frac{1}{p + 2} + \frac{\delta}{2})|v|^2 < (1 - \frac{2}{p + 2} - \delta)\widetilde{\Phi}$. We have
	\begin{equation*}
		0 \le |v|^2 < \frac{2p - 2\delta(p + 2)}{2 + \delta(p + 2)}\widetilde{\Phi} \le \frac{2p - 2\delta(p + 2)}{2 + \delta(p + 2)}\frac{1}{p}|x|^p.
	\end{equation*}
	Then
	\begin{equation*}
		(1 - \varepsilon)|x|^p - (1 + \varepsilon)|v|^2 > (1 - \varepsilon)|x|^p - \frac{1 + \varepsilon}{p} \times \frac{2p - 2\delta(p + 2)}{2 + \delta(p + 2)}|x|^p.
	\end{equation*}
	Observing that
	\begin{equation*}
		p(1 - \varepsilon)(2 + \delta(p + 2)) > (1 + \varepsilon)(2p - 2\delta(p + 2))
	\end{equation*}
	is equivalent to
	\begin{equation*}
		\varepsilon < \frac{\delta(p^2 + 4p + 4)}{4p + \delta(p^2 - 4)}.
	\end{equation*}
	Thus we choose such $\varepsilon$ and then there exist constants $C, C_1>0$ such that
	\begin{equation*}
		(1 - \varepsilon)|x|^p - (1 + \varepsilon)|v|^2 > 2C|x|^p > C(|x|^p + \frac{p(2 + \delta(p + 2))}{2p - \delta(p + 2)}|v|^2) \ge C_1(|x|^p + |v|^2).
	\end{equation*}
		It yields that
	\begin{equation*}
		\begin{aligned}
			\mbox{The left hand side of \eqref{WeightTransferEq2}}	&\ge  - \mathsf{T}(x \cdot v)e^{2(\lambda - \frac{1}{p + 2} - \delta)(\widetilde{\Phi} + \frac{1}{2}|v|^2)} \\
			&\ge C_1(|x|^p + |v|^2)e^{2(\lambda - \frac{1}{p + 2} - \delta)(\widetilde{\Phi} + \frac{1}{2}|v|^2)}.
		\end{aligned}
	\end{equation*}

	$\bullet$ Case 3: $|x| < M_\varepsilon$. Note that $\varepsilon$ is chosen, and related to $p$ and $\delta$. Then $M_\varepsilon$ is a constant depending on  $p$ and $\delta$. And $|x| < M_\varepsilon$ shows that $\widetilde{\Phi}$ is actually bounded, then there exists a constant $c>0$ such that
	\begin{equation*}
		(|x|^p - |v|^2 - 2\mathsf{R}x \cdot v - |\mathsf{R}x|^2)e^{2(\lambda - \frac{1}{p + 2} - \delta)(\widetilde{\Phi} + \frac{1}{2}|v|^2)} \ge -c(|v|^2 + 1)e^{(\lambda - \frac{1}{p + 2} - \delta)|v|^2}.
	\end{equation*}
	Thus we have
	\begin{multline*}
		\mbox{The left hand side of \eqref{WeightTransferEq2}}\ge N(|v|^2 + 1)^{-3}e^{\lambda|v|^2} - c(|v|^2 + 1)e^{(\lambda - \frac{1}{p + 2} - \delta)|v|^2} \\\ge c_1(M^p + |v|^2)e^{(\lambda - \frac{1}{p + 2} - \delta)|v|^2} \ge C_1(|x|^p + |v|^2)e^{2(\lambda - \frac{1}{p + 2} - \delta)(\widetilde{\Phi} + \frac{1}{2}|v|^2)}
	\end{multline*}

	Finally, we divide into two cases, i.e.
	$|x|^p + |v|^2 \ge 1$ and $|x|^p + |v|^2 < 1$, it is easy to obtain that
	\beno
	C_1(|x|^p + |v|^2)e^{2(\lambda - \frac{1}{p + 2} - \delta)(\widetilde{\Phi} + \frac{1}{2}|v|^2)} \ge C_1e^{2(\lambda - \frac{1}{p + 2} - \delta)(\widetilde{\Phi} + \frac{1}{2}|v|^2)} - C_2.
	\eeno
	Then we complete the proof of this lemma.
%	\begin{equation*}
%		\widetilde{\Phi} + \frac{1}{2}|v|^2 = \frac{1}{p}|x|^p - \frac{1}{2}|\mathsf{R}x|^2 + \frac{1}{2}|v|^2 < \frac{1}{2} + \frac{1}{p}.
%	\end{equation*}
%	Thus $e^{2(\lambda - \frac{1}{p + 2} - \delta)(\widetilde{\Phi} + \frac{1}{2}|v|^2)}$ has a maximum value, and complete the proof.
\end{proof}

As a direct corollary of Lemma \ref{WeightTransferLemma1} and \ref{WeightTransferLemma2}, we have
\begin{corollary}\label{WeightTransferLemma}
	In both cases, whether \( p \in (1,2) \) or \( p \in (2,\infty) \),
	 for any $\lambda \in (\frac{1}{2}, 1), 0 < \delta < \frac{1}{3} - \frac{1}{p + 2}$. There exists a large constant $N$, and constants $C_1, C_2>0$ such that
	\begin{equation}\label{WeightTransferEq}
		N\langle v \rangle^{-6}e^{(2\lambda - 1)\widetilde{\Phi} + \lambda|v|^2} - \mathsf{T}(x \cdot v)e^{2(\lambda - \frac{1}{p + 2} - \delta)(\widetilde{\Phi} + \frac{1}{2}|v|^2)} \ge C_1e^{2(\lambda - \frac{1}{p + 2} - \delta)(\widetilde{\Phi} + \frac{1}{2}|v|^2)} - C_2.
	\end{equation}
\end{corollary}

%%%%%%%%%%%%%%%%%%%%%%%%%%%%%%%%%%%%%%%%%%%%%%%%%%

	\section{Weighted Poincar\'e--Korn tools for the macroscopic equations}

	The macroscopic closure uses the weighted Poincar\'e and Poincar\'e--Korn framework developed in \cite{carrapatoso_weighted_2022} and employed in \cite{carrapatoso_special_2024}. We do not claim the general Poincar\'e--Korn principle as new. We formulate and prove the exact zero-order and negative-order estimates, with the nested exponential weights generated by $\widetilde\Phi$, that are needed for the transformed moment equations in Section 6.

\subsection{Weighted functional setting}

	For $0 < \delta < 1/2$, we define the weight function
	\begin{equation}\label{084}
		V_\delta(x) := 2(1 - \delta)\widetilde{\Phi}(x)
	\end{equation}
	with $\widetilde{\Phi}$ defined in  \eqref{013}, and Hilbert space $L^2(V_\delta)$ equip with the inner product
	
	\begin{equation*}
		(f, g)_{L^2(V_\delta)} := \int_{\R^3} f(x)\overline{g(x)}e^{-V_\delta(x)}dx.
	\end{equation*}
	The Laplace operator associated with $V_\delta$ is defined by
	\begin{equation}\label{085}
		\Delta_{V_\delta}f := \Delta f - \nabla V_\delta \cdot \nabla f.
	\end{equation}
	By using integral by parts, one may check that
	\begin{equation*}
		(-\Delta_{V_\delta}f, g)_{L^2(V_\delta)} = \int_{\R^3}(-\Delta f + \nabla V_\delta \cdot \nabla f)ge^{-V_\delta}dx = \int_{\R^3}(\nabla f \cdot \nabla g)e^{-V_\delta}dx = (\nabla f, \nabla g)_{L^2(V_\delta)},
	\end{equation*}
	which means that $-\Delta_{V_\delta}$ is nonnegative and self-adjoint in Hilbert space $L^2(V_\delta)$. Thus the operator
	\begin{equation*}
		\Lambda_\delta := I - \Delta_{V_\delta}
	\end{equation*}
	is positive, self-adjoint and invertible. When not causing confusion, we abbreviate $V_\delta, \Lambda_\delta$ as $V, \Lambda$. 

	\bigskip In \cite{carrapatoso_weighted_2022}, authors have provided very detailed studies on Hilbert space $L^2(V)$ and operator $\Lambda$. We list some of the conclusions, which would be used later.

	\begin{lemma}[Weighted elliptic operator bounds]\label{BoundedLambda}
		The following operators are all bounded in space $L^2(V)$:
		\begin{equation*}
			\begin{aligned}
			&\Lambda^{-\frac{1}{2}}\nabla, \nabla\Lambda^{-\frac{1}{2}},	\<\na V\>\Lambda^{-\frac{1}{2}},\Lambda^{-\frac{1}{2}}\<\na V\>,\Lambda^{\frac{1}{2}}\nabla\Lambda^{-1}, \Lambda\nabla\Lambda^{-\frac{3}{2}}, \Lambda^{-1}\nabla\Lambda^{\frac{1}{2}};\\
			&\na^2\Lam^{-1},\Lam^{-1}\na^2,\<\na V\>\na\Lam^{-1},\Lam^{-1}\<\na V\>\na, \<\na V\>^2\Lam^{-1},\Lam^{-1}\<\na V\>^2.
      \end{aligned}
		\end{equation*}
		More precisely, for example, there exists a constant $C_V$ such that for any $f \in L^2(V)$,
		\beno
		\|\Lambda^{-\frac{1}{2}}\nabla f\|_{L^2(V)} \le \|f\|_{L^2(V)}.
		\eeno
		
	%	\begin{equation*}
	%	\begin{aligned}
	%		&\|\Lambda^{-\frac{1}{2}}\nabla f\|_{L^2(V)} \le \|f\|_{L^2(V)},~~\|\Lambda^{-\frac{1}{2}}\<\na V\> f\|_{L^2(V)} \le \|f\|_{L^2(V)},\quad
	%		\|\Lambda^{\frac{1}{2}}\nabla\Lambda^{-1}f\|_{L^2(V)} \le C_V\|f\|_{L^2(V)};\\
	%		&\|\Lambda\nabla\Lambda^{-\frac{3}{2}}f\|_{L^2(V)} \le C_V\|f\|_{L^2(V)};\quad
	%		\|\Lambda^{-1}\nabla f\|_{L^2(V)} \le C_V\|\Lambda^{-\frac{1}{2}}f\|_{L^2(V)}.
	%	\end{aligned}
	%	\end{equation*}
	\end{lemma}
	\begin{proof}
		These results come from $(28),(29),(31),(32),(40)$ and Lemma 10 in \cite{carrapatoso_weighted_2022}.
	\end{proof}

%%%%%%%%%%%%%%%%%%%%%%%%%%%%%%%%%%%%%%%%%%%%%%%%%%
\smallskip
	\subsection{Weighted Poincar\'e inequality}
	We start with two auxiliary lemmas.
	\begin{lemma}[Confinement at spatial infinity]\label{PILem}
		For all sufficiently large $M>0$,
		\begin{equation*}
			\frac{1}{4}|\nabla V|^2 - \frac{1}{2}\Delta V \ge \frac{1}{2} - (1 - \delta)(\Delta\widetilde{\Phi} + 1)1_{|x| < M}.
		\end{equation*}
	\end{lemma}
	\begin{proof}
		Recall that $V = 2(1 - \delta)\widetilde{\Phi}$, we need to show that
		\begin{equation*}
			(1 - \delta)^2|\nabla\widetilde{\Phi}|^2 - (1 - \delta)\Delta\widetilde{\Phi} \ge \frac{1}{2} - (1 - \delta)(\Delta\widetilde{\Phi} + 1)1_{|x| < M}.
		\end{equation*}
		It definitely holds for $|x| < M$, as for $|x| \ge M$, it turns to
		\begin{equation}\label{PILemproof}
			(1 - \delta)^2|\nabla\widetilde{\Phi}|^2 - (1 - \delta)\Delta\widetilde{\Phi} - \frac{1}{2} \ge 0.
		\end{equation}
		If $p \in (1, 2)$ and $\widetilde{\Phi}(x) = \frac{\alpha}{p}\langle x \rangle^p$, then
		\begin{equation*}
			|\nabla \widetilde{\Phi}|^2 = |\alpha\langle x \rangle^{p - 2}x|^2 = \alpha^2\langle x \rangle^{2p - 4}|x|^2, \:\:\: \Delta\widetilde{\Phi} = 3\alpha\langle x \rangle^{p - 2} - \alpha(2 - p)\langle x \rangle^{p - 4}|x|^2.
		\end{equation*}
		If $p \in (2, \infty)$ and $\widetilde{\Phi}(x) = \frac{1}{p}|x|^p - \frac{1}{2}|\mathsf{R}x|^2$, then $\Delta\widetilde{\Phi} = (p + 1)|x|^{p - 2} - 2r^2$ and
		\begin{equation*}
			|\nabla \widetilde{\Phi}|^2 = ||x|^{p - 2}x + \mathsf{R}^2x|^2 \ge \frac{1}{2}|x|^{2p - 2} - |\mathsf{R}^2x|^2 \ge \frac{1}{2}|x|^{2p - 2} - r^4|x|^2.
		\end{equation*}
		Thus, in both cases, \eqref{PILemproof} holds for sufficiently large $M$, which proves the lemma.
	\end{proof}

	\begin{lemma}\label{LPI}
		(Local Poincar\'e Inequality). If $W = W(x)$ satisfies $W, W^{-1} \in L_{loc}^\infty$, then for any ball $\Omega$ and smooth function $f$, there exists some constant $C_\Omega$ such that
		\begin{equation*}
			\int_\Omega f^2Wdx \le C_\Omega\int_\Omega|\nabla f|^2Wdx + \frac{1}{W(\Omega)}\big(\int_\Omega fWdx\big)^2, \: \text{where} \:\: W(\Omega) := \int_\Omega Wdx.
		\end{equation*}
	\end{lemma}

	Next we introduce the variant Poincar\'e inequality with exponential weight $V$:
	\begin{lemma}[Weighted Poincar\'e inequality]\label{PI}
		If smooth function $f$ satisfies one of the following conditions:
		\begin{equation*}
			\int_{\R^3} fe^{-\widetilde{\Phi}}dx = 0, \:\:\: \int_{\R^3}(\Lambda f)e^{-\widetilde{\Phi}}dx = 0.
		\end{equation*}
		Then there exists a constant $C_V$ such that
		\begin{equation*}
			\int_{\R^3} f^2e^{-V}dx \le C_V\int_{\R^3}|\nabla f|^2e^{-V}dx.
		\end{equation*}
	\end{lemma}
	\begin{proof}
		%Before starting the proof, we emphasize that $V$ is determined by $\delta$ and $\widetilde{\Phi}$, and function $\widetilde{\Phi}(x)$ is determined by $p$ and $\alpha$ in $p \in (1, 2)$ case, or $p$ and $r$ in $p \in (2, \infty)$. Therefore constant $C_V$ actually means $C_{p, \alpha, r, \delta}$. 
		We set $g := fe^{-\frac{1}{2}V}$, then
		\begin{equation*}
			\nabla g = (\nabla f)e^{-\frac{1}{2}V} - \frac{1}{2}(\nabla V)fe^{-\frac{1}{2}V}.
		\end{equation*}
	It leads to
		\begin{equation*}
			0 \le \int_{\R^3}|\nabla g|^2dx = \int_{\R^3}|\nabla f|^2e^{-V}dx + \frac{1}{4}\int_{\R^3}|\nabla V|^2f^2e^{-V}dx - \int_{\R^3}f\nabla f \cdot \nabla Ve^{-V}dx.
		\end{equation*}
		By using integral by parts, we have
		\begin{equation*}
			\int_{\R^3} f\nabla f \cdot \nabla Ve^{-V}dx = \frac{1}{2}\int_{\R^3}\nabla(f^2) \cdot \nabla Ve^{-V}dx = -\frac{1}{2}\int_{\R^3}(\Delta V)f^2e^{-V} + \frac{1}{2}\int_{\R^3}|\nabla V|^2f^2e^{-V}dx.
		\end{equation*}
		Thus we have
		\begin{equation*}
			0 \le \int_{\R^3}|\nabla f|^2e^{-V}dx + \frac{1}{2}\int_{\R^3}(\Delta V)f^2e^{-V}dx - \frac{1}{4}\int_{\R^3}|\nabla V|^2f^2e^{-V}dx,
		\end{equation*}
		and reaches that
		\begin{equation*}
			\int_{\R^3}|\nabla f|^2e^{-V}dx \ge \int_{\R^3}(\frac{1}{4}|\nabla V|^2 - \frac{1}{2}\Delta V)f^2e^{-V}dx.
		\end{equation*}
		By Lemma \ref{PILem},
		\begin{equation*}
			\int_{\R^3}|\nabla f|^2e^{-V}dx \ge \frac{1}{2}\int_{\R^3} f^2e^{-V}dx - (1 - \delta)\int_{|x| < M}(\Delta\widetilde{\Phi} + 1)f^2e^{-V}dx.
		\end{equation*}
		Observing that
		\begin{equation*}
			(\Delta\widetilde{\Phi} + 1)e^{-V} = (\Delta\widetilde{\Phi} + 1)e^{-(1 - 2\delta)\widetilde{\Phi}}e^{-\widetilde{\Phi}},
		\end{equation*}
		where $|\Delta\widetilde{\Phi} + 1|e^{-(1 - 2\delta)\widetilde{\Phi}}$ is bounded in $\mathbb{R}^3$ since $\delta<\f12$. Thus we have
		\begin{equation}\label{PIproof1}
			\int_{\R^3} f^2e^{-V}dx \le 2\int_{\R^3}|\nabla f|^2e^{-V}dx + C_V\int_{|x| < M}f^2e^{-\widetilde{\Phi}}dx.
		\end{equation}
		By local Poincar\'e Inequality, i.e. Lemma \ref{LPI},
		\begin{equation}\label{PIproof2}
			\int_{|x| < M}f^2e^{-\widetilde{\Phi}}dx \le C_M\int_{|x| < M}|\nabla f|^2e^{-\widetilde{\Phi}}dx + \frac{1}{\mathcal{Z}_M}\big(\int_{|x| < M}fe^{-\widetilde{\Phi}}dx\big)^2,
		\end{equation}
		where
		\begin{equation*}
			\mathcal{Z}_M = \int_{|x| < M}e^{-\widetilde{\Phi}}dx, \:\:\: \text{and} \: \mathcal{Z}_M \to \mathcal{Z}_\infty < +\infty \:\: \text{as} \:\: M \to \infty.
		\end{equation*}
		
		In the condition that $\int_{\R^3} fe^{-\widetilde{\Phi}}dx = 0$, we have
		\begin{equation*}
			(\int_{|x| < M}fe^{-\widetilde{\Phi}}dx\big)^2 = (\int_{|x| \ge M}fe^{-\widetilde{\Phi}}dx\big)^2 \le \int_{|x| \ge M}f^2e^{-2(1 - \delta)\widetilde{\Phi}}dx \int_{|x| \ge M}e^{-2\delta\widetilde{\Phi}}dx \le \mathcal{E}_M\int_{\R^3} f^2e^{-V}dx,
		\end{equation*}
		where
		\begin{equation*}
			\mathcal{E}_M = \int_{|x| \ge M}e^{-2\delta\widetilde{\Phi}}dx, \:\:\: \text{and} \: \mathcal{E}_M \to 0 \:\: \text{as} \:\: M \to \infty.
		\end{equation*}

		In the condition that $\int_{\R^3}(\Lambda f)e^{-\widetilde{\Phi}}dx = 0$, by integration by parts, we first write
		\begin{equation*}
			0 = \int_{\R^3}(f - \Delta f + \nabla V \cdot \nabla f)e^{-\widetilde{\Phi}}dx = \int_{\R^3}(f + (1 - 2\delta)\nabla\widetilde{\Phi}\cdot \nabla f)e^{-\widetilde{\Phi}}dx.
		\end{equation*}
		The Cauchy--Schwarz inequality gives
		\begin{equation*}
		\begin{aligned}
			(\int_{|x| < M}fe^{-\widetilde{\Phi}}dx\big)^2 &= (\int_{|x| \ge M}fe^{-\widetilde{\Phi}}dx + (1 - 2\delta)\int_{\R^3}\nabla\widetilde{\Phi} \cdot \nabla fe^{-\widetilde{\Phi}}dx\big)^2 \\
			&\le 2(\int_{|x| \ge M}fe^{-\widetilde{\Phi}}dx\big)^2 + 2(1 - \delta)^2\big(\int_{\R^3}\nabla\widetilde{\Phi} \cdot \nabla fe^{-\widetilde{\Phi}}dx\big)^2 \\
			&\le 2\mathcal{E}_M\int_{|x| \ge M}f^2e^{-V}dx + 2(1 - \delta)^2\int_{\R^3}|\nabla f|^2e^{-2(1 - \delta)\widetilde{\Phi}}dx  \int_{\R^3}|\nabla\widetilde{\Phi}|^2e^{-2\delta\widetilde{\Phi}}dx \\
			&\le 2\mathcal{E}_M\int_{\R^3} f^2e^{-V}dx + C_V\int_{\R^3}|\nabla f|^2e^{-V}dx.
		\end{aligned}
		\end{equation*}
		Thus in either cases, we plug it into \eqref{PIproof2} and get
		\begin{equation*}
			\int_{|x| < M}f^2e^{-\widetilde{\Phi}}dx \le C_M\int_{|x| < M}|\nabla f|^2e^{-\widetilde{\Phi}}dx + \frac{2\mathcal{E}_M}{\mathcal{Z}_M}\int_{\R^3} f^2e^{-V}dx + \frac{C_V}{\mathcal{Z}_M}\int_{\R^3}|\nabla f|^2e^{-V}dx.
		\end{equation*}
	Observe that
		\begin{equation*}
			\mathbf{1}_{|x| < M}e^{-\widetilde{\Phi}} = 	\mathbf{1}_{x < M}e^{(1 - 2\delta)\widetilde{\Phi}}e^{-V} \le C_{V, M}e^{-V}
		\end{equation*}
		Therefore
		\begin{equation*}
			\int_{|x| < M}f^2e^{-\widetilde{\Phi}}dx \le (C_{V, M} + \frac{C_V}{\mathcal{Z}_M})\int_{\R^3}|\nabla f|^2e^{-V}dx + \frac{2\mathcal{E}_M}{\mathcal{Z}_M}\int_{\R^3} f^2e^{-V}dx.
		\end{equation*}
		We use the above inequality in \eqref{PIproof1} to get that
		\begin{equation*}
			\int_{\R^3} f^2e^{-V}dx \le (2+C_{V, M} + \frac{C_V}{\mathcal{Z}_M})\int_{\R^3}|\nabla f|^2e^{-V}dx + \frac{2\mathcal{E}_M}{\mathcal{Z}_M}\int_{|x| \ge M}f^2e^{-V}dx.
		\end{equation*}
		Recall that, as $M\to\infty$,
		\begin{equation*}
			\mathcal{Z}_M \to \mathcal{Z}_\infty < +\infty, \:\:\: \mathcal{E}_M \to 0.
		\end{equation*}
		Choose $M$ sufficiently large that $4\mathcal E_M<\mathcal Z_M$. Then
		\begin{equation*}
			\int_{\R^3} f^2e^{-V}dx \le C_V\int_{\R^3}|\nabla f|^2e^{-V}dx,
		\end{equation*}
		which completes the proof.
	\end{proof}

%%%%%%%%%%%%%%%%%%%%%%%%%%%%%%%%%%%%%%%%%%%%%%%%%%

	\subsection{Negative-order weighted Poincar\'e inequalities}
	\begin{lemma}[Negative-order weighted Poincar\'e inequalities]\label{VPI}
		If smooth function $f$ satisfies $\int_{\R^3} fe^{-\widetilde{\Phi}}dx = 0$. Then there exists some constant $C_V$ such that
		\begin{equation*}
			\|f\|_{L^2(V)} \le C_V\|\Lambda^{-\frac{1}{2}}\nabla f\|_{L^2(V)}; \:\:\:\:\:\:
			\|\Lambda^{-\frac{1}{2}}f\|_{L^2(V)} \le C_V\|\Lambda^{-1}\nabla f\|_{L^2(V)}.
		\end{equation*}
	\end{lemma}
	\begin{proof}
		For the first inequality, set $g:=\Lambda^{-1}f$. Since $\int_{\R^3}(\Lambda g)e^{-\widetilde{\Phi}}dx=0$, Lemma \ref{PI} gives
		\begin{equation*}
			(-\Delta_Vg, g)_{L^2(V)} = \|\nabla g\|_{L^2(V)}^2 \ge C_V\|g\|_{L^2(V)}^2 = C_V(g, g)_{L^2(V)}
		\end{equation*}
		for some constant $C_V$. Then
		\begin{equation*}
			(-\Delta_Vg, g)_{L^2(V)} \ge \frac{1}{2}(-\Delta_Vg, g)_{L^2(V)} + \frac{1}{2}C_V(g, g)_{L^2(V)} \ge C'_V((I - \Delta_V)g, g)_{L^2(V)}.
		\end{equation*}
		Therefore we have
		\begin{equation*}
			\begin{aligned}
				(f, f)_{L^2(V)} &= ((I - \Delta_V)g, (I - \Delta_V)g)_{L^2(V)}
				= ((I - \Delta_V)g, g)_{L^2(V)} \\
				&\quad + (-\Delta_Vg, g)_{L^2(V)}
				 + (-\Delta_Vg, -\Delta_Vg)_{L^2(V)} \\
			&\ge (1 + C'_V)((I - \Delta_V)g, g)_{L^2(V)} = (1 + C'_V)(f, \Lambda^{-1}f)_{L^2(V)}
			= (1 + C'_V)(f, (\Lambda + \Delta_V)\Lambda^{-1}f)_{L^2(V)} \\
			&= (1 + C'_V)(f, f)_{L^2(V)} - (1 + C'_V)(-\Delta_V\Lambda^{-1}f, f)_{L^2(V)}.
		\end{aligned}
		\end{equation*}
	It leads to
		\begin{equation*}
		\begin{aligned}
			\frac{C'_V}{1 + C'_V}\|f\|_{L^2(V)}^2 &\le (-\Delta_V\Lambda^{-1}f, f)_{L^2(V)}
			= (\nabla\Lambda^{-1}f, \nabla f)_{L^2(V)} \\
			&= (\Lambda^{\frac{1}{2}}\nabla\Lambda^{-1}f, \Lambda^{-\frac{1}{2}}\nabla f)_{L^2(V)}
			\le \|\Lambda^{\frac{1}{2}}\nabla\Lambda^{-1}f\|_{L^2(V)}\|\Lambda^{-\frac{1}{2}}\nabla f\|_{L^2(V)}.
		\end{aligned}
		\end{equation*}
		Lemma \ref{BoundedLambda} gives $\|\Lambda^{1/2}\nabla\Lambda^{-1}f\|_{L^2(V)}\leq\|f\|_{L^2(V)}$, proving the first inequality.

		\bigskip For the second inequality, again we set $g := \Lambda^{-1}f$ and Lemma \ref{PI} shows that
		\begin{equation*}
			(-\Delta_Vg, g)_{L^2(V)} = \|\nabla g\|_{L^2(V)}^2 \ge C_V\|g\|_{L^2(V)}^2 = C_V(g, g)_{L^2(V)}
		\end{equation*}
		for some constant $C_V$. Then
		\begin{equation*}
			C_V(g, g)_{L^2(V)} \le (-\Delta_Vg, g)_{L^2(V)} = ((\Lambda - I)g, g)_{L^2(V)} = (\Lambda g, g)_{L^2(V)} - (g, g)_{L^2(V)},
		\end{equation*}
		which leads to
		\begin{equation*}
			\frac{1}{1 + C_V}(\Lambda g, g)_{L^2(V)} \ge (g, g)_{L^2(V)} = (\Lambda g, g)_{L^2(V)} - (-\Delta_Vg, g)_{L^2(V)}.
		\end{equation*}
		Therefore we have
		\begin{equation*}
			\frac{C_V}{1 + C_V}(\Lambda g, g)_{L^2(V)} \le (-\Delta_Vg, g)_{L^2(V)}.
		\end{equation*}
		Note that $g = \Lambda^{-1}f$, then we have 
		\begin{multline*}
			\frac{C_V}{1 + C_V}\|\Lambda^{-\frac{1}{2}}f\|_{L^2(V)} \le (-\Delta_V\Lambda^{-1}f, \Lambda^{-1}f)_{L^2(V)} = (-\Delta_V\Lambda^{-2}f, f)_{L^2(V)} = (\nabla\Lambda^{-2}f, \nabla f)_{L^2(V)} \\
			= (\Lambda\nabla\Lambda^{-2}f, \Lambda^{-1}\nabla f)_{L^2(V)} \le \|\Lambda\nabla\Lambda^{-\frac{3}{2}}\Lambda^{-\frac{1}{2}}f\|_{L^2(V)}\|\Lambda^{-1}\nabla f\|_{L^2(V)}.
		\end{multline*}
		Applying Lemma \ref{BoundedLambda}, we find $\|\Lambda\nabla\Lambda^{-3/2}\Lambda^{-1/2}f\|_{L^2(V)}\leq\|\Lambda^{-1/2}f\|_{L^2(V)}$, which proves the second inequality.
		%We complete the proof of this lemma.
	\end{proof}

	\subsection{Weighted Poincar\'e--Korn inequality}
	For the vector-valued function $f = (f_1, f_2, f_3)$, we denote $\nabla_x^{sym}f$ and $\nabla_x^{skew}f$ are $3 \times 3$ matrix, where
	\begin{equation}\label{Defisymskew}
		\nabla_x^{sym}f = (\nabla_x^{sym}f)_{ij} = \frac{\partial_jf_i + \partial_if_j}{2}; \:\:\:\:\:\:
		\nabla_x^{skew}f = (\nabla_x^{skew}f)_{ij} = \frac{\partial_jf_i - \partial_if_j}{2}.
	\end{equation}
	Besides, we denote
	\begin{equation*}
		|\nabla f|^2 = \sum_{i, j = 1}^3|\partial_if_j|^2; \:\:\:
		|\nabla_x^{sym}f|^2 = \sum_{i, j = 1}^3|(\nabla_x^{sym}f)_{ij}|^2; \:\:\:
		|\nabla_x^{skew}f|^2 = \sum_{i, j = 1}^3|(\nabla_x^{skew}f)_{ij}|^2.
	\end{equation*}

The following zeroth order variant Poincar\'e-Korn inequality holds:
	\begin{lemma}[Weighted Poincar\'e--Korn inequality]\label{VPKI}
		If the vector-valued function $f = (f_1, f_2, f_3)$ satisfies
		\begin{equation*}
			\int_{\R^3} fe^{-\widetilde{\Phi}}dx = 0, \quad\mbox{and}\quad \int_{\R^3}\nabla_x^{skew}fe^{-\widetilde{\Phi}}dx = 0.
		\end{equation*}
		Then there exists a constant $C_V$ such that
		\begin{equation*}
			\|f\|_{L^2(V)} \le C_V\|\Lambda^{-\frac{1}{2}}\nabla_x^{sym}f\|_{L^2(V)}.
		\end{equation*}
	\end{lemma}
	\begin{proof}
		Firstly, it is easy to check that
		\begin{equation*}
			\|\Lambda^{-\frac{1}{2}}\nabla f\|_{L^2(V)}^2 = \|\Lambda^{-\frac{1}{2}}\nabla_x^{sym}f\|_{L^2(V)}^2 + \|\Lambda^{-\frac{1}{2}}\nabla_x^{skew}f\|_{L^2(V)}^2.
		\end{equation*}
		The first inequality in Lemma \ref{VPI} shows that
		\begin{equation}\label{VPKIproof}
			\|f\|_{L^2(V)}^2 \le C_V\|\Lambda^{-\frac{1}{2}}\nabla f\|_{L^2(V)}^2 = C_V(\|\Lambda^{-\frac{1}{2}}\nabla_x^{sym}f\|_{L^2(V)}^2 + \|\Lambda^{-\frac{1}{2}}\nabla_x^{skew}f\|_{L^2(V)}^2).
		\end{equation}
		Now by the second inequality in Lemma \ref{VPI}, we have
		\begin{equation*}
			\|\Lambda^{-\frac{1}{2}}\nabla_x^{skew}f\|_{L^2(V)}^2 = \sum_{i, j = 1}^3\|\Lambda^{-\frac{1}{2}}(\nabla_x^{skew}f)_{ij}\|_{L^2(V)}^2 \le C_V\sum_{i, j, k = 1}^3\|\Lambda^{-1}\partial_k(\nabla_x^{skew}f)_{ij}\|_{L^2(V)}^2.
		\end{equation*}
		We use the following differential compatibility identity:
		\begin{equation*}
			\partial_k(\nabla_x^{skew}f)_{ij} = \partial_j(\nabla_x^{sym}f)_{ik} - \partial_i(\nabla_x^{sym}f)_{jk}.
		\end{equation*}
		One may check its correctness by definition of $\nabla_x^{sym}$ and $\nabla_x^{skew}$ \eqref{Defisymskew}. Therefore we get
		\begin{equation*}
			\|\Lambda^{-\frac{1}{2}}\nabla_x^{skew}f\|_{L^2(V)}^2 \le C_V\sum_{i, j, k = 1}^3\|\Lambda^{-1}\partial_k(\nabla_x^{sym}f)_{ij}\|_{L^2(V)}^2 = C_V\sum_{i, j = 1}^3\|\Lambda^{-1}\nabla(\nabla_x^{sym}f)_{ij}\|_{L^2(V)}^2.
		\end{equation*}
		From Lemma \ref{BoundedLambda}, we know that $\Lam^{-1}\na\Lam^{\f12}$ is bounded in $L^2(V)$, thus
		\begin{equation*}
			\|\Lambda^{-\frac{1}{2}}\nabla_x^{skew}f\|_{L^2(V)}^2 \le C_V\sum_{i, j = 1}^3\|\Lambda^{-\frac{1}{2}}(\nabla_x^{sym}f)_{ij}\|_{L^2(V)}^2 = C_V\|\Lambda^{-\frac{1}{2}}\nabla_x^{sym}f\|_{L^2(V)}^2.
		\end{equation*}
		Together with \eqref{VPKIproof}, this is enough for the proof.
	\end{proof}

%%%%%%%%%%%%%%%%%%%%%%%%%%%%%%%%%%%%%%%%%%%%%%%%%%

	\subsection{Comparison and interpolation of weighted spaces}
At the end of this section, we compare the spaces associated with different values of $\delta$ and record the interpolation estimates used later.
	\begin{lemma}[Comparison and interpolation of weights]\label{lemma5.7}
		$(1)$ For any given $0<\de_1<\f12$, there exists a constant $\varepsilon>0$ depending on $\de_1$ such that if $|\de_2-\de_1|<\varepsilon$, then for any regular function $f$, it holds that
		\begin{equation*}
		\|\Lam_{\de_2}f\|_{L^2(V_{\de_1})}\sim_{\de_1,\de_2} \|\Lam_{\de_1}f\|_{L^2(V_{\de_1})}.
		\end{equation*}
		As a corollary, for any given $\de_0\in (0,\f1{12})$, we can choose $\de_0>\de_1>\de_2>\de_3>\de_4>0$ such that $\de_0+\de_4>2\de_1,\de_1+\de_4>2\de_2,\de_2+\de_4>2\de_3$ and 
		\begin{equation*}
			\|\Lam_{\de_1}f\|_{L^2(V_{\de_0})}\sim \|\Lam_{\de_0}f\|_{L^2(V_{\de_0})},\quad 	\|\Lam_{\de_2}f\|_{L^2(V_{\de_1})}\sim \|\Lam_{\de_1}f\|_{L^2(V_{\de_1})},\quad\|\Lam_{\de_3}f\|_{L^2(V_{\de_2})}\sim \|\Lam_{\de_2}f\|_{L^2(V_{\de_2})}.
		\end{equation*}

		\noindent $(2)$ For every $m\in\mathbb R$, $0<\delta_1<\delta_2<1/2$, and every smooth function $f$,
		\begin{equation*}
			\|\langle \cdot \rangle^mf\|_{L^2(V_{\delta_1})} \lesssim \|f\|_{L^2(V_{\delta_2})}.
		\end{equation*}

		\noindent $(3)$ Let $0<\delta_1<\delta_2<\delta_3<1/2$ and let $f,g$ be smooth. If $2\delta_2\le\delta_1+\delta_3$, then
		\begin{equation*}
			(f, g)_{L^2(V_{\delta_2})} \lesssim \|f\|_{L^2(V_{\delta_1})}\|g\|_{L^2(V_{\delta_3})}.
		\end{equation*}
	\end{lemma}
	\begin{proof}
First, by the definitions of $V_{\de_2}$ and $\Lam_{\de_2}$ in \eqref{084} and \eqref{085},
\beno
\Lam_{\de_2}=I-\D+\na V_{\de_2}\cdot\na=I-\D+2(1-\de_2)\na \widetilde{\Phi}\cdot \na,
\eeno
thus from Lemma \ref{BoundedLambda}, we have that $\|\Lam_{\de_2}f\|_{L^2(V_{\de_1})}\ls_{\de_1,\de_2} \|\Lam_{\de_1}f\|_{L^2(V_{\de_1})}$.
    
  For the reverse inequality, note that
	    \begin{align*}
	    \|\Lam_{\de_2}f\|^2_{L^2(V_{\de_1})}
	    &= (\Lam_{\de_1}f,\Lam_{\de_1}f)_{L^2(V_{\de_1})}
	      +((\Lam_{\de_2}-\Lam_{\de_1})f,\Lam_{\de_2}f)_{L^2(V_{\de_1})}\\
	    &\quad +(\Lam_{\de_1}f,(\Lam_{\de_2}-\Lam_{\de_1})f)_{L^2(V_{\de_1})}\\
	    &=\|\Lam_{\de_1}f\|^2_{L^2(V_{\de_1})}
	      +2(\de_1-\de_2)
	      (\na_x\widetilde{\Phi}\cdot\na f,\Lam_{\de_1}f+\Lam_{\de_2}f)_{L^2(V_{\de_1})}\\
	    &\geq \bigl(1-C_{\de_1}|\de_1-\de_2|\bigr)
	      \|\Lam_{\de_1}f\|^2_{L^2(V_{\de_1})}
	      -C_{\de_1,\de_2}\|\Lam_{\de_2}f\|^2_{L^2(V_{\de_1})}.
	    \end{align*}
For fixed $\de_1$, choose $|\de_1-\de_2|<(2C_{\de_1})^{-1}$. It follows that
\[
\|\Lam_{\de_2}f\|^2_{L^2(V_{\de_1})}
\geq \frac{1}{2(1+C_{\de_1,\de_2})}
\|\Lam_{\de_1}f\|^2_{L^2(V_{\de_1})}.
\]

	    We next prove the corollary. Fix $\de_0\in(0,\f1{12})$. There exists a sufficiently small constant $\varepsilon_{01}>0$ such that, for $\de_1=\de_0-\varepsilon_{01}$, one has $\|\Lam_{\de_1}f\|_{L^2(V_{\de_0})}\sim \|\Lam_{\de_0}f\|_{L^2(V_{\de_0})}$. For this fixed $\de_1$, there exists a sufficiently small constant $\varepsilon_{12}<\f{\varepsilon_{01}}3$ such that, for $\de_2=\de_1-\varepsilon_{12}$, one has
    $\|\Lam_{\de_2}f\|_{L^2(V_{\de_1})}\sim \|\Lam_{\de_1}f\|_{L^2(V_{\de_1})}$. Then for fixed $\de_2=\de_1-\varepsilon_{12}$, there exists $\varepsilon_{23}<\varepsilon_{12}$ such that for $\de_3=\de_2-\varepsilon_{23}$, it holds that $\|\Lam_{\de_3}f\|_{L^2(V_{\de_2})}\sim \|\Lam_{\de_2}f\|_{L^2(V_{\de_2})}$. It is easy to verify that
    \beno
	    \max\{\de_0-2\varepsilon_{01},\de_1-2\varepsilon_{12},\de_2-2\varepsilon_{23}\}<\de_3
    \eeno
	    by the choice of $\varepsilon_{01},\varepsilon_{12}$ and $\varepsilon_{23}$. Finally, choose $\de_4=\f12(\de_3+\max\{\de_0-2\varepsilon_{01},\de_1-2\varepsilon_{12},\de_2-2\varepsilon_{23}\})$. Then $0<\de_4<\de_3$ and $\de_0+\de_4>2\de_1,\de_1+\de_4>2\de_2,\de_2+\de_4>2\de_3$.

For part (2), the identity $V_{\delta_1}-V_{\delta_2}=2(\delta_2-\delta_1)\widetilde\Phi$ and the confinement of $\widetilde\Phi$ imply
\[
\langle x\rangle^{2m}e^{-V_{\delta_1}(x)}le C_{m,\delta_1,\delta_2}e^{-V_{\delta_2}(x)}.
\]
Integration proves the asserted norm comparison. For part (3), write
\[
e^{-V_{\delta_2}}
=e^{-V_{\delta_1}/2}e^{-V_{\delta_3}/2}
\exp\left(\frac{V_{\delta_1}+V_{\delta_3}}2-V_{\delta_2}\right).
\]
The last factor is at most one when $2\delta_2\le\delta_1+\delta_3$. The conclusion follows from the Cauchy--Schwarz inequality.
\end{proof}

%%%%%%%%%%%%%%%%%%%%%%%%%%%%%%%%%%%%%%%%%%%%%%%%%%

	\section{Macroscopic closure under the transformed conservation constraints}
This section is where the Poincar\'e--Korn architecture recalled in Section 5 is adapted to the normalized equation of Section 3. The transformed transport contains the rotation matrix $\mathsf R$, and the orthogonality relations are the transformed mass, energy, and angular-momentum constraints rather than the standard Maxwellian constraints. The argument has four stages: we first derive the moment system; next we split the thirteen local macroscopic coefficients into weighted Poincar\'e--Korn-coercive components and a finite-dimensional remainder; we then control these two classes by evolution equations and the five conservation laws, respectively; finally we combine the estimates into a single macroscopic coercivity inequality. Recall from \eqref{Defig0} that the macroscopic part
	\begin{equation}\label{g0form}
		g_0 = [a(t, x) + v \cdot b(t, x) + (|v|^2 - 3)c(t, x)]\mu.
	\end{equation}
	Thus the estimation for $g_0$ can be reduced to the estimation for $a, b, c$.

%%%%%%%%%%%%%%%%%%%%%%%%%%%%%%%%%%%%%%%%%%%%%%%%%%
\subsection{Macro--micro decomposition and moment equations}
We write $g = g_0 + g_1$ in \eqref{Lineareqg} and get that
	\begin{equation}\label{Lineareqg0}
		\partial_tg_0 + \mathsf{T}g_0 = -\partial_tg_1 - \mathsf{T}g_1 + C_Me^{-\widetilde{\Phi}}\mathsf{L}g.
	\end{equation}
	Here we have used notations
	\begin{equation*}
		\mathsf{T} := v \cdot \nabla_x + \mathsf{R}x \cdot \nabla_x - \mathsf{R}v \cdot \nabla_v - \nabla_x\widetilde{\Phi} \cdot \nabla_v.
	\end{equation*}
	By the definition of collision ${Q}$ and $\mathsf{L}$, we write
	\begin{equation*}
		\mathsf{L}g = \int_{\R^3\times\S^2}(\mu'_*g' + \mu'g'_* - \mu_*g - \mu g_*)Bd\sigma dv_*,
	\end{equation*}
	together with conservation of momentum and energy \eqref{conserv}, we have
	\begin{equation}\label{propertyL1}
		\mathsf{L}\mu = \mathsf{L}(v\mu) = \mathsf{L}(|v|^2\mu) = 0 \:\:\: \Rightarrow \:\:\: \mathsf{L}g_0 = 0, \: \mathsf{L}g = \mathsf{L}g_1.
	\end{equation}
	By using pre-post change of variables, we have
	\begin{equation*}
		(\mathsf{L}g, h)_{L_v^2} = \int_{\R^6\times\S^2}(\mu'_*g' + \mu'g'_* - \mu_*g - \mu g_*)hBd\sigma dvdv_*
		= \int_{\R^6\times\S^2}\mu_*g(h + h_* - h' - h'_*)Bd\sigma dvdv_*.
	\end{equation*}
	Again by the conservation of momentum and energy \eqref{conserv}, we have
	\begin{equation}\label{propertyL2}
		(\mathsf{L}g, 1)_{L_v^2} = (\mathsf{L}g, v)_{L_v^2} = (\mathsf{L}g, |v|^2)_{L_v^2} = 0.
	\end{equation}
	
	\subsubsection{Moment equations for $a$, $b$, and $c$}
We take the inner products of \eqref{Lineareqg0} with $1$, $v$, and $|v|^2-3$ in $L^2(\R^3_v)$. Using \eqref{g0form}, \eqref{propertyL1}, and \eqref{propertyL2}, we obtain
	\begin{equation*}
	\begin{aligned}
		\partial_ta + \nabla_x \cdot b + \mathsf{R}x \cdot \nabla_xa &= -(\mathsf{T}g_1, 1)_{L_v^2};\\
		\partial_tb + \nabla_x(a + 2c) + \mathsf{R}x \cdot \nabla_xb + \mathsf{R}b + (\nabla_x\widetilde{\Phi})a &= -(\mathsf{T}g_1, v)_{L_v^2};\\
		\partial_tc + \frac{1}{3}\nabla_x \cdot b + \mathsf{R}x \cdot \nabla_xc + \frac{1}{3}\nabla_x\widetilde{\Phi} \cdot b &= -\frac{1}{6}(\mathsf{T}g_1, |v|^2 - 3)_{L_v^2}.\\
	\end{aligned}
	\end{equation*}
	Next, take inner product of \eqref{Lineareqg0} with $v_iv_j - \delta_{ij}, 1 \le i, j \le 3$ in $L^2(\R^3_v)$ space, we can obtain that
	\begin{equation*}
		\delta_{ij}\partial_tc + \frac{\partial_jb_i + \partial_ib_j}{2} + \delta_{ij}\mathsf{R}x \cdot \nabla_xc + \frac{(\partial_j\widetilde{\Phi})b_i + (\partial_i\widetilde{\Phi})b_j}{2} = \frac{1}{2}(-\partial_tg_1 - \mathsf{T}g_1 + C_Me^{-\widetilde{\Phi}}\mathsf{L}g_1, v_iv_j - \delta_{ij})_{L_v^2}.\\
	\end{equation*}
	Finally, take inner product  of \eqref{Lineareqg0} with $v(|v|^2 - 5)$ in $L^2(\R^3_v)$ space, we have 
	\begin{equation*}
		\nabla_xc + (\nabla_x\widetilde{\Phi})c = \frac{1}{10}(-\partial_tg_1 - \mathsf{T}g_1 + C_Me^{-\widetilde{\Phi}}\mathsf{L}g_1, v(|v|^2 - 5))_{L_v^2}.\\
	\end{equation*}
	Multiply $e^{\widetilde{\Phi}}$ on both sides of  each formula and use notations
	\begin{equation}\label{Defilk}
	\begin{gathered}
		l_a = -e^{\widetilde{\Phi}}(\mathsf{T}g_1, 1)_{L_v^2} = 0; \:\:\: l_b = -e^{\widetilde{\Phi}}(\mathsf{T}g_1, v)_{L_v^2}; \:\:\: l_c = -\frac{1}{6}e^{\widetilde{\Phi}}(\mathsf{T}g_1, |v|^2 - 3)_{L_v^2} \\
		K = (k_{ij}) = -\frac{1}{2}e^{\widetilde{\Phi}}(g_1, v_iv_j - \delta_{ij})_{L_v^2}; \:\:\: L = (l_{ij}) = \frac{1}{2}(C_M\mathsf{L}g_1 - e^{\widetilde{\Phi}}\mathsf{T}g_1, v_iv_j - \delta_{ij})_{L_v^2};\\
		k_d = -\frac{1}{10}e^{\widetilde{\Phi}}(g_1, v(|v|^2 - 5))_{L_v^2}; \:\:\: l_d = \frac{1}{10}(C_M\mathsf{L}g_1 - e^{\widetilde{\Phi}}\mathsf{T}g_1, v(|v|^2 - 5))_{L_v^2},
	\end{gathered}
	\end{equation}
	the above formulas become
	\begin{align}
		\partial_t(ae^{\widetilde{\Phi}}) + &(\nabla_x \cdot b)e^{\widetilde{\Phi}} + \mathsf{R}x \cdot \nabla_x(ae^{\widetilde{\Phi}}) = 0;\label{pta}\\
		\partial_t(be^{\widetilde{\Phi}}) + &\nabla_x(ae^{\widetilde{\Phi}}) + 2(\nabla_xc)e^{\widetilde{\Phi}} + \mathsf{R}be^{\widetilde{\Phi}} + \mathsf{R}x \cdot \nabla_x(be^{\widetilde{\Phi}}) = l_b;\label{ptb}\\
		\partial_t(ce^{\widetilde{\Phi}}) + &\frac{1}{3}\nabla_x \cdot (be^{\widetilde{\Phi}}) + \mathsf{R}x \cdot \nabla_x(ce^{\widetilde{\Phi}}) = l_c;\label{ptc}\\
		\mathbb{I}_{3 \times 3}\partial_t(ce^{\widetilde{\Phi}}) + &\nabla_x^{sym}(be^{\widetilde{\Phi}}) + \mathbb{I}_{3 \times 3}\mathsf{R}x \cdot \nabla_x(ce^{\widetilde{\Phi}}) = \partial_tK + L;\label{ptcI}\\
		&\nabla_x(ce^{\widetilde{\Phi}}) = \partial_tk_d + l_d,\label{pxc}
	\end{align}
	where we also use the fact that $\mathsf{R}x \cdot \nabla_x\widetilde{\Phi} = 0$ since $\nabla_x\widetilde{\Phi} = \al\<x\>^{p-2}x,p\in(1,2),\nabla_x\widetilde{\Phi} = |x|^{p - 2}x + \mathsf{R}^2x,p\in(2,\infty)$ and $\mathsf{R}$ is skew-symmetric and the notation $\nabla_x^{sym}$ is defined in \eqref{Defisymskew}.

 Before going further, we introduce the notation that for $f = f(x)$,
	\begin{equation}\label{deofAandCphi}
		\mathscr{A}(f) := \frac{\int_{\R^3} f(x)e^{-\widetilde{\Phi}(x)}dx}{\int_{\R^3} e^{-\widetilde{\Phi}(x)}dx} = C_{\widetilde{\Phi}}\int_{\R^3} f(x)e^{-\widetilde{\Phi}(x)}dx,\quad C_{\widetilde{\Phi}}=\Big(\int_{\R^3} e^{-\widetilde{\Phi}(x)}dx\Big)^{-1}.
	\end{equation}
	If $f$ is odd in any coordinate, then $\mathscr A(f)=0$. Moreover, every sufficiently regular $f$ satisfies
	\begin{equation*}
	\begin{aligned}
		&\mathscr{A}((\nabla_xf)e^{\widetilde{\Phi}}) = C_{\widetilde{\Phi}}\int_{\R^3}\nabla_xfdx = 0;\\
		&\mathscr{A}(\nabla_xf) = C_{\widetilde{\Phi}}\int_{\R^3}(\nabla_xf)e^{-{\widetilde{\Phi}}}dx = C_{\widetilde{\Phi}}\int_{\R^3}(\nabla_x{\widetilde{\Phi}})fe^{-{\widetilde{\Phi}}}dx = \mathscr{A}((\nabla_x{\widetilde{\Phi}})f);\\
		&\mathscr{A}(\mathsf{R}x \cdot \nabla_xf) = \mathscr{A}(\nabla_x \cdot (\mathsf{R}xf)) = \mathscr{A}((\nabla_x{\widetilde{\Phi}}) \cdot (\mathsf{R}xf)) = 0.
	\end{aligned}
	\end{equation*}
	Besides, we note that
	\begin{equation*}
	\begin{aligned}
		\mathscr{A}(l_b) &= -C_{\widetilde{\Phi}}\int(v \cdot \nabla_xg_1 + \mathsf{R}x \cdot \nabla_xg_1 - \mathsf{R}v \cdot \nabla_vg_1 - \nabla_x{\widetilde{\Phi}} \cdot \nabla_vg_1)vdxdv = 0;\\
		\mathscr{A}(l_c) &= -\frac{1}{6}C_{\widetilde{\Phi}}\int(v \cdot \nabla_xg_1 + \mathsf{R}x \cdot \nabla_xg_1 - \mathsf{R}v \cdot \nabla_vg_1 - \nabla_x{\widetilde{\Phi}} \cdot \nabla_vg_1)(|v|^2 - 3)dxdv = 0.
	\end{aligned}
	\end{equation*}
%	We use notation $A \lesssim B$, if there exists some contant $C$ such that $A \le CB$.

%%%%%%%%%%%%%%%%%%%%%%%%%%%%%%%%%%%%%%%%%%%%%%%%%%

\subsection{Coercive and finite-dimensional macroscopic modes}
The following splitting removes precisely the finite-dimensional obstructions to the weighted Poincaré and Poincaré--Korn inequalities. The components carrying the subscript $\mathsf n$ are the coercive parts; the remaining scalar, vector, and matrix coefficients will be controlled later by their evolution equations and by the five conservation laws.
We introduce
	\begin{equation}\label{Defianbncn}
	\begin{gathered}
		c_\mathsf{n}e^{\widetilde{\Phi}} = ce^{\widetilde{\Phi}} - \mathscr{A}(ce^{\widetilde{\Phi}});\:\:\:\:\:\:\:\:\:\:\:\:
		b_\mathsf{n}e^{\widetilde{\Phi}} = be^{\widetilde{\Phi}} - \mathscr{A}(be^{\widetilde{\Phi}}) - B(t)x - \frac{1}{3}\mathscr{A}(\nabla_x \cdot (be^{\widetilde{\Phi}}))x;\\
		a_\mathsf{n}e^{\widetilde{\Phi}} = ae^{\widetilde{\Phi}} - 2({\widetilde{\Phi}} - \mathscr{A}({\widetilde{\Phi}}))\mathscr{A}(ce^{\widetilde{\Phi}}) - \mathscr{A}(\nabla_x(ae^{\widetilde{\Phi}})) \cdot x + (|x|^2 - \mathscr{A}(|x|^2))A(t) + T(t)x \cdot x - \mathscr{A}(T(t)x \cdot x),
	\end{gathered}
	\end{equation}
	where $B(t) := \mathscr{A}(\nabla_x^{skew}(be^{\widetilde{\Phi}}))$ is skew-symmetric, $T(t) := \frac{1}{2}(\mathsf{R}B(t) + B(t)\mathsf{R})$ is symmetric, and
	\begin{equation}\label{defiAt}
		A(t) := \frac{1}{6}(2\mathscr{A}((\Delta_x{\widetilde{\Phi}})ce^{\widetilde{\Phi}}) - \mathscr{A}(\Delta_x(ae^{\widetilde{\Phi}})) - \mathscr{A}(\nabla_x{\widetilde{\Phi}} \cdot \mathsf{R}be^{\widetilde{\Phi}}) - \mathscr{A}(\nabla_x{\widetilde{\Phi}} \cdot (\mathsf{R}x \cdot \nabla_xbe^{\widetilde{\Phi}}))).
	\end{equation}
		By construction, $\mathscr A(c_\mathsf n e^{\widetilde\Phi})=0$. The mass condition in \eqref{Conserveg}, together with \eqref{g0form}, gives $\mathscr A(ae^{\widetilde\Phi})=0$ and hence $\mathscr A(a_\mathsf n e^{\widetilde\Phi})=0$. Thus the normalized scalar fields satisfy the orthogonality hypothesis of Lemma \ref{VPI}. Moreover, $\mathscr A(x)=0$ and $\mathscr A(\nabla_x^{\rm skew}x)=0$, so
		\[
		\mathscr A(b_\mathsf n e^{\widetilde\Phi})=0,
		\qquad
		\mathscr A\!\left(\nabla_x^{\rm skew}(b_\mathsf n e^{\widetilde\Phi})\right)=0.
		\]
		Therefore $b_\mathsf n e^{\widetilde\Phi}$ satisfies the orthogonality hypotheses of the weighted Poincar\'e--Korn inequality in Lemma \ref{VPKI}.
	
\subsection{Estimates for $a_\mathsf{n}$, $b_\mathsf{n}$, and $c_\mathsf{n}$}
We now estimate the normalized macroscopic coefficients. The goal of this subsection is the following lemma.
	\begin{lemma}[Coercive macroscopic components]\label{Estimateanbncn}
		Fix $\de_0\in(0,1/12)$ and choose $\de_i$, $1\le i\le4$, as in Lemma \ref{lemma5.7}; namely,
		\begin{equation}\label{de01234}
			\begin{aligned}
	&\de_0>\de_1>\de_2>\de_3>\de_4>0, \de_0+\de_4>2\de_1,\de_1+\de_4>2\de_2,\de_2+\de_4>2\de_3;\\
	&\|\Lam_{\de_1}f\|_{L^2(V_{\de_0})}\sim \|\Lam_{\de_0}f\|_{L^2(V_{\de_0})},	\|\Lam_{\de_2}f\|_{L^2(V_{\de_1})}\sim \|\Lam_{\de_1}f\|_{L^2(V_{\de_1})},\|\Lam_{\de_3}f\|_{L^2(V_{\de_2})}\sim \|\Lam_{\de_2}f\|_{L^2(V_{\de_2})}.
			\end{aligned}
		\end{equation}
Set $\mathsf H(x,v):=\widetilde\Phi(x)+|v|^2/2$. Then
		\begin{equation*}
		\begin{aligned}
			&\frac{d}{dt}\mathcal{J}_c + \frac{1}{2}\|\Lambda_{\delta_1}^{-\frac{1}{2}}\nabla_x(c_\mathsf{n}e^{\widetilde{\Phi}})\|_{L^2(V_{\delta_1})}^2
			\lesssim \|e^{\delta_0{\mathsf{H}}}g_1\|_{L_{x, v}^2}^2 + \|e^{\delta_0{\mathsf{H}}}g_1\|_{L_{x, v}^2}\|\Lambda_{\delta_4}^{-1}\partial_t\nabla_x(c_\mathsf{n}e^{\widetilde{\Phi}})\|_{L^2(V_{\delta_4})};\\
			&\frac{d}{dt}\mathcal{J}_b + \frac{1}{2}\|\Lambda_{\delta_2}^{-\frac{1}{2}}\nabla_x^{sym}(b_\mathsf{n}e^{\widetilde{\Phi}})\|_{L^2(V_{\delta_2})}^2
			\lesssim \|e^{\delta_0{\mathsf{H}}}g_1\|_{L_{x, v}^2}^2 + \|c_\mathsf{n}e^{\widetilde{\Phi}}\|_{L^2(V_{\delta_1})}^2 \\
			&\qquad\qquad\qquad\qquad\qquad+ (\|e^{\delta_0{\mathsf{H}}}g_1\|_{L_{x, v}^2} + \|c_\mathsf{n}e^{\widetilde{\Phi}}\|_{L^2(V_{\delta_1})})\|\Lambda_{\delta_4}^{-1}\partial_t\nabla_x^{sym}(b_\mathsf{n}e^{\widetilde{\Phi}})\|_{L^2(V_{\delta_4})};\\
			&\frac{d}{dt}\mathcal{J}_a + \frac{1}{2}\|\Lambda_{\delta_3}^{-\frac{1}{2}}\nabla_x(a_\mathsf{n}e^{\widetilde{\Phi}})\|_{L^2(V_{\delta_3})}^2
			\lesssim \|e^{\delta_0{\mathsf{H}}}g_1\|_{L_{x, v}^2}^2 + \|c_\mathsf{n}e^{\widetilde{\Phi}}\|_{L^2(V_{\delta_1})}^2 + \|b_\mathsf{n}e^{\widetilde{\Phi}}\|_{L^2(V_{\delta_2})}^2\\
			&\qquad\qquad\qquad\qquad\qquad+ (\|e^{\delta_0{\mathsf{H}}}g_1\|_{L_{x, v}^2} + \|b_\mathsf{n}e^{\widetilde{\Phi}}\|_{L^2(V_{\delta_2})})\|\Lambda_{\delta_4}^{-1}\partial_t\nabla_x(a_\mathsf{n}e^{\widetilde{\Phi}})\|_{L^2(V_{\delta_4})},
		\end{aligned}
		\end{equation*}
		where $\mathcal{J}_c, \mathcal{J}_b, \mathcal{J}_a$ are defined as
		\begin{equation*}
		\begin{aligned}
			\mathcal{J}_c &:= -(k_d, \Lambda_{\delta_1}^{-1}\nabla_x(c_\mathsf{n}e^{\widetilde{\Phi}}))_{L^2(V_{\delta_1})},\\
			\mathcal{J}_b &:= -(K - \mathbb{I}_{3 \times 3}c_\mathsf{n}e^{\widetilde{\Phi}}, \Lambda_{\delta_2}^{-1}\nabla_x^{sym}(b_\mathsf{n}e^{\widetilde{\Phi}}))_{L^2(V_{\delta_2})},\\
			\mathcal{J}_a &:= -(\mathscr{A}(K^*)x - 2k_d - b_\mathsf{n}e^{\widetilde{\Phi}}, \Lambda_{\delta_3}^{-1}\nabla_x(a_\mathsf{n}e^{\widetilde{\Phi}}))_{L^2(V_{\delta_3})}
		\end{aligned}
		\end{equation*}
		and $K^*$ is defined in \eqref{DefiK*}.
	\end{lemma}

%%%%%%%%%%%%%%%%%%%%%%%%%%%%%%%%%%%%%%%%%%%%%%%%%%
We will provide the proof of this lemma in several steps. 

	\subsubsection{Estimates for $c_\mathsf{n}$} We start with the estimate of $c_\mathsf{n}$.
	\begin{lemma}\label{015}
		Let $\de_i$, $i=0,\ldots,4$, be defined by \eqref{de01234}. Then
		\begin{equation*}
		\frac{d}{dt}\mathcal{J}_c + \frac{1}{2}\|\Lambda_{\delta_1}^{-\frac{1}{2}}\nabla_x(c_\mathsf{n}e^{\widetilde{\Phi}})\|_{L^2(V_{\delta_1})}^2
		\lesssim \|\Lambda_{\delta_1}^{-\frac{1}{2}}l_d\|_{L^2(V_{\delta_1})}^2 + \|k_d\|_{L^2(V_{\delta_0})}\|\Lambda_{\delta_4}^{-1}\partial_t\nabla_x(c_\mathsf{n}e^{\widetilde{\Phi}})\|_{L^2(V_{\delta_4})},
	\end{equation*}
	where
$
		\mathcal{J}_c := -(k_d, \Lambda_{\delta_1}^{-1}\nabla_x(c_\mathsf{n}e^{\widetilde{\Phi}}))_{L^2(V_{\delta_1})}.
$

	\end{lemma}
	\begin{proof}
	By \eqref{pxc},
	\begin{equation*}
		\nabla_x(c_\mathsf{n}e^{\widetilde{\Phi}}) = \nabla_x(ce^{\widetilde{\Phi}}) = \partial_tk_d + l_d.
	\end{equation*}
	Thus we have
		\begin{equation}\label{020}
		\begin{aligned}
			\|\Lambda_{\delta_1}^{-\frac{1}{2}}\nabla_x(c_\mathsf{n}e^{\widetilde{\Phi}})\|_{L^2(V_{\delta_1})}^2
			&= (\partial_tk_d + l_d,
			\Lambda_{\delta_1}^{-1}\nabla_x(c_\mathsf{n}e^{\widetilde{\Phi}}))_{L^2(V_{\delta_1})} \\
			&= \frac{d}{dt}(k_d,
			\Lambda_{\delta_1}^{-1}\nabla_x(c_\mathsf{n}e^{\widetilde{\Phi}}))_{L^2(V_{\delta_1})}\\
			&\quad -(k_d,
			\Lambda_{\delta_1}^{-1}\partial_t\nabla_x(c_\mathsf{n}e^{\widetilde{\Phi}}))_{L^2(V_{\delta_1})}\\
			&\quad +(l_d,
			\Lambda_{\delta_1}^{-1}\nabla_x(c_\mathsf{n}e^{\widetilde{\Phi}}))_{L^2(V_{\delta_1})}.
		\end{aligned}
		\end{equation}
	By the Cauchy--Schwarz inequality,
		\begin{equation}\label{021}
		\begin{aligned}
		&(l_d, \Lambda_{\delta_1}^{-1}\nabla_x(c_\mathsf{n}e^{\widetilde{\Phi}}))_{L^2(V_{\delta_1})}\\
		&\qquad\le \frac{1}{2}\|\Lambda_{\delta_1}^{-\frac{1}{2}}l_d\|_{L^2(V_{\delta_1})}^2
		+\frac{1}{2}\|\Lambda_{\delta_1}^{-\frac{1}{2}}\nabla_x(c_\mathsf{n}e^{\widetilde{\Phi}})\|_{L^2(V_{\delta_1})}^2.
		\end{aligned}
		\end{equation}
	On the other hand, we denote $\Lam^{*\de_1}_{\de_4}$ as the dual operator of $\Lam_{\de_4}$ in space $L^2(V_{\de_1})$, i.e. $(\Lam^{*\de_1}_{\de_4} f,g)_{L^2(V_{\de_1})}=( f,\Lam_{\de_4}g)_{L^2(V_{\de_1})}$ for any regular functions $f,g$. Indeed, we can compute that \begin{equation}\label{088}
		\Lam^{*\de_1}_{\de_4}=\Lam_{\de_1}+\na(V_{\de_4}-V_{\de_1})\cdot\na V_{\de_1}-\D(V_{\de_4}-V_{\de_1})-\na(V_{\de_4}-V_{\de_1})\cdot\na.
		\end{equation} 
		Then we have 
	\begin{equation*}
		\begin{aligned}
		-(k_d, \Lambda_{\delta_1}^{-1}\partial_t\nabla_x(c_\mathsf{n}e^{\widetilde{\Phi}}))_{L^2(V_{\delta_1})}&=-(\Lambda_{\delta_4}^{*\de_1}\Lambda_{\delta_1}^{-1}k_d, \Lambda_{\delta_4}^{-1}\partial_t\nabla_x(c_\mathsf{n}e^{\widetilde{\Phi}}))_{L^2(V_{\delta_1})}\\
		&\leq\|\Lambda_{\delta_4}^{*\de_1}\Lambda_{\delta_1}^{-1}k_d\|_{L^2(V_{\de_0})}\|\Lambda_{\delta_4}^{-1}\partial_t\nabla_x(c_\mathsf{n}e^{\widetilde{\Phi}})\|_{L^2(V_{\de_4})}
		\end{aligned}
	\end{equation*}
	where we use the fact that $2\de_1<\de_0+\de_4$ and Lemma \ref{lemma5.7}(3).  We claim that
	\begin{equation}\label{027}
	\|\Lambda_{\delta_4}^{*\de_1}\Lambda_{\delta_1}^{-1}k_d\|_{L^2(V_{\de_0})}\ls \|k_d\|_{L^2(V_{\de_0})},
	\end{equation}
	which is equivalent to
	\beno
		\|\Lambda_{\delta_4}^{*\de_1}\tilde{k}\|_{L^2(V_{\de_0})}\ls \|\Lambda_{\delta_1}\tilde{k}\|_{L^2(V_{\de_0})},\quad \tilde{k}=\Lambda_{\delta_1}^{-1}k_d.
	\eeno
By the choice of $\de_i$ in \eqref{de01234}, we know that
	\beno
\|\Lambda_{\delta_0}\tilde{k}\|_{L^2(V_{\de_0})}\sim \|\Lambda_{\delta_1}\tilde{k}\|_{L^2(V_{\de_0})}.
\eeno
Thus we only need to prove 
\beno
\|\Lambda_{\delta_4}^{*\de_1}\tilde{k}\|_{L^2(V_{\de_0})}\ls\|\Lambda_{\delta_0}\tilde{k}\|_{L^2(V_{\de_0})}
\eeno
which is equivalent to
		\beno
	\|\Lambda_{\delta_4}^{*\de_1}\Lambda_{\delta_0}^{-1}\hat{k}\|_{L^2(V_{\de_0})}\ls \|\hat{k}\|_{L^2(V_{\de_0})},\quad \hat{k}=\Lambda_{\delta_0}\tilde{k}.
	\eeno
This follows from \eqref{088} and Lemma \ref{BoundedLambda}. Therefore,
\begin{equation}\label{022}
	\begin{aligned}
		-(k_d, \Lambda_{\delta_1}^{-1}\partial_t\nabla_x(c_\mathsf{n}e^{\widetilde{\Phi}}))_{L^2(V_{\delta_1})}&=-(\Lambda_{\delta_4}^*\Lambda_{\delta_1}^{-1}k_d, \Lambda_{\delta_4}^{-1}\partial_t\nabla_x(c_\mathsf{n}e^{\widetilde{\Phi}}))_{L^2(V_{\delta_1})}\\
		&\leq\|k_d\|_{L^2(V_{\de_0})}\|\Lambda_{\delta_4}^{-1}\partial_t\nabla_x(c_\mathsf{n}e^{\widetilde{\Phi}})\|_{L^2(V_{\de_4})}.
	\end{aligned}
\end{equation}
Thus we obtain the desired result by combining \eqref{020}, \eqref{021} and \eqref{022}.
\end{proof}

%%%%%%%%%%%%%%%%%%%%%%%%%%%%%%%%%%%%%%%%%%%%%%%%%%
\subsubsection{Estimates for $b_\mathsf{n}$} Next, we give the estimates of $b_\mathsf{n}$.
	\begin{lemma}\label{016}
		Let $\de_i$, $i=0,\ldots,4$, be defined by \eqref{de01234}. Then
			\begin{equation*}
			\begin{aligned}
				\frac{d}{dt}\mathcal{J}_b + \frac{1}{2}\|\Lambda_{\delta_2}^{-\frac{1}{2}}\nabla_x^{sym}(b_\mathsf{n}e^{\widetilde{\Phi}})\|_{L^2(V_{\delta_2})}^2 &\lesssim \|\Lambda_{\delta_2}^{-\frac{1}{2}}L\|_{L^2(V_{\delta_2})}^2 + \|c_\mathsf{n}e^{\widetilde{\Phi}}\|_{L^2(V_{\delta_1})}^2 \\+& (\|c_\mathsf{n}e^{\widetilde{\Phi}}\|_{L^2(V_{\delta_1})} + \|K\|_{L^2(V_{\delta_0})})\|\Lambda_{\delta_4}^{-1}\partial_t\nabla_x^{sym}(b_\mathsf{n}e^{\widetilde{\Phi}})\|_{L^2(V_{\delta_4})},
			\end{aligned}
		\end{equation*}
where
$
			\mathcal{J}_b := -(K - \mathbb{I}_{3 \times 3}c_\mathsf{n}e^{\widetilde{\Phi}}, \Lambda_{\delta_2}^{-1}\nabla_x^{sym}(b_\mathsf{n}e^{\widetilde{\Phi}}))_{L^2(V_{\delta_2})}.
$
%	\begin{equation*}
%		\begin{aligned}
%			\frac{d}{dt}\mathcal{J}_b + %\frac{1}{2}\|\Lambda_{\delta_2}^{-\frac{1}{2}}\nabla_x^{sym}(b_\mathsf{n}e^{\widetilde{\Phi}})\|_{L^2(V_{\delta_2})}^2
%			&\lesssim \|e^{\delta_0{\mathsf{H}}}g_1\|_{L_{x, v}^2}^2 + \|c_\mathsf{n}e^{\widetilde{\Phi}}\|_{L^2(V_{\delta_1})}^2 \\
%			&+ (\|e^{\delta_0{\mathsf{H}}}g_1\|_{L_{x, v}^2} + \|c_\mathsf{n}e^{\widetilde{\Phi}}\|_{L^2(V_{\delta_1})})\|\Lambda_{\delta_4}^{-1}\partial_t\nabla_x^{sym}(b_\mathsf{n}e^{\widetilde{\Phi}})\|_{L^2(V_{\delta_4})}
%		\end{aligned}
%	\end{equation*}
%	where 
%			\begin{equation*}
%		\begin{aligned}
%			\mathcal{J}_b &:= -(K - \mathbb{I}_{3 \times 3}c_\mathsf{n}e^{\widetilde{\Phi}}, \Lambda_{\delta_2}^{-1}\nabla_x^{sym}(b_\mathsf{n}e^{\widetilde{\Phi}}))_{L^2(V_{\delta_2})}.
%		\end{aligned}
%	\end{equation*}
	\end{lemma}
	\begin{proof}
	Recall that  $B(t) = \mathscr{A}(\nabla_x^{skew}(be^{\widetilde{\Phi}}))$ is skew-symmetric, we compute that
	\begin{equation*}
		\partial_j(B(t)x)_i = \partial_j\sum_{k = 1}^3B_{ik}(t)x_k = B_{ij}(t) = \frac{1}{2}\mathscr{A}(\partial_j(b_ie^{\widetilde{\Phi}}) - \partial_i(b_je^{\widetilde{\Phi}})).
	\end{equation*}
By the definition of $b_\mathsf n$ in \eqref{Defianbncn},
	\begin{equation*}
	\begin{aligned}
		\partial_j(b_{\mathsf{n}, i}e^{\widetilde{\Phi}}) &= \partial_j(b_ie^{\widetilde{\Phi}}) - \frac{1}{2}\mathscr{A}(\partial_j(b_ie^{\widetilde{\Phi}}) - \partial_i(b_je^{\widetilde{\Phi}})) - \frac{1}{3}\delta_{ij}\mathscr{A}(\nabla_x \cdot (be^{\widetilde{\Phi}}));\\
		\partial_i(b_{\mathsf{n}, j}e^{\widetilde{\Phi}}) &= \partial_i(b_je^{\widetilde{\Phi}}) - \frac{1}{2}\mathscr{A}(\partial_i(b_je^{\widetilde{\Phi}}) - \partial_j(b_ie^{\widetilde{\Phi}})) - \frac{1}{3}\delta_{ij}\mathscr{A}(\nabla_x \cdot (be^{\widetilde{\Phi}})).\\
	\end{aligned}
	\end{equation*}
Hence $\mathscr{A}(\nabla_x^{skew}(b_\mathsf{n}e^{\widetilde{\Phi}}))=0$, and
	\begin{equation}\label{bnbeq}
		\nabla_x^{sym}(b_\mathsf{n}e^{\widetilde{\Phi}}) = \nabla_x^{sym}(be^{\widetilde{\Phi}}) - \frac{1}{3}\mathbb{I}_{3 \times 3}\mathscr{A}(\nabla_x \cdot (be^{\widetilde{\Phi}})).
	\end{equation}
	To deduce the energy for $b_\mathsf{n}$, we start with \eqref{ptcI}, i.e.,
	\begin{equation}\label{ptcI+}
		\nabla_x^{sym}(be^{\widetilde{\Phi}}) = -\mathbb{I}_{3 \times 3}\partial_t(ce^{\widetilde{\Phi}}) - \mathbb{I}_{3 \times 3}\mathsf{R}x \cdot \nabla_x(ce^{\widetilde{\Phi}}) + \partial_tK + L.
	\end{equation}
	By the definition of $c_\mathsf{n}$ \eqref{Defianbncn}, we have
	\begin{equation*}
		ce^{\widetilde{\Phi}} = c_\mathsf{n}e^{\widetilde{\Phi}} + \mathscr{A}(ce^{\widetilde{\Phi}}).
	\end{equation*}
	Then $\nabla_x(ce^{\widetilde{\Phi}}) = \nabla_x(c_\mathsf{n}e^{\widetilde{\Phi}})$, and by \eqref{ptc} we have
	\begin{equation*}
		\partial_t(ce^{\widetilde{\Phi}}) = \partial_t(c_\mathsf{n}e^{\widetilde{\Phi}}) + \mathscr{A}(\partial_t(ce^{\widetilde{\Phi}})) = \partial_t(c_\mathsf{n}e^{\widetilde{\Phi}}) - \frac{1}{3}\mathscr{A}(\nabla_x \cdot (be^{\widetilde{\Phi}})) - \mathscr{A}(\mathsf{R}x \cdot \nabla_x(ce^{\widetilde{\Phi}})) + \mathscr{A}(l_c).
	\end{equation*}
	Note the fact that $\mathscr{A}(l_c) = 0$, and $\mathscr{A}(\mathsf{R}x \cdot \nabla_xf) = 0$ for any function $f$, thus \eqref{ptcI+} becomes
	\begin{equation*}
		\nabla_x^{sym}(be^{\widetilde{\Phi}}) = -\mathbb{I}_{3 \times 3}\partial_t(c_\mathsf{n}e^{\widetilde{\Phi}}) + \frac{1}{3}\mathbb{I}_{3 \times 3}\mathscr{A}(\nabla_x \cdot (be^{\widetilde{\Phi}})) - \mathbb{I}_{3 \times 3}\mathsf{R}x \cdot \nabla_x(c_\mathsf{n}e^{\widetilde{\Phi}}) + \partial_tK + L.
	\end{equation*}
	We rewrite it as
	\begin{equation*}
		\nabla_x^{sym}(be^{\widetilde{\Phi}}) - \frac{1}{3}\mathbb{I}_{3 \times 3}\mathscr{A}(\nabla_x \cdot (be^{\widetilde{\Phi}})) = \partial_t(K - \mathbb{I}_{3 \times 3}(c_\mathsf{n}e^{\widetilde{\Phi}})) + L - \mathbb{I}_{3 \times 3}\mathsf{R}x \cdot \nabla_x(c_\mathsf{n}e^{\widetilde{\Phi}}).
	\end{equation*}
	Together with \eqref{bnbeq}, we get
\begin{equation}\label{023}
	\begin{aligned}
		\|\Lambda_{\delta_2}^{-\frac{1}{2}}\nabla_x^{sym}(b_\mathsf{n}e^{\widetilde{\Phi}})\|_{L^2(V_{\delta_2})}^2
&= (\nabla_x^{sym}(b_\mathsf{n}e^{\widetilde{\Phi}}), \Lambda_{\delta_2}^{-1}\nabla_x^{sym}(b_\mathsf{n}e^{\widetilde{\Phi}}))_{L^2(V_{\delta_2})} \\
&= (\partial_t(K - \mathbb{I}_{3 \times 3}(c_\mathsf{n}e^{\widetilde{\Phi}})) + L - \mathbb{I}_{3 \times 3}\mathsf{R}x \cdot \nabla_x(c_\mathsf{n}e^{\widetilde{\Phi}}), \Lambda_{\delta_2}^{-1}\nabla_x^{sym}(b_\mathsf{n}e^{\widetilde{\Phi}}))_{L^2(V_{\delta_2})}.
	\end{aligned}
\end{equation}
On one hand, we have that
	\begin{multline}\label{024}
		(L - \mathbb{I}_{3 \times 3}\mathsf{R}x \cdot \nabla_x(c_\mathsf{n}e^{\widetilde{\Phi}}), \Lambda_{\delta_2}^{-1}\nabla_x^{sym}(b_\mathsf{n}e^{\widetilde{\Phi}}))_{L^2(V_{\delta_2})} \\
		\le 3(\|\Lambda_{\delta_2}^{-\frac{1}{2}}L\|_{L^2(V_{\delta_2})}^2 + \|\Lambda_{\delta_2}^{-\frac{1}{2}}\mathsf{R}x \cdot \nabla_x(c_\mathsf{n}e^{\widetilde{\Phi}})\|_{L^2(V_{\delta_2})}^2) + \frac{1}{2}\|\Lambda_{\delta_2}^{-\frac{1}{2}}\nabla_x^{sym}(b_\mathsf{n}e^{\widetilde{\Phi}})\|_{L^2(V_{\delta_2})}^2.
	\end{multline}
	For the second term on the right-hand side, Lemma \ref{BoundedLambda} gives
	\begin{equation}\label{025}
		\|\Lambda_{\delta_2}^{-\frac{1}{2}}\mathsf{R}x \cdot \nabla_x(c_\mathsf{n}e^{\widetilde{\Phi}})\|_{L^2(V_{\delta_2})}^2 \lesssim \|c_\mathsf{n}e^{\widetilde{\Phi}}\|_{L^2(V_{\delta_2})}^2
	 \lesssim \|c_\mathsf{n}e^{\widetilde{\Phi}}\|_{L^2(V_{\delta_1})}^2.
	\end{equation}
On the other hand, we write
	\begin{equation}\label{026}
		\begin{aligned}
	(\partial_t(K - \mathbb{I}_{3 \times 3}(c_\mathsf{n}e^{\widetilde{\Phi}})), \Lambda_{\delta_2}^{-1}\nabla_x^{sym}(b_\mathsf{n}e^{\widetilde{\Phi}}))_{L^2(V_{\delta_2})}	&= \frac{d}{dt}(K - \mathbb{I}_{3 \times 3}(c_\mathsf{n}e^{\widetilde{\Phi}}), \Lambda_{\delta_2}^{-1}\nabla_x^{sym}(b_\mathsf{n}e^{\widetilde{\Phi}}))_{L^2(V_{\delta_2})}\\
	 +& (\mathbb{I}_{3 \times 3}(c_\mathsf{n}e^{\widetilde{\Phi}})-K, \Lambda_{\delta_2}^{-1}\partial_t\nabla_x^{sym}(b_\mathsf{n}e^{\widetilde{\Phi}}))_{L^2(V_{\delta_2})},
		\end{aligned}
	\end{equation}
	and moreover,
	\begin{equation}\label{028}
		\begin{aligned}
		&(\mathbb{I}_{3 \times 3}(c_\mathsf{n}e^{\widetilde{\Phi}}) - K, \Lambda_{\delta_2}^{-1}\partial_t\nabla_x^{sym}(b_\mathsf{n}e^{\widetilde{\Phi}}))_{L^2(V_{\delta_2})}=(\Lambda_{\delta_4}^{*\de_2}\Lambda_{\delta_2}^{-1}\mathbb{I}_{3 \times 3}(c_\mathsf{n}e^{\widetilde{\Phi}}) - \Lambda_{\delta_4}^{*\de_2}\Lambda_{\delta_2}^{-1}K, \Lambda_{\delta_4}^{-1}\partial_t\nabla_x^{sym}(b_\mathsf{n}e^{\widetilde{\Phi}}))_{L^2(V_{\delta_2})}\\
		&\leq(\|\Lambda_{\delta_4}^{*\de_2}\Lambda_{\delta_2}^{-1}(c_\mathsf{n}e^{\widetilde{\Phi}})\|_{L^2(V_{\delta_1})} + \|\Lambda_{\delta_4}^{*\de_2}\Lambda_{\delta_2}^{-1}K\|_{L^2(V_{\delta_1})})\|\Lambda_{\delta_4}^{-1}\partial_t\nabla_x^{sym}(b_\mathsf{n}e^{\widetilde{\Phi}})\|_{L^2(V_{\delta_4})},
		\end{aligned}
	\end{equation}
	where we use the fact that $2\de_2< \de_1+\de_4$. Furthermore, by the same argument as \eqref{027}, we have 
	\begin{equation}\label{029}
	\|\Lambda_{\delta_4}^{*\de_2}\Lambda_{\delta_2}^{-1}(c_\mathsf{n}e^{\widetilde{\Phi}})\|_{L^2(V_{\delta_1})} + \|\Lambda_{\delta_4}^{*\de_2}\Lambda_{\delta_2}^{-1}K\|_{L^2(V_{\delta_1})}\ls \|c_\mathsf{n}e^{\widetilde{\Phi}}\|_{L^2(V_{\de_1})}+\|K\|_{L^2(V_{\de_1})}.
	\end{equation}
	Thus we get the desired result by combining \eqref{023}-\eqref{029}.
\end{proof}
%%%%%%%%%%%%%%%%%%%%%%%%%%%%%%%%%%%%%%%%%%%%%%%%%%

\subsubsection{Estimates for $a_\mathsf{n}$} Finally, we give the estimates of $a_\mathsf{n}$.
	\begin{lemma}\label{017}
Let $\de_i$, $i=0,\ldots,4$, be defined by \eqref{de01234}. Then
		\begin{align*}
		\frac{d}{dt}\mathcal{J}_a
		&+ \frac{1}{2}\|\Lambda_{\delta_3}^{-\frac{1}{2}}
		\nabla_x(a_\mathsf{n}e^{\widetilde{\Phi}})\|_{L^2(V_{\delta_3})}^2 \\
		&\lesssim \|\Lambda_{\delta_3}^{-\frac{1}{2}}l_b\|_{L^2(V_{\delta_3})}^2
		+ \|\Lambda_{\delta_3}^{-\frac{1}{2}}l_d\|_{L^2(V_{\delta_3})}^2 \\
		&\quad + |\mathscr{A}(\nabla_x \cdot l_b)|^2+|\mathscr{A}(L^*)|^2
		+|\mathscr{A}(\nabla_x^{skew}l_b)|^2 \\
		&\quad + \|c_\mathsf{n}e^{\widetilde{\Phi}}\|_{L^2(V_{\delta_1})}^2
		+ \|b_\mathsf{n}e^{\widetilde{\Phi}}\|_{L^2(V_{\delta_2})}^2 \\
		&\quad +\big(\|b_\mathsf{n}e^{\widetilde{\Phi}}\|_{L^2(V_{\delta_2})}
		+\|k_d\|_{L^2(V_{\delta_0})}+|\mathscr{A}(K^*)|\big) \\
		&\qquad\times
		\|\Lambda_{\delta_4}^{-1}\partial_t\nabla_x^{sym}
		(b_\mathsf{n}e^{\widetilde{\Phi}})\|_{L^2(V_{\delta_4})},
	\end{align*}
where $
\mathcal{J}_a := -(\mathscr{A}(K^*)x - 2k_d - b_\mathsf{n}e^{\widetilde{\Phi}}, \Lambda_{\delta_3}^{-1}\nabla_x(a_\mathsf{n}e^{\widetilde{\Phi}}))_{L^2(V_{\delta_3})}
$, where $K^*$ and $L^*$ are defined in \eqref{DefiK*}.
\end{lemma}
\begin{proof}
We will divide the calculation into two parts.

\underline{Step 1: Derivation of the $a$-equation.}
Recall that $\mathsf R$ and $B(t)$ are skew-symmetric matrices; therefore $T(t):=(\mathsf RB(t)+B(t)\mathsf R)/2$ is symmetric. The definition of $a_\mathsf n$ in \eqref{Defianbncn} gives
	\begin{equation}\label{anaeq}
		\nabla_x(a_\mathsf{n}e^{\widetilde{\Phi}}) = \nabla_x(ae^{\widetilde{\Phi}}) - 2(\nabla_x{\widetilde{\Phi}})\mathscr{A}(ce^{\widetilde{\Phi}}) - \mathscr{A}(\nabla_x(ae^{\widetilde{\Phi}})) + 2A(t)x + 2T(t)x.
	\end{equation}
	To deduce the energy for $a_\mathsf{n}$, we start with \eqref{ptb} to get that
	\begin{equation}\label{Energyan0}
		\nabla_x(ae^{\widetilde{\Phi}}) = -\partial_t(be^{\widetilde{\Phi}}) - 2(\nabla_xc)e^{\widetilde{\Phi}} - \mathsf{R}be^{\widetilde{\Phi}} - \mathsf{R}x \cdot \nabla_x(be^{\widetilde{\Phi}}) + l_b.
	\end{equation}
	We first recall the definition of $c_\mathsf{n}$ in \eqref{Defianbncn}, we have $ce^{\widetilde{\Phi}} = c_\mathsf{n}e^{\widetilde{\Phi}} + \mathscr{A}(ce^{\widetilde{\Phi}})$, together with \eqref{pxc} we get
	\begin{equation}\label{Energyan1}
		(\nabla_xc)e^{\widetilde{\Phi}} = \nabla_x(ce^{\widetilde{\Phi}}) - (\nabla_x{\widetilde{\Phi}})ce^{\widetilde{\Phi}} = \partial_tk_d + l_d - (\nabla_x{\widetilde{\Phi}})c_\mathsf{n}e^{\widetilde{\Phi}} - (\nabla_x{\widetilde{\Phi}})\mathscr{A}(ce^{\widetilde{\Phi}}).
	\end{equation}
	Next we recall the definition of $b_\mathsf{n}$ in \eqref{Defianbncn} reads like
	\begin{equation*}
		be^{\widetilde{\Phi}} = b_\mathsf{n}e^{\widetilde{\Phi}} + \mathscr{A}(be^{\widetilde{\Phi}}) + B(t)x + \frac{1}{3}\mathscr{A}(\nabla_x \cdot (be^{\widetilde{\Phi}}))x.
	\end{equation*}
	Thus we compute that
	\begin{equation}\label{Energyan2}
	\begin{gathered}
		\mathsf{R}be^{\widetilde{\Phi}} = \mathsf{R}b_\mathsf{n}e^{\widetilde{\Phi}} + \mathscr{A}(\mathsf{R}be^{\widetilde{\Phi}}) + \mathsf{R}B(t)x + \frac{1}{3}\mathscr{A}(\nabla_x \cdot (be^{\widetilde{\Phi}}))\mathsf{R}x;\\
		\mathsf{R}x \cdot \nabla_x(be^{\widetilde{\Phi}}) = \mathsf{R}x \cdot \nabla_x(b_\mathsf{n}e^{\widetilde{\Phi}}) + B(t)\mathsf{R}x + \frac{1}{3}\mathscr{A}(\nabla_x \cdot (be^{\widetilde{\Phi}}))\mathsf{R}x.
	\end{gathered}
	\end{equation}
	We plug \eqref{Energyan1} and \eqref{Energyan2} into \eqref{Energyan0} to get that
	\begin{multline}\label{Energyan3}
		\nabla_x(ae^{\widetilde{\Phi}}) = -\partial_t(be^{\widetilde{\Phi}}) - 2\partial_tk_d - 2l_d + 2(\nabla_x{\widetilde{\Phi}})c_\mathsf{n}e^{\widetilde{\Phi}} + 2(\nabla_x{\widetilde{\Phi}})\mathscr{A}(ce^{\widetilde{\Phi}}) \\- \mathsf{R}b_\mathsf{n}e^{\widetilde{\Phi}} - \mathsf{R}x \cdot \nabla_x(b_\mathsf{n}e^{\widetilde{\Phi}}) - \mathscr{A}(\mathsf{R}be^{\widetilde{\Phi}}) - 2T(t)x - \frac{2}{3}\mathscr{A}(\nabla_x \cdot (be^{\widetilde{\Phi}}))\mathsf{R}x + l_b.
	\end{multline}
	Subsequently, we are going to deal with $\partial_t(be^{\widetilde{\Phi}})$, again recall the definition of $b_\mathsf{n}$ in \eqref{Defianbncn},
	\begin{equation}\label{Energyan4}
		\partial_t(be^{\widetilde{\Phi}}) = \partial_t(b_\mathsf{n}e^{\widetilde{\Phi}}) + \partial_t\mathscr{A}(be^{\widetilde{\Phi}}) + \partial_t\mathscr{A}(\nabla_x^{skew}(be^{\widetilde{\Phi}}))x + \frac{1}{3}\partial_t\mathscr{A}(\nabla_x \cdot (be^{\widetilde{\Phi}}))x.
	\end{equation}
	We next compute $\partial_t\mathscr{A}(be^{\widetilde{\Phi}})$, $\partial_t\mathscr{A}(\nabla_x^{skew}(be^{\widetilde{\Phi}}))$, and $\partial_t\mathscr{A}(\nabla_x\cdot(be^{\widetilde{\Phi}}))$. Rewrite \eqref{ptb} as
	\begin{equation}\label{Energyan5}
		\partial_t(be^{\widetilde{\Phi}}) = -\nabla_x(ae^{\widetilde{\Phi}}) - 2(\nabla_xc)e^{\widetilde{\Phi}} - \mathsf{R}be^{\widetilde{\Phi}} - \mathsf{R}x \cdot \nabla_x(be^{\widetilde{\Phi}}) + l_b.
	\end{equation}
	Take $\mathscr{A}(\cdot)$ on \eqref{Energyan5}, and note that $\mathscr{A}(l_b) = 0$ and for any function $f$,
	\begin{equation*}
		\mathscr{A}((\nabla_xf)e^{\widetilde{\Phi}}) = 0, \:\:\: \mathscr{A}(\mathsf{R}x \cdot \nabla_xf) = 0,
	\end{equation*}
	we get
	\begin{equation}\label{Energyan6}
		\partial_t\mathscr{A}(be^{\widetilde{\Phi}}) = -\mathscr{A}(\nabla_x(ae^{\widetilde{\Phi}})) - \mathscr{A}(\mathsf{R}be^{\widetilde{\Phi}}).
	\end{equation}
	Next we take $\partial_j$ to the $i-$th component of \eqref{Energyan5} to get
	\begin{equation*}
		\partial_t\partial_j(b_ie^{\widetilde{\Phi}}) = -\partial_j\partial_i(ae^{\widetilde{\Phi}}) - 2\partial_j(\partial_ice^{\widetilde{\Phi}}) - \partial_j((\mathsf{R}b)_ie^{\widetilde{\Phi}}) - \partial_j(\mathsf{R}x \cdot \nabla_xb_ie^{\widetilde{\Phi}}) + \partial_j(l_{b, i}).
	\end{equation*}
	Applying $\mathscr A$ yields
	\begin{equation*}
		\partial_t\mathscr{A}(\partial_j(b_ie^{\widetilde{\Phi}})) = -\mathscr{A}(\partial_j\partial_i(ae^{\widetilde{\Phi}})) + 2\mathscr{A}(\partial_j\partial_i{\widetilde{\Phi}} ce^{\widetilde{\Phi}}) - \mathscr{A}(\partial_j((\mathsf{R}b)_ie^{\widetilde{\Phi}})) - \mathscr{A}(\partial_j(\mathsf{R}x \cdot \nabla_xb_ie^{\widetilde{\Phi}})) + \mathscr{A}(\partial_j(l_{b, i})),
	\end{equation*}
	where we used $\mathscr{A}(\partial_j(\partial_ice^{\widetilde{\Phi}}))=\mathscr{A}(\partial_j\widetilde{\Phi}\,\partial_ice^{\widetilde{\Phi}})=-\mathscr{A}(\partial_j\partial_i\widetilde{\Phi}\,ce^{\widetilde{\Phi}})$. Since $\mathscr A(\nabla_xf)=\mathscr A((\nabla_x\widetilde\Phi)f)$,
		\begin{equation}\label{Energyan7}
		\begin{aligned}
			\partial_t\mathscr{A}(\nabla_x \cdot (be^{\widetilde{\Phi}}))
			&= -\mathscr{A}(\Delta_x(ae^{\widetilde{\Phi}}))
			 + 2\mathscr{A}(\Delta_x{\widetilde{\Phi}} ce^{\widetilde{\Phi}})\\
			&\quad - \mathscr{A}(\nabla_x{\widetilde{\Phi}} \cdot \mathsf{R}be^{\widetilde{\Phi}})
			 - \mathscr{A}(\nabla_x{\widetilde{\Phi}} \cdot (\mathsf{R}x \cdot \nabla_xb_i)e^{\widetilde{\Phi}})\\
			&\quad + \mathscr{A}(\nabla_x \cdot l_b) \\
			&= 6A(t) + \mathscr{A}(\nabla_x \cdot l_b);\\
			\partial_t\mathscr{A}(\nabla_x^{skew}(be^{\widetilde{\Phi}}))
			&= -\mathscr{A}(\nabla_x^{skew}(\mathsf{R}be^{\widetilde{\Phi}}))\\
			&\quad - \mathscr{A}(\nabla_x^{skew}(\mathsf{R}x \cdot \nabla_xbe^{\widetilde{\Phi}}))
			 + \mathscr{A}(\nabla_x^{skew}l_b),
		\end{aligned}
		\end{equation}
	where $A(t)$ is defined in \eqref{defiAt}. We plug \eqref{Energyan6} and \eqref{Energyan7} into \eqref{Energyan4} to get
	\begin{multline*}
		\partial_t(be^{\widetilde{\Phi}}) = \partial_t(b_\mathsf{n}e^{\widetilde{\Phi}}) - \mathscr{A}(\nabla_x(ae^{\widetilde{\Phi}})) - \mathscr{A}(\mathsf{R}be^{\widetilde{\Phi}}) - \mathscr{A}(\nabla_x^{skew}(\mathsf{R}be^{\widetilde{\Phi}}))x \\- \mathscr{A}(\nabla_x^{skew}(\mathsf{R}x \cdot \nabla_xbe^{\widetilde{\Phi}}))x + \mathscr{A}(\nabla_x^{skew}l_b)x + 2A(t)x + \frac{1}{3}\mathscr{A}(\nabla_x \cdot l_b)x.
	\end{multline*}
	Therefore, \eqref{Energyan3} becomes
	\begin{multline}\label{Energyan8}
		\nabla_x(ae^{\widetilde{\Phi}}) - 2(\nabla_x{\widetilde{\Phi}})\mathscr{A}(ce^{\widetilde{\Phi}}) - \mathscr{A}(\nabla_x(ae^{\widetilde{\Phi}})) + 2A(t)x + 2T(t)x = -\partial_t(2k_d + b_\mathsf{n}e^{\widetilde{\Phi}}) \\
		+ l_b - 2l_d - \mathscr{A}(\nabla_x^{skew}l_b)x - \frac{1}{3}\mathscr{A}(\nabla_x \cdot l_b)x - \mathsf{R}b_\mathsf{n}e^{\widetilde{\Phi}} - \mathsf{R}x \cdot \nabla_x(b_\mathsf{n}e^{\widetilde{\Phi}}) + 2(\nabla_x{\widetilde{\Phi}})c_\mathsf{n}e^{\widetilde{\Phi}} \\
		+ \mathscr{A}(\nabla_x^{skew}(\mathsf{R}be^{\widetilde{\Phi}}))x + \mathscr{A}(\nabla_x^{skew}(\mathsf{R}x \cdot \nabla_xbe^{\widetilde{\Phi}}))x - \frac{2}{3}\mathscr{A}(\nabla_x \cdot (be^{\widetilde{\Phi}}))\mathsf{R}x.
	\end{multline}

Now we focus on the skew-symmetric matrix in \eqref{Energyan8},
	\begin{equation*}
		\mathcal{R} := \mathscr{A}(\nabla_x^{skew}(\mathsf{R}be^{\widetilde{\Phi}})) + \mathscr{A}(\nabla_x^{skew}(\mathsf{R}x \cdot \nabla_xbe^{\widetilde{\Phi}})) - \frac{2}{3}\mathscr{A}(\nabla_x \cdot (be^{\widetilde{\Phi}}))\mathsf{R}.
	\end{equation*}
	Recall that $C_{\widetilde\Phi}=(\int_{\R^3}e^{-\widetilde\Phi}dx)^{-1}$. For every sufficiently regular $f$,
	\begin{equation*}
		\mathscr{A}(\nabla_x(fe^{\widetilde{\Phi}})) = C_{\widetilde{\Phi}}\int_{\R^3}\nabla_x(fe^{\widetilde{\Phi}})e^{-{\widetilde{\Phi}}}dx = C_{\widetilde{\Phi}}\int_{\R^3}\nabla_x{\widetilde{\Phi}} fdx,
	\end{equation*}
	and recall the form of matrix $\mathsf{R}$ in \eqref{Rform}, we compute that
	\begin{equation*}
	\begin{aligned}
		2\mathcal{R}_{23} &= \mathscr{A}(\partial_3(Rbe^{\widetilde{\Phi}})_2) - \mathscr{A}(\partial_2(Rbe^{\widetilde{\Phi}})_3) + \mathscr{A}(\partial_3(Rx \cdot \nabla_xb_2e^{\widetilde{\Phi}})) - \mathscr{A}(\partial_2(Rx \cdot \nabla_xb_3e^{\widetilde{\Phi}})) \\
		&= C_{\widetilde{\Phi}}\int_{\R^3}[(\partial_3{\widetilde{\Phi}})(Rb)_2 - (\partial_2{\widetilde{\Phi}})(Rb)_3 + (\partial_3{\widetilde{\Phi}})(Rx \cdot \nabla_xb_2) - (\partial_2{\widetilde{\Phi}})(Rx \cdot \nabla_xb_3)]dx \\
		&= rC_{\widetilde{\Phi}}\int_{\R^3}[-(\partial_3{\widetilde{\Phi}})b_1 + (\partial_3{\widetilde{\Phi}})(x_2\partial_1b_2 - x_1\partial_2b_2) - (\partial_2{\widetilde{\Phi}})(x_2\partial_1b_3 - x_1\partial_2b_3)]dx \\
		&= rC_{\widetilde{\Phi}}\int_{\R^3}{\widetilde{\Phi}}[\partial_3b_1 - x_2\partial_1\partial_3b_2 + x_1\partial_2\partial_3b_2 + \partial_1b_3 + x_2\partial_1\partial_2b_3 - x_1\partial_2^2b_3]dx.
	\end{aligned}
	\end{equation*}
	We note that fact that
	\begin{equation*}
		\nabla_x{\widetilde{\Phi}} = (|x|^{p - 2}x_1 - r^2x_1, |x|^{p - 2}x_2 - r^2x_2, |x|^{p - 2}x_3)^\tau,
	\end{equation*}
	which implies $x_1\partial_2\widetilde\Phi=x_2\partial_1\widetilde\Phi$. Consequently, every sufficiently regular $f$ satisfies
	\begin{equation*}
		\int_{\R^3}{\widetilde{\Phi}}(x_1\partial_2f)dx = -\int_{\R^3}(x_1\partial_2{\widetilde{\Phi}})fdx = -\int_{\R^3}(x_2\partial_1{\widetilde{\Phi}})fdx = \int_{\R^3}{\widetilde{\Phi}}(x_2\partial_1f)dx.
	\end{equation*}
	Therefore
	\begin{equation*}
	\begin{aligned}
		2\mathcal{R}_{23} &= rC_{\widetilde{\Phi}}\int_{\R^3}{\widetilde{\Phi}}[\partial_3b_1 + (x_1\partial_2 - x_2\partial_1)\partial_3b_2 + \partial_1b_3 + (x_2\partial_1 - x_1\partial_2)\partial_2b_3]dx \\
		&= rC_{\widetilde{\Phi}}\int_{\R^3}{\widetilde{\Phi}}(\partial_3b_1 + \partial_1b_3)dx = -rC_{\widetilde{\Phi}}\int_{\R^3}[(\partial_3{\widetilde{\Phi}})b_1 + (\partial_1{\widetilde{\Phi}})b_3]dx \\
		&= -r\mathscr{A}(\partial_3(b_1e^{\widetilde{\Phi}})) - r\mathscr{A}(\partial_1(b_3e^{\widetilde{\Phi}})) = -2r\mathscr{A}(\nabla_x^{sym}(be^{\widetilde{\Phi}}))_{13}.
	\end{aligned}
	\end{equation*}
	Thus by \eqref{ptcI}, we get
	\begin{equation*}
		\mathcal{R}_{23} = -r\mathscr{A}(\nabla_x^{sym}(be^{\widetilde{\Phi}}))_{13} = -r\partial_t\mathscr{A}(K_{13}) - r\mathscr{A}(L_{13}).
	\end{equation*}

By the similar argument, we compute that
	\begin{equation*}
	\begin{aligned}
		2\mathcal{R}_{13} &= \mathscr{A}(\partial_3(Rbe^{\widetilde{\Phi}})_1) - \mathscr{A}(\partial_1(Rbe^{\widetilde{\Phi}})_3) + \mathscr{A}(\partial_3(Rx \cdot \nabla_xb_1e^{\widetilde{\Phi}})) - \mathscr{A}(\partial_1(Rx \cdot \nabla_xb_3e^{\widetilde{\Phi}})) \\
		&= C_{\widetilde{\Phi}}\int_{\R^3}[(\partial_3{\widetilde{\Phi}})(Rb)_1 - (\partial_1{\widetilde{\Phi}})(Rb)_3 + (\partial_3{\widetilde{\Phi}})(Rx \cdot \nabla_xb_1) - (\partial_1{\widetilde{\Phi}})(Rx \cdot \nabla_xb_3)]dx \\
		&= rC_{\widetilde{\Phi}}\int_{\R^3}[(\partial_3{\widetilde{\Phi}})b_2 + (\partial_3{\widetilde{\Phi}})(x_2\partial_1b_1 - x_1\partial_2b_1) - (\partial_1{\widetilde{\Phi}})(x_2\partial_1b_3 - x_1\partial_2b_3)]dx \\
		&= rC_{\widetilde{\Phi}}\int_{\R^3}{\widetilde{\Phi}}[-\partial_3b_2 - x_2\partial_1\partial_3b_1 + x_1\partial_2\partial_3b_1 + x_2\partial_1^2b_3 - \partial_2b_3 - x_1\partial_1\partial_2b_3]dx \\
			\end{aligned}
	\end{equation*}
		\begin{equation*}
		\begin{aligned}
		&= rC_{\widetilde{\Phi}}\int_{\R^3}{\widetilde{\Phi}}[-\partial_3b_2 + (x_1\partial_2 - x_2\partial_1)\partial_3b_1 + (x_2\partial_1 - x_1\partial_2)\partial_1b_3 - \partial_2b_3]dx \\
		&= -rC_{\widetilde{\Phi}}\int_{\R^3}{\widetilde{\Phi}}(\partial_3b_2 + \partial_2b_3)dx = rC_{\widetilde{\Phi}}\int_{\R^3}[(\partial_3{\widetilde{\Phi}})b_2 + (\partial_2{\widetilde{\Phi}})b_3]dx \\
		&= r\mathscr{A}(\partial_3(b_2e^{\widetilde{\Phi}})) + r\mathscr{A}(\partial_2(b_3e^{\widetilde{\Phi}})) = 2r\mathscr{A}(\nabla_x^{sym}(be^{\widetilde{\Phi}}))_{23}.
	\end{aligned}
	\end{equation*}
	Again by \eqref{ptcI}, we get
	\begin{equation*}
		\mathcal{R}_{13} = r\mathscr{A}(\nabla_x^{sym}(be^{\widetilde{\Phi}}))_{23} = r\partial_t\mathscr{A}(K_{23}) + r\mathscr{A}(L_{23}).
	\end{equation*}

	Finally, we compute that
	\begingroup\small
	\begin{equation*}
	\begin{aligned}
		2\mathcal{R}_{12} &= \mathscr{A}(\partial_2(Rbe^{\widetilde{\Phi}})_1) - \mathscr{A}(\partial_1(Rbe^{\widetilde{\Phi}})_2) + \mathscr{A}(\partial_2(Rx \cdot \nabla_xb_1e^{\widetilde{\Phi}})) - \mathscr{A}(\partial_1(Rx \cdot \nabla_xb_2e^{\widetilde{\Phi}})) - \frac{4r}{3}\mathscr{A}(\nabla_x \cdot (be^{\widetilde{\Phi}})) \\
		&= C_{\widetilde{\Phi}}\int_{\R^3}[(\partial_2{\widetilde{\Phi}})(Rb)_1 - (\partial_1{\widetilde{\Phi}})(Rb)_2 + (\partial_2{\widetilde{\Phi}})(Rx \cdot \nabla_xb_1) - (\partial_1{\widetilde{\Phi}})(Rx \cdot \nabla_xb_2) - \frac{4r}{3}\nabla_x{\widetilde{\Phi}} \cdot b]dx \\
		&= rC_{\widetilde{\Phi}}\int_{\R^3}[(\partial_2{\widetilde{\Phi}})b_2 + (\partial_1{\widetilde{\Phi}})b_1 + (\partial_2{\widetilde{\Phi}})(x_2\partial_1b_1 - x_1\partial_2b_1) - (\partial_1{\widetilde{\Phi}})(x_2\partial_1b_2 - x_1\partial_2b_2) - \frac{4}{3}\nabla_x{\widetilde{\Phi}} \cdot b]dx \\
		&= rC_{\widetilde{\Phi}}\int_{\R^3}{\widetilde{\Phi}}[-\partial_2b_2 - \partial_1b_1 - \partial_1b_1 - x_2\partial_1\partial_2b_1 + x_1\partial_2^2b_1 + x_2\partial_1^2b_2 - \partial_2b_2 - x_1\partial_1\partial_2b_2 + \frac{4}{3}\nabla_x \cdot b]dx \\
		&= rC_{\widetilde{\Phi}}\int_{\R^3}{\widetilde{\Phi}}[-\frac{2}{3}(\partial_1b_1 + \partial_2b_2 - 2\partial_3b_3) + (x_1\partial_2 - x_2\partial_1)\partial_2b_1 + (x_2\partial_1 - x_1\partial_2)\partial_1b_2]dx \\
		&= -\frac{2}{3}rC_{\widetilde{\Phi}}\int_{\R^3}{\widetilde{\Phi}}(\partial_1b_1 + \partial_2b_2 - 2\partial_3b_3)dx = \frac{2}{3}rC_{\widetilde{\Phi}}\int_{\R^3}[(\partial_1{\widetilde{\Phi}})b_1 + (\partial_2{\widetilde{\Phi}})b_2 - 2(\partial_3{\widetilde{\Phi}})b_3]dx \\
		&= \frac{2r}{3}[\mathscr{A}(\partial_1(b_1e^{\widetilde{\Phi}})) + \mathscr{A}(\partial_2(b_2e^{\widetilde{\Phi}})) - 2\mathscr{A}(\partial_3(b_3e^{\widetilde{\Phi}}))] \\
		&= \frac{2r}{3}[\mathscr{A}(\nabla_x^{sym}(be^{\widetilde{\Phi}}))_{11} + \mathscr{A}(\nabla_x^{sym}(be^{\widetilde{\Phi}}))_{22} - 2\mathscr{A}(\nabla_x^{sym}(be^{\widetilde{\Phi}}))_{33}].
	\end{aligned}
	\end{equation*}
	\endgroup
 By \eqref{ptcI}, we get
	\begin{equation*}
	\begin{aligned}
		\mathcal{R}_{12} &= \frac{r}{3}[\mathscr{A}(\nabla_x^{sym}(be^{\widetilde{\Phi}}))_{11} + \mathscr{A}(\nabla_x^{sym}(be^{\widetilde{\Phi}}))_{22} - 2\mathscr{A}(\nabla_x^{sym}(be^{\widetilde{\Phi}}))_{33}] \\
		&= \frac{r}{3}\partial_t\mathscr{A}(K_{11} + K_{22} - 2K_{33}) + \frac{r}{3}\mathscr{A}(L_{11} + L_{22} - 2L_{33}).
	\end{aligned}
	\end{equation*}

	Based on the previous calculations, we set
	\begin{equation}\label{DefiK*}
		K^* := \frac{r}{3}\begin{pmatrix} 0 & K_{11} + K_{22} - 2K_{33} & 3K_{23} \\ -K_{11} - K_{22} + 2K_{33} & 0 & -3K_{13} \\ -3K_{23} & 3K_{13} & 0 \end{pmatrix},
	\end{equation}
	and
	\begin{equation}\label{DefiL*}
		L^* := \frac{r}{3}\begin{pmatrix} 0 & L_{11} + L_{22} - 2L_{33} & 3L_{23} \\ -L_{11} - L_{22} + 2L_{33} & 0 & -3L_{13} \\ -3L_{23} & 3L_{13} & 0 \end{pmatrix}.
	\end{equation}
	Our calculations imply that
	\begin{equation*}
		\mathcal{R} = \partial_t\mathscr{A}(K^*) + \mathscr{A}(L^*).
	\end{equation*}

	Now we recall \eqref{Energyan8}, it becomes
	\begin{multline}\label{014}
		\nabla_x(ae^{\widetilde{\Phi}}) - 2(\nabla_x{\widetilde{\Phi}})\mathscr{A}(ce^{\widetilde{\Phi}}) - \mathscr{A}(\nabla_x(ae^{\widetilde{\Phi}})) + 2A(t)x + 2T(t)x = \partial_t(\mathscr{A}(K^*)x - 2k_d - b_\mathsf{n}e^{\widetilde{\Phi}}) \\
		+ l_b - 2l_d + \mathscr{A}(L^*)x - \mathscr{A}(\nabla_x^{skew}l_b)x - \frac{1}{3}\mathscr{A}(\nabla_x \cdot l_b)x - \mathsf{R}b_\mathsf{n}e^{\widetilde{\Phi}} - \mathsf{R}x \cdot \nabla_x(b_\mathsf{n}e^{\widetilde{\Phi}}) + 2(\nabla_x{\widetilde{\Phi}})c_\mathsf{n}e^{\widetilde{\Phi}}.
	\end{multline}
	
\underline{Step 2: Energy estimate for $a_\mathsf{n}$.}
Set
\[
\begin{aligned}
\mathcal F_a={}&l_b-2l_d+\mathscr{A}(L^*)x
-\mathscr{A}(\nabla_x^{skew}l_b)x
-\frac13\mathscr{A}(\nabla_x\cdot l_b)x\\
&-\mathsf{R}b_\mathsf{n}e^{\widetilde{\Phi}}
-\mathsf{R}x\cdot\nabla_x(b_\mathsf{n}e^{\widetilde{\Phi}})
+2(\nabla_x\widetilde{\Phi})c_\mathsf{n}e^{\widetilde{\Phi}}.
\end{aligned}
\]
Combining \eqref{014} and \eqref{anaeq}, we obtain
		\begin{equation}\label{030}
		\begin{aligned}
			\|\Lambda_{\delta_3}^{-\frac{1}{2}}\nabla_x(a_\mathsf{n}e^{\widetilde{\Phi}})\|_{L^2(V_{\delta_3})}^2
			&= \big(\partial_t(\mathscr{A}(K^*)x-2k_d-b_\mathsf{n}e^{\widetilde{\Phi}}),
			\Lambda_{\delta_3}^{-1}\nabla_x(a_\mathsf{n}e^{\widetilde{\Phi}})\big)_{L^2(V_{\delta_3})}\\
			&\quad +\big(\mathcal F_a,
			\Lambda_{\delta_3}^{-1}\nabla_x(a_\mathsf{n}e^{\widetilde{\Phi}})\big)_{L^2(V_{\delta_3})}.
		\end{aligned}
		\end{equation}
By the Cauchy--Schwarz inequality and Lemma \ref{BoundedLambda},
	\begin{equation}\label{031}
		\begin{aligned}
	&\big(l_b - 2l_d, \Lambda_{\delta_3}^{-1}\nabla_x(a_\mathsf{n}e^{\widetilde{\Phi}})\big)_{L^2(V_{\delta_3})} \le 16(\|\Lambda_{\delta_3}^{-\frac{1}{2}}l_b\|_{L^2(V_{\delta_3})}^2 + \|\Lambda_{\delta_3}^{-\frac{1}{2}}l_d\|_{L^2(V_{\delta_3})}^2) + \frac{1}{8}\|\Lambda_{\delta_3}^{-\frac{1}{2}}\nabla_x(a_\mathsf{n}e^{\widetilde{\Phi}})\|_{L^2(V_{\delta_3})}^2;\\
	&\big(2(\nabla_x{\widetilde{\Phi}})c_\mathsf{n}e^{\widetilde{\Phi}}, \Lambda_{\delta_3}^{-1}\nabla_x(a_\mathsf{n}e^{\widetilde{\Phi}})\big)_{L^2(V_{\delta_3})} \le2\|\Lambda_{\delta_3}^{-\frac{1}{2}}(\nabla_x{\widetilde{\Phi}} c_\mathsf{n}e^{\widetilde{\Phi}})\|_{L^2(V_{\delta_3})}\|\Lambda_{\delta_3}^{-\frac{1}{2}}\nabla_x(a_\mathsf{n}e^{\widetilde{\Phi}})\|_{L^2(V_{\delta_3})}\\
		&~~~\qquad\quad\qquad\qquad\qquad\qquad\qquad\qquad\le C\|c_\mathsf{n}e^{\widetilde{\Phi}}\|_{L^2(V_{\delta_1})}^2 + \frac{1}{8}\|\Lambda_{\delta_3}^{-\frac{1}{2}}\nabla_x(a_\mathsf{n}e^{\widetilde{\Phi}})\|_{L^2(V_{\delta_3})}^2;\\
		&\big(\mathscr{A}(L^*)x - \mathscr{A}(\nabla_x^{skew}l_b)x - \frac{1}{3}\mathscr{A}(\nabla_x \cdot l_b)x, \Lambda_{\delta_3}^{-1}\nabla_x(a_\mathsf{n}e^{\widetilde{\Phi}})\big)_{L^2(V_{\delta_3})} \\
		&\qquad\qquad\qquad\qquad\qquad\le C(|\mathscr{A}(L^*)|^2 + |\mathscr{A}(\nabla_x^{skew}l_b)|^2 + |\mathscr{A}(\nabla_x \cdot l_b)|^2) + \frac{1}{8}\|\Lambda_{\delta_3}^{-\frac{1}{2}}\nabla_x(a_\mathsf{n}e^{\widetilde{\Phi}})\|_{L^2(V_{\delta_3})}^2;
	\end{aligned}
	\end{equation}
	and
	\begin{multline}\label{032}
		-\big(\mathsf{R}b_\mathsf{n}e^{\widetilde{\Phi}} + \mathsf{R}x \cdot \nabla_x(b_\mathsf{n}e^{\widetilde{\Phi}}), \Lambda_{\delta_3}^{-1}\nabla_x(a_\mathsf{n}e^{\widetilde{\Phi}})\big)_{L^2(V_{\delta_3})} \\
		\le C(\|b_\mathsf{n}e^{\widetilde{\Phi}}\|_{L^2(V_{\delta_3})}^2 + \|\Lambda_{\delta_3}^{-\frac{1}{2}}\mathsf{R}x \cdot \nabla_x(b_\mathsf{n}e^{\widetilde{\Phi}})\|_{L^2(V_{\delta_3})}^2) + \frac{1}{8}\|\Lambda_{\delta_3}^{-\frac{1}{2}}\nabla_x(a_\mathsf{n}e^{\widetilde{\Phi}})\|_{L^2(V_{\delta_3})}^2;
	\end{multline}
	Due to Lemma \ref{BoundedLambda}, we have
	\begin{equation*}
		\|\Lambda_{\delta_3}^{-\frac{1}{2}}\mathsf{R}x \cdot \nabla_x(b_\mathsf{n}e^{\widetilde{\Phi}})\|_{L^2(V_{\delta_3})}^2
		= \|\Lambda_{\delta_3}^{-\frac{1}{2}}\nabla_x \cdot (\mathsf{R}xb_\mathsf{n}e^{\widetilde{\Phi}})\|_{L^2(V_{\delta_3})}^2
		\lesssim \|xb_\mathsf{n}e^{\widetilde{\Phi}}\|_{L^2(V_{\delta_3})}^2 \lesssim \|b_\mathsf{n}e^{\widetilde{\Phi}}\|_{L^2(V_{\delta_2})}^2.
	\end{equation*}
	Therefore, we have that
		\begin{equation}\label{033}
		\begin{aligned}
		&-(\mathsf{R}b_\mathsf{n}e^{\widetilde{\Phi}}
		+ \mathsf{R}x \cdot \nabla_x(b_\mathsf{n}e^{\widetilde{\Phi}}),
		\Lambda_{\delta_3}^{-1}\nabla_x(a_\mathsf{n}e^{\widetilde{\Phi}}))_{L^2(V_{\delta_3})} \\
		&\qquad\le C\|b_\mathsf{n}e^{\widetilde{\Phi}}\|_{L^2(V_{\delta_2})}^2
		+ \frac{1}{8}\|\Lambda_{\delta_3}^{-\frac{1}{2}}
		\nabla_x(a_\mathsf{n}e^{\widetilde{\Phi}})\|_{L^2(V_{\delta_3})}^2.
		\end{aligned}
		\end{equation}
	
	Besides, we write
		\begin{equation*}
			\begin{aligned}
			&(\partial_t(\mathscr{A}(K^*)x - 2k_d - b_\mathsf{n}e^{\widetilde{\Phi}}),
			\Lambda_{\delta_3}^{-1}\nabla_x(a_\mathsf{n}e^{\widetilde{\Phi}}))_{L^2(V_{\delta_3})}\\
			&\quad= \frac{d}{dt}(\mathscr{A}(K^*)x - 2k_d - b_\mathsf{n}e^{\widetilde{\Phi}},
			\Lambda_{\delta_3}^{-1}\nabla_x(a_\mathsf{n}e^{\widetilde{\Phi}}))_{L^2(V_{\delta_3})} \\
			&\qquad- (\mathscr{A}(K^*)x - 2k_d - b_\mathsf{n}e^{\widetilde{\Phi}},
			\Lambda_{\delta_3}^{-1}\partial_t\nabla_x(a_\mathsf{n}e^{\widetilde{\Phi}}))_{L^2(V_{\delta_3})}.
	\end{aligned}
		\end{equation*}
Set $Z_a=b_\mathsf{n}e^{\widetilde{\Phi}}+2k_d-\mathscr{A}(K^*)x$. Then
\begin{equation}\label{034}
\begin{aligned}
	&\big(Z_a,\Lambda_{\delta_3}^{-1}\partial_t\nabla_x
	(a_\mathsf{n}e^{\widetilde{\Phi}})\big)_{L^2(V_{\delta_3})}\\
	&\quad=\big(\Lambda_{\delta_4}^{*\de_3}\Lambda_{\delta_3}^{-1}Z_a,
	\Lambda_{\delta_4}^{-1}\partial_t\nabla_x
	(a_\mathsf{n}e^{\widetilde{\Phi}})\big)_{L^2(V_{\delta_3})}\\
	&\quad\leq \|\Lambda_{\delta_4}^{*\de_3}\Lambda_{\delta_3}^{-1}Z_a\|_{L^2(V_{\de_2})}
	\|\Lambda_{\delta_4}^{-1}\partial_t\nabla_x
	(a_\mathsf{n}e^{\widetilde{\Phi}})\|_{L^2(V_{\de_4})}\\
	&\quad\lesssim \big(\|b_\mathsf{n}e^{\widetilde{\Phi}}\|_{L^2(V_{\delta_2})}
	+\|k_d\|_{L^2(V_{\delta_0})}+|\mathscr{A}(K^*)|\big)\\
	&\qquad\times\|\Lambda_{\delta_4}^{-1}\partial_t\nabla_x^{sym}
	(b_\mathsf{n}e^{\widetilde{\Phi}})\|_{L^2(V_{\delta_4})}.
\end{aligned} 
\end{equation}
where we use the same argument as \eqref{027} and the fact that $2\de_3<\de_2+\de_4$.

Combining \eqref{030}--\eqref{034} proves the lemma.
\end{proof}
%%%%%%%%%%%%%%%%%%%%%%%%%%%%%%%%%%%%%%%%%%%%%%%%%%

	\subsubsection{Further estimates for $a_\mathsf{n}, b_\mathsf{n}, c_\mathsf{n}$}
	To estimate the terms on the right-hand side of the energy equation of $a_\mathsf{n}, b_\mathsf{n}, c_\mathsf{n}$, we provide the following lemma:

	\begin{lemma}[Bounds for the microscopic source terms]\label{MicroBound}
		We recall that	${\mathsf{H}} = {\mathsf{H}}(x, v) = {\widetilde{\Phi}}(x) + \frac{1}{2}|v|^2$, then the following assertions hold:

		$(1)$ For $k = k_d, K_{ij}$ or $ K_{ij}^*$ defined in \eqref{Defilk} and \eqref{DefiK*}, then
		\begin{equation*}
			\|k\|_{L^2(V_{\delta_0})}^2 + |\mathscr{A}(k)|^2 \lesssim \|e^{\delta_0{\mathsf{H}}}g_1\|_{L_{x, v}^2}^2.
		\end{equation*}
		
		$(2)$ For $l = l_b, l_d, L_{ij}$ or $L_{ij}^*$ defined in \eqref{Defilk} and \eqref{DefiL*}, $\delta \in \{\delta_1, \delta_2, \delta_3, \delta_4\}$, then
		\begin{equation*}
			\|\Lambda_{\delta}^{-\frac{1}{2}}l\|_{L^2(V_{\delta})}^2 + |\mathscr{A}(l)|^2 + |\mathscr{A}(\nabla_xl)|^2 \lesssim \|e^{\delta_0{\mathsf{H}}}g_1\|_{L_{x, v}^2}^2.
		\end{equation*}
	\end{lemma}
	\begin{proof}
		$(1)$ By definition of $k_d, K, K^*$ shown in \eqref{Defilk} and \eqref{DefiK*}. We need only consider $k$ in the form
		\begin{equation*}
			k = e^{\widetilde{\Phi}}(g_1, P(v))_{L_v^2}
		\end{equation*}
		for $P(v)\in\{1,v_i,v_iv_j,v_i|v|^2\}$, $i,j=1,2,3$. By the Cauchy--Schwarz inequality,
		\begin{equation*}
		\begin{aligned}
			\|k\|_{L^2(V_{\delta_0})}^2 &= \int_{\R^3} e^{2\delta_0{\widetilde{\Phi}}}(g_1, P(v))_{L_v^2}^2dx \le \int_{\R^3} e^{2\delta_0{\widetilde{\Phi}}}\Big(\int_{\R^3} e^{\delta_0|v|^2}g_1^2dv \int_{\R^3} P(v)^2e^{-\delta_0|v|^2}dv\Big)dx \lesssim \|e^{\delta_0{\mathsf{H}}}g_1\|_{L_{x, v}^2}^2;\\
			|\mathscr{A}(k)|^2 &= \Big(\int_{\R^6} P(v)g_1dxdv\Big)^2 \le \int_{\R^6} e^{2\delta_0{\widetilde{\Phi}} + \delta_0|v|^2}g_1^2dxdv  \int_{\R^6} P(v)^2e^{-2\delta_0{\widetilde{\Phi}} - \delta_0|v|^2}dxdv \lesssim \|e^{\delta_0{\mathsf{H}}}g_1\|_{L_{x, v}^2}^2.
		\end{aligned}
		\end{equation*}
		This proves the first assertion.

		$(2)$ By definition of $l_b, l_d, L, L^*$ shown in \eqref{Defilk} and \eqref{DefiL*} and recall the definition of operator $\mathsf{T}$ as:
		\begin{equation*}
			\mathsf{T} = v \cdot \nabla_x + \mathsf{R}x \cdot \nabla_x - \mathsf{R}v \cdot \nabla_v - \nabla_x{\widetilde{\Phi}} \cdot \nabla_v.
		\end{equation*}
		We only need to consider $l$ in the form $l_1, l_2$ and $l_3$, where
		\begin{equation*}
			l_1 = e^{\widetilde{\Phi}}((v + \mathsf{R}x) \cdot \nabla_xg_1, P(v))_{L_v^2}, \:\:\:
			l_2 = e^{\widetilde{\Phi}}((\mathsf{R}v + \nabla_x{\widetilde{\Phi}}) \cdot \nabla_vg_1, P(v))_{L_v^2}, \:\:\:
			l_3 = (\mathsf{L}g_1, P(v))_{L_v^2}
		\end{equation*}
		for $P(v) \in \{1,v_i, v_iv_j, v_i|v|^2\},i,j=1,2,3$. Since $\mathscr{A}(\nabla_xl) = \mathscr{A}(\nabla_x{\widetilde{\Phi}} l)$, for $l = l_1, l_2, l_3$ we have
		\begin{equation*}
			|\mathscr{A}(l)|^2 + |\mathscr{A}(\nabla_xl)|^2 \lesssim |\mathscr{A}((|\nabla_x{\widetilde{\Phi}}| + 1)l)|^2.
		\end{equation*}
			$\bullet$ For $l_1$, Lemma \ref{BoundedLambda} gives
			\begingroup\small
			\begin{equation*}
		\begin{aligned}
			\|\Lambda_{\delta}^{-\frac{1}{2}}l_1\|_{L^2(V_{\delta})}^2 &= \|\Lambda_{\delta}^{-\frac{1}{2}}e^{\widetilde{\Phi}}\nabla_x \cdot \int_{\R^3}(v + \mathsf{R}x)P(v)g_1dv\|_{L^2(V_{\delta})}^2 \\
			&\le \|\Lambda_{\delta}^{-\frac{1}{2}}\nabla_x \cdot e^{\widetilde{\Phi}}\int_{\R^3}(v + \mathsf{R}x)P(v)g_1dv\|_{L^2(V_{\delta})}^2 + \|\Lambda_{\delta}^{-\frac{1}{2}}\nabla_x{\widetilde{\Phi}} e^{\widetilde{\Phi}} \cdot \int_{\R^3}(v + \mathsf{R}x)P(v)g_1dv\|_{L^2(V_{\delta})}^2 \\
			&\lesssim \|e^{\widetilde{\Phi}}\int_{\R^3}(v + \mathsf{R}x)P(v)g_1dv\|_{L^2(V_{\delta})}^2 + \|\nabla_x{\widetilde{\Phi}} e^{\widetilde{\Phi}} \cdot \int_{\R^3}(v + \mathsf{R}x)P(v)g_1dv\|_{L^2(V_{\delta})}^2 \\
			&\le \int_{\R^3} e^{2\delta{\widetilde{\Phi}}}(|\nabla_x{\widetilde{\Phi}}|^2 + 1)\big(\int_{\R^3}|(v + \mathsf{R}x)P(v)|g_1dv\big)^2dx \\
			&\lesssim \int_{\R^6} e^{2\delta{\widetilde{\Phi}}}(|\nabla_x{\widetilde{\Phi}}|^2 + 1)(|x|^2 + 1)e^{\delta_0|v|^2}g_1^2dxdv
			\lesssim \|e^{\delta_0{\mathsf{H}}}g_1\|_{L_{x, v}^2}^2;\\
			|\mathscr{A}((|\nabla_x{\widetilde{\Phi}}| + 1)l_1)|^2 &= \big(\int_{\R^3}(|\nabla_x{\widetilde{\Phi}}| + 1)\nabla_x \cdot \big(\int_{\R^3}(v + \mathsf{R}x)P(v)g_1dv\big)dx\big)^2 \\
			&= \big(\int_{\R^6}\nabla_x|\nabla_x{\widetilde{\Phi}}| \cdot (v + \mathsf{R}x)P(v)g_1dxdv\big)^2
			\lesssim \|e^{\delta_0{\mathsf{H}}}g_1\|_{L_{x, v}^2}^2.
			\end{aligned}
			\end{equation*}
			\endgroup
			$\bullet$ For $l_2$, we have
		\begin{equation*}
		\begin{aligned}
			l_2
			&=e^{\widetilde{\Phi}}((\mathsf{R}v + \nabla_x{\widetilde{\Phi}})
			\cdot \nabla_vg_1, P(v))_{L_v^2}\\
			&=-e^{\widetilde{\Phi}}((\mathsf{R}v + \nabla_x{\widetilde{\Phi}})
			\cdot \nabla_vP(v), g_1)_{L_v^2},
		\end{aligned}
		\end{equation*}
		Then
		\begingroup\small
		\begin{equation*}
		\begin{aligned}
			\|\Lambda_{\delta}^{-\frac{1}{2}}l_2\|_{L^2(V_{\delta})}^2 &\lesssim \|l_2\|_{L^2(V_{\delta})}^2 = \int_{\R^3} e^{2\delta{\widetilde{\Phi}}}\big(\int_{\R^3}(Rv + \nabla_x{\widetilde{\Phi}}) \cdot \nabla_vP(v)g_1dv\big)^2 dx\\
			&\lesssim \int e^{2\delta{\widetilde{\Phi}}}(|\nabla_x{\widetilde{\Phi}}|^2 + 1)e^{\delta_0|v|^2}g_1^2dxdv \lesssim \|e^{\delta_0{\mathsf{H}}}g_1\|_{L_{x, v}^2}^2;\\
			|\mathscr{A}((|\nabla_x{\widetilde{\Phi}}| + 1)l_2)|^2 &= \big(\int(|\nabla_x{\widetilde{\Phi}}| + 1)(Rv + \nabla_x{\widetilde{\Phi}}) \cdot \nabla_vP(v)g_1dxdv\big)^2 \lesssim \|e^{\delta_0{\mathsf{H}}}g_1\|_{L_{x, v}^2}^2.
		\end{aligned}
		\end{equation*}
		\endgroup
		$\bullet$ For $l_3$, we apply Lemma \ref{UpperBoundL0} to obtain that
		\begin{equation*}
			l_3^2 \lesssim \int_{\R^3} e^{\delta|v|^2}g_1^2dv, \:\:\: \forall\delta > 0,
		\end{equation*}
		then
		\begin{equation*}
		\begin{aligned}
			\|\Lambda_{\delta}^{-\frac{1}{2}}l_3\|_{L^2(V_{\delta})}^2 &\lesssim \|l_3\|_{L^2(V_{\delta})}^2 = \int_{\R^3} e^{-2(1 - \delta){\widetilde{\Phi}}}l_3^2dx \lesssim \int_{\R^3} e^{-2(1 - \delta){\widetilde{\Phi}} + \delta_0|v|^2}g_1^2 dvdx\lesssim \|e^{\delta_0{\mathsf{H}}}g_1\|_{L_{x, v}^2}^2;\\
			|\mathscr{A}((|\nabla_x{\widetilde{\Phi}}| + 1)l_3)|^2 &= \big(\int_{\R^3}(|\nabla_x{\widetilde{\Phi}}| + 1)l_3e^{-{\widetilde{\Phi}}}dx\big)^2 
			\le \int_{\R^3} e^{2\delta_0{\widetilde{\Phi}}}l_3^2dx \int_{\R^3}(|\nabla_x{\widetilde{\Phi}}| + 1)^2e^{-2(1 + \delta_0){\widetilde{\Phi}} }dx\\
			&\lesssim \int _{\R^6}e^{2\delta_0{\widetilde{\Phi}} + \delta_0|v|^2}g_1^2dxdv = \|e^{\delta_0{\mathsf{H}}}g_1\|_{L_{x, v}^2}^2.
		\end{aligned}
		\end{equation*}
	This completes the proof.
	\end{proof}
	
	We can now prove Lemma \ref{Estimateanbncn}.
	\begin{proof}[Proof of Lemma \ref{Estimateanbncn}]
		With the help of the Lemma \ref{MicroBound}, together with Lemma \ref{015}, Lemma \ref{016} and Lemma \ref{017}, we can conclude Lemma \ref{Estimateanbncn} directly.
	\end{proof}

%%%%%%%%%%%%%%%%%%%%%%%%%%%%%%%%%%%%%%%%%%%%%%%%%%

	\subsection{Estimates for the finite-dimensional remainder}
	Now we have obtained the energy estimate for $a_\mathsf{n}, b_\mathsf{n}, c_\mathsf{n}$ and we will estimate the difference terms between $a_\mathsf{n}, b_\mathsf{n}, c_\mathsf{n}$ and $a, b, c$. Recall the definition of $a_\mathsf{n}, b_\mathsf{n}, c_\mathsf{n}$ in \eqref{Defianbncn} and rewrite it as
	\begin{equation}\label{041}
	\begin{gathered}
		ce^{\widetilde{\Phi}} = c_\mathsf{n}e^{\widetilde{\Phi}} + \mathscr{A}(ce^{\widetilde{\Phi}});\:\:\:\:\:\:\:\:\:\:\:\:
		be^{\widetilde{\Phi}} = b_\mathsf{n}e^{\widetilde{\Phi}} + \mathscr{A}(be^{\widetilde{\Phi}}) + B(t)x + \frac{1}{3}\mathscr{A}(\nabla_x \cdot (be^{\widetilde{\Phi}}))x;\\
		ae^{\widetilde{\Phi}} = a_\mathsf{n}e^{\widetilde{\Phi}} + 2({\widetilde{\Phi}} - \mathscr{A}({\widetilde{\Phi}}))\mathscr{A}(ce^{\widetilde{\Phi}}) + \mathscr{A}(\nabla_x(ae^{\widetilde{\Phi}})) \cdot x - (|x|^2 - \mathscr{A}(|x|^2))A(t) - T(t)x \cdot x + \mathscr{A}(T(t)x \cdot x).
	\end{gathered}
	\end{equation}
	where $B(t) := \mathscr{A}(\nabla_x^{skew}(be^{\widetilde{\Phi}}))$ is skew-symmetric, $T(t) := \frac{1}{2}(\mathsf{R}B(t) + B(t)\mathsf{R})$ is symmetric and $A(t)$ is defined in \eqref{defiAt}.
	We introduce notations
	\begin{equation}\label{051}
		c_\mathfrak{r}(t) = \mathscr{A}(ce^{\widetilde{\Phi}}), \:\:\: b_\mathfrak{r}(t) = \mathscr{A}(be^{\widetilde{\Phi}}), \:\:\: b_\mathfrak{R}(t) = \frac{1}{3}\mathscr{A}(\nabla_x \cdot (be^{\widetilde{\Phi}})), \:\:\: a_\mathfrak{R}(t) = \mathscr{A}(\nabla_x(ae^{\widetilde{\Phi}})).
	\end{equation}
	Then \eqref{041} becomes
	\begin{equation}\label{abcremainingform}
	\begin{gathered}
		ce^{\widetilde{\Phi}} = c_\mathsf{n}e^{\widetilde{\Phi}} + c_\mathfrak{r}(t);\:\:\:\:\:\:\:\:\:\:\:\: be^{\widetilde{\Phi}} = b_\mathsf{n}e^{\widetilde{\Phi}} + b_\mathfrak{r}(t) + B(t)x + b_\mathfrak{R}(t)x;\\
		ae^{\widetilde{\Phi}} = a_\mathsf{n}e^{\widetilde{\Phi}} + 2({\widetilde{\Phi}} - \mathscr{A}({\widetilde{\Phi}}))c_\mathfrak{r}(t) + a_\mathfrak{R}(t) \cdot x - (|x|^2 - \mathscr{A}(|x|^2))A(t) - T(t)x \cdot x + \mathscr{A}(T(t)x \cdot x).
	\end{gathered}
	\end{equation}
	We refer to $c_\mathfrak r$, $b_\mathfrak r$, $b_\mathfrak R$, $a_\mathfrak R$, $A$, and $B$ as the remaining macroscopic modes. They depend only on time. We estimate both these modes and their time derivatives. Since $2T(t)=\mathsf RB(t)+B(t)\mathsf R$, one has $|T(t)|\lesssim|B(t)|$, and hence
	\begin{equation*}
	\begin{aligned}
		\|ce^{\widetilde{\Phi}}\|_{L^2(V_{\delta_1})} + \|be^{\widetilde{\Phi}}\|_{L^2(V_{\delta_2})} + \|ae^{\widetilde{\Phi}}\|_{L^2(V_{\delta_3})}
		&\lesssim \|c_\mathsf{n}e^{\widetilde{\Phi}}\|_{L^2(V_{\delta_1})} + \|b_\mathsf{n}e^{\widetilde{\Phi}}\|_{L^2(V_{\delta_2})} + \|a_\mathsf{n}e^{\widetilde{\Phi}}\|_{L^2(V_{\delta_3})} \\
		&+ |c_\mathfrak{r}(t)| + |b_\mathfrak{r}(t)| + |b_\mathfrak{R}(t)| + |a_\mathfrak{R}(t)| + |A(t)| + |B(t)|.
	\end{aligned}
	\end{equation*}
	Thus if we denote total energy by
	\begin{equation}\label{DefiE}
	\begin{aligned}
		\mathsf{E}^2 := \|e^{\delta_0{\mathsf{H}}}g_1\|_{L_{x, v}^2}^2 + \|c_\mathsf{n}e^{\widetilde{\Phi}}\|_{L^2(V_{\delta_1})}^2 + \|b_\mathsf{n}e^{\widetilde{\Phi}}\|_{L^2(V_{\delta_2})}^2 + \|a_\mathsf{n}e^{\widetilde{\Phi}}\|_{L^2(V_{\delta_3})}^2 \\
		+ |c_\mathfrak{r}(t)|^2 + |b_\mathfrak{r}(t)|^2 + |b_\mathfrak{R}(t)|^2 + |a_\mathfrak{R}(t)|^2 + |A(t)|^2 + |B(t)|^2,
		\end{aligned}
	\end{equation}
then we have
	\begin{equation*}
		\|ce^{\widetilde{\Phi}}\|_{L^2(V_{\delta_1})} + \|be^{\widetilde{\Phi}}\|_{L^2(V_{\delta_2})} + \|ae^{\widetilde{\Phi}}\|_{L^2(V_{\delta_3})} \lesssim \mathsf{E}.
	\end{equation*}
	Our goal in this subsection is to prove the following lemma:
	\begin{lemma}[Finite-dimensional macroscopic remainder]\label{EstimatecrbrbsasAB}
		There exists a function $\mathcal{J}_0 = \mathcal{J}_0(t)$ satisfying $|\mathcal{J}_0(t)| \lesssim \mathsf{E}^2$ and
	\begin{equation*}
	\begin{aligned}
			&\frac{d}{dt}\mathcal{J}_0 + |c_\mathfrak{r}(t)|^2 + |b_\mathfrak{r}(t)|^2 + |b_\mathfrak{R}(t)|^2 + |a_\mathfrak{R}(t)|^2 + |A(t)|^2 + |B(t)|^2 \\
			\lesssim& \|e^{\delta_0{\mathsf{H}}}g_1\|_{L_{x, v}^2}^2 + \|b_\mathsf{n}e^{\widetilde{\Phi}}\|_{L^2(V_{\delta_2})}^2 + \|a_\mathsf{n}e^{\widetilde{\Phi}}\|_{L^2(V_{\delta_3})}^2 + \|a_\mathsf{n}e^{\widetilde{\Phi}}\|_{L^2(V_{\delta_3})}\mathsf{E}.
		\end{aligned}
\end{equation*}
	\end{lemma}

%%%%%%%%%%%%%%%%%%%%%%%%%%%%%%%%%%%%%%%%%%%%%%%%%%
\medskip Let us do some preliminary preparations first.
	\subsubsection{Estimates for the derivatives of the remaining modes}
	\begin{lemma}\label{crbrbsasAB'}
		The following estimate holds:
		\begin{equation}\label{042}
			|c'_\mathfrak{r}(t)| + |b'_\mathfrak{r}(t)| + |b'_\mathfrak{R}(t)| + |a'_\mathfrak{R}(t)| + |A'(t)| + |B'(t)| \lesssim \mathsf{E}.
		\end{equation}
	\end{lemma}
	\begin{proof}
		For the first three terms on the left-hand side of \eqref{042}, apply $\mathscr A$ to \eqref{ptc} and \eqref{ptb}. Since $\mathscr A(l_b)=\mathscr A(l_c)=0$ and $\mathscr A(Rx\cdot\nabla_xf)=\mathscr A((\nabla_xf)e^{\widetilde\Phi})=0$,
		\begin{equation}\label{Remainingeq}
			c'_\mathsf{r}(t) + b_\mathfrak{R}(t) = 0; \:\:\: b'_\mathsf{r}(t) + a_\mathfrak{R}(t) + \mathsf{R}b_\mathfrak{r}(t) = 0.
		\end{equation}
		Besides, it follows from \eqref{Energyan7} that
		\begin{equation}\label{Remainingeq+}
			3b'_\mathsf{R}(t) = 6A(t) + \mathscr{A}(\nabla_x \cdot l_b).
		\end{equation}
		Therefore, due to \eqref{DefiE} and Lemma \ref{MicroBound}, we have
		\begin{equation*}
			|c'_\mathsf{r}(t)| + |b'_\mathfrak{r}(t)| + |b'_\mathfrak{R}(t)|\ls |b_\mathfrak{R}(t)|+|a_\mathfrak{R}(t)| +|b_\mathfrak{r}(t)|+|A(t)|+|\mathscr{A}(\nabla_x \cdot l_b)|\lesssim \mathsf{E}.
		\end{equation*}
		
		For the last three terms on the left hand side of \eqref{042}, we provide estimates separately.\\
		$\bullet$ Estimates for $|a'_\mathfrak{R}(t)|$: We recall $C_{\widetilde{\Phi}} = (\int e^{-{\widetilde{\Phi}}}dx)^{-1}$, then by \eqref{pta} we have
		\begin{equation*}
		\begin{aligned}
			a'_\mathfrak{R}(t) &= \partial_t\mathscr{A}(\nabla_x(ae^{\widetilde{\Phi}})) = \partial_t\mathscr{A}((\nabla_x{\widetilde{\Phi}})(ae^{\widetilde{\Phi}})) = \mathscr{A}((\nabla_x{\widetilde{\Phi}})\partial_t(ae^{\widetilde{\Phi}})) \\
			&= -\mathscr{A}((\nabla_x{\widetilde{\Phi}})(\nabla_x \cdot b)e^{\widetilde{\Phi}}) - \mathscr{A}((\nabla_x{\widetilde{\Phi}})(\mathsf{R}x \cdot \nabla_xae^{\widetilde{\Phi}})) \\
			&= C_{\widetilde{\Phi}}\int_{\R^3}(\nabla_x^2{\widetilde{\Phi}})bdx + C_{\widetilde{\Phi}}\int_{\R^3}(\nabla_x^2{\widetilde{\Phi}})(\mathsf{R}xa)dx.
		\end{aligned}
		\end{equation*}
		Therefore, by the Cauchy--Schwarz inequality,
		\begin{equation*}
		\begin{aligned}
			|a'_\mathfrak{R}(t)|^2 &\lesssim \int_{\R^3} e^{2\delta_2{\widetilde{\Phi}}}|b|^2dx  \int_{\R^3}|\nabla_x^2{\widetilde{\Phi}}|^2e^{-2\delta_2{\widetilde{\Phi}}}dx + \int_{\R^3} e^{2\delta_3{\widetilde{\Phi}}}a^2dx \times \int_{\R^3}|\mathsf{R}x|^2|\nabla_x^2{\widetilde{\Phi}}|^2e^{-2\delta_3{\widetilde{\Phi}}}dx \\
		&\lesssim \int_{\R^3} e^{2\delta_2{\widetilde{\Phi}}}|b|^2dx + \int_{\R^3} e^{2\delta_3{\widetilde{\Phi}}}a^2dx = \|be^{\widetilde{\Phi}}\|_{L^2(V_{\delta_2})} + \|ae^{\widetilde{\Phi}}\|_{L^2(V_{\delta_3})} \lesssim \mathsf{E}.
		\end{aligned}
		\end{equation*}
	%	In the following proof, we will omit the detailed usage of Cauchy-Schwartz inequality.\\
		$\bullet$ Estimates for $|B'(t)|$: Recall that $B(t) = \mathscr{A}(\nabla_x^{skew}(be^{\widetilde{\Phi}}))$, then it follows from \eqref{Energyan7} that
		\begin{equation*}
			|B'(t)| \lesssim \sum_{i, j = 1}^3(|\mathscr{A}(\partial_i((\mathsf{R}b)_je^{\widetilde{\Phi}}))| + |\mathscr{A}(\partial_i(\mathsf{R}x \cdot \nabla_x(b_je^{\widetilde{\Phi}})))| + |\mathscr{A}(\partial_il_{b, j})|).
		\end{equation*}
	Due to Lemma \ref{MicroBound}, we have that $|\mathscr{A}(\partial_il_{b, j})| \lesssim \|e^{\delta_0{\mathsf{H}}}g_1\|_{L_{x, v}^2} \lesssim \mathsf{E}$. For the rest terms, we have 
		\begin{equation*}
		\begin{aligned}
			|\mathscr{A}(\partial_i((\mathsf{R}b)_je^{\widetilde{\Phi}}))| &= |\mathscr{A}((\partial_i{\widetilde{\Phi}})(\mathsf{R}b)_je^{\widetilde{\Phi}})| = C_{\widetilde{\Phi}}\Big|\int_{\R^3}(\partial_i{\widetilde{\Phi}})(\mathsf{R}b)_jdx\Big| \lesssim \|be^{\widetilde{\Phi}}\|_{L^2(V_{\delta_2})} \lesssim \mathsf{E};\\
			|\mathscr{A}(\partial_i(\mathsf{R}x \cdot \nabla_x(b_je^{\widetilde{\Phi}})))| &= |\mathscr{A}((\partial_i{\widetilde{\Phi}})(\mathsf{R}x \cdot \nabla_xb_j)e^{\widetilde{\Phi}})| = C_{\widetilde{\Phi}}\Big|\int_{\R^3}(\partial_i{\widetilde{\Phi}})(\mathsf{R}x \cdot \nabla_xb_j)dx\Big| \\
			&= C_{\widetilde{\Phi}}\Big|\int_{\R^3}\nabla_x(\partial_i{\widetilde{\Phi}}) \cdot \mathsf{R}xb_jdx\Big| \lesssim \|be^{\widetilde{\Phi}}\|_{L^2(V_{\delta_2})} \lesssim \mathsf{E}.
		\end{aligned}
		\end{equation*}
		Thus we conclude that $|B'(t)| \lesssim \mathsf{E}$.\\
		$\bullet$ Estimates for $|A'(t)|$: Recall the definition of $A(t)$ in \eqref{defiAt}, we have that
		\begin{equation*}
			6A'(t) = 2\mathscr{A}((\Delta_x{\widetilde{\Phi}})\pa_t(ce^{\widetilde{\Phi}})) - \mathscr{A}(\Delta_x\pa_t(ae^{\widetilde{\Phi}})) - \mathscr{A}(\nabla_x{\widetilde{\Phi}} \cdot \mathsf{R}\pa_tbe^{\widetilde{\Phi}}) - \mathscr{A}(\nabla_x{\widetilde{\Phi}} \cdot (\mathsf{R}x \cdot \nabla_x\pa_tbe^{\widetilde{\Phi}})).
		\end{equation*}
		It suffices to estimate one representative term; the others are analogous. By \eqref{ptc},
		\begin{equation*}
		\begin{aligned}
			&|\mathscr{A}((\Delta_x{\widetilde{\Phi}})\partial_t(ce^{\widetilde{\Phi}}))| \lesssim |\mathscr{A}((\Delta_x{\widetilde{\Phi}})\nabla_x \cdot (be^{\widetilde{\Phi}}))| + |\mathscr{A}((\Delta_x{\widetilde{\Phi}})\mathsf{R}x \cdot \nabla_x(ce^{\widetilde{\Phi}}))|+ |\mathscr{A}((\Delta_x{\widetilde{\Phi}})l_c)|.
		\end{aligned}
		\end{equation*}
		For the first term and the second term on the right hand side, by integration by parts, we have that
			\begingroup\small
			\begin{equation*}
			\begin{aligned}
				&|\mathscr{A}((\Delta_x{\widetilde{\Phi}})\nabla_x \cdot (be^{\widetilde{\Phi}}))| + |\mathscr{A}((\Delta_x{\widetilde{\Phi}})\mathsf{R}x \cdot \nabla_x(ce^{\widetilde{\Phi}}))|\\
				\lesssim& \Big|\int_{\R^3}(\Delta_x{\widetilde{\Phi}})(\nabla_x{\widetilde{\Phi}} \cdot b + \nabla_x \cdot b)dx\Big| + \Big|\int_{\R^3}(\Delta_x{\widetilde{\Phi}})(\mathsf{R}x \cdot \nabla_xc)dx\Big| \\
				\lesssim& \Big|\int_{\R^3}(\Delta_x{\widetilde{\Phi}})(\nabla_x{\widetilde{\Phi}} \cdot b)dx\Big| + \Big|\int_{\R^3}(\nabla_x(\Delta_x{\widetilde{\Phi}}) \cdot b)dx\Big| + \Big|\int_{\R^3}\nabla_x(\Delta_x{\widetilde{\Phi}}) \cdot (\mathsf{R}xc)dx\Big| \\
				\lesssim& \|be^{\widetilde{\Phi}}\|_{L^2(V_{\delta_2})} + \|ce^{\widetilde{\Phi}}\|_{L^2(V_{\delta_1})} \lesssim \mathsf{E}.
			\end{aligned}
			\end{equation*}
			\endgroup
			Here we emphasize that for $p \in (1, 2)$, function $|\nabla_x(\Delta_x{\widetilde{\Phi}})|^2$ has no singularity at $x=0$. As for $p \in (2, \infty)$, although function $|\nabla_x(\Delta_x{\widetilde{\Phi}})|^2$ has singularity, it is still locally integrable
		in $\mathbb{R}^3$.
		
		For the third term, we can use the similar method as in Lemma \ref{MicroBound} to get that	\begin{equation*}
	\begin{aligned}
		|\mathscr{A}((\Delta_x{\widetilde{\Phi}})l_c)|\ls \mathsf{E}.
	\end{aligned}
\end{equation*}
	 We omit the estimates for other terms and conclude that $|A'(t)|\ls \mathsf{E}.$ 

Combining the preceding estimates completes the proof.
	\end{proof}

%%%%%%%%%%%%%%%%%%%%%%%%%%%%%%%%%%%%%%%%%%%%%%%%%%

	\subsubsection{Estimates for $\partial_t\nabla_x(a_\mathsf{n}, b_\mathsf{n}, c_\mathsf{n})$}
	We next estimate the terms on the right-hand side of Lemma \ref{Estimateanbncn}.
	\begin{lemma}[Time derivatives of the coercive components]\label{Estimateptanbncn}
		One has
		\begin{equation*}
			\|\Lambda_{\delta_4}^{-1}\partial_t\nabla_x(c_\mathsf{n}e^{\widetilde{\Phi}})\|_{L^2(V_{\delta_4})}
			+ \|\Lambda_{\delta_4}^{-1}\partial_t\nabla_x^{sym}(b_\mathsf{n}e^{\widetilde{\Phi}})\|_{L^2(V_{\delta_4})}
			+ \|\Lambda_{\delta_4}^{-1}\partial_t\nabla_x(a_\mathsf{n}e^{\widetilde{\Phi}})\|_{L^2(V_{\delta_4})}
			\lesssim \mathsf{E}.
		\end{equation*}
	\end{lemma}
	\begin{proof}
		We give the proof one to one.\\
		$\bullet$ Estimation for $c_\mathsf{n}$: Note that
		\begin{equation*}
			\partial_t\nabla_x(c_\mathsf{n}e^{\widetilde{\Phi}}) = \partial_t\nabla_x(ce^{\widetilde{\Phi}}) = \nabla_x\partial_t(c_\mathsf{n}e^{\widetilde{\Phi}}),
		\end{equation*}
		thus by the first and last inequality in Lemma \ref{BoundedLambda}, Lemma \ref{MicroBound} and \eqref{ptc}, we have
		\begin{equation*}
		\begin{aligned}
			\|\Lambda_{\delta_4}^{-1}\partial_t\nabla_x(c_\mathsf{n}e^{\widetilde{\Phi}})\|_{L^2(V_{\delta_4})}
			&= \|\Lambda_{\delta_4}^{-1}\nabla_x\partial_t(ce^{\widetilde{\Phi}})\|_{L^2(V_{\delta_4})}
			\lesssim \|\Lambda_{\delta_4}^{-\frac{1}{2}}\partial_t(ce^{\widetilde{\Phi}})\|_{L^2(V_{\delta_4})} \\
			&\lesssim \|\Lambda_{\delta_4}^{-\frac{1}{2}}\nabla_x \cdot (be^{\widetilde{\Phi}})\|_{L^2(V_{\delta_4})} + \|\Lambda_{\delta_4}^{-\frac{1}{2}}Rx \cdot \nabla_x(ce^{\widetilde{\Phi}})\|_{L^2(V_{\delta_4})} + \|\Lambda_{\delta_4}^{-\frac{1}{2}}l_c\|_{L^2(V_{\delta_4})} \\
			&\lesssim \|be^{\widetilde{\Phi}}\|_{L^2(V_{\delta_2})} + \|ce^{\widetilde{\Phi}}\|_{L^2(V_{\delta_1})} + \|e^{\delta_0{\mathsf{H}}}g_1\|_{L_{x, v}^2} \lesssim \mathsf{E}.
		\end{aligned}
		\end{equation*}
		$\bullet$ Estimation for $b_\mathsf{n}$: We recall \eqref{bnbeq} as
		\begin{equation*}
			\nabla_x^{sym}(b_\mathsf{n}e^{\widetilde{\Phi}}) = \nabla_x^{sym}(be^{\widetilde{\Phi}}) - \frac{1}{3}\mathbb{I}_{3 \times 3}\mathscr{A}(\nabla_x \cdot (be^{\widetilde{\Phi}})) = \nabla_x^{sym}(be^{\widetilde{\Phi}}) - \mathbb{I}_{3 \times 3}b_\mathfrak{R}(t).
		\end{equation*}
		Again apply the first and last inequality in Lemma \ref{BoundedLambda}, Lemma \ref{MicroBound}, Lemma \ref{crbrbsasAB'} and \eqref{ptb}, we have
		\begin{equation*}
		\begin{aligned}
			\|\Lambda_{\delta_4}^{-1}\partial_t\nabla_x^{sym}(b_\mathsf{n}e^{\widetilde{\Phi}})\|_{L^2(V_{\delta_4})}
			&\lesssim \|\Lambda_{\delta_4}^{-1}\nabla_x^{sym}\partial_t(be^{\widetilde{\Phi}})\|_{L^2(V_{\delta_4})} + |b'_\mathsf{R}(t)|
			\lesssim \|\Lambda_{\delta_4}^{-\frac{1}{2}}\partial_t(be^{\widetilde{\Phi}})\|_{L^2(V_{\delta_4})} + \mathsf{E} \\
			&\lesssim \|\Lambda_{\delta_4}^{-\frac{1}{2}}\nabla_x(ae^{\widetilde{\Phi}})\|_{L^2(V_{\delta_4})} + \|\Lambda_{\delta_4}^{-\frac{1}{2}}(\nabla_xc)e^{\widetilde{\Phi}}\|_{L^2(V_{\delta_4})} + \|\Lambda_{\delta_4}^{-\frac{1}{2}}Rbe^{\widetilde{\Phi}}\|_{L^2(V_{\delta_4})} \\
			&+ \|\Lambda_{\delta_4}^{-\frac{1}{2}}Rx \cdot \nabla_x(be^{\widetilde{\Phi}})\|_{L^2(V_{\delta_4})} + \|\Lambda_{\delta_4}^{-\frac{1}{2}}l_b\|_{L^2(V_{\delta_4})} + \mathsf{E} \\
			&\lesssim \|ae^{\widetilde{\Phi}}\|_{L^2(V_{\delta_3})} + \|be^{\widetilde{\Phi}}\|_{L^2(V_{\delta_2})} + \|ce^{\widetilde{\Phi}}\|_{L^2(V_{\delta_1})} + \|e^{\delta_0{\mathsf{H}}}g_1\|_{L_{x, v}^2} + \mathsf{E} \lesssim \mathsf{E}.
		\end{aligned}
		\end{equation*}
		Where we have used the trick
		\begin{multline}\label{trick}
			\|\Lambda_{\delta_4}^{-\frac{1}{2}}(\nabla_xc)e^{\widetilde{\Phi}}\|_{L^2(V_{\delta_4})} = \|\Lambda_{\delta_4}^{-\frac{1}{2}}(\nabla_x(ce^{\widetilde{\Phi}}) - (\nabla_x{\widetilde{\Phi}})ce^{\widetilde{\Phi}})\|_{L^2(V_{\delta_4})} \\
			\lesssim \|\Lambda_{\delta_4}^{-\frac{1}{2}}\nabla_x(ce^{\widetilde{\Phi}})\|_{L^2(V_{\delta_4})} + \|\Lambda_{\delta_4}^{-\frac{1}{2}}(\nabla_x{\widetilde{\Phi}})ce^{\widetilde{\Phi}}\|_{L^2(V_{\delta_4})} \lesssim \|ce^{\widetilde{\Phi}}\|_{L^2(V_{\delta_1})}.
		\end{multline}
		$\bullet$ Estimation for $a_\mathsf{n}$: We recall \eqref{anaeq} as
		\begin{equation*}
		\begin{aligned}
			\nabla_x(a_\mathsf{n}e^{\widetilde{\Phi}}) &= \nabla_x(ae^{\widetilde{\Phi}}) - 2(\nabla_x{\widetilde{\Phi}})\mathscr{A}(ce^{\widetilde{\Phi}}) - \mathscr{A}(\nabla_x(ae^{\widetilde{\Phi}})) + 2A(t)x + 2T(t)x \\
			&= \nabla_x(ae^{\widetilde{\Phi}}) - 2(\nabla_x{\widetilde{\Phi}})c_\mathfrak{r}(t) - a_\mathfrak{R}(t) + 2A(t)x + 2T(t)x.
		\end{aligned}
		\end{equation*}
		First we apply the last inequality in Lemma \ref{BoundedLambda}, and get
		\begin{equation*}
		\begin{aligned}
			\|\Lambda_{\delta_4}^{-1}\partial_t\nabla_x(a_\mathsf{n}e^{\widetilde{\Phi}})\|_{L^2(V_{\delta_4})}
			&\lesssim \|\Lambda_{\delta_4}^{-1}\nabla_x\partial_t(ae^{\widetilde{\Phi}})\|_{L^2(V_{\delta_4})} + |c'_\mathsf{r}(t)| + |a'_\mathsf{R}(t)| + |A'(t)| + |T'(t)| \\
			&\lesssim \|\Lambda_{\delta_4}^{-\frac{1}{2}}\partial_t(ae^{\widetilde{\Phi}})\|_{L^2(V_{\delta_4})} + \mathsf{E} \\
			&\lesssim \|\Lambda_{\delta_4}^{-\frac{1}{2}}(\nabla_x \cdot b)e^{\widetilde{\Phi}}\|_{L^2(V_{\delta_4})} + \|\Lambda_{\delta_4}^{-\frac{1}{2}}Rx \cdot \nabla_x(ae^{\widetilde{\Phi}})\|_{L^2(V_{\delta_4})} + \mathsf{E} \\
			&\lesssim \|be^{\widetilde{\Phi}}\|_{L^2(V_{\delta_2})} + \|ae^{\widetilde{\Phi}}\|_{L^2(V_{\delta_3})} + \mathsf{E} \lesssim \mathsf{E},
		\end{aligned}
		\end{equation*}
		where we have used the same trick as in \eqref{trick}:
		\begin{multline*}
			\|\Lambda_{\delta_4}^{-\frac{1}{2}}(\nabla_x \cdot b)e^{\widetilde{\Phi}}\|_{L^2(V_{\delta_4})} = \|\Lambda_{\delta_4}^{-\frac{1}{2}}(\nabla_x \cdot (be^{\widetilde{\Phi}}) - (\nabla_x{\widetilde{\Phi}}) \cdot be^{\widetilde{\Phi}})\|_{L^2(V_{\delta_4})} \\
			\lesssim \|\Lambda_{\delta_4}^{-\frac{1}{2}}\nabla_x \cdot (be^{\widetilde{\Phi}})\|_{L^2(V_{\delta_4})} + \|\Lambda_{\delta_4}^{-\frac{1}{2}}(\nabla_x{\widetilde{\Phi}}) \cdot be^{\widetilde{\Phi}}\|_{L^2(V_{\delta_4})} \lesssim \|be^{\widetilde{\Phi}}\|_{L^2(V_{\delta_2})}.
		\end{multline*}
		We complete the proof of this lemma.
	\end{proof}

%%%%%%%%%%%%%%%%%%%%%%%%%%%%%%%%%%%%%%%%%%%%%%%%%%

	\subsubsection{Evolution equations for the remaining modes}
We first derive the evolution equations needed to estimate the remaining modes.
	\begin{lemma}[Evolution system for the remaining modes]\label{052}
		Let $c_\mathfrak r$, $b_\mathfrak r$, $b_\mathfrak R$, $a_\mathfrak R$, and $A$ be defined by \eqref{defiAt} and \eqref{051}, and set $B(t):=\mathscr A(\nabla_x^{skew}(be^{\widetilde\Phi}))$. Then
				\begin{multline}\label{brbsAB+}
			(e_1,e_2,e_3)(x) \cdot b_\mathfrak{r}(t) + e_4(x)b_\mathfrak{R}(t) + e_5(x)A'(t) + e_6(x)B'_{12}(t) + e_7(x)B'_{13}(t) + e_8(x)B'_{23}(t) \\+ (e_9,e_{10},e_{11})(x) \cdot (a'_\mathfrak{R}(t) + \mathsf{R}b'_\mathfrak{r}(t)) = \Lambda_{\delta_4}^{-\frac{1}{2}}\partial_t(a_\mathsf{n}e^{\widetilde{\Phi}}) + \Lambda_{\delta_4}^{-\frac{1}{2}}((\nabla_x \cdot b_\mathsf{n})e^{\widetilde{\Phi}}) + \Lambda_{\delta_4}^{-\frac{1}{2}}(\mathsf{R}x \cdot \nabla_x(a_\mathsf{n}e^{\widetilde{\Phi}})),
		\end{multline}
		where $e_i = e_i(x)$ is defined as
		\begin{equation*}
			\begin{aligned}
				&(e_1,e_2,e_3)(x) = \Lambda_{\delta_4}^{-\frac{1}{2}}(\nabla_x{\widetilde{\Phi}} - \mathsf{R}^2x); \:\:\:
				&&e_4(x) = \Lambda_{\delta_4}^{-\frac{1}{2}}(2{\widetilde{\Phi}} - 2\mathscr{A}({\widetilde{\Phi}}) + \nabla_x{\widetilde{\Phi}} \cdot x - 3);\\
				&e_5(x) = \Lambda_{\delta_4}^{-\frac{1}{2}}(|x|^2 - \mathscr{A}(|x|^2));\:\:\:
				&&e_6(x) = -r\Lambda_{\delta_4}^{-\frac{1}{2}}(x_1^2 + x_2^2 - \mathscr{A}(x_1^2 + x_2^2));\\
				&e_7(x) = -r\Lambda_{\delta_4}^{-\frac{1}{2}}(x_2x_3); \:\:\:
				&&e_8(x) = r\Lambda_{\delta_4}^{-\frac{1}{2}}(x_1x_3); \:\:\:\:\:\:\:\:\:\:\:\:
				(e_9,e_{10},e_{11})(x) = -\Lambda_{\delta_4}^{-\frac{1}{2}}x
			\end{aligned}
		\end{equation*}
		with $\mathsf{R}$ and $r$ showed in \eqref{Rform}. Moreover, $\{e_i(x)\}_{i = 1}^{11}$ are linearly independent in $L^2(V_{\delta_4})$ space.
	\end{lemma}
	\begin{proof}
It follows from \eqref{pta} that
	\begin{equation*}
		\partial_t(ae^{\widetilde{\Phi}}) + (\nabla_x \cdot b)e^{\widetilde{\Phi}} + \mathsf{R}x \cdot \nabla_x(ae^{\widetilde{\Phi}}) = 0.
	\end{equation*}
	According to \eqref{abcremainingform}, we compute
	\begin{equation*}
	\begin{gathered}
		\mathsf{R}x \cdot \nabla_x(ae^{\widetilde{\Phi}}) = \mathsf{R}x \cdot \nabla_x(a_ne^{\widetilde{\Phi}}) + \mathsf{R}x \cdot a_\mathfrak{R}(t) - 2\mathsf{R}x \cdot T(t)x;\\
		(\nabla_x \cdot b)e^{\widetilde{\Phi}} = (\nabla_x \cdot b_\mathsf{n})e^{\widetilde{\Phi}} - \nabla_x{\widetilde{\Phi}} \cdot b_\mathfrak{r}(t) - \nabla_x{\widetilde{\Phi}} \cdot B(t)x + 3b_\mathfrak{R}(t) - (\nabla_x{\widetilde{\Phi}} \cdot x)b_\mathfrak{R}(t);\\
		\partial_t(ae^{\widetilde{\Phi}}) = \partial_t(a_\mathsf{n}e^{\widetilde{\Phi}}) + 2({\widetilde{\Phi}} - \mathscr{A}({\widetilde{\Phi}}))c'_\mathfrak{r}(t) + a'_\mathfrak{R}(t) \cdot x - (|x|^2 - \mathscr{A}(|x|^2))A'(t) - T'(t)x \cdot x + \mathscr{A}(T'(t)x \cdot x),
	\end{gathered}
	\end{equation*}
	where we use the fact that $\mathsf{R}x \cdot x = \mathsf{R}x \cdot \nabla_x{\widetilde{\Phi}} = 0$ since $R$ is skew-symmetric. Now \eqref{pta} becomes
	\begin{multline*}
		\partial_t(a_\mathsf{n}e^{\widetilde{\Phi}}) + 2({\widetilde{\Phi}} - \mathscr{A}({\widetilde{\Phi}}))c'_\mathfrak{r}(t) + a'_\mathfrak{R}(t) \cdot x - (|x|^2 - \mathscr{A}(|x|^2))A'(t) - T'(t)x \cdot x + \mathscr{A}(T'(t)x \cdot x) + (\nabla_x \cdot b_\mathsf{n})e^{\widetilde{\Phi}} \\
		- \nabla_x{\widetilde{\Phi}} \cdot b_\mathfrak{r}(t) - \nabla_x{\widetilde{\Phi}} \cdot B(t)x + 3b_\mathfrak{R}(t) - (\nabla_x{\widetilde{\Phi}} \cdot x)b_\mathfrak{R}(t) + \mathsf{R}x \cdot \nabla_x(a_\mathsf{n}e^{\widetilde{\Phi}}) + \mathsf{R}x \cdot a_\mathfrak{R}(t) - 2\mathsf{R}x \cdot T(t)x = 0.
	\end{multline*}
	For $p\in(1,2)$, one has $\widetilde\Phi=\alpha\langle x\rangle^p/p$ and $\mathsf R=0$. Since $B(t)$ is skew-symmetric,
	\begin{equation*}
		\nabla_x{\widetilde{\Phi}} \cdot B(t)x + 2\mathsf{R}x \cdot T(t)x = 0.
	\end{equation*}
	For $p\in(2,\infty)$, recall that $\widetilde\Phi=|x|^p/p-|\mathsf Rx|^2/2$, and hence $\nabla_x\widetilde\Phi=|x|^{p-2}x+\mathsf R^2x$. Since $\mathsf R$ and $B(t)$ are skew-symmetric matrices,
	\begin{multline*}
		\nabla_x{\widetilde{\Phi}} \cdot B(t)x + 2\mathsf{R}x \cdot T(t)x = (|x|^{p - 2}x + \mathsf{R}^2x) \cdot B(t)x + \mathsf{R}x \cdot (\mathsf{R}B(t)x + B(t)\mathsf{R}x) \\
		= \mathsf{R}^2x \cdot B(t)x + \mathsf{R}x \cdot \mathsf{R}B(t)x = 0.
	\end{multline*}
	We use \eqref{Remainingeq} to write $c'_\mathsf{r}(t) = -b_\mathfrak{R}(t)$, together with the above observation, we have
	\begin{multline*}
		-(2{\widetilde{\Phi}} - 2\mathscr{A}({\widetilde{\Phi}}) + \nabla_x{\widetilde{\Phi}} \cdot x - 3)b_\mathfrak{R}(t) + a'_\mathfrak{R}(t) \cdot x - (|x|^2 - \mathscr{A}(|x|^2))A'(t) - T'(t)x \cdot x + \mathscr{A}(T'(t)x \cdot x) \\- \nabla_x{\widetilde{\Phi}} \cdot b_\mathfrak{r}(t) + \mathsf{R}x \cdot a_\mathfrak{R}(t) = -\partial_t(a_\mathsf{n}e^{\widetilde{\Phi}}) - (\nabla_x \cdot b_\mathsf{n})e^{\widetilde{\Phi}} - \mathsf{R}x \cdot \nabla_x(a_\mathsf{n}e^{\widetilde{\Phi}}).
	\end{multline*}
	Again we use \eqref{Remainingeq} to write
	\begin{equation*}
	\begin{aligned}
		a'_\mathfrak{R}(t) \cdot x - \nabla_x{\widetilde{\Phi}} \cdot b_\mathfrak{r}(t) + \mathsf{R}x \cdot a_\mathfrak{R}(t) &= a'_\mathfrak{R}(t) \cdot x - \nabla_x{\widetilde{\Phi}} \cdot b_\mathfrak{r}(t) - \mathsf{R}x \cdot b'_\mathfrak{r}(t) - \mathsf{R}x \cdot \mathsf{R}b_\mathfrak{r}(t) \\
		&= x \cdot (a'_\mathfrak{R}(t) + \mathsf{R}b'_\mathfrak{r}(t)) - (\nabla_x{\widetilde{\Phi}} - \mathsf{R}^2x) \cdot b_\mathfrak{r}(t).
	\end{aligned}
	\end{equation*}
	Now we get
	\begin{multline*}
		(\nabla_x{\widetilde{\Phi}} - \mathsf{R}^2x) \cdot b_\mathfrak{r}(t) + (2{\widetilde{\Phi}} - 2\mathscr{A}({\widetilde{\Phi}}) + \nabla_x{\widetilde{\Phi}} \cdot x - 3)b_\mathfrak{R}(t) + (|x|^2 - \mathscr{A}(|x|^2))A'(t) + T'(t)x \cdot x - \mathscr{A}(T'(t)x \cdot x) \\- x \cdot (a'_\mathfrak{R}(t) + \mathsf{R}b'_\mathfrak{r}(t)) = \partial_t(a_\mathsf{n}e^{\widetilde{\Phi}}) + (\nabla_x \cdot b_\mathsf{n})e^{\widetilde{\Phi}} + \mathsf{R}x \cdot \nabla_x(a_\mathsf{n}e^{\widetilde{\Phi}}).
	\end{multline*}
	Recall the form of matrix $\mathsf{R}$ in \eqref{Rform}, we compute
	\begin{equation*}
		T(t) = \frac{1}{2}(\mathsf{R}B(t) + B(t)\mathsf{R}) = \frac{r}{2}\begin{pmatrix} -2B_{12}(t) & 0 & B_{23}(t) \\ 0 & -2B_{12}(t) & -B_{13}(t) \\ B_{23}(t) & -B_{13}(t) & 0 \end{pmatrix},
	\end{equation*}
	and
	\begin{equation}\label{ComputeTxx}
		T(t)x \cdot x = -r(x_1^2 + x_2^2)B_{12}(t) - rx_2x_3B_{13}(t) + rx_1x_3B_{23}(t).
	\end{equation}
	Note the fact that $\mathscr{A}(x_2x_3) = \mathscr{A}(x_1x_3) = 0$, we finally get
	\begin{multline*}
		(\nabla_x{\widetilde{\Phi}} - \mathsf{R}^2x) \cdot b_\mathfrak{r}(t) + (2{\widetilde{\Phi}} - 2\mathscr{A}({\widetilde{\Phi}}) + \nabla_x{\widetilde{\Phi}} \cdot x - 3)b_\mathfrak{R}(t) + (|x|^2 - \mathscr{A}(|x|^2))A'(t) \\- r(x_1^2 + x_2^2 - \mathscr{A}(x_1^2 + x_2^2))B'_{12}(t) - rx_2x_3B'_{13}(t) + rx_1x_3B'_{23}(t) - x \cdot (a'_\mathsf{R}(t) + Rb'_\mathsf{r}(t)) \\= \partial_t(a_\mathsf{n}e^{\widetilde{\Phi}}) + (\nabla_x \cdot b_\mathsf{n})e^{\widetilde{\Phi}} + \mathsf{R}x \cdot \nabla_x(a_\mathsf{n}e^{\widetilde{\Phi}}),
	\end{multline*}
	which is our desired result.

Next, we will prove the linear independence. Otherwise,  $\{\Lambda_{\delta_4}^{\frac{1}{2}}e_i(x)\}$ are linearly dependent and furthermore $\{1, \Lambda_{\delta_4}^{\frac{1}{2}}e_i(x)\}$ are linearly dependent. Thus we consider linear subspace $X := \mathrm{span}\{1, \Lambda_{\delta_4}^{\frac{1}{2}}e_i(x)\}_{i = 1}^{11}$. One may check that
	\begin{equation*}
		X = \mathrm{span}\{1, x, x_1x_3, x_2x_3, x_1^2 + x_2^2, x_3^2, \nabla_x{\widetilde{\Phi}} - \mathsf{R}^2x, 2{\widetilde{\Phi}} + \nabla_x{\widetilde{\Phi}} \cdot x\}.
	\end{equation*}
	More precisely, for $p \in (1, 2)$ we have
	\begin{equation*}
		\nabla_x{\widetilde{\Phi}} - R^2x = \alpha\langle x \rangle^{p - 2}x, \:\:\: 2{\widetilde{\Phi}} + \nabla_x{\widetilde{\Phi}} \cdot x = \alpha(1 + \frac{2}{p})\langle x \rangle^p - \alpha\langle x \rangle^{p - 2}.
	\end{equation*}
	As for $p \in (2, \infty)$, we compute
	\begin{equation*}
		\nabla_x{\widetilde{\Phi}} - \mathsf{R}^2x = |x|^{p - 2}x, \:\:\: 2{\widetilde{\Phi}} + \nabla_x{\widetilde{\Phi}} \cdot x = (1 + \frac{2}{p})|x|^p - 2r^2(x_1^2 + x_2^2).
	\end{equation*}
	Thus in either cases, we have $\dim X = 12$, which means $\{1, \Lambda_{\delta_4}^{\frac{1}{2}}e_i(x)\}$ are linearly independent, and so are $\{e_i(x)\}_{i = 1}^{11}$.  Then we complete the proof of this lemma.
\end{proof}

%%%%%%%%%%%%%%%%%%%%%%%%%%%%%%%%%%%%%%%%%%%%%%%%%%
We begin with estimates for $b_\mathfrak r$, $b_\mathfrak R$, $a_\mathfrak R$, and $A$.
	\subsubsection{Estimates for $|b_\mathfrak{r}(t)| + |b_\mathfrak{R}(t)| + |a_\mathfrak{R}(t)| + |A(t)|$}
	\begin{lemma}[Evolution-controlled finite-dimensional modes]\label{EstimatebrbsasA}
		There exists a function $\mathcal{J}_0 = \mathcal{J}_0(t)$ satisfying $|\mathcal{J}_0(t)| \lesssim \mathsf{E}^2$ and
			\begin{equation*}
	\begin{aligned}
		&\frac{d}{dt}\mathcal{J}_0 + |b_\mathfrak{r}(t)|^2 + |b_\mathfrak{R}(t)|^2 + |a_\mathfrak{R}(t)|^2 + |A(t)|^2 \\
		\lesssim& \|e^{\delta_0{\mathsf{H}}}g_1\|_{L_{x, v}^2}^2 + \|b_\mathsf{n}e^{\widetilde{\Phi}}\|_{L^2(V_{\delta_2})}^2 + \|a_\mathsf{n}e^{\widetilde{\Phi}}\|_{L^2(V_{\delta_3})}^2 + \|a_\mathsf{n}e^{\widetilde{\Phi}}\|_{L^2(V_{\delta_3})}\mathsf{E}.
	\end{aligned}
\end{equation*}
	\end{lemma}
	\begin{proof}
			As a direct corollary of Lemma \ref{052}, there exist $\tilde{e}_i(x) \in \mathrm{span}\{e_j(x)\}_{j = 1}^{11}$, $i=1,\ldots,11$, such that
		\begin{equation}\label{053}
			(e_i(x), \tilde{e}_j(x))_{L^2(V_{\delta_4})} = \de_{ij},\qquad i,j=1,\ldots,11.
		\end{equation}\\
		$\bullet$ Estimates for $b_\mathfrak{r}(t)$: In \eqref{brbsAB+}, we take inner product with $(\tilde{e}_{1},\tilde{e}_{2},\tilde{e}_{3})(x) \cdot b_\mathfrak{r}(t)$ in $L^2(V_{\delta_4})$ space. Then by \eqref{053}, we have
		\begin{equation*}
			|b_\mathfrak{r}(t)|^2 = (\Lambda_{\delta_4}^{-\frac{1}{2}}\partial_t(a_\mathsf{n}e^{\widetilde{\Phi}}) + \Lambda_{\delta_4}^{-\frac{1}{2}}(\nabla_x \cdot b_\mathsf{n})e^{\widetilde{\Phi}} + \Lambda_{\delta_4}^{-\frac{1}{2}}\mathsf{R}x \cdot \nabla_x(a_\mathsf{n}e^{\widetilde{\Phi}}), (\tilde{e}_{1},\tilde{e}_{2},\tilde{e}_{3})(x) \cdot b_\mathfrak{r}(t))_{L^2(V_{\delta_4})}
		\end{equation*}
		Note that
		\begin{equation*}
		\begin{aligned}
			(\Lambda_{\delta_4}^{-\frac{1}{2}}(\nabla_x \cdot b_\mathsf{n})e^{\widetilde{\Phi}}, (\tilde{e}_{1},\tilde{e}_{2},\tilde{e}_{3})(x) \cdot b_\mathfrak{r}(t))_{L^2(V_{\delta_4})} &\le C\|b_\mathsf{n}e^{\widetilde{\Phi}}\|_{L^2(V_{\delta_2})}^2 + \frac{1}{4}|b_\mathfrak{r}(t)|^2;\\
			(\Lambda_{\delta_4}^{-\frac{1}{2}}\mathsf{R}x \cdot \nabla_x(a_\mathsf{n}e^{\widetilde{\Phi}}), (\tilde{e}_{1},\tilde{e}_{2},\tilde{e}_{3})(x) \cdot b_\mathfrak{r}(t))_{L^2(V_{\delta_4})} &\le C\|a_\mathsf{n}e^{\widetilde{\Phi}}\|_{L^2(V_{\delta_3})}^2 + \frac{1}{4}|b_\mathfrak{r}(t)|^2.
		\end{aligned}
		\end{equation*}
		And we write
		\begin{multline*}
			(\Lambda_{\delta_4}^{-\frac{1}{2}}\partial_t(a_\mathsf{n}e^{\widetilde{\Phi}}), (\tilde{e}_{1},\tilde{e}_{2},\tilde{e}_{3})(x) \cdot b_\mathfrak{r}(t))_{L^2(V_{\delta_4})} \\
			= \frac{d}{dt}(\Lambda_{\delta_4}^{-\frac{1}{2}}(a_\mathsf{n}e^{\widetilde{\Phi}}), (\tilde{e}_{1},\tilde{e}_{2},\tilde{e}_{3})(x) \cdot b_\mathfrak{r}(t))_{L^2(V_{\delta_4})} - (\Lambda_{\delta_4}^{-\frac{1}{2}}(a_\mathsf{n}e^{\widetilde{\Phi}}), (\tilde{e}_{1},\tilde{e}_{2},\tilde{e}_{3})(x) \cdot b'_\mathfrak{r}(t))_{L^2(V_{\delta_4})}.
		\end{multline*}
		By Lemma \ref{crbrbsasAB'},
		\begin{equation*}
		- (\Lambda_{\delta_4}^{-\frac{1}{2}}(a_\mathsf{n}e^{\widetilde{\Phi}}), (\tilde{e}_{1},\tilde{e}_{2},\tilde{e}_{3})(x) \cdot b'_\mathfrak{r}(t))_{L^2(V_{\delta_4})} \lesssim \|a_\mathsf{n}e^{\widetilde{\Phi}}\|_{L^2(V_{\delta_3})}\mathsf{E},
		\end{equation*}
		Thus we get
		\begin{equation}\label{Estimatebr}
			\frac{d}{dt}\mathcal{J}_{0, 1} + |b_\mathfrak{r}(t)|^2 \lesssim \|b_\mathsf{n}e^{\widetilde{\Phi}}\|_{L^2(V_{\delta_2})}^2 + \|a_\mathsf{n}e^{\widetilde{\Phi}}\|_{L^2(V_{\delta_3})}^2 + \|a_\mathsf{n}e^{\widetilde{\Phi}}\|_{L^2(V_{\delta_3})}\mathsf{E},
		\end{equation}
		with
		\begin{equation*}
			|\mathcal{J}_{0, 1}| = 2|(\Lambda_{\delta_4}^{-\frac{1}{2}}(a_\mathsf{n}e^{\widetilde{\Phi}}), (\tilde{e}_{1},\tilde{e}_{2},\tilde{e}_{3})(x) \cdot b_\mathfrak{r}(t))_{L^2(V_{\delta_4})}| \lesssim \|a_\mathsf{n}e^{\widetilde{\Phi}}\|_{L^2(V_{\delta_3})}|b_\mathfrak{r}(t)| \lesssim \mathsf{E}^2.
		\end{equation*}\\
		$\bullet$ Estimates for $b_\mathfrak{R}(t)$: Similarly, in \eqref{brbsAB+}, we take inner product with $\tilde{e}_4(x)b_\mathfrak{R}(t)$ in $L^2(V_{\delta_4})$ space to get that
		\begin{equation*}
		\begin{aligned}
			|b_\mathfrak{R}(t)|^2 &= (\Lambda_{\delta_4}^{-\frac{1}{2}}\partial_t(a_\mathsf{n}e^{\widetilde{\Phi}}) + \Lambda_{\delta_4}^{-\frac{1}{2}}(\nabla_x \cdot b_\mathsf{n})e^{\widetilde{\Phi}} + \Lambda_{\delta_4}^{-\frac{1}{2}}\mathsf{R}x \cdot \nabla_x(a_\mathsf{n}e^{\widetilde{\Phi}}), \tilde{e}_4(x)b_\mathfrak{R}(t))_{L^2(V_{\delta_4})} \\
			&= \frac{d}{dt}(\Lambda_{\delta_4}^{-\frac{1}{2}}(a_\mathsf{n}e^{\widetilde{\Phi}}), \tilde{e}_4(x)b_\mathfrak{R}(t))_{L^2(V_{\delta_4})} - (\Lambda_{\delta_4}^{-\frac{1}{2}}(a_\mathsf{n}e^{\widetilde{\Phi}}), \tilde{e}_4(x)b'_\mathfrak{R}(t))_{L^2(V_{\delta_4})} \\
			&+ (\Lambda_{\delta_4}^{-\frac{1}{2}}(\nabla_x \cdot b_\mathsf{n})e^{\widetilde{\Phi}} + \Lambda_{\delta_4}^{-\frac{1}{2}}\mathsf{R}x \cdot \nabla_x(a_\mathsf{n}e^{\widetilde{\Phi}}), \tilde{e}_4(x)b_\mathfrak{R}(t))_{L^2(V_{\delta_4})} \\
			&\le \frac{d}{dt}(\Lambda_{\delta_4}^{-\frac{1}{2}}(a_\mathsf{n}e^{\widetilde{\Phi}}), \tilde{e}_4(x)b_\mathfrak{R}(t))_{L^2(V_{\delta_4})} \\
			&+ C(\|b_\mathsf{n}e^{\widetilde{\Phi}}\|_{L^2(V_{\delta_2})}^2 + \|a_\mathsf{n}e^{\widetilde{\Phi}}\|_{L^2(V_{\delta_3})}^2 + \|a_\mathsf{n}e^{\widetilde{\Phi}}\|_{L^2(V_{\delta_3})}\mathsf{E}) + \frac{1}{2}|b_\mathfrak{R}(t)|^2.
		\end{aligned}
		\end{equation*}
		Thus we get
		\begin{equation}\label{Estimatebs}
			\frac{d}{dt}\mathcal{J}_{0, 2} + |b_\mathfrak{R}(t)|^2 \lesssim \|b_\mathsf{n}e^{\widetilde{\Phi}}\|_{L^2(V_{\delta_2})}^2 + \|a_\mathsf{n}e^{\widetilde{\Phi}}\|_{L^2(V_{\delta_3})}^2 + \|a_\mathsf{n}e^{\widetilde{\Phi}}\|_{L^2(V_{\delta_3})}\mathsf{E},
		\end{equation}
		with
		\begin{equation*}
			|\mathcal{J}_{0, 2}| = 2|(\Lambda_{\delta_4}^{-\frac{1}{2}}(a_\mathsf{n}e^{\widetilde{\Phi}}), \tilde{e}_4(x)b_\mathfrak{R}(t))_{L^2(V_{\delta_4})}| \lesssim \|a_\mathsf{n}e^{\widetilde{\Phi}}\|_{L^2(V_{\delta_3})}|b_\mathfrak{R}(t)| \lesssim \mathsf{E}^2.
		\end{equation*}\\
		$\bullet$ Estimates for $a_\mathfrak{R}(t)$: Set
		\[
		\widetilde e_{\mathfrak r}(t,x)
		=(\tilde e_9,\tilde e_{10},\tilde e_{11})(x)\cdot b_\mathfrak r(t).
		\]
		Testing \eqref{brbsAB+} against $\widetilde e_{\mathfrak r}$ in $L^2(V_{\delta_4})$ gives
		\begin{equation*}
		\begin{aligned}
			&(a'_\mathfrak{R}(t)+\mathsf{R}b'_\mathfrak{r}(t))b_\mathfrak{r}(t)\\
			&\quad=\big(\Lambda_{\delta_4}^{-\frac12}\partial_t(a_\mathsf n e^{\widetilde\Phi})
			+\Lambda_{\delta_4}^{-\frac12}(\nabla_x\cdot b_\mathsf n)e^{\widetilde\Phi}\\
			&\hspace{25mm}
			+\Lambda_{\delta_4}^{-\frac12}\mathsf R x\cdot\nabla_x(a_\mathsf n e^{\widetilde\Phi}),
			\widetilde e_{\mathfrak r}\big)_{L^2(V_{\delta_4})}\\
			&\quad=\frac d{dt}
			\big(\Lambda_{\delta_4}^{-\frac12}(a_\mathsf n e^{\widetilde\Phi}),
			\widetilde e_{\mathfrak r}\big)_{L^2(V_{\delta_4})}\\
			&\qquad-\big(\Lambda_{\delta_4}^{-\frac12}(a_\mathsf n e^{\widetilde\Phi}),
			(\tilde e_9,\tilde e_{10},\tilde e_{11})(x)\cdot b'_\mathfrak r(t)\big)_{L^2(V_{\delta_4})}\\
			&\qquad+\big(\Lambda_{\delta_4}^{-\frac12}(\nabla_x\cdot b_\mathsf n)e^{\widetilde\Phi}
			+\Lambda_{\delta_4}^{-\frac12}\mathsf R x\cdot\nabla_x(a_\mathsf n e^{\widetilde\Phi}),
			\widetilde e_{\mathfrak r}\big)_{L^2(V_{\delta_4})}\\
			&\quad\le \frac d{dt}
			\big(\Lambda_{\delta_4}^{-\frac12}(a_\mathsf n e^{\widetilde\Phi}),
			\widetilde e_{\mathfrak r}\big)_{L^2(V_{\delta_4})} \\
				&\quad + C\big(\|b_\mathsf{n}e^{\widetilde{\Phi}}\|_{L^2(V_{\delta_2})}^2
				+ \|a_\mathsf{n}e^{\widetilde{\Phi}}\|_{L^2(V_{\delta_3})}^2\\
				&\hspace{37mm}
				+ \|a_\mathsf{n}e^{\widetilde{\Phi}}\|_{L^2(V_{\delta_3})}\mathsf{E}
				+ |b_\mathfrak{r}(t)|^2\big).
		\end{aligned}
		\end{equation*}
		By \eqref{Remainingeq},
		\begin{equation*}
		\begin{aligned}
			(a'_\mathfrak{R}(t) + \mathsf{R}b'_\mathfrak{r}(t))b_\mathfrak{r}(t) &= \frac{d}{dt}[(a_\mathfrak{R}(t) + \mathsf{R}b_\mathfrak{r}(t))b_\mathfrak{r}(t)] - (a_\mathfrak{R}(t) + \mathsf{R}b_\mathfrak{r}(t))b'_\mathfrak{r}(t) \\
			&= \frac{d}{dt}[(a_\mathfrak{R}(t) + \mathsf{R}b_\mathfrak{r}(t))b_\mathfrak{r}(t)] + |a_\mathfrak{R}(t) + \mathsf{R}b_\mathfrak{r}(t)|^2 \\
			&\ge \frac{d}{dt}[(a_\mathfrak{R}(t) + \mathsf{R}b_\mathfrak{r}(t))b_\mathfrak{r}(t)] + \frac{1}{2}|a_\mathfrak{R}(t)|^2 - 2|\mathsf{R}b_\mathfrak{r}(t)|^2.
		\end{aligned}
		\end{equation*}
		Thus we get
		\begin{equation}\label{Estimateas}
			\frac{d}{dt}\mathcal{J}_{0, 3} + |a_\mathfrak{R}(t)|^2 \lesssim \|b_\mathsf{n}e^{\widetilde{\Phi}}\|_{L^2(V_{\delta_2})}^2 + \|a_\mathsf{n}e^{\widetilde{\Phi}}\|_{L^2(V_{\delta_3})}^2 + \|a_\mathsf{n}e^{\widetilde{\Phi}}\|_{L^2(V_{\delta_3})}\mathsf{E} + |b_\mathfrak{r}(t)|^2,
		\end{equation}
		with
		\begin{equation*}
		\begin{aligned}
			|\mathcal{J}_{0, 3}| &= 2|(\Lambda_{\delta_4}^{-\frac{1}{2}}(a_\mathsf{n}e^{\widetilde{\Phi}}), (\tilde{e}_{9},\tilde{e}_{10},\tilde{e}_{11})(x) \cdot b_\mathfrak{r}(t))_{L^2(V_{\delta_4})} - (a_\mathfrak{R}(t) + \mathsf{R}b_\mathfrak{r}(t))b_\mathfrak{r}(t)| \\
			&\lesssim \|a_\mathsf{n}e^{\widetilde{\Phi}}\|_{L^2(V_{\delta_3})}|b_\mathfrak{r}(t)| + |a_\mathfrak{R}(t)|^2 + |b_\mathfrak{r}(t)|^2 \lesssim \mathsf{E}^2.
		\end{aligned}
		\end{equation*}\\
		$\bullet$ Estimates for $A(t)$: In \eqref{brbsAB+}, we take inner product with $-\tilde{e}_5(x)b_\mathfrak{R}(t)$ in $L^2(V_{\delta_4})$ space.
		\begin{equation*}
		\begin{aligned}
			-A'(t)b_\mathfrak{R}(t) &= -(\Lambda_{\delta_4}^{-\frac{1}{2}}\partial_t(a_\mathsf{n}e^{\widetilde{\Phi}}) + \Lambda_{\delta_4}^{-\frac{1}{2}}(\nabla_x \cdot b_\mathsf{n})e^{\widetilde{\Phi}} + \Lambda_{\delta_4}^{-\frac{1}{2}}\mathsf{R}x \cdot \nabla_x(a_\mathsf{n}e^{\widetilde{\Phi}}), \tilde{e}_5(x)b_\mathfrak{R}(t))_{L^2(V_{\delta_4})} \\
			&= -\frac{d}{dt}(\Lambda_{\delta_4}^{-\frac{1}{2}}(a_\mathsf{n}e^{\widetilde{\Phi}}), \tilde{e}_5(x)b_\mathfrak{R}(t))_{L^2(V_{\delta_4})} + (\Lambda_{\delta_4}^{-\frac{1}{2}}(a_\mathsf{n}e^{\widetilde{\Phi}}), \tilde{e}_5(x)b'_\mathfrak{R}(t))_{L^2(V_{\delta_4})} \\
			&- (\Lambda_{\delta_4}^{-\frac{1}{2}}(\nabla_x \cdot b_\mathsf{n})e^{\widetilde{\Phi}} + \Lambda_{\delta_4}^{-\frac{1}{2}}\mathsf{R}x \cdot \nabla_x(a_\mathsf{n}e^{\widetilde{\Phi}}), \tilde{e}_5(x)b_\mathfrak{R}(t))_{L^2(V_{\delta_4})} \\
			&\le -\frac{d}{dt}(\Lambda_{\delta_4}^{-\frac{1}{2}}(a_\mathsf{n}e^{\widetilde{\Phi}}), \tilde{e}_5(x)b_\mathfrak{R}(t))_{L^2(V_{\delta_4})} \\
			&+ C(\|b_\mathsf{n}e^{\widetilde{\Phi}}\|_{L^2(V_{\delta_2})}^2 + \|a_\mathsf{n}e^{\widetilde{\Phi}}\|_{L^2(V_{\delta_3})}^2 + \|a_\mathsf{n}e^{\widetilde{\Phi}}\|_{L^2(V_{\delta_3})}\mathsf{E} + |b_\mathfrak{R}(t)|^2).
		\end{aligned}
		\end{equation*}
		By \eqref{Remainingeq+},
		\begin{equation*}
		\begin{aligned}
			-A'(t)b_\mathfrak{R}(t) &= -\frac{d}{dt}[A(t)b_\mathfrak{R}(t)] + A(t)b'_\mathfrak{R}(t) = -\frac{d}{dt}[A(t)b_\mathfrak{R}(t)] + 2|A(t)|^2 + \frac{1}{3}\mathscr{A}(\nabla_x \cdot l_b)A(t) \\
			&\le -\frac{d}{dt}[A(t)b_\mathfrak{R}(t)] + |A(t)|^2 - C\|e^{\delta_0{\mathsf{H}}}g_1\|_{L_{x, v}^2}^2.
		\end{aligned}
		\end{equation*}
		Thus we get
		\begin{equation}\label{EstimateA}
			\frac{d}{dt}\mathcal{J}_{0, 4} + |A(t)|^2 \lesssim \|e^{\delta_0{\mathsf{H}}}g_1\|_{L_{x, v}^2}^2 + \|b_\mathsf{n}e^{\widetilde{\Phi}}\|_{L^2(V_{\delta_2})}^2 + \|a_\mathsf{n}e^{\widetilde{\Phi}}\|_{L^2(V_{\delta_3})}^2 + \|a_\mathsf{n}e^{\widetilde{\Phi}}\|_{L^2(V_{\delta_3})}\mathsf{E} + |b_\mathfrak{R}(t)|^2,
		\end{equation}
		with
		\begin{equation*}
		\begin{aligned}
			|\mathcal{J}_{0, 4}| &= 2|(\Lambda_{\delta_4}^{-\frac{1}{2}}(a_\mathsf{n}e^{\widetilde{\Phi}}), \tilde{e}_5(x) \cdot b_\mathfrak{R}(t))_{L^2(V_{\delta_4})} - A(t)b_\mathfrak{R}(t)| \\
			&\lesssim \|a_\mathsf{n}e^{\widetilde{\Phi}}\|_{L^2(V_{\delta_3})}|b_\mathfrak{R}(t)| + |A(t)|^2 + |b_\mathfrak{R}(t)|^2 \lesssim \mathsf{E}^2.
		\end{aligned}
		\end{equation*}

Patching together \eqref{Estimatebr}, \eqref{Estimatebs}, \eqref{Estimateas} and \eqref{EstimateA}, we conclude that there exists a constant $C>0$ such that
		\begin{equation*}
		\begin{aligned}
			&\frac{d}{dt}\mathcal{J}_{0, 1} + |b_\mathfrak{r}(t)|^2 \le C(\|b_\mathsf{n}e^{\widetilde{\Phi}}\|_{L^2(V_{\delta_2})}^2 + \|a_\mathsf{n}e^{\widetilde{\Phi}}\|_{L^2(V_{\delta_3})}^2 + \|a_\mathsf{n}e^{\widetilde{\Phi}}\|_{L^2(V_{\delta_3})}\mathsf{E});\\
			&\frac{d}{dt}\mathcal{J}_{0, 2} + |b_\mathfrak{R}(t)|^2 \le C(\|b_\mathsf{n}e^{\widetilde{\Phi}}\|_{L^2(V_{\delta_2})}^2 + \|a_\mathsf{n}e^{\widetilde{\Phi}}\|_{L^2(V_{\delta_3})}^2 + \|a_\mathsf{n}e^{\widetilde{\Phi}}\|_{L^2(V_{\delta_3})}\mathsf{E});\\
			&\frac{d}{dt}\mathcal{J}_{0, 3} + |a_\mathfrak{R}(t)|^2 \le C(\|b_\mathsf{n}e^{\widetilde{\Phi}}\|_{L^2(V_{\delta_2})}^2 + \|a_\mathsf{n}e^{\widetilde{\Phi}}\|_{L^2(V_{\delta_3})}^2 + \|a_\mathsf{n}e^{\widetilde{\Phi}}\|_{L^2(V_{\delta_3})}\mathsf{E} + |b_\mathfrak{r}(t)|^2);\\
			&\frac{d}{dt}\mathcal{J}_{0, 4} + |A(t)|^2 \le C(\|e^{\delta_0{\mathsf{H}}}g_1\|_{L_{x, v}^2}^2 + \|b_\mathsf{n}e^{\widetilde{\Phi}}\|_{L^2(V_{\delta_2})}^2 + \|a_\mathsf{n}e^{\widetilde{\Phi}}\|_{L^2(V_{\delta_3})}^2 + \|a_\mathsf{n}e^{\widetilde{\Phi}}\|_{L^2(V_{\delta_3})}\mathsf{E} + |b_\mathfrak{R}(t)|^2).
		\end{aligned}
		\end{equation*}
		Set
		\begin{equation*}
			\mathcal{J}_0 = (C + 1)(\mathcal{J}_{0, 1} + \mathcal{J}_{0, 2}) + \mathcal{J}_{0, 3} + \mathcal{J}_{0, 4}.
		\end{equation*}
		Then $|\mathcal{J}_0| \lesssim \mathsf{E}^2$ and one may check that
			\begin{equation*}
		\begin{aligned}
			&\frac{d}{dt}\mathcal{J}_0 + |b_\mathfrak{r}(t)|^2 + |b_\mathfrak{R}(t)|^2 + |a_\mathfrak{R}(t)|^2 + |A(t)|^2 \\
			\lesssim& \|e^{\delta_0{\mathsf{H}}}g_1\|_{L_{x, v}^2}^2 + \|b_\mathsf{n}e^{\widetilde{\Phi}}\|_{L^2(V_{\delta_2})}^2 + \|a_\mathsf{n}e^{\widetilde{\Phi}}\|_{L^2(V_{\delta_3})}^2 + \|a_\mathsf{n}e^{\widetilde{\Phi}}\|_{L^2(V_{\delta_3})}\mathsf{E}.
		\end{aligned}
     \end{equation*}
	This completes the proof.
	\end{proof}

%%%%%%%%%%%%%%%%%%%%%%%%%%%%%%%%%%%%%%%%%%%%%%%%%%
So far, we have used the evolution equation \eqref{brbsAB+} to obtain the estimates for $b_\mathfrak{r},b_\mathfrak{R},a_\mathfrak{R}$ and $A$. Note that, up to this point, apart from using the conservation law \(\int_{\mathbb{R}^6} g(t, x, v) \, dx \, dv = 0\) to show that \(\mathscr{A}(a_\mathsf{n} e^{\widetilde{\Phi}}) = 0\), we still have four conservation equalities left(see \eqref{054}). We will use them to estimate the remaining four terms, i.e. $c_\mathfrak{r}$ and $B$.

\subsubsection{Estimates for $|c_\mathfrak{r}(t)| + |B(t)|$}
	\begin{lemma}[Conservation-controlled finite-dimensional modes]\label{EstimatecrB}
	It holds that
		\begin{equation*}
			|c_\mathfrak{r}(t)| + |B(t)| \lesssim \|b_\mathsf{n}e^{\widetilde{\Phi}}\|_{L^2(V_{\delta_2})} + \|a_\mathsf{n}e^{\widetilde{\Phi}}\|_{L^2(V_{\delta_3})} + |A(t)|.
		\end{equation*}
	\end{lemma}
	\begin{proof}
		Recall our conserved conditions \eqref{Conserveg} and \eqref{g0form}, we have
		\begin{equation}\label{054}
			\int_{\R^3}( ({\widetilde{\Phi}}(x) + |\mathsf{R}x|^2)a + \mathsf{R}x \cdot b + 3c, x \wedge b + x \wedge \mathsf{R}xa)dx = 0.
		\end{equation}\\
		$\bullet$ Estimates for $B_{13}(t)$ and $B_{23}(t)$. From \eqref{054} and \eqref{abcremainingform},
		\begin{equation*}
		\begin{aligned}
			0 &= C_{\widetilde{\Phi}}\int_{\R^3}(x \wedge b + x \wedge \mathsf{R}xa)_{23}dx = C_{\widetilde{\Phi}}\int_{\R^3}(x_2b_3 - x_3b_2 + x_2(\mathsf{R}x)_3a - x_3(\mathsf{R}x)_2a)dx \\
			&= C_{\widetilde{\Phi}}\int_{\R^3}(x_2b_3 - x_3b_2 + rx_1x_3a)dx = \mathscr{A}(x_2b_3e^{\widetilde{\Phi}}) - \mathscr{A}(x_3b_2e^{\widetilde{\Phi}}) + r\mathscr{A}(x_1x_3ae^{\widetilde{\Phi}}) \\
			&= \mathscr{A}(x_2b_{\mathsf{n}, 3}e^{\widetilde{\Phi}}) + B_{32}(t)\mathscr{A}(x_2^2) - \mathscr{A}(x_3b_{\mathsf{n}, 2}e^{\widetilde{\Phi}}) - B_{23}\mathscr{A}(x_3^2) +  r\mathscr{A}(x_1x_3a_\mathsf{n}e^{\widetilde{\Phi}}) - r\mathscr{A}(x_1x_3Tx \cdot x).
		\end{aligned}
		\end{equation*}	
		Recall \eqref{ComputeTxx}, we have
		\begin{equation}\label{ComputeTxx+}
			T(t)x \cdot x = -r(x_1^2 + x_2^2)B_{12}(t) - rx_2x_3B_{13}(t) + rx_1x_3B_{23}(t).
		\end{equation}
		Then we can derive that
		\begin{equation*}
			B_{23}(t) = \frac{\mathscr{A}(x_2b_{\mathsf{n}, 3}e^{\widetilde{\Phi}}) - \mathscr{A}(x_3b_{\mathsf{n}, 2}e^{\widetilde{\Phi}}) + r\mathscr{A}(x_1x_3a_\mathsf{n}e^{\widetilde{\Phi}})}{\mathscr{A}(x_2^2 + x_3^2 + r^2x_1^2x_3^2)},
		\end{equation*}
		which implies that $|B_{23}(t)| \lesssim \|b_\mathsf{n}e^{\widetilde{\Phi}}\|_{L^2(V_{\delta_2})} + \|a_\mathsf{n}e^{\widetilde{\Phi}}\|_{L^2(V_{\delta_3})}$ since $\mathscr{A}(x_2^2 + x_3^2 + r^2x_1^2x_3^2)>0$. By the similar calculations, the same estimates holds for $B_{13}(t)$ and we omit the details.\\
		$\bullet$ Estimates for $B_{12}(t)$ and $c_\mathfrak{r}(t)$. Conserved conditions \eqref{054} show that
		\begin{equation}
		\begin{aligned}
			0 &= \int_{\R^3}(x \wedge b + x \wedge \mathsf{R}xa)_{12}dx = \int_{\R^3}(x_1b_2 - x_2b_1 - r(x_1^2 + x_2^2)a)dx;\\\label{055}
			0 &= \int_{\R^3}({\widetilde{\Phi}} a + r^2(x_1^2 + x_2^2)a + r(x_2b_1 - x_1b_2) + 3c)dx.
		\end{aligned}
		\end{equation}
		Simplify to obtain that
		\begin{equation}\label{056}
			0 = \int_{\R^3}(x_1b_2 - x_2b_1 - r(x_1^2 + x_2^2)a)dx = \int_{\R^3}({\widetilde{\Phi}} a + 3c)dx.
		\end{equation}
		Now we use the first equality in \eqref{055}, \eqref{abcremainingform} and \eqref{ComputeTxx+} to compute that
		\begin{equation*}
		\begin{aligned}
			0 &= \mathscr{A}(x_1b_2e^{\widetilde{\Phi}}) - \mathscr{A}(x_2b_1e^{\widetilde{\Phi}}) - r\mathscr{A}((x_1^2 + x_2^2)ae^{\widetilde{\Phi}}) \\
			&= \mathscr{A}(x_1b_{\mathsf{n}, 2}e^{\widetilde{\Phi}}) + B_{21}(t)\mathscr{A}(x_1^2) - \mathscr{A}(x_2b_{\mathsf{n}, 1}e^{\widetilde{\Phi}}) - B_{12}(t)\mathscr{A}(x_2^2) - r\mathscr{A}((x_1^2 + x_2^2)a_\mathsf{n}e^{\widetilde{\Phi}}) \\
			&- 2r(\mathscr{A}((x_1^2 + x_2^2){\widetilde{\Phi}}) - \mathscr{A}(x_1^2 + x_2^2)\mathscr{A}({\widetilde{\Phi}}))c_\mathfrak{r}(t) + r(\mathscr{A}((x_1^2 + x_2^2){\widetilde{\Phi}}) - \mathscr{A}(x_1^2 + x_2^2)\mathscr{A}({\widetilde{\Phi}}))A(t) \\
			& -r^2B_{12}(t)\mathscr{A}((x_1^2 + x_2^2)^2) + r^2B_{12}(t)|\mathscr{A}(x_1^2 + x_2^2)|^2.
		\end{aligned}
		\end{equation*}
It leads to
		\begin{multline}\label{crBeq1}
			[r^2(\mathscr{A}((x_1^2 + x_2^2)^2) - |\mathscr{A}(x_1^2 + x_2^2)|^2) + \mathscr{A}(x_1^2 + x_2^2)]B_{12}(t) + 2r(\mathscr{A}((x_1^2 + x_2^2){\widetilde{\Phi}}) - \mathscr{A}(x_1^2 + x_2^2)\mathscr{A}({\widetilde{\Phi}}))c_\mathfrak{r}(t) \\= \mathscr{A}(x_1b_{\mathsf{n}, 2}e^{\widetilde{\Phi}}) - \mathscr{A}(x_2b_{\mathsf{n}, 1}e^{\widetilde{\Phi}}) - r\mathscr{A}((x_1^2 + x_2^2)a_ne^{\widetilde{\Phi}}) + r(\mathscr{A}((x_1^2 + x_2^2){\widetilde{\Phi}}) - \mathscr{A}(x_1^2 + x_2^2)\mathscr{A}({\widetilde{\Phi}}))A(t).
		\end{multline}
		Next we use \eqref{056}, \eqref{abcremainingform} and \eqref{ComputeTxx+} to compute that
		\begin{equation*}
		\begin{aligned}
			0 &= \mathscr{A}({\widetilde{\Phi}} ae^{\widetilde{\Phi}}) + 3\mathscr{A}(ce^{\widetilde{\Phi}}) = \mathscr{A}(a_\mathsf{n}e^{\widetilde{\Phi}}) + (2\mathscr{A}({\widetilde{\Phi}}^2) - 2|\mathscr{A}({\widetilde{\Phi}})|^2 + 3)c_\mathfrak{r}(t) \\
			&- (\mathscr{A}(|x|^2{\widetilde{\Phi}}) - \mathscr{A}(|x|^2)\mathscr{A}({\widetilde{\Phi}}))A(t) + rB_{12}(t)\mathscr{A}((x_1^2 + x_2^2){\widetilde{\Phi}}) - rB_{12}(t)\mathscr{A}(x_1^2 + x_2^2)\mathscr{A}({\widetilde{\Phi}}).
		\end{aligned}
		\end{equation*}
	It leads to
		\begin{multline}\label{crBeq2}
			r(\mathscr{A}((x_1^2 + x_2^2){\widetilde{\Phi}}) - \mathscr{A}(x_1^2 + x_2^2)\mathscr{A}({\widetilde{\Phi}}))B_{12}(t) + (2\mathscr{A}({\widetilde{\Phi}}^2) - 2|\mathscr{A}({\widetilde{\Phi}})|^2 + 3)c_\mathfrak{r}(t) \\= -\mathscr{A}(a_\mathsf{n}e^{\widetilde{\Phi}}) + (\mathscr{A}(|x|^2{\widetilde{\Phi}}) - \mathscr{A}(|x|^2)\mathscr{A}({\widetilde{\Phi}}))A(t).
		\end{multline}
		Put together \eqref{crBeq1} and \eqref{crBeq2}. Observing that the right hand side can be bounded by $\|b_\mathsf{n}e^{\widetilde{\Phi}}\|_{L^2(V_{\delta_2})} + \|a_\mathsf{n}e^{\widetilde{\Phi}}\|_{L^2(V_{\delta_3})} + |A(t)|$, thus we need to prove that the determinant of their coefficient matrix is non-zero, i.e.,
		\begin{multline}\label{crBeq3}
			[r^2(\mathscr{A}((x_1^2 + x_2^2)^2) - |\mathscr{A}(x_1^2 + x_2^2)|^2) + \mathscr{A}(x_1^2 + x_2^2)][2\mathscr{A}({\widetilde{\Phi}}^2) - 2|\mathscr{A}({\widetilde{\Phi}})|^2 + 3] \\> 2r^2[\mathscr{A}((x_1^2 + x_2^2){\widetilde{\Phi}}) - \mathscr{A}(x_1^2 + x_2^2)\mathscr{A}({\widetilde{\Phi}})]^2.
		\end{multline}
		For every $f=f(x)$, the Cauchy--Schwarz inequality gives
		\begin{equation*}
			\mathscr{A}(f^2) - |\mathscr{A}(f)|^2 = \frac{\int_{\R^3} f^2e^{-{\widetilde{\Phi}}}dx}{\int_{\R^3} e^{-{\widetilde{\Phi}}}dx} - \Big(\frac{\int_{\R^3} fe^{-{\widetilde{\Phi}}}dx}{\int_{\R^3} e^{-{\widetilde{\Phi}}}dx}\Big)^2 = \frac{\int_{\R^3} f^2e^{-{\widetilde{\Phi}}}dx\int_{\R^3} e^{-{\widetilde{\Phi}}}dx - (\int_{\R^3} fe^{-{\widetilde{\Phi}}}dx)^2}{(\int_{\R^3} e^{-{\widetilde{\Phi}}}dx)^2} \ge 0.
		\end{equation*}
		Thus \eqref{crBeq3} holds for $r = 0$. If $r> 0$, it is sufficient to prove
		\begin{equation}\label{101}
			[\mathscr{A}((x_1^2 + x_2^2)^2) - |\mathscr{A}(x_1^2 + x_2^2)|^2][\mathscr{A}({\widetilde{\Phi}}^2) - |\mathscr{A}({\widetilde{\Phi}})|^2] \ge [\mathscr{A}((x_1^2 + x_2^2){\widetilde{\Phi}}) - \mathscr{A}(x_1^2 + x_2^2)\mathscr{A}({\widetilde{\Phi}})]^2,
		\end{equation}
		which is equivalent to
		\begin{multline*}
			\Big[\int_{\R^3}(x_1^2 + x_2^2)^2e^{-{\widetilde{\Phi}}}dx\int_{\R^3} e^{-{\widetilde{\Phi}}}dx - \Big(\int_{\R^3}(x_1^2 + x_2^2)e^{-{\widetilde{\Phi}}}dx\Big)^2\Big]\Big[\int_{\R^3}{\widetilde{\Phi}}^2e^{-{\widetilde{\Phi}}}dx\int_{\R^3} e^{-{\widetilde{\Phi}}}dx - \Big(\int_{\R^3}{\widetilde{\Phi}} e^{-{\widetilde{\Phi}}}dx\Big)^2\Big] \\\ge \Big[\int_{\R^3}(x_1^2 + x_2^2){\widetilde{\Phi}} e^{-{\widetilde{\Phi}}}dx\int_{\R^3} e^{-{\widetilde{\Phi}}}dx - \int_{\R^3}(x_1^2 + x_2^2)e^{-{\widetilde{\Phi}}}dx\int_{\R^3}{\widetilde{\Phi}} e^{-{\widetilde{\Phi}}}dx\Big]^2.
		\end{multline*}
		This can be directly obtained from Lemma \ref{LagrangeIneq}. 
		
		 We  complete the proof of this lemma by combining all above estimates.
	\end{proof}
%%%%%%%%%%%%%%%%%%%%%%%%%%%%%%%%%%%%%%%%%%%%%%%%%%
We now prove Lemma \ref{EstimatecrbrbsasAB}.
\begin{proof}[Proof of Lemma \ref{EstimatecrbrbsasAB}]
The conclusion follows by combining Lemmas \ref{EstimatebrbsasA} and \ref{EstimatecrB}.
\end{proof}

	\subsection{Closing the macroscopic coercivity estimate}
We now combine the preceding lemmas to obtain the macroscopic estimate.
	\begin{proposition}[Macroscopic coercivity]\label{060}
	There exists a function $\mathcal{J} = \mathcal{J}(t)$ satisfying $|\mathcal{J}(t)| \lesssim \mathsf{E}^2$ and
	\begin{equation}\label{059}
		\frac{d}{dt}\mathcal{J} + \mathsf{E}^2 \lesssim \|e^{\delta_0{\mathsf{H}}}g_1\|_{L_{x, v}^2}^2,
	\end{equation}
	where $g_1$ is the microscopic part of $g$.
\end{proposition}
\begin{proof}
Lemmas \ref{Estimateanbncn} and \ref{Estimateptanbncn} give
	\begin{equation}\label{057}
	\begin{aligned}
		&\frac{d}{dt}\mathcal{J}_c + \frac{1}{2}\|\Lambda_{\delta_1}^{-\frac{1}{2}}\nabla_x(c_\mathsf{n}e^{\widetilde{\Phi}})\|_{L^2(V_{\delta_1})}^2
		\lesssim \|e^{\delta_0{\mathsf{H}}}g_1\|_{L_{x, v}^2}^2 + \|e^{\delta_0{\mathsf{H}}}g_1\|_{L_{x, v}^2}\mathsf{E};\\
		&\frac{d}{dt}\mathcal{J}_b + \frac{1}{2}\|\Lambda_{\delta_2}^{-\frac{1}{2}}\nabla_x^{sym}(b_\mathsf{n}e^{\widetilde{\Phi}})\|_{L^2(V_{\delta_2})}^2
		\lesssim \|e^{\delta_0{\mathsf{H}}}g_1\|_{L_{x, v}^2}^2 + \|c_\mathsf{n}e^{\widetilde{\Phi}}\|_{L^2(V_{\delta_1})}^2 \\
		&\qquad\qquad\qquad\qquad\qquad\qquad+ (\|e^{\delta_0{\mathsf{H}}}g_1\|_{L_{x, v}^2} + \|c_\mathsf{n}e^{\widetilde{\Phi}}\|_{L^2(V_{\delta_1})})\mathsf{E};\\
		&\frac{d}{dt}\mathcal{J}_a + \frac{1}{2}\|\Lambda_{\delta_3}^{-\frac{1}{2}}\nabla_x(a_\mathsf{n}e^{\widetilde{\Phi}})\|_{L^2(V_{\delta_3})}^2
		\lesssim \|e^{\delta_0{\mathsf{H}}}g_1\|_{L_{x, v}^2}^2 + \|c_\mathsf{n}e^{\widetilde{\Phi}}\|_{L^2(V_{\delta_1})}^2 + \|b_\mathsf{n}e^{\widetilde{\Phi}}\|_{L^2(V_{\delta_2})}^2\\
		&\qquad\qquad\qquad\qquad\qquad\qquad+(\|e^{\delta_0{\mathsf{H}}}g_1\|_{L_{x, v}^2} + \|b_\mathsf{n}e^{\widetilde{\Phi}}\|_{L^2(V_{\delta_2})})\mathsf{E}.
	\end{aligned}
	\end{equation}
	Moreover, we can also derive that
	\begin{equation}\label{058}
	\begin{aligned}
		|\mathcal{J}_c| &= |(k_d, \Lambda_{\delta_1}^{-1}\nabla_x(c_\mathsf{n}e^{\widetilde{\Phi}}))_{L^2(V_{\delta_1})}| \lesssim \|e^{\delta_0{\mathsf{H}}}g_1\|_{L_{x, v}^2}\|c_\mathsf{n}e^{\widetilde{\Phi}}\|_{L^2(V_{\delta_1})} \lesssim \mathsf{E}^2;\\
		|\mathcal{J}_b| &= |(K - \mathbb{I}_{3 \times 3}c_\mathsf{n}e^{\widetilde{\Phi}}, \Lambda_{\delta_2}^{-1}\nabla_x^{sym}(b_\mathsf{n}e^{\widetilde{\Phi}}))_{L^2(V_{\delta_2})}| \\
		&\lesssim (\|e^{\delta_0{\mathsf{H}}}g_1\|_{L_{x, v}^2} + \|c_\mathsf{n}e^{\widetilde{\Phi}}\|_{L^2(V_{\delta_1})})\|b_\mathsf{n}e^{\widetilde{\Phi}}\|_{L^2(V_{\delta_2})} \lesssim \mathsf{E}^2;\\
		|\mathcal{J}_a| &= |(\mathscr{A}(K^*)x - 2k_d - b_\mathsf{n}e^{\widetilde{\Phi}}, \Lambda_{\delta_3}^{-1}\nabla_x(a_\mathsf{n}e^{\widetilde{\Phi}}))_{L^2(V_{\delta_3})}| \\
		&\lesssim (\|e^{\delta_0{\mathsf{H}}}g_1\|_{L_{x, v}^2} + \|b_\mathsf{n}e^{\widetilde{\Phi}}\|_{L^2(V_{\delta_1})})\|a_\mathsf{n}e^{\widetilde{\Phi}}\|_{L^2(V_{\delta_2})} \lesssim \mathsf{E}^2.
	\end{aligned}
	\end{equation}
	We have already verified the hypotheses of the weighted Poincar\'e and Poincar\'e--Korn inequalities for $c_\mathsf n e^{\widetilde\Phi}$, $b_\mathsf n e^{\widetilde\Phi}$, and $a_\mathsf n e^{\widetilde\Phi}$. Hence \eqref{057} and Lemmas \ref{VPI}--\ref{VPKI} imply
	\begin{equation}\label{090}
		\begin{aligned}
		\frac{d}{dt}\mathcal{J}_c + \|e^{\delta_1{\widetilde{\Phi}}}c_\mathsf{n}\|_{L_x^2}^2 &\le C\|e^{\delta_0{\mathsf{H}}}g_1\|_{L_{x, v}^2}^2 + C\|e^{\delta_0{\mathsf{H}}}g_1\|_{L_{x, v}^2}\mathsf{E};\\
		\frac{d}{dt}\mathcal{J}_b + \|e^{\delta_2{\widetilde{\Phi}}}b_\mathsf{n}\|_{L_x^2}^2 &\le C\|e^{\delta_0{\mathsf{H}}}g_1\|_{L_{x, v}^2}^2 + C\|e^{\delta_1{\widetilde{\Phi}}}c_\mathsf{n}\|_{L_x^2}^2 + (\|e^{\delta_0{\mathsf{H}}}g_1\|_{L_{x, v}^2} + \|e^{\delta_1{\widetilde{\Phi}}}c_\mathsf{n}\|_{L_x^2})\mathsf{E};\\
		\frac{d}{dt}\mathcal{J}_a + \|e^{\delta_3{\widetilde{\Phi}}}a_\mathsf{n}\|_{L_x^2}^2 &\le C\|e^{\delta_0{\mathsf{H}}}g_1\|_{L_{x, v}^2}^2 + C\|e^{\delta_1{\widetilde{\Phi}}}c_\mathsf{n}\|_{L_x^2}^2 + C\|e^{\delta_2{\widetilde{\Phi}}}b_\mathsf{n}\|_{L_x^2}^2 + C(\|e^{\delta_0{\mathsf{H}}}g_1\|_{L_{x, v}^2} + \|e^{\delta_2{\widetilde{\Phi}}}b_\mathsf{n}\|_{L_x^2})\mathsf{E},
		\end{aligned}
	\end{equation}
where $C>0$ is a constant.
	
	On the other hand, Lemma \ref{EstimatecrbrbsasAB} gives
	\begin{equation}\label{MacroEnergy4}
		\frac{d}{dt}\mathcal{J}_0 + \mathcal{R}^2 \le C\|e^{\delta_0{\mathsf{H}}}g_1\|_{L_{x, v}^2}^2 + C\|e^{\delta_2{\widetilde{\Phi}}}b_\mathsf{n}\|_{L_x^2}^2 + C\|e^{\delta_3{\widetilde{\Phi}}}a_\mathsf{n}\|_{L_x^2}^2 + C\|e^{\delta_3{\widetilde{\Phi}}}a_\mathsf{n}\|_{L_x^2}\mathsf{E},
	\end{equation}
	where $|\mathcal{J}_0(t)| \lesssim \mathsf{E}^2$ and 
	\begin{equation}\label{DefiR}
		\mathcal{R}^2 := |c_\mathfrak{r}(t)|^2 + |b_\mathfrak{r}(t)|^2 + |b_\mathfrak{R}(t)|^2 + |a_\mathfrak{R}(t)|^2 + |A(t)|^2 + |B(t)|^2.
	\end{equation}
Choose $\eps>0$ and set $\eps_1=\eps^{2/3}$, $\eps_2=\eps^{6/7}$, and $\eps_3=\eps$. For brevity, write
\[
\mathcal G_0=\|e^{\delta_0\mathsf H}g_1\|_{L_{x,v}^2}^2,
\quad \mathcal C_n=\|e^{\delta_1\widetilde\Phi}c_\mathsf n\|_{L_x^2}^2,
\quad \mathcal B_n=\|e^{\delta_2\widetilde\Phi}b_\mathsf n\|_{L_x^2}^2,
\quad \mathcal A_n=\|e^{\delta_3\widetilde\Phi}a_\mathsf n\|_{L_x^2}^2.
\]
The Cauchy--Schwarz inequality and \eqref{090} give
\begin{equation*}
	\begin{aligned}
	&\frac{d}{dt}\mathcal{J}_c+\mathcal C_n
	\le \left(C+\frac C{\eps\eps_3}\right)\mathcal G_0
	+C\eps\eps_3\mathsf E^2;\\
	&\frac{d}{dt}(\eps_1\mathcal{J}_b)+\eps_1\mathcal B_n
	\le \left(C\eps_1+\frac{C\eps_1^2}{\eps\eps_3}\right)\mathcal G_0
	+\left(C\eps_1+\frac14\right)\mathcal C_n
	+(C\eps\eps_3+C\eps_1^2)\mathsf E^2;\\
	&\frac{d}{dt}(\eps_2\mathcal{J}_a)+\eps_2\mathcal A_n
	\le \left(C\eps_2+\frac{C\eps_2^2}{\eps\eps_3}\right)\mathcal G_0
	+C\eps_2\mathcal C_n+\left(C\eps_2+\frac{\eps_1}{4}\right)\mathcal B_n\\
	&\hspace{57mm}
	+\left(C\eps\eps_3+\frac{C^2\eps_2^2}{\eps_1}\right)\mathsf E^2;\\
	&\frac{d}{dt}(\eps_3\mathcal{J}_0)+\eps_3\mathcal R^2
	\le C\eps_3\mathcal G_0+C\eps_3\mathcal B_n
	+\left(C\eps_3+\frac{\eps_2}{4}\right)\mathcal A_n
	+\frac{C^2\eps_3^2}{\eps_2}\mathsf E^2.
	\end{aligned}
\end{equation*}
Adding these inequalities yields
\begin{equation*}
	\begin{aligned}
		&\frac{d}{dt}(\mathcal{J}_c+\eps_1\mathcal{J}_b
		+\eps_2\mathcal{J}_a+\eps_3\mathcal{J}_0)
		+\left(\frac34-C(\eps_1+\eps_2)\right)\mathcal C_n\\
		&\quad+\left(\frac34\eps_1-C\eps_2-C\eps_3\right)\mathcal B_n
		+\frac{3\eps_2}{4}\mathcal A_n+\eps_3\mathcal R^2\\
		&\le \left(C\eps\eps_3+C\eps_1^2
		+\frac{C^2\eps_2^2}{\eps_1}+\frac{C^2\eps_3^2}{\eps_2}\right)
		(\mathsf E^2-\mathcal G_0)
		+C_{\eps,\eps_1,\eps_2,\eps_3}\mathcal G_0.
	\end{aligned}
\end{equation*}
For sufficiently small $\eps>0$,
\beno
\f34-C(\eps_1+\eps_2)>\f14,~~\f34\eps_1-C\eps_2-C\eps_3>\f14\eps_1.
\eeno
Moreover, it also holds that
\beno
\Big(C\eps\eps_3+C\eps_1^2+\f{C^2\eps_2^2}{\eps_1}+\f{C^2\eps_3^2}{\eps_2}\Big)\leq C\eps_3^{\f{22}{21}}.
\eeno
Then by the definition of total energy $\mathsf{E}$ in \eqref{DefiE}, we can choose $\eps$ suitably small such that 
\beno
\f d {dt}\mathcal{J}+ (\mathsf{E}^2-\|e^{\delta_0{\mathsf{H}}}g_1\|_{L_{x, v}^2}^2)\ls \|e^{\delta_0{\mathsf{H}}}g_1\|_{L_{x, v}^2}^2,
\eeno
where $\mathcal{J}\sim\mathcal{J}_c+\eps_1\mathcal{J}_b+\eps_2\mathcal{J}_a+\eps_3\mathcal{J}_0$. Together with the upper bound of $\mathcal{J}_0$ in Lemma \ref{EstimatebrbsasA} and \eqref{058}, we conclude the desired result.
\end{proof}
%%%%%%%%%%%%%%%%%%%%%%%%%%%%%%%%%%%%%%%%%%%%%%%%%%

	\section{Algebraic relaxation modulo stationary modes}
	We now combine the spatially degenerate microscopic estimate of Section 4 with the transformed macroscopic closure of Section 6. The resulting weighted energy estimate first gives algebraic decay for zero-moment data in the normalized variables. We then return to the original variables and to arbitrary initial moments by means of the stationary projection.

\begin{proposition}[Decay for zero conserved moments]\label{Conv}
Let $M$ be a positive equilibrium of the form \eqref{071} or \eqref{072}, and suppose that
\[
\int_{\mathbb R^6}\chi f_0\,dx\,dv=0.
\]
Let $1/2<\lambda_2<\lambda_1<1$. If $M^{-\lambda_1}f_0\in L^2(\mathbb R^6)$, then the semigroup solution of \eqref{Lineareqf} satisfies, for every $\delta>0$,
\[
\|M^{-\lambda_2}f(t)\|_{L^2}
\le C_{M,B,\lambda_1,\lambda_2,\delta}
\langle t\rangle^{-\frac{(\lambda_1-\lambda_2)(p+2)}{2[1+\delta(p+2)]}}
\|M^{-\lambda_1}f_0\|_{L^2}.
\]
\end{proposition}

\begin{theorem}[Normalized zero-moment decay]\label{082}
Let $\mathcal{M}=C_Me^{-\widetilde{\Phi}(x)}\mu(v)$, where $C_M$ and $\widetilde{\Phi}$ are defined in \eqref{083} and \eqref{013}. Suppose that $g_0$ satisfies \eqref{Conserveg} at $t=0$. Let $1/2<\lambda_2<\lambda_1<1$ and assume that $\mathcal{M}^{-\lambda_1}g_0\in L^2(\mathbb R^6)$. Then the semigroup solution of \eqref{Lineareqg} satisfies, for every $\delta>0$,
\[
\|\mathcal{M}^{-\lambda_2}g(t)\|_{L^2}
\le C\langle t\rangle^{-\frac{(\lambda_1-\lambda_2)(p+2)}{2[1+\delta(p+2)]}}
\|\mathcal{M}^{-\lambda_1}g_0\|_{L^2}.
\]
\end{theorem}

\subsection{Proof of the normalized zero-moment theorem}
	
	\begin{proof}[Proof of Theorem \ref{082}]
\medskip\noindent\textit{Step 1: Microscopic energy estimate.}
 We begin with the normalized equation \eqref{Lineareqg} for $g$:
	\begin{equation}\label{Lineareqg+}
			\partial_tg + \mathsf{T}g = C_Me^{-{\widetilde{\Phi}}}\mathsf{L}g, \qquad \mathsf{T} = v \cdot \nabla_x + \mathsf{R}x \cdot \nabla_x - \mathsf{R}v \cdot \nabla_v - \nabla_x{\widetilde{\Phi}} \cdot \nabla_v.
	\end{equation}
	Recall ${\mathsf{H}}(x,v)={\widetilde{\Phi}}(x)+|v|^2/2$ and note that $\mathsf{T}\mathsf{H}=0$. Multiplying the equation by $e^{\mathsf{H}}g$ and integrating gives
	\begin{equation*}
		\frac{1}{2}\frac{d}{dt}\|e^{\frac{1}{2}{\mathsf{H}}}g\|_{L_{x, v}^2}^2 = C_M(\mathsf{L}g, e^{\frac{1}{2}|v|^2}g)_{L_{x, v}^2}.
	\end{equation*}
	By Lemma \ref{UpperBoundL1}, there exists $c>0$ such that
	\begin{equation}\label{MicroEnergy}
		\frac{d}{dt}\|e^{\frac{1}{2}{\mathsf{H}}}g\|_{L_{x, v}^2}^2 + c\normm{e^{\frac{1}{4}|v|^2}g_1}^2_{L^2_x} \le 0,
	\end{equation}
	where $\normm{\cdot}_{L^2_x}:=\Big(\int_{\R^3}\normm{\cdot}^2dx\Big)^{\f12}$.
%%%%%%%%%%%%%%%%%%%%%%%%%%%%%%%%%%%%%%%%%%%%%%%%%%

\medskip\noindent\textit{Step 2: Weighted energy and transfer estimate.}
	For $\lambda\in(1/2,1)$, multiply both sides of \eqref{Lineareqg+} by $e^{2\lambda\mathsf H}g$ and integrate to obtain
	\begin{equation*}
		\frac{1}{2}\frac{d}{dt}\|e^{\lambda{\mathsf{H}}}g\|_{L_{x, v}^2}^2 = C_M\int_{\R^3} e^{(2\lambda - 1){\widetilde{\Phi}}}(\mathsf{L}g, e^{\lambda|v|^2}g)_{L_v^2}dx.
	\end{equation*}
	By Lemma \ref{UpperBoundL2}, there exist $c,C>0$ such that
	\begin{equation}\label{ProofThm31}
		\frac{d}{dt}\|e^{\lambda{\mathsf{H}}}g\|_{L_{x, v}^2}^2 + c\|e^{(\lam-\f12)\widetilde{\Phi}}e^{\lambda |v|^2}g\|_{L^2_x,B} \le C \|e^{(\lambda - \frac{1}{2}){\widetilde{\Phi}}}g\|_{L_{x, v}^2}^2.
	\end{equation}
	Next, fix $0<\delta<1/3-1/(p+2)$, multiply \eqref{Lineareqg+} by $(x\cdot v)e^{2(\lambda-1/(p+2)-\delta)\mathsf H}g$, and integrate. Since
	\begin{equation*}
		(\mathsf{T}g, (x \cdot v)e^{2(\lambda - \frac{1}{p + 2} - \delta){\mathsf{H}}}g)_{L_{x, v}^2}
		= -(g, \mathsf{T}(x \cdot v) \times e^{2(\lambda - \frac{1}{p + 2} - \delta){\mathsf{H}}}g)_{L_{x, v}^2}
		- (g, (x \cdot v)e^{2(\lambda - \frac{1}{p + 2} - \delta){\mathsf{H}}}\mathsf{T}g)_{L_{x, v}^2},
	\end{equation*}
we obtain
	\begin{equation*}
		(\mathsf{T}g, (x \cdot v)e^{2(\lambda - \frac{1}{p + 2} - \delta){\mathsf{H}}}g)_{L_{x, v}^2} = -\frac{1}{2}(g, \mathsf{T}(x \cdot v) \times e^{2(\lambda - \frac{1}{p + 2} - \delta){\mathsf{H}}}g)_{L_{x, v}^2}.
	\end{equation*}
Consequently,
	\begin{multline*}
		\frac{d}{dt}\int_{\R^6}(x \cdot v)e^{2(\lambda - \frac{1}{p + 2} - \delta){\mathsf{H}}}g^2dxdv - \int_{\R^6}\mathsf{T}(x \cdot v)e^{2(\lambda - \frac{1}{p + 2} - \delta){\mathsf{H}}}g^2dxdv \\= 2C_M\int_{\R^6} e^{(2\lambda - \frac{2}{p + 2} - 2\delta - 1){\widetilde{\Phi}}}(\mathsf{L}g, x\cdot ve^{(\lambda - \frac{1}{p + 2} - \delta)|v|^2}g)_{L_v^2}dx.
	\end{multline*}
	A second application of Lemma \ref{UpperBoundL2} gives
	\begin{equation}\label{ProofThm32}
		\frac{d}{dt}\int_{\R^6}(x \cdot v)e^{2(\lambda - \frac{1}{p + 2} - \delta){\mathsf{H}}}g^2dxdv - \int_{\R^6}\mathsf{T}(x \cdot v)e^{2(\lambda - \frac{1}{p + 2} - \delta){\mathsf{H}}}g^2dxdv \le C \normm{e^{(\lambda-\f1{p+2}-\f12)\widetilde{\Phi}}e^{(\f \lambda 2-\f1{2(p+2)})|v|^2}g}^2_{L^2_x}.
	\end{equation}
	Observe that
	\beno
	&&\|\<v\>^{-3}e^{(\lam-\f12)\widetilde{\Phi}}e^{\lambda |v|^2}g\|_{L^2_{x,v}}\ls \normm{e^{(\lam-\f12)\widetilde{\Phi}}e^{\lambda |v|^2}g}_{L^2_x},\\ &&\normm{e^{(\lambda-\f1{p+2}-\f12)\widetilde{\Phi}}e^{(\f \lambda 2-\f1{2(p+2)})|v|^2}g}_{L^2_x}\ls \normm{e^{(\lam-\f12)\widetilde{\Phi}}e^{\lambda |v|^2}g}_{L^2_x}.
	\eeno 
	Combining \eqref{ProofThm31} and \eqref{ProofThm32} and applying Corollary \ref{WeightTransferLemma}, we find that for every $0<\delta<1/3-1/(p+2)$ there exists a sufficiently large $N$ such that
	\begin{equation}\label{ProofThm33}
		\frac{d}{dt}\mathcal{F}_N + \|e^{(\lambda - \frac{1}{p + 2} - \delta){\mathsf{H}}}g\|_{L_{x, v}^2}^2 \le C \|e^{(\lambda - \frac{1}{2}){\widetilde{\Phi}}}g\|_{L_{x, v}^2}^2 +C \|g\|_{L_{x, v}^2}^2,
	\end{equation}
	where
	\begin{equation*}
		\mathcal{F}_N = N\|e^{\lambda{\mathsf{H}}}g\|_{L_{x, v}^2}^2 + \int_{\R^6}(x \cdot v)e^{2(\lambda - \frac{1}{p + 2} - \delta){\mathsf{H}}}g^2dxdv.
	\end{equation*}
	Moreover,
	\begin{equation*}
		\Big|\int(x \cdot v)e^{2(\lambda - \frac{1}{p + 2} - \delta){\mathsf{H}}}g^2dxdv\Big| \lesssim \|e^{\lambda{\mathsf{H}}}g\|_{L_{x, v}^2}^2.
	\end{equation*}
	Choose $N$ sufficiently large that $\|e^{\lambda\mathsf H}g\|_{L_{x,v}^2}^2\le\mathcal F_N\le2N\|e^{\lambda\mathsf H}g\|_{L_{x,v}^2}^2$. Since $\lambda-1/(p+2)-\de>\lambda-1/2$, H\"older's inequality applied to the right-hand side of \eqref{ProofThm33} gives
	\begin{equation*}
		\|e^{(\lambda - \frac{1}{2}){\widetilde{\Phi}}}g\|_{L_{x, v}^2} \le \|e^{(\lambda - \frac{1}{p + 2} - \delta){\widetilde{\Phi}}}g\|_{L_{x, v}^2}^\theta\|g\|_{L_{x, v}^2}^{1 - \theta} \lesssim \|e^{(\lambda - \frac{1}{p + 2} - \delta){\mathsf{H}}}g\|_{L_{x, v}^2}^\theta\|g\|_{L_{x, v}^2}^{1 - \theta}.
	\end{equation*}
	where $\theta := \frac{\lambda - \frac{1}{2}}{\lambda - \frac{1}{p + 2} - \delta}$.
	Moreover, by Young's inequality, for any $\varepsilon > 0$, we have
	\begin{equation*}
		\|e^{(\lambda - \frac{1}{2}){\widetilde{\Phi}}}g\|_{L_{x, v}^2}^2 \le \varepsilon\|e^{(\lambda - \frac{1}{p + 2} - \delta){\mathsf{H}}}g\|_{L_{x, v}^2}^2 + C_\varepsilon\|g\|_{L_{x, v}^2}^2.
	\end{equation*}
	Choosing $\eps>0$ sufficiently small in \eqref{ProofThm33}, we obtain
	\begin{equation}\label{ProofThm34}
		\frac{d}{dt}\mathcal{F}_N + \frac{1}{2}\|e^{(\lambda - \frac{1}{p + 2} - \delta){\mathsf{H}}}g\|_{L_{x, v}^2}^2 \lesssim \|g\|_{L_{x, v}^2}^2, \: ~~\text{with}~~ \: \|e^{\lambda{\mathsf{H}}}g\|_{L_{x, v}^2}^2 \le \mathcal{F}_N \le 2N\|e^{\lambda{\mathsf{H}}}g\|_{L_{x, v}^2}^2.
	\end{equation}

%%%%%%%%%%%%%%%%%%%%%%%%%%%%%%%%%%%%%%%%%%%%%%%%%%

\medskip\noindent\textit{Step 3: Macro--micro energy closure.}
	By Proposition \ref{060}, there exists a function $\mathcal J=\mathcal J(t)$ satisfying $|\mathcal J(t)|\lesssim\mathsf E^2$ and
	\begin{equation*}
		\frac{d}{dt}\mathcal{J} + \mathsf{E}^2 \lesssim \|e^{\delta_0{\mathsf{H}}}g_1\|_{L_{x, v}^2}^2.
	\end{equation*}
	The definition of $\mathsf E$ in \eqref{DefiE} implies
	\begin{equation*}
		\|g\|_{L_{x, v}^2}^2 \le \|g_0\|_{L_{x, v}^2}^2 + \|g_1\|_{L_{x, v}^2}^2 \lesssim \|a\|_{L_x^2}^2 + \|b\|_{L_x^2}^2 + \|c\|_{L_x^2}^2 + \|g_1\|_{L_{x, v}^2}^2 \lesssim \mathsf{E}^2.
	\end{equation*}
	Thus there exists a constant $C>0$ such that
	\begin{equation*}
	\begin{gathered}
		\frac{d}{dt}\mathcal{F}_N + \frac{1}{2}\|e^{(\lambda - \frac{1}{p + 2} - \delta){\mathsf{H}}}g\|_{L_{x, v}^2}^2 \le C\mathsf{E}^2;\quad
		\frac{d}{dt}\mathcal{J} + \mathsf{E}^2 \le C\|e^{\delta_0{\mathsf{H}}}g_1\|_{L_{x, v}^2}^2.
	\end{gathered}
	\end{equation*}
Therefore,
	\begin{equation}\label{ProofThm35}
		\frac{d}{dt}(\mathcal{F}_N + C_1\mathcal{J}) + \frac{1}{2}\|e^{(\lambda - \frac{1}{p + 2} - \delta){\mathsf{H}}}g\|_{L_{x, v}^2}^2 \le C_1C\|e^{\delta_0{\mathsf{H}}}g_1\|_{L_{x, v}^2}^2.
	\end{equation}
	By the Cauchy--Schwarz inequality,
	\begin{equation*}
		\|e^{\delta_0{\mathsf{H}}}g_1\|_{L_{x, v}^2}^2 \le \|e^{2\delta_0{\mathsf{H}}}g_1\|_{L_{x, v}^2}\|g_1\|_{L_{x, v}^2}.
	\end{equation*}
	By the choices of $\delta_0$, $\delta$, and $\lambda$, one has $2\delta_0<1/6<\lambda-1/(p+2)-\delta$. Thus
	\begin{equation*}
		\|e^{2\delta_0{\mathsf{H}}}g_1\|_{L_{x, v}^2} \le \|e^{2\delta_0{\mathsf{H}}}g_0\|_{L_{x, v}^2} + \|e^{2\delta_0{\mathsf{H}}}g\|_{L_{x, v}^2} \lesssim \|e^{(\lambda - \frac{1}{p + 2} - \delta){\mathsf{H}}}g\|_{L_{x, v}^2}.
	\end{equation*}
	Now we have
	\begin{equation*}
		\|e^{\delta_0{\mathsf{H}}}g_1\|_{L_{x, v}^2}^2 \le \varepsilon\|e^{(\lambda - \frac{1}{p + 2} - \delta){\mathsf{H}}}g\|_{L_{x, v}^2}^2 + C_\varepsilon\|g_1\|_{L_{x, v}^2}^2,
	\end{equation*}
	and \eqref{ProofThm35} becomes
	\begin{equation*}
		\frac{d}{dt}(\mathcal{F}_N + C_1\mathcal{J}) + \frac{1}{4}\|e^{(\lambda - \frac{1}{p + 2} - \delta){\mathsf{H}}}g\|_{L_{x, v}^2}^2 \ls \|g_1\|_{L_{x, v}^2}^2 \le C \|e^{\frac{1}{4}|v|^2}g_1\|_{L_{x, v}^2}^2.
	\end{equation*}
	Together with our microscopic estimates \eqref{MicroEnergy}, there exists sufficiently large $C_2$ such that
	\begin{equation*}
		\frac{d}{dt}(\mathcal{F}_N + C_1\mathcal{J} + C_2\|e^{\frac{1}{2}{\mathsf{H}}}g\|_{L_{x, v}^2}^2) + \frac{1}{4}\|e^{(\lambda - \frac{1}{p+2} - \delta){\mathsf{H}}}g\|_{L_{x, v}^2}^2 \le 0.
	\end{equation*}
	The definition of $\mathsf E$ in \eqref{DefiE} gives
\begin{equation*}
	\begin{aligned}
		\mathsf{E}^2 &\lesssim \|e^{\delta_0{\mathsf{H}}}(\mathsf{I} - \mathsf{P})g\|_{L_{x, v}^2}^2 + \|e^{\delta_0{\widetilde{\Phi}}}a\|_{L_x^2}^2 + \|e^{\delta_0{\widetilde{\Phi}}}b\|_{L_x^2}^2 + \|e^{\delta_0{\widetilde{\Phi}}}c\|_{L_x^2}^2 \\
		&\lesssim \|e^{\frac{1}{2}{\mathsf{H}}}(\mathsf{I} - \mathsf{P})g\|_{L_{x, v}^2}^2 + \|e^{\frac{1}{2}{\mathsf{H}}}\mathsf{P}g\|_{L_{x, v}^2}^2 = \|e^{\frac{1}{2}{\mathsf{H}}}g\|_{L_{x, v}^2}^2.
			\end{aligned}
\end{equation*}

	Since $|\mathcal J(t)|\lesssim\mathsf E^2$, choose $C_2$ sufficiently large that
	\begin{equation*}
		\mathcal{F}_N \le \mathcal{F} := \mathcal{F}_N + C_1\mathcal{J} + C_2\|e^{\frac{1}{2}{\mathsf{H}}}g\|_{L_{x, v}^2}^2 \le \mathcal{F}_N + 2C_2\|e^{\frac{1}{2}{\mathsf{H}}}g\|_{L_{x, v}^2}^2.
	\end{equation*}
	Since $\|e^{\lambda\mathsf H}g\|_{L_{x,v}^2}^2\le\mathcal F_N\le2N\|e^{\lambda\mathsf H}g\|_{L_{x,v}^2}^2$, it follows that $\|e^{\lambda\mathsf H}g\|_{L_{x,v}^2}^2\le\mathcal F\lesssim\|e^{\lambda\mathsf H}g\|_{L_{x,v}^2}^2$.
	
	\medskip
	We conclude that for every $\lambda\in(1/2,1)$ and fixed $\delta\in(0,1/3-1/(p+2))$, there exist $C_\lambda>0$ and an energy $\mathcal F_\lambda(t)$ such that $\|e^{\lambda\mathsf H}g\|_{L_{x,v}^2}^2\le\mathcal F_\lambda\le C_\lambda\|e^{\lambda\mathsf H}g\|_{L_{x,v}^2}^2$ and
	\begin{equation*}
		\frac{d}{dt}\mathcal{F}_\lambda + \|e^{(\lambda - \frac{1}{p +2} - \delta){\mathsf{H}}}g\|_{L_{x, v}^2}^2 \le 0.
	\end{equation*}

\medskip\noindent\textit{Step 4: Algebraic decay by interpolation.}
Fix $1/2<\lam_2<\lam_1<1$. Then
	\begin{equation*}
		\frac{d}{dt}\mathcal{F}_{\lambda_1} + \|e^{(\lambda_1 - \frac{1}{p + 2} - \delta){\mathsf{H}}}g\|_{L_{x, v}^2}^2 \le 0, \:\:\: \frac{d}{dt}\mathcal{F}_{\lambda_2} + \|e^{(\lambda_2 - \frac{1}{p + 2} - \delta){\mathsf{H}}}g\|_{L_{x, v}^2}^2 \le 0.
	\end{equation*}
	The first inequality gives $\mathcal F_{\lambda_1}(t)\le\mathcal F_{\lambda_1}(0)$. For the second inequality, H\"older's inequality implies
	\begin{equation*}
		\|e^{\lambda_2{\mathsf{H}}}g\|_{L_{x, v}^2} \le \|e^{\lambda_1{\mathsf{H}}}g\|_{L_{x, v}^2}^{1 - \theta}\|e^{(\lambda_2 - \frac{1}{p + 2} - \delta){\mathsf{H}}}g\|_{L_{x, v}^2}^\theta,
	\end{equation*}
	where 
	\begin{equation*}
		\lambda_2 = (1 - \theta)\lambda_1 + \theta(\lambda_2 - \frac{1}{p + 2} - \delta), \:\:\:
		\theta = \frac{\lambda_1 - \lambda_2}{\lambda_1 - \lambda_2 + \frac{1}{p + 2} + \delta}.
	\end{equation*}
	Since $\|e^{\lambda_1\mathsf H}g\|_{L_{x,v}^2}^2\le\mathcal F_{\lambda_1}\le\mathcal F_{\lambda_1}(0)$,
	\begin{equation*}
		C_{\lambda_2}^{-1}\mathcal{F}_{\lambda_2} \le \|e^{\lambda_2{\mathsf{H}}}g\|_{L_{x, v}^2}^2 \le \mathcal{F}_{\lambda_1}(0)^{1 - \theta}\|e^{(\lambda_2 - \frac{1}{p + 2} - \delta){\mathsf{H}}}g\|_{L_{x, v}^2}^{2\theta},
	\end{equation*}
	then
	\begin{equation*}
		\|e^{(\lambda_2 - \frac{1}{p + 2} - \delta){\mathsf{H}}}g\|_{L_{x, v}^2}^2 \ge C_{\lambda_2}^{-\frac{1}{\theta}}\mathcal{F}_{\lambda_1}(0)^{-\frac{1 - \theta}{\theta}}\mathcal{F}_{\lambda_2}^{\frac{1}{\theta}}.
	\end{equation*}
	Consequently,
	\begin{equation*}
		\frac{d}{dt}\mathcal{F}_{\lambda_2} + C \mathcal{F}_{\lambda_1}(0)^{-\frac{1 - \theta}{\theta}}\mathcal{F}_{\lambda_2}^{\frac{1}{\theta}} \le 0,
	\end{equation*}
	which shows that
	\begin{equation*}
		\begin{aligned}
		&\|\mathcal{M}^{-\lam_2}g(t)\|^2_{L^2_{x,v}}\ls \|e^{\lambda_2{\mathsf{H}}}g(t)\|_{L_{x, v}^2}^2 \ls \mathcal{F}_{\lambda_2} \lesssim (1 + t)^{-\frac{\theta}{1 - \theta}}\mathcal{F}_{\lambda_1}(0) 
		= (1 + t)^{-\frac{(\lambda_1 - \lambda_2)(p + 2)}{1 + \delta(p + 2)}}\mathcal{F}_{\lambda_1}(0)\\
		&\qquad\qquad\qquad\quad\ls (1 + t)^{-\frac{(\lambda_1 - \lambda_2)(p + 2)}{1 + \delta(p + 2)}}\|e^{\lam_1\mathsf{H}}g_0\|^2_{L^2_{x,v}}\ls (1 + t)^{-\frac{(\lambda_1 - \lambda_2)(p + 2)}{1 + \delta(p + 2)}}\|\mathcal{M}^{-\lam_1}g_0\|^2_{L^2_{x,v}},
		\end{aligned}
		\end{equation*}
		where $\mathcal M^{-1}\sim e^{\mathsf H}$. 

\iffalse
The preceding argument was written for the technical range
$0<\delta<1/3-1/(p+2)$. Given an arbitrary prescribed $\delta>0$, choose
$0<\widetilde\delta\le\delta$ in this admissible range and apply the estimate
with $\widetilde\delta$. Since the decay exponent obtained with
$\widetilde\delta$ is no smaller than the one displayed with $\delta$, the
same bound follows for every $\delta>0$ after changing the constant.
\fi

\medskip\noindent\textit{Step 5: Passage from the core to mild solutions.}
Let $g_0$ satisfy the assumptions of the theorem. By the standard core property of the normalized semigroup realization, choose $\widehat g_{0,n}$ in the smooth core such that
\[
\|\mathcal M^{-\lambda_1}(\widehat g_{0,n}-g_0)\|_{L^2}\longrightarrow0.
\]
The five moment functionals in \eqref{Conserveg} are continuous on $\mathcal X_{1/2}(\mathcal M)$: this follows from Cauchy--Schwarz because each transformed collision invariant multiplied by $\mathcal M^{1/2}$ belongs to $L^2(\mathbb R^6)$. We may therefore correct $\widehat g_{0,n}$ by a linear combination of five fixed smooth core functions whose moment matrix is invertible. The resulting data $g_{0,n}$ still converge to $g_0$ in $\mathcal X_{\lambda_1}(\mathcal M)$ and satisfy \eqref{Conserveg} exactly.

Let $g_n(t)$ be the corresponding strong solutions. Steps 1--4 apply to $g_n$ and give, with a constant independent of $n$,
\[
\|\mathcal M^{-\lambda_2}g_n(t)\|_{L^2}
\le C\langle t\rangle^{-\frac{(\lambda_1-\lambda_2)(p+2)}{2[1+\delta(p+2)]}}
\|\mathcal M^{-\lambda_1}g_{0,n}\|_{L^2}.
\]
The core approximation and strong continuity of the semigroup give $g_n\to g$ in $C([0,T];\mathcal X_{1/2}(\mathcal M))$ for every $T>0$, where $g$ is the semigroup solution issued from $g_0$. For each fixed $t\ge0$, the preceding estimate makes $g_n(t)$ bounded in $\mathcal X_{\lambda_2}(\mathcal M)$. Any weakly convergent subsequence in that space has limit $g(t)$, because the embedding $\mathcal X_{\lambda_2}(\mathcal M)\hookrightarrow\mathcal X_{1/2}(\mathcal M)$ is continuous. Weak lower semicontinuity of the norm and convergence of $g_{0,n}$ in $\mathcal X_{\lambda_1}(\mathcal M)$ yield the asserted estimate for $g(t)$. The moment identities also pass to the limit by their continuity on the base space. This proves Theorem \ref{082}.
	\end{proof}

\subsection{Return to the original variables}
Applying the reductions of Section 3 and using the equivalence of the corresponding weighted norms gives Proposition \ref{Conv}.

\subsection{Proof of the general linear stability theorem}
\begin{proof}[Proof of Theorem \ref{linearstabilityp<2}]
	Set $h=\Pi_Mf_0$ and $g=f-h$. Proposition \ref{onetoone2} shows that $h$ is stationary for the full linearized equation, while \eqref{coef2} gives
	\[
	\int_{\mathbb R^6}\chi g(0)\,dx\,dv=0.
	\]
	Hence
	\[
	\partial_tg+v\cdot\nabla_xg-\nabla_x\Phi\cdot\nabla_vg
	=Q(M,g)+Q(g,M),
	\qquad g(0)=f_0-\Pi_Mf_0.
	\]
	The estimate now follows from Proposition \ref{Conv}. To identify the kernel, let $f$ be stationary in $\mathcal X_\lambda(M)$ for some $1/2<\lambda<1$. Then $f-\Pi_Mf$ is also stationary and has zero conserved moments. Applying Proposition \ref{Conv} to this time-independent function gives $f=\Pi_Mf$. Thus the stationary space is exactly the range of $\Pi_M$.
	\end{proof}

%%%%%%%%%%%%%%%%%%%%%%%%%%%%%%%%%%%%%%%%%%%%%%%%%%
\section{Nonlinear obstruction for $1<p<2$}
\iffalse
\subsection{Nonlinear equilibria and non-integrable neutral directions}
Assume $1<p<2$. The nonlinear equilibrium manifold is
\[
\mathscr M_{p<2}
=\left\{m\exp\left[-\alpha\left(\Phi(x)+\frac12|v|^2\right)\right]:m>0,\ \alpha>0\right\},
\]
and hence has a two-dimensional tangent space at each equilibrium. By contrast, Theorem \ref{linearstabilityp<2} shows that the stationary space of the linearized equation is five-dimensional. The three additional directions are the angular-momentum modes $(x_iv_j-x_jv_i)M$.

To see why these directions do not generate nearby nonlinear equilibria, fix $R\in\mathfrak{so}(3)\setminus\{0\}$ and formally set
\[
M_{\varepsilon R}(x,v)
=m\exp\left\{-\alpha\left(\Phi(x)+\frac12|v|^2\right)+\varepsilon Rx\cdot v\right\}.
\]
Integration in $v$ gives
\[
\int_{\mathbb R^3}M_{\varepsilon R}(x,v)\,dv
=m\left(\frac{2\pi}{\alpha}\right)^{3/2}
\exp\left\{-\alpha\Phi(x)+\frac{\varepsilon^2|Rx|^2}{2\alpha}\right\}.
\]
Because $\Phi(x)\sim |x|^p$ with $p<2$, this function is not integrable in $x$ for any $\varepsilon\ne0$, although
\[
\left.\frac{d}{d\varepsilon}M_{\varepsilon R}\right|_{\varepsilon=0}
=(Rx\cdot v)M
\]
is an integrable stationary solution of the linearized equation. Thus the angular modes are neutral directions of the linearized problem that cannot be integrated into the nonlinear equilibrium manifold.
\fi

\subsection{Angular-momentum topologies and quantitative separation}
We first make the functional setting precise. For every signed function $G$ for which the integral is absolutely convergent, set
\[
J(G):=\int_{\mathbb R^6}(x\wedge v)G(x,v)\,dx\,dv
\in\mathbb R^{3\wedge3}.
\]
Here $x\wedge v=xv^\tau-vx^\tau$, $\mathbb R^{3\wedge3}$ is the space of real $3\times3$ skew-symmetric matrices, and $\|\cdot\|_2$ below denotes the Frobenius norm.
Let $X$ be a normed linear space of signed functions such that $\mathscr M_{p<2}\subset X$. We say that $X$ \emph{controls angular momentum} if $J$ is bounded on $X$, and we denote its operator norm by
\[
C_X:=\|J\|_{\mathcal L(X,\mathbb R^{3\wedge3})}.
\]
The corresponding closed zero-angular-momentum subspace and the distance to a subset $\mathcal A\subset X$ are
\[
\ker_XJ:=\{G\in X:J(G)=0\},
\qquad
\operatorname{dist}_X(G,\mathcal A)
:=\inf_{H\in\mathcal A}\|G-H\|_X.
\]
The basic concrete example used below is
\[
X_J:=L^1\!\left(\mathbb R^6,
  (1+\|x\wedge v\|_2)\,dx\,dv\right),
\qquad
\|G\|_{X_J}:=
\int_{\mathbb R^6}(1+\|x\wedge v\|_2)|G|\,dx\,dv.
\]
Indeed, $\|J(G)\|_2\le\int\|x\wedge v\|_2|G|\le\|G\|_{X_J}$.

To state the dynamical consequence, fix a class $\mathfrak S_X$ of global nonnegative solutions of \eqref{Boltzmann} whose trajectories belong to $X$. We call $\mathscr M_{p<2}$ \emph{locally attractive near} $M\in\mathscr M_{p<2}$ relative to $\mathfrak S_X$ if there exists $\eta>0$ such that
\[
F\in\mathfrak S_X,\quad \|F_0-M\|_X<\eta
\quad\Longrightarrow\quad
\operatorname{dist}_X(F(t),\mathscr M_{p<2})\longrightarrow0.
\]
We call it \emph{Lyapunov stable near} $M$ relative to $\mathfrak S_X$ if, for every $\epsilon>0$, there exists $\eta>0$ such that
\[
F\in\mathfrak S_X,\quad \|F_0-M\|_X<\eta
\quad\Longrightarrow\quad
\sup_{t\ge0}\operatorname{dist}_X(F(t),\mathscr M_{p<2})<\epsilon.
\]
Local asymptotic stability near $M$ means that both properties hold. The next result rules out local attraction, provided the solution class contains global trajectories with nonzero angular momentum; it makes no claim about Lyapunov stability.

\begin{proposition}[Quantitative separation from the equilibrium manifold]\label{nonlinear-angular-obstruction}
Let $1<p<2$, let $X$ control angular momentum, and suppose that $F(t)\in X$ is a global solution of \eqref{Boltzmann} whose angular momentum is conserved. If $J(F_0)\ne0$, then $C_X>0$ and, for every $t\ge0$,
\begin{equation}\label{abstract-angular-separation}
\operatorname{dist}_X(F(t),\ker_XJ)
\ge \frac{\|J(F_0)\|_2}{C_X},
\qquad
\operatorname{dist}_X(F(t),\mathscr M_{p<2})
\ge \frac{\|J(F_0)\|_2}{C_X}.
\end{equation}
Consequently, there are no sequences $t_n\to\infty$ and $M_n\in\mathscr M_{p<2}$ such that $\|F(t_n)-M_n\|_X\to0$.

In the concrete space $X_J$, the following sharper moment estimate holds for every $t\ge0$:
\begin{equation}\label{weighted-angular-separation}
\begin{aligned}
\inf_{G\in\ker_{X_J}J}
\int_{\mathbb R^6}\|x\wedge v\|_2|F(t)-G|\,dx\,dv
&\ge \|J(F_0)\|_2,\\
\inf_{M_*\in\mathscr M_{p<2}}
\int_{\mathbb R^6}\|x\wedge v\|_2|F(t)-M_*|\,dx\,dv
&\ge \|J(F_0)\|_2.
\end{aligned}
\end{equation}
Finally, if $\mathfrak S_X$ contains solutions with nonzero angular momentum whose initial data lie arbitrarily close in $X$ to some $M\in\mathscr M_{p<2}$, then $\mathscr M_{p<2}$ is not locally attractive near $M$ and hence is not locally asymptotically stable there relative to $\mathfrak S_X$.
\end{proposition}
\begin{proof}
Every $M_*\in\mathscr M_{p<2}$ is radial in both $x$ and $v$, and therefore $J(M_*)=0$. Thus
\[
\mathscr M_{p<2}\subset\ker_XJ.
\]
For every $G\in\ker_XJ$, conservation and boundedness of $J$ give
\[
\|J(F_0)\|_2
=\|J(F(t))-J(G)\|_2
=\|J(F(t)-G)\|_2
\le C_X\|F(t)-G\|_X.
\]
Taking the infimum over $G\in\ker_XJ$ proves the first inequality in \eqref{abstract-angular-separation}; the inclusion above proves the second. It also follows immediately that convergence to a time-dependent sequence of equilibria is impossible.

For $X_J$, the same argument before applying the full $X_J$-norm yields
\[
\|J(F_0)\|_2
\le\int_{\mathbb R^6}\|x\wedge v\|_2|F(t)-G|\,dx\,dv
\qquad(G\in\ker_{X_J}J).
\]
Taking the two infima gives \eqref{weighted-angular-separation}.

It remains to verify the stability conclusion. Let $F\in\mathfrak S_X$ have $J(F_0)\ne0$. The second estimate in \eqref{abstract-angular-separation} implies
\[
\inf_{t\ge0}\operatorname{dist}_X(F(t),\mathscr M_{p<2})
\ge \frac{\|J(F_0)\|_2}{C_X}>0,
\]
so this trajectory is not attracted to the equilibrium manifold. If such trajectories start arbitrarily close to $M$, local attraction near $M$ is impossible.
\end{proof}

\section{Appendix: an auxiliary inequality}
We record an auxiliary inequality used in the proofs of Lemmas \ref{010} and \ref{EstimatecrB}.

	\begin{lemma}\label{LagrangeIneq}
		For $a_i,b_i,c_i\in\mathbb R$, $1\le i\le n$, one has
		\begin{equation}\label{061}
			\begin{aligned}
			&\Big[\Big(\sum_{i = 1}^na_i^2\Big)\Big(\sum_{i = 1}^nb_i^2\Big) - \Big(\sum_{i = 1}^na_ib_i\Big)^2\Big]\Big[\Big(\sum_{i = 1}^na_i^2\Big)\Big(\sum_{i = 1}^nc_i^2\Big) - \Big(\sum_{i = 1}^na_ic_i\Big)^2\Big] \\
			&\qquad\qquad\qquad\quad\quad\ge \Big[\Big(\sum_{i = 1}^na_i^2\Big)\Big(\sum_{i = 1}^nb_ic_i\Big) - \Big(\sum_{i = 1}^na_ib_i\Big)\Big(\sum_{i = 1}^na_ic_i\Big)\Big]^2.
			\end{aligned}
		\end{equation}
		For smooth functions $f=f(x)$, $g=g(x)$, and $h=h(x)$, one also has
		\begin{equation*}
				\begin{aligned}
			&\Big[\int_{\R^3} f^2dx\int_{\R^3} g^2dx - \Big(\int_{\R^3} fgdx\Big)^2\Big]\Big[\int_{\R^3} f^2dx\int_{\R^3} h^2dx - \Big(\int_{\R^3} fhdx\Big)^2\Big] \\
			&\qquad\qquad\qquad\quad\quad\qquad\quad\ge \Big[\int_{\R^3} f^2dx\int_{\R^3} ghdx - \int_{\R^3} fgdx\int_{\R^3} fhdx\Big]^2.
				\end{aligned}
		\end{equation*}
	\end{lemma}
	\begin{proof}
	 We first introduce Lagrange's identity as
		\begin{equation*}
			\Big(\sum_{i = 1}^na_i^2\Big)\Big(\sum_{i = 1}^nb_i^2\Big) - \Big(\sum_{i = 1}^na_ib_i\Big)^2 = \sum_{1 \le i < j \le n}(a_ib_j - a_jb_i)^2.
		\end{equation*}
		The right-hand side of \eqref{061} satisfies
		\begin{equation*}
			\begin{aligned}
				&\Big[\Big(\sum_{i = 1}^na_i^2\Big)\Big(\sum_{i = 1}^nb_ic_i\Big) - \Big(\sum_{i = 1}^na_ib_i\Big)\Big(\sum_{i = 1}^na_ic_i\Big)\Big]^2\\
				=& \big[\sum_{i = 1}^na_i^2b_ic_i + \sum_{1 \le i < j \le n}(a_i^2b_jc_j + a_j^2b_ic_i) - \sum_{i = 1}^na_i^2b_ic_i - \sum_{1 \le i < j \le n}(a_ib_ia_jc_j - a_jb_ja_ic_i)\big]^2 \\
				=& \big[\sum_{1 \le i < j \le n}(a_ib_j - a_jb_i)(a_ic_j - a_jc_i)\big]^2 \le \sum_{1 \le i < j \le n}(a_ib_j - a_jb_i)^2\sum_{1 \le i < j \le n}(a_ic_j - a_jc_i)^2 \\
				=&\big[(\sum_{i = 1}^na_i^2)(\sum_{i = 1}^nb_i^2) - (\sum_{i = 1}^na_ib_i)^2\big]\big[(\sum_{i = 1}^na_i^2)(\sum_{i = 1}^nc_i^2) - (\sum_{i = 1}^na_ic_i)^2\big].
			\end{aligned}
		\end{equation*}
		The integral inequality follows from the finite-dimensional one by approximation.
\end{proof}

{\bf Conflict of interested.} The authors declare that they have no conflict of interest.

{\bf Data availability statement.} Data sharing not applicable to this article as no datasets were generated or analyzed.

%\cite{carrapatoso_weighted_2022}\cite{bardos_global_2016}

\section*{Acknowledgements}
The research of L.-B. He was supported by the National Natural Science Foundation of China under Grant No.~11771236 and by the New Cornerstone Investigator Program under Grant No.~100001127. J. Ji was supported by the Jiangsu Funding Program for Excellent Postdoctoral Talent and the Basic Research Program of Jiangsu Province under Grant No.~BK20251376.

\bibliographystyle{abbrv}

\bibliography{references}

%\begin{thebibliography}{99}

%\bibitem{KJFSC}
%K. Carrapatoso, J. Dolbeault, F. H\'erau, S. Mischler and C. Mouhot,
%{ Weighted Korn and Poincar\'e-Korn inequalities in the Euclidean space and associated operators,}
%Archive for Rational Mechanics and Analysis, 243(2022), 1565-1596 .

%\end{thebibliography}
\end{document}